\documentclass[DIV=12,numbers=enddot,leqno,bibliography=totoc]{scrartcl}
\usepackage{./q-HodgeStyle_minimal}
\title{\texorpdfstring{$q$}{q}-Witt vectors and $q$-Hodge complexes}
\author{Ferdinand Wagner}

\begin{document}
	\maketitle
	
	\begin{abstract}
		\textbf{Abstract. --- }In this article, we'll introduce a \enquote{$q$-variant} of Witt vectors and de~Rham--Witt complexes. This variant is closely related to the Habiro ring of a number field constructed by Garoufalidis, Scholze, Wheeler, and Zagier \cite{HabiroRingOfNumberField}, to $q$-Hodge cohomology, and to $\THH(-/\ku)$. While most of these connections will only be explored in forthcoming work \cite{qHodge,qWittHabiro}, the goal of this article is to provide the necessary technical foundation.
	\end{abstract}

	\tableofcontents
	\renewcommand{\SectionPrefix}{\textrm{\S}}
	\renewcommand{\SubsectionPrefix}{\textrm{\S}}
	\newpage
	%\addtocounter{section}{-1}
	
	\section{Introduction}\label{sec:Intro}
	In $p$-adic geometry, one often encounters the \emph{ring of \embrace{$p$-typical} Witt vectors} $W(k)$, where $k$ is an $\IF_p$-algebra. This ring comes equipped with two natural endomorphisms: A \emph{Frobenius} $F_p\colon W(k)\rightarrow W(k)$ and a \emph{Verschiebung} $V_p\colon W(k)\rightarrow W(k)$. These satisfy the well-known relations $F_p\circ V_p=p=V_p\circ F_p$.
	
	The $p$-typical Witt vector ring $W(k)$ has a global analogue, given by the ring of \emph{big Witt vectors}. This ring still admits Frobenii $F_p$ and Verschiebungen $V_p$ for all primes $p$, but the commutativity of $F_p$ and $V_p$ is lost. The idea that we'll explore in this paper is that upon introducing an auxiliary variable $q$, one can force $F_p$ and $V_p$ to commute \enquote{up to $q$-twist}, without losing any information.
	\begin{numpar}[$q$-Witt vectors.]\label{par:IntroqWitt}
		Let us now explain this idea in more detail. Fix a positive integer $m$, a commutative ring $R$, and let $\IW_m(R)$ denote the ring of big Witt vectors of $R$ with respect to the \emph{truncation set} $T_m\coloneqq \{\text{divisors of }m\}$. Here the terminology is taken from \cite[\S\href{https://arxiv.org/pdf/1006.3125\#section.1}{1}]{HesselholtBigDeRhamWitt}; we'll review it in \cref{par:BigWitt}. If $d$ is a divisor of $m$, then the rings $\IW_m(R)$ and $\IW_d(R)$ are related via a \emph{Frobenius} $F_{m/d}\colon \IW_m(R)\rightarrow \IW_d(R)$ and a \emph{Verschiebung} $V_{m/d}\colon \IW_d(R)\rightarrow \IW_m(R)$. The Frobenius $F_{m/d}$ is a morphism of rings, whereas the Verschiebung $V_{m/d}$ is only a morphism of abelian groups. These morphisms satisfy
		\begin{equation*}
			F_{m/d}\circ V_{m/d}=m/d\,;
		\end{equation*}
		however, there is no equally nice formula for the composition $V_{m/d}\circ F_{m/d}$. It can be described as multiplication by the element $V_{m/d}(1)$; still, this element is not very explicit. But we can make it explicit as follows. Let $(\qIW_m(R))_{m\in\IN}$ be the initial system of $\IZ[q]$-algebras equipped with the following structure:
		\begin{alphanumerate}
			\item For all $m\in \IN$, a $\IZ[q]$-algebra map $\IW_{m}(R)[q]/(q^{m}-1)\rightarrow \qIW_m(R)$.
			\item For all divisors $d\mid m$, a $\IZ[q]$-algebra morphism $F_{m/d}\colon W_m\rightarrow W_d$ and a $\IZ[q]$-module morphism $V_{m/d}\colon W_d\rightarrow W_m$. These must be compatible with the usual Frobenii and Verschiebungen on ordinary Witt vectors and satisfy
			\begin{equation*}
				F_{m/d}\circ V_{m/d}=m/d\quad\text{and}\quad V_{m/d}\circ 	F_{m/d}=[m/d]_{q^d}\coloneqq \frac{q^m-1}{q^d-1}\,.
			\end{equation*}
		\end{alphanumerate}
		It'll be shown in \cref{lem:BigqWittUniversal} that such an initial system does indeed exist and that $\qIW_m(R)$ is given by an explicit quotient of $\IW_m(R)[q]/(q^m-1)$. We call $\qIW_m(R)$ the ring of \emph{$m$-truncated big $q$-Witt vectors over $R$}.
	\end{numpar}
	\begin{rem}
		Despite the name, $\qIW_m(R)$ is not a $q$-deformation of $\IW_m(R)$. Indeed, in $\qIW_m(R)/(q-1)$ the condition $V_{m/d}\circ F_{m/d}=m/d$ is enforced. In fact, if $(\ov\IW_m(R))_{m\in\IN}$ denotes the universal quotient of $(\IW_m(R))_{m\in\IN}$ such that Frobenius and Verschiebung commute, then it's straightforward to check $\qIW_m(R)/(q-1)\cong \ov\IW_m(R)$, so $\qIW_m(R)$ is a $q$-deformation of $\ov\IW_m(R)$ instead.
		
		There's also a clash of terminology with a construction of Andre Chatzistamatiou. In unpublished work, he introduces \emph{$q$-Witt vectors} and \emph{$q$-de Rham--Witt complexes} of $\Lambda$-rings and uses them to obtain a partial result towards coordinate-independence of the $q$-de Rham complex (see \cref{par:qDeRhamComplex} and \cref{thm:qDeRhamGlobalIntro} below). In particular, he was able to construct a homotopy equivalence $\qOmega_{\IZ[T]/\IZ,\square_1}^*\simeq \qOmega_{\IZ[T]/\IZ,\square_2}^*$, where $\square_1$ is the identical framing and $\square_2\colon \IZ[T]\rightarrow \IZ[T]$ is the framing that maps $T\mapsto T-1$ .
		
		In constrast to our constructions, Chatzistamatiou's $q$-Witt vectors and $q$-de Rham--Witt complexes are honest $q$-deformations of their classical counterparts. The author doesn't know whether there is a connection between the constructions in this paper and Chatzistamatiou's.
	\end{rem}

	$q$-Witt vectors have a number of nice properties. First of all, it can be shown that the canonical map $\IW_m(R)\rightarrow \qIW_m(R)$ is always injective (\cref{prop:WittToqWittInjective}), so we really don't lose any information by enforcing $V_{m/d}\circ F_{m/d}=[m/d]_{q^d}$. Second, formulas often become easier than for ordinary Witt vectors. For example, $\qIW_m(\IZ)\cong \IZ[q]/(q^m-1)$ couldn't be simpler (\cref{cor:qWittOfPerfectLambdaRing}). Third, it turns out that most constructions with and properties of ordinary Witt vectors have analogues for $q$-Witt vectors, as we'll see throughout \crefrange{sec:qWitt}{sec:qDRW}.
	\begin{numpar}[A theory without restrictions.]\label{par:IntroNoRestrictions}
		The only real exception is that the restriction maps for ordinary Witt vectors do not extend to maps $\operatorname{Res}_{m/d}\colon \qIW_m(R)\rightarrow \qIW_d(R)$. In particular, we can't define a big $q$-Witt ring $\qIW(R)\coloneqq \limit_{m\in\IN,\operatorname{Res}}\qIW_m(R)$. It is then perhaps a little surprising that the classical theory of de Rham--Witt complexes, as developed by Illusie \cite{Illusie} (building on earlier work of Bloch, Deligne, and Lubkin) for $\IF_p$-algebras and by Langer--Zink \cite{LangerZink} in an arbitrary relative setting, should have an analogue for $q$-Witt vectors. Nevertheless, it works, as we'll demonstrate in \cref{sec:qDRW}. What this really shows is that the restrictions weren't actually necessary to set up the classical theory: You can take the universal property from \cite{LangerZink} and delete all restrictions from it---this will still give you the same truncated relative de Rham--Witt complexes. The restrictions were only used a posteriori to combine all the truncated complexes into one single complex by taking the limit.
		
		Nevertheless, the lack of restrictions can be annoying. But there seems to be at least some use in considering the limit $\limit_{m\in \IN,\, F}\qIW_m(R)$ along the Frobenius maps; see \cite{qWittHabiro}.
	\end{numpar}
	We've made a point why the study of $q$-Witt vectors and $q$-de Rham--Witt complexes can be of independent interest, but the real reason we're interested in them is that they do appear in nature.%
	\footnote{In fact, the definition was guessed from a computation of $\H^0(\qHodge_{R,\square}^*/(q^m-1))$; see \cref{thm:qDeRhamWittGlobalIntro} below.} For one, they can be used to construct the \emph{generalised Habiro rings} of Garoufalidis and Zagier \cite{HabiroRingOfNumberField}. We won't discuss this here (although the name \emph{Habiro ring} will appear again and be defined below), but refer to the forthcoming paper \cite{qWittHabiro} instead. What we will discuss is the connection between $q$-de Rham--Witt complexes and \emph{$q$-Hodge complexes}. Before we introduce the latter, let's briefly review the \emph{$q$-de Rham complex}.
	
	\begin{numpar}[$q$-de Rham complexes.]\label{par:qDeRhamComplex}
		Jackson \cite{Jackson} defined the \emph{$q$-derivative} of a function $f(T)$ via the formula
		\begin{equation*}
			\q\partial f(T)\coloneqq \frac{f(qT)-f(T)}{qT-T}\,.
		\end{equation*}
		For example, if $f(T)=T^m$ for some integer $m\geqslant 0$, then $\q\partial f(T)=[m]_qT^{m-1}$, where $[m]_q=1+q+\dotsb+q^{m-1}$ denotes Gauß's $q$-analogue of $m$. Given some base ring $A$, for a polynomial ring in several variables $A[T_1,\dotsc,T_n]$, one can consider \emph{partial $q$-derivatives} $\q\partial_i$ as well as a \emph{$q$-gradient} $\q\nabla\coloneqq \sum_{i=1}^d\q\partial_i\d T_i\colon A[T_1,\dotsc,T_n,q]\rightarrow \Omega_{A[T_1,\dotsc,T_n]/A}^1[q]$; furthermore, these can be organised into a \emph{$q$-de Rham complex}. This was first done by Aomoto \cite{AomotoI}.
		
		Unfortunately, the $q$-derivative does not interact well with coordinate transformations and so there's no way to make Aomoto's $q$-de Rham complex independent of the choice of coordinates. An insight how to fix this came from Scholze \cite{Toulouse}: First, he observed that after completion at $(q-1)$, the $q$-de Rham complex can be defined in more general situations. A \emph{framed smooth $A$-algebra} is a pair $(R,\square)$, where $R$ is smooth over $A$ and $\square\colon A[T_1,\dotsc,T_n]\rightarrow R$ is an étale map from a polynomial ring; we'll often call $\square$ an \emph{étale framing} and think of it as a choice of coordinates on $\Spec R$. Using the infinitesimal lifting properties of étale morphisms, one can show that the $q$-gradient for $A[T_1,\dotsc,T_n]$ extends to a map $\q\nabla\colon R\qpower\rightarrow \Omega_{R/A}^1\qpower$.%
		\footnote{Here it's crucial that we complete at $(q-1)$ or the lifting wouldn't work.}
		The precise construction will be recalled in \cref{par:qDeRhamqHodge}. One can then form the \emph{$q$-de Rham complex of $(R,\square)$}
		\begin{equation*}
			\qOmega_{R/A,\square}^*\coloneqq\left(R\qpower \xrightarrow{\q\nabla}\Omega_{R/A}^1\qpower\xrightarrow{\q\nabla} \dotsb \xrightarrow{\q\nabla} \Omega^n_{R/A}\qpower \right)\,.
		\end{equation*}
		One immediately checks $\qOmega_{R/A,\square}^*/(q-1)\cong \Omega_{R/A}^*$, so we get a $q$-deformation of the de Rham complex. As a complex, $\qOmega_{R/A,\square}^*$ suffers from the same coordinate-dependence as before. However, Scholze observed that if $A$ is a $\Lambda$-ring, then $\qOmega_{R/A,\square}^*$ is coordinate-independent as an object in the derived category $D(A\qpower)$! More precisely, Bhatt and Scholze were able to show the following theorem:
	\end{numpar}
	\begin{thm}[see {\cite[\S\href{https://arxiv.org/pdf/1905.08229v4\#section.16}{16}]{Prismatic}} for the essential case]\label{thm:qDeRhamGlobalIntro}
		If $A$ is equipped with a $\Lambda$-structure and $\IZ$-torsion free, then there exists a functor
		\begin{equation*}
			\qOmega_{-/A}\colon \cat{Sm}_A\longrightarrow \CAlg\left(\widehat{\Dd}_{(q-1)}\bigl(A\qpower\bigr)\right)
		\end{equation*}
		from the category of smooth $A$-algebras into the $\infty$-category of $(q-1)$-complete $\IE_\infty$-algebras over $A\qpower$, such that $\qOmega_{-/A}/(q-1)\simeq \Omega_{-/A}$ agrees with the de Rham complex functor and for every framed smooth $\IZ$-algebra $(R,\square)$, the underlying object of $\qOmega_{R/A}$ in the derived $\infty$-category of $A\qpower$ can be represented by the complex $\qOmega_{R/A,\square}^*$.
	\end{thm}
	\begin{numpar}[$q$-Hodge complexes.]\label{par:qHodgeComplex}
		Given a framed smooth $A$-algebra $(R,\square)$ as above, one can also form the \emph{$q$-Hodge complex}
		\begin{equation*}
			\qHodge_{R/A,\square}^*\coloneqq\left(R\qpower \xrightarrow{(q-1)\q\nabla}\Omega_{R/A}^1\qpower\xrightarrow{(q-1)\q\nabla} \dotsb \xrightarrow{(q-1)\q\nabla} \Omega^n_{R/A}\qpower \right)
		\end{equation*}
		by multiplying all differentials in $\qOmega_{R/A,\square}^*$ with $(q-1)$. It's not immediately obvious why that would be an interesting construction---or even a sensible one---so let's give some motivation why one should look at the $q$-Hodge complex.
		
		The $q$-Hodge complex was first introduced by Pridham \cite{Pridham}%
		\footnote{While Pridham used the notation \enquote{$\widehat{\operatorname{qDR}}$}, we've opted for the perhaps more descriptive \enquote{$\qHodge$}.}
		who used it to obtain a partial result towards \cref{thm:qDeRhamGlobalIntro}. Many results in Pridham's paper are proved for the $q$-Hodge complex first and then deduced for the $q$-de Rham complex via $\qOmega_{R/A,\square}^*\cong \eta_{(q-1)}\qHodge_{R/A,\square}^*$, where $\eta_{(q-1)}$ denotes the Berthelot--Ogus décalage functor (see \cite{BerthelotOgus} or \cite[\stackstag{0F7N}]{Stacks}). This is a first hint that $\qHodge_{R,\square}^*$ might be a more fundamental object. A second hint comes from Waßmuth's paper \cite{Nils}: He introduced a version of the prismatic site in characteristic~$0$ and showed that the cohomology of that site can be computed by a similar complex as above. A third piece of motivation is the following question:
		\begin{alphanumerate}\itshape
			\item[\boxtimes] Can the $q$-de Rham complex, or some modification of it, be descended along $\Hh\rightarrow \IZ\qpower$? Here $\Hh$ denotes the \emph{Habiro ring}\label{qst:Habiro}
			\begin{equation*}
				\Hh\coloneqq \limit_{m\in\IN}\IZ[q]_{(q^m-1)}^\complete\,.
			\end{equation*}
		\end{alphanumerate}
		Question~\cref{qst:Habiro} was raised by Peter Scholze in the hope that a cohomology theory with values in $\Hh$-modules could explain the results of \cite{HabiroRingOfNumberField}; in particular, the mysterious regulator map from $K$-theory to line bundles over extensions of $\Hh$ should arise via a realisation map from motivic cohomology to this hypothetical cohomology theory. Question~\cref{qst:Habiro} is also very natural in view of the general principle that whenever one has a deformation at $q=1$, one should evaluate it at other roots of unity as well. By design, $\Hh$ consists precisely of those power series in $\IZ\qpower$ that can be evaluated at arbitrary roots of unity.%
		\footnote{In fact, $\Hh$ can be viewed as as the ring of those power series in $\IZ\qpower$ that can also be Taylor-expanded around each root of unity.}
		
		Now the $q$-Hodge complex is a more natural candidate to descend to the Habiro ring than the $q$-de Rham complex itself. One intuitive reason goes as follows: As we've seen, the $q$-derivative satisfies $\q\partial(T^m)=[m]_qT^{m-1}$. But a theory that descends to the Habiro ring should \enquote{treat all roots of unity equally} and thus send $T^m\mapsto(q^m-1)T^{m-1}$ instead. This leads immediately to the definition of $\qHodge_{R/A,\square}^*$. A more mathematical reason to expect such a descent for the $q$-Hodge complex is \cref{thm:qDeRhamWittqHodge} below, which we'll restate here in slightly less precise and less general form:
	\end{numpar}
	\begin{thm}[see \cref{thm:qDeRhamWittqHodge}]\label{thm:qDeRhamWittGlobalIntro}
		For all framed smooth $\IZ$-algebras $(R,\square)$ and all $m\in\IN$, there is an isomorphism of commutative differential-graded $\IZ[q]$-algebras
		\begin{equation*}
			\bigl(\qIW_{m}\Omega_{R/\IZ}^*\bigr)_{(q-1)}^\complete\overset{\cong}{\longrightarrow} \H^*\bigl(\qHodge_{R/\IZ,\square}^*/(q^{m}-1)\bigr)\,.
		\end{equation*}
		Here $\qIW_m\Omega_{R/\IZ}^*$ denotes the $q$-de Rham--Witt complex from \cref{def:qDRW}, $(-)_{(q-1)}^\complete$ refers to degree-wise $(q-1)$-completion, and we turn the cohomology $\H^*(\qHodge_{R/\IZ,\square}^*/(q^{m}-1))$ into a commutative differential-graded $\IZ[q]$-algebra via the Bockstein differential.
	\end{thm}
	\begin{rem}
		The fact that $\H^*(\qHodge_{R/\IZ,\square}^*/(q^m-1))$ is canonically the $(q-1)$-completion of something else is precisely what we would expect to see if $\qHodge_{R/\IZ,\square}^*$ were really the $(q-1)$-completion of an object over the Habiro ring!
		
		\cref{thm:qDeRhamWittGlobalIntro} also nicely illustrates why the $q$-Hodge complex is more likely to descend to the Habiro ring than the $q$-de Rham complex. For the $q$-de Rham complex, there's a similar isomorphism $\H^*(\qOmega_{R/\IZ,\square}^*/\Phi_p(q))\cong (\Omega_{R/\IZ}^*\otimes_\IZ\IZ[\zeta_p])_p^\complete$ for all primes $p$, see \cite[Proposition~\href{https://arxiv.org/pdf/1606.01796\#theorem.3.4}{3.4}]{Toulouse}. More generally, Molokov \cite{Molokov} relates $\H^*(\qOmega_{R/\IZ,\square}^*/[p^\alpha]_q)$ to $\IW_{p^{\alpha-1}}\Omega_{R/\IZ}^*$ for all $\alpha\geqslant 1$. But note the shift in the index! This shift prevents us from, say, relating $\H^*(\qOmega_{R/\IZ,\square}^*/[m]_q)$ to $\IW_m\Omega_{R/\IZ}^*$, and so it's entirely unclear whether these $p$-typical results for varying~$p$ can be combined into a global result like \cref{thm:qDeRhamWittGlobalIntro}. 
	\end{rem}
	\cref{thm:qDeRhamWittGlobalIntro} also looks promising regarding the question whether $\qHodge_{R/\IZ,\square}^*$ can be made coordinate-independent (at least in the derived category). But something strange goes wrong:
	\begin{thm}[see \cref{thm:qHodgeNotFunctorial}]\label{thm:qHodgeNotFunctorialIntro}
		There is no functor $\qHodge_{-/\IZ}\colon \cat{Sm}_\IZ\rightarrow \widehat{\Dd}_{(q-1)}(\IZ\qpower)$ that also makes the identifications from \cref{thm:qDeRhamWittGlobalIntro} functorial.
	\end{thm}
	\cref{thm:qHodgeNotFunctorialIntro} is a very unwelcome surprise. It doesn't rule out that the construction $\qHodge_{R/\IZ,\square}^*$ can somehow be made functorial, but we consider this unlikely.\footnote{And it would be hard to say anything about such a functor, since we can't access its cohomology via \cref{thm:qDeRhamWittGlobalIntro}.} In forthcoming work \cite{qWittHabiro}, we'll explain a partial fix, showing that a functorial \emph{derived $q$-Hodge complex} exists on a certain full subcategory of all commutative rings and satisfies a derived version of \cref{thm:qDeRhamWittGlobalIntro}. Furthermore, we'll relate this functor to $\THH(-/\ku)$ and show that it does descend to the Habiro ring, thus giving at least a partial affirmative answer to question~\cref{qst:Habiro}.
	\begin{numpar}[Notation and conventions.]\label{par:Notation}
		As usual, we'll write $[m]_q=1+q+\dotsb+q^{m-1}$ for the Gaußian $q$-analogue of an integer $m\geqslant 0$. More generally, if $d$ is any positive divisor of $m$, we'll use the notation
		\begin{equation*}
			[m/d]_{q^{d}}\coloneqq 1+q^{d}+(q^d)^2\dotsb+(q^{d})^{m/d-1}=\frac{q^{m}-1}{q^{d}-1}\,.
		\end{equation*}
		We also let $\Phi_m(q)$ denote the $m$\textsuperscript{th} cyclotomic polynomial. So $[p]_q=\Phi_p(q)$ and we'll sometimes switch back and forth between these two notations.
		
		%Beware that the notation between these two papers is not consistent: The ring that was denoted $\qIW_m(R)$ in \cite[Definition~\href{https://guests.mpim-bonn.mpg.de/ferdinand/q-deRham.pdf\#theorem.5.3}{5.3}]{MasterThesis} would now be denoted $\qIW_m(R)_{(q-1)}^\complete$, at least as long $\widehat{R}_p$ is static for all prime factors $p\mid m$ (so that \cref{cor:qWittCompletion} is applicable); indeed, this follows from a simple comparison of universal properties. In light of the upcoming applications \cite{qWittHabiro,qHodge}, it seemed the right thing to change the notation, despite the confusion this may cause.
		
		Since we're mostly working with cochain complexes, we'll use cohomological indexing. We'll also use some $\infty$-categoric language. If $R$ is a ring, the derived $\infty$-category of $R$ will be denoted $\Dd(R)$. If $M^*$ is a cochain complex, then its image in $\Dd(R)$ will usually be denoted $M$. We'll usually say that a sequence $K\rightarrow L\rightarrow M$ in $\Dd(R)$ is a \emph{fibre/cofibre sequence} instead of writing that $K\rightarrow L\rightarrow M\rightarrow K[1]$ is a distinguished triangle in the ordinary derived category $D(R)$. Following Clausen--Scholze, we'll say that an object $M\in\Dd(R)$ is \emph{static} (\enquote{\emph{un-animated}}) if $M$ is concentrated in degree $0$. Furthermore, we often use the derived quotient notation: If $f\in R$ and $M\in \Dd(R)$,
		\begin{equation*}
			M/^\L f\coloneqq \cofib\left(f\colon M\rightarrow M\right)
		\end{equation*}
		denotes the cofibre taken in $\Dd(R)$, or equivalently the cone in $D(R)$, of the multiplication map $f\colon M\rightarrow M$. For multiple elements $f_1,\dotsc,f_r\in R$, we let
		\begin{equation*}
			M/^\L(f_1,\dotsc,f_r)\coloneqq \bigl(\dotso(M/^\L f_1)/^\L\dotso\bigr)/^\L f_r\,.
		\end{equation*}
		In the case where $M$ is static, so that it can be regarded as an $R$-module, the object $M/^\L(f_1,\dotsc,f_r)\in \Dd(R)$ has an explicit representative given by the homological Koszul complex $\Kos_*(M,(f_1,\dotsc,f_r))$, which, according to our indexing conventions, we regard as a cochain complex in nonpositive degrees.

		Finally, the notion of \emph{derived $I$-completeness} for finitely generated ideals $I\subseteq R$ will be ubiquitous throughout the text. If $I=(f)$ is principal, we say that $M\in\Dd(R)$ is \emph{derived $f$-complete} if $M\simeq \limit_{n\geqslant 1}M/^\L f^n$, where the limit is taken in the derived $\infty$-category (so it corresponds to a derived limit in the ordinary derived category). In general, $M$ is called \emph{derived $I$-complete} if it is derived $f$-complete for all $f\in I$, or equivalently, for all $f$ in a generating set of $I$, see \cite[\stackstag{091Q}]{Stacks}. We'll use the following (abuse of) notation: If $M\in \Dd(R)$, we denote its derived $I$-completion by $\widehat{M}_I$ (or $(-)_I^\complete$ for larger expressions). If, instead, $M^*$ is a cochain complex, then $\widehat{M}_I^*$ denotes its degree-wise underived $I$-completion. However, whenever we use the latter notation in this paper, it will always be true that $\widehat{M}_I^*$ represents the derived $I$-completion of $M$ (which we'll usually have to justify), so the notation will never be inconsistent! We also denote by $\widehat{\Dd}_I(R)$ the full sub-$\infty$-category of $\Dd(R)$ spanned by the derived $I$-complete objects.
		
		A complex $M\in \Dd(R)$ is called \emph{$I$-completely flat} if $M\lotimes_RR/I$ is discrete and flat over $R/I$. A ring morphism $R\rightarrow S$ is called \emph{$I$-completely smooth} if $S$ is derived $I$-complete, $I$-completely flat over $R$, and $S\lotimes_RR/I$ is smooth over $R/I$. In the same way, the terms \emph{$I$-completely étale} and \emph{$I$-completely ind-smooth/étale} are defined. It can be shown that $S$ is $I$-completely smooth/étale over $R$ if and only if it is the derived $I$-completion of a smooth/étale $R$-algebra, see \cite[footnote~\href{https://arxiv.org/pdf/1905.08229v4\#Hfootnote.6}{6} on page~\href{https://arxiv.org/pdf/1905.08229v4\#page=11}{11}]{Prismatic}.
	\end{numpar}
	\begin{numpar}[Leitfaden of this paper.]
		The main ideas in this paper are already contained in the author's master thesis \cite{MasterThesis}, but here we develop the theory of $q$-Witt vectors and $q$-de Rham--Witt complexes in a much more systematic way, in more generality, and most importantly, without the $(q-1)$-completeness assumption. This results in quite some additional work, but also in a much simpler proof of \cref{thm:qDeRhamWittGlobalIntro}, and the additional generality will be needed in \cite{qWittHabiro}.
		
		In \cref{sec:qWitt}, we'll introduce $q$-Witt vectors and prove many technical results about them that will be needed later on. If you're mainly interested in the application to $q$-Hodge complexes, you may want read up to \cref{prop:qWittKoszulExactSequence} and then skip the rest of \cref{subsec:qWitt} as well as \cref{subsec:WittToqWitt}.  In \cref{sec:qDRW}, we'll introduce $q$-de Rham--Witt complexes. The most work in that section goes into proving that $q$-de Rham--Witt complexes carry a natural choice of Frobenius operators. Again, this is technical, and if you're willing to take it on faith, you can skip \cref{subsec:ConstructionOfFrobenii}. In \cref{sec:qDeRhamWittSmooth}, we'll study $q$-de Rham--Witt complexes for smooth $\IZ$-algebras. This includes a proof of \cref{thm:qDeRhamWittGlobalIntro}, but we'll also show that they are degree-wise $p$-torsion free. Finally, in \cref{sec:Functoriality} we'll give a proof of \cref{thm:qHodgeNotFunctorialIntro}.
	\end{numpar}
	\begin{numpar}[Acknowledgements.]
		Due to the unsatisfying nature of \cref{thm:qHodgeNotFunctorialIntro}, I've long hesitated to turn my master thesis into a paper. With at least a partial fix in sight \cite{qHodge,qWittHabiro}, I've now finally decided to put these ideas forward. I'd like to thank my advisor Peter Scholze for his support throughout this project. I'd also like to thank Johannes Anschütz and Quentin Gazda for their interest in my work and their encouragement to finally turn this work into a preprint, as well as Bora Yalkinoglu for helpful comments on an earlier version.
		
		This work was carried out while I was a master/Ph.D.\ student at the University/Max Planck Institute for Mathematics in Bonn and I'd like to thank these institutions for their hospitality. I was supported by DFG through Peter Scholze's Leibniz-Preis.
	\end{numpar}

	\newpage

	\section{\texorpdfstring{$q$}{q}-Witt vectors}\label{sec:qWitt}
	In this section we'll introduce a functor which associates to any ring $R$ a system of rings $(\qIW_m(R))_{m\in\IN}$, called the \emph{truncated $q$-Witt vectors of $R$}. After a brief recollection of some facts about cyclotomic polynomials in \cref{subsec:Cyclotomic}, we'll give the definition of $\qIW_m(-)$ and study some basic properties in \cref{subsec:qWitt}. In \cref{subsec:qWittLambda} and \cref{subsec:qWittEtale} we'll study the behaviour of $\qIW_m(-)$ on $\Lambda$-rings and under étale ring morphisms.
	\subsection{Some technical preliminaries}\label{subsec:Cyclotomic}
	We record some elementary facts about cyclotomic polynomials that will be used throughout the text.
	\begin{lem}\label{lem:CyclotomicPolynomialsCoprime}
		Let $m$ and $n$ be positive integers and let $R=\IZ[q]/(\Phi_m(q),\Phi_n(q))$. Let $d=\gcd(m,n)$. If $p$ is a prime factor of $m/d$ or $n/d$, then $p=0$ in $R$. In particular, the ring $R$ vanishes unless $m/n=p^\alpha$ for some prime $p$ and some $\alpha\in\IZ$. In the latter case, $R\cong \IF_p[q]/\Phi_{\min\{m,n\}}(q)$.
	\end{lem}
	\begin{proof}
		Clearly $q^m=q^n=1$ in $R$, hence also $q^d=1$ in $R$. If $p$ divides $m/d$, then this implies $q^{m/p}=1$. But also
		\begin{equation*}
			[p]_{q^{m/p}}=1+q^{m/p}+(q^{m/p})^2+\dotsb+(q^{m/p})^{p-1}=0
		\end{equation*}
		in $R$, because $\Phi_m(q)$ divides the left-hand side. Thus $p=0$, as claimed. The case where $p$ divides $n/d$ is analogous. This immediately implies the second assertion. For the third one, assume $\alpha\geqslant 0$ without restriction and use that
		\begin{equation*}
			\Phi_n(q)\equiv\Phi_{mp^\alpha}(q)\equiv\begin{cases*}
				\Phi_m(q)^{p^\alpha} & if $p\mid m$\\
				\Phi_m(q)^{(p-1)p^\alpha} & if $p\nmid m$
			\end{cases*}\mod p
		\end{equation*}
		holds in this case.
	\end{proof}
	\begin{lem}\label{lem:IdealGeneratedByPhi}
		Let $m$ be a positive integer. Then the following ideals of $\IZ[q]$ are equal:
		\begin{equation*}
			\bigl([p]_{q^{m/p}}\ \big|\  p\text{ prime factor of }m\bigr)=\bigl(\Phi_m(q)\bigr)\,.
		\end{equation*}
	\end{lem}
	\begin{proof}
		The inclusion \enquote{$\subseteq$} is clear, so it suffices to show that $[p]_{q^{m/p}}/\Phi_m(q)$ generate the unit ideal in $\IZ[q]$. If $m$ has only one prime factor, this is trivial because then $[p]_{q^{m/p}}/\Phi_m(q)=1$. So assume $m$ has at least two prime factors. For any prime factor $p$ of $m$, let $I_p$ be the set of divisors $d\mid m$ such that $d\neq m$ and $v_p(d)=v_p(m)$. Then
		\begin{equation*}
			\frac{[p]_{q^{p^m}}}{\Phi_m(q)}=\prod_{d\in I_p}\Phi_d(q)\,.
		\end{equation*}
		We wish to apply \cref{lem:UnitIdeal} below. To verify the condition, we have to check that for any choice of elements $(d_p\in I_p)_{p\text{ prime factor of }m}$, the ideal $(\Phi_{d_p}(q)\ |\ p\text{ prime factor of }m)$ is the unit ideal in $\IZ[q]$. In fact, we claim that there must be prime factor $p\neq \ell$ of $m$ such that $d_p\nmid d_\ell$ and $d_\ell\nmid d_p$, so that already $\Phi_{d_p}(q)$ and $\Phi_{d_\ell}(q)$ generate the unit ideal by \cref{lem:CyclotomicPolynomialsCoprime}. Indeed, if no such $p$ and $\ell$ exist, then the set $\left\{d_p\ \middle|\ p\text{ prime factor of }m\right\}$ would be totally ordered with respect to division, but then the maximal $d_p$ would have $v_\ell(d_p)\geqslant v_\ell(d_\ell)=v_\ell(m)$ for all prime factors $\ell\mid m$, forcing $d_p=m$, in contradiction to our assumptions. This shows that \cref{lem:UnitIdeal} can be applied and we're done.
	\end{proof}
	\begin{lem}\label{lem:UnitIdeal}
		Let $(I_j)_{j\in J}$ be finite sets indexed by another finite set $J$. Let $((x_{i_j})_{i_j\in I_j})_{j\in J}$ be a $J$-tuple of $I_j$-tuples of elements of a ring $R$. Suppose that for any choice of indices $(i_j\in I_j)_{j\in J}$ the ideal $(x_{i_j}\ |\ j\in J)$ is the unit ideal in $R$. Then $\bigl(\prod_{i_j\in I_j}x_{i_j}\ \big|\ j\in J\bigr)$ is the unit ideal as well.
	\end{lem}
	\begin{proof}
		We do induction on $\sum_{j\in J}\abs{I_j}$. If $\abs{I_j}=1$ for all $j\in J$, the assertion is trivial. So choose elements $s,t\in I_k$, $x\neq y$, for some index $k\in J$ such that $\abs{I_k}\geqslant 2$. We claim that the condition still holds if we remove $x_s$ and $x_t$ from the tuple $(x_{i_k})_{i_k\in I_k}$ and add $x_sx_t$ instead. Indeed, the only thing we have to check is the following: For any choice of indices $(i_j\in I_j)_{j\in J,\, j\neq k}$, the ideal $(x_sx_t,x_{i_j}\ |\ j\in J,\,j\neq k)$ is the unit ideal in $R$. But this follows from
		\begin{equation*}
			\bigl(x_sx_t,x_{i_j}\ \big|\ j\in J,\,j\neq k\bigr)\supseteq \bigl(x_s,x_{i_j}\ \big|\ j\in J,\,j\neq k\bigr)\bigl(x_t,x_{i_j}\ \big|\ j\in J,\,j\neq k\bigr)
		\end{equation*}
		and the fact that the right-hand side is $R$ by assumption.
	\end{proof}
	Furthermore, we'll frequently use the following technical lemma.
	\begin{lem}\label{lem:DerivedBeauvilleLaszlo}
		Let $R$ be a ring. Let $I_1,\dotsc,I_r$ be finitely many finitely generated ideals of $R$ and let $f_1,\dotsc,f_s$ be a finitely many elements of $A$ such that on the level of underlying sets we have $\Spec R=\bigcup_{j=1}^r\Spec R/I_j \cup\bigcup_{k=1}^s\Spec R[1/f_k]$. Then the functors
		\begin{equation*}
			(-)_{I_j}^\complete\colon \Dd(R)\longrightarrow \Dd(R)\quad\text{and}\quad (-)\left[\localise{f_k}\right]\colon \Dd(R)\rightarrow \Dd(R)
		\end{equation*}
		are jointly conservative for $j=1,\dotsc,r$, $k=1,\dotsc,s$.
	\end{lem}
	\begin{proof}
		We do induction on $r$. The case $r=0$ follows from the case $r=1$ by choosing $I_1$ to be the unit ideal. So let's first consider $r=1$. In this case we're given a finitely generated ideal $I\coloneqq I_1$ and elements $f_1,\dotsc,f_s$ such that $\Spec R=\Spec R/I\cup\bigcup_{k=1}^s\Spec R[1/f_k]$. Then $I$ is contained in the radical of the ideal $(f_1,\dotsc,f_s)$, hence derived $(f_1,\dotsc,f_s)$-adic completion factors through derived $I$-adic completion. So we may assume that $I$ is generated by the $f_k$.
		
		Now let $M\in \Dd(R)$. It suffices to show that $\widehat{M}_I\simeq 0$ and $M[1/f_k]\simeq 0$ for all $k=1,\dotsc,s$ together imply $M\simeq 0$. Write $U=\Spec R\smallsetminus V(I)$ and let
		\begin{equation*}
			\R\Gamma(U,\Oo_U)\simeq \left(\prod_{k}R\left[\localise{f_{k}}\right]\rightarrow \prod_{k<\ell}R\left[\localise{f_{k}f_{\ell}}\right]\rightarrow \dotso\rightarrow R\left[\localise{f_1\dotsm f_s}\right]\right)
		\end{equation*}
		denote the derived global sections of $U$, which can be computed by an alternating \v Cech complex as indicated.
		Let $M\in \Dd(R)$. Then $\widehat{M}_I\simeq \RHom_R(\cofib(R\rightarrow \R\Gamma(U,\Oo_U)),M)$, see e.g.\ \cite[\stackstag{091V}]{Stacks}. Hence $\widehat{M}_I\simeq 0$ implies that $R\rightarrow \R\Gamma(U,\Oo_U)$ induces an equivalence
		\begin{equation*}
			\RHom_R\big(\R\Gamma(U,\Oo_U),M\big)\overset{\simeq}{\longrightarrow} \RHom_R(R,M)\simeq M\,.
		\end{equation*}
		Applying $\RHom_R(M,-)$ shows $\RHom_R(M\lotimes_R\R\Gamma(U,\Oo_U),M)\simeq\RHom_R(M,M)$. In particular, the identity on $M$ factors through $M\lotimes_R\R\Gamma(U,\Oo_U)$, which means that $M$ must be a direct summand of $M\lotimes_R\R\Gamma(U,\Oo_U)$. Since $\R\Gamma(U,\Oo_U)$ is an idempotent $\IE_\infty$-$R$-algebra, one can even show $M\simeq M\lotimes_A\R\Gamma(U,\Oo_U)$, but we won't need that.
		Using the above representation as a \v Cech complex, we see that $M[1/f_k]\simeq 0$ for all $k=1,\dotsc,s$ implies $M\lotimes_R\R\Gamma(U,\Oo_U)\simeq 0$. Hence also its direct summand $M$ must vanish, as claimed.
		
		Now let $r\geqslant 2$ and assume that the assertion has been proved for $r-1$ many ideals. Again, it's enough to show that $\widehat{M}_{I_j}\simeq 0$ and $M[1/f_k]\simeq 0$ jointly imply $M\simeq 0$. By the case $r=1$ it's enough to show $M[1/f]\simeq 0$, where $f$ ranges through a finite generating set of $I_r$. But $\Spec R[1/f]=\bigcup_{j=1}^{r-1}\Spec R[1/f]/I_j\cup\bigcup_{k=1}^s\Spec R[1/(f_kf)]$, so the desired vanishing of  $M[1/f]$ follows by applying the induction hypothesis to the ring $R[1/f]$.
	\end{proof}
	\begin{rem}\label{rem:DerivedBeauvilleLaszlo}
		As a consequence of \cref{lem:DerivedBeauvilleLaszlo}, the functors $-\lotimes_\IZ\IQ$ and $(-)_p^\complete$ are jointly conservative on $\Dd(\mathbb Z)$. Indeed, for any integer $N\neq 0$, \cref{lem:DerivedBeauvilleLaszlo} shows that $(-)[1/N]$ and $(-)_p^\complete$ for $p\mid N$ are jointly conservative. Thus, if $M\in\Dd(\IZ)$ satisfies $\widehat{M}_p\simeq 0$ for all primes $p$, then $M\rightarrow M[1/N]$ is an equivalence for all $N$, thus $M\simeq M\lotimes_\IZ\IQ$. 
	\end{rem}
	
	\subsection{Definitions and basic properties}\label{subsec:qWitt}	
	\begin{numpar}[Truncated big Witt vectors à la {\cite[\S\href{https://arxiv.org/pdf/1006.3125\#section.1}{1}]{HesselholtBigDeRhamWitt}}.]\label{par:BigWitt}
		We'll briefly recall the construction of truncated big Witt vectors, as well as the Frobenius and Verschiebung maps, following Hesselholt's exposition.
		
		Let $R$ be an arbitrary commutative, but not necessarily unital ring, and $S\subseteq \IN$ a subset which is closed under divisors (a \emph{truncation set} in Hesselholt's terminology). The \emph{\embrace{$S$-truncated} big Witt ring} $\IW_S(R)$ is constructed as follows: As a set, $\IW_S(R)$ is given by $R^S$. Its ring structure is uniquely determined by the condition that for all $n\in S$ the \emph{ghost map} $\gh_n\colon \IW_S(R) \rightarrow R$ given by
		\begin{equation*}
			\gh_n\bigl((x_i)_{i\in S}\bigr)\coloneqq \sum_{d\mid n}dx_d^{n/d}
		\end{equation*}
		is a morphism of rings and functorial in $R$.
		
		For us, $S$ will always be the set $T_m$ of positive divisors of some integer $m$, and we'll write $\IW_m(R)=\IW_{T_m}(R)$ for short. By \cite[Lemmas~\href{https://arxiv.org/pdf/1006.3125\#equation.1.3}{1.3}--\href{https://arxiv.org/pdf/1006.3125\#equation.1.5}{1.5}]{HesselholtBigDeRhamWitt}, for every divisor $d\mid m$ there are \emph{Frobenius} and \emph{Verschiebung} maps
		\begin{equation*}
			F_{m/d}\colon \IW_{m}(R) \longrightarrow \IW_d(R)\quad\text{and}\quad V_{m/d}\colon 	\IW_d(R) \longrightarrow \IW_{d}(R)
		\end{equation*}
		such that $F_{m/d}$ is a ring map and $V_{m/d}$ is a map of abelian groups (in fact, a map of $\IW_m(R)$-modules if we equip $\IW_d(R)$ with the $\IW_m(R)$-module structure induced by $F_{m/d}$). If $n=m/d$ and the numbers $m$ and $d$ are clear from the context (or irrelevant), we abuse notation and write just $F_n\coloneqq F_{m/d}$ and $V_n\coloneqq V_{m/d}$. These maps fulfil the following relations: For all chains of divisors $e\mid d\mid m$ we have
		\begin{equation*}
			F_{d/e}\circ F_{m/d}=F_{m/e}\quad\text{and}\quad V_{m/d}\circ 	V_{d/e}=V_{m/e}\,.
		\end{equation*}
		Furthermore, if $n\geqslant 1$ is arbitrary and $k$ is coprime to $n$, then $F_{n}\circ V_{n}=n$ and $F_n\circ V_k=V_k\circ F_n$, where we use the abuse of notation we just warned about. Finally, there's a multiplicative section of $\gh_m\colon \IW_m(R) \rightarrow R$, called the \emph{Teichmüller lift}\footnote{We choose $\tau_m(-)$ over the standard-notation $[-]$ to distinguish between Teichmüller lifts to $\IW_m(R)$ for various $m$, but also to avoid confusion with the elements $[m/d]_{q^d}=(q^m-1)/(q^d-1)$ in $\IZ[q]$}
		\begin{equation*}
			\tau_m(-)\colon R \rightarrow \IW_m(R)\,.
		\end{equation*}
		The Teichmüller lift interacts with the Frobenius and the Verschiebung via the formulas
		\begin{equation*}
			F_{m/d}\tau_m(r)=\tau_d(r)^{m/d}\quad\text{and}\quad x=\sum_{d\mid 	m}V_{m/d}\tau_d(x_{m/d})
		\end{equation*}
		for all $r\in R$ and all $x=(x_d)_{d\mid m}\in \IW_m(R)$.
	\end{numpar}
	
	\begin{rem}\label{rem:BigWittAndWitt}
		If $m=p^n$ is a prime power, then $\IW_{p^n}(R)\cong W_{n+1}(R)$ equals the ring of truncated $p$-typical Witt vectors of length $n+1$. Furthermore, the Frobenii and Verschiebungen $F_p$ and $V_p$ coincide with their $p$-typical namesakes $F$ and $V$, as does the Teichmüller lift.
	\end{rem}
	
	Now we can start to define what $q$-Witt vectors are.
	\begin{defi}\label{def:qFVSystemOfRings}
		Fix a commutative, but not necessarily unital ring $R$. A \emph{$q$-$FV$-system of rings over $R$} is a system of $\IZ[q]$-algebras $(W_m)_{m\in \IN}$, together with the following structure:
		\begin{alphanumerate}
			\item For all $m\in \IN$, a $\IZ[q]$-algebra map $\IW_{m}(R)[q]/(q^{m}-1)\rightarrow W_m$.\label{enum:qWittConditionA}
			\item For all divisors $d\mid m$, a $\IZ[q]$-algebra morphism $F_{m/d}\colon W_m\rightarrow W_d$ and a $\IZ[q]$-module morphism $V_{m/d}\colon W_d\rightarrow W_m$. These must be compatible with the usual Frobenii and Verschiebungen on ordinary Witt vectors (via the morphisms from \cref{enum:qWittGeneratorsI}) and satisfy\label{enum:qWittConditionB}
			\begin{equation*}
				F_{m/d}\circ V_{m/d}=m/d\quad\text{and}\quad V_{m/d}\circ 	F_{m/d}=[m/d]_{q^{d}}\,.
			\end{equation*}
		\end{alphanumerate}
		These objects form an obvious category, which we denote $\cat{CRing}_R^{\q FV}$.
	\end{defi}
	\begin{lem}\label{lem:BigqWittUniversal}
		Let $R$ be a commutative, but not necessarily unital ring. The category $\cat{CRing}_R^{\q FV}$ has an inital object $(\qIW_m(R))_{m\in\IN}$. It can be explicitly described as
		\begin{equation*}
			\qIW_{m}(R)\cong \IW_{m}(R)[q]/\II_{m}\,,
		\end{equation*}
		where $\II_{m}$ is the ideal generated by the following two kinds of generators:
		\begin{alphanumerate}
			\item $(q^{d}-1)\im V_{m/d}$ for all divisors $d\mid m$, and\label{enum:qWittGeneratorsI}
			\item $\im ([d/e]_{q^{e}}V_{m/d}-V_{m/e}F_{d/e})$ for all chains of divisors $e\mid d\mid m$.\label{enum:qWittGeneratorsII}
		\end{alphanumerate}
	\end{lem}
	\begin{defi}\label{def:BigqWitt}
		Let $R$ a commutative but not necessarily unital ring and let $m$ be a positive integer. The ring $\qIW_m(R)$ from \cref{lem:BigqWittUniversal} is called the ring of \emph{$m$-truncated big $q$-Witt vectors over $R$}.%
		\footnote{Beware that this definition is not consistent with \cite[Definition~\href{https://guests.mpim-bonn.mpg.de/ferdinand/q-deRham.pdf\#theorem.5.3}{5.3}]{MasterThesis}. The ring that was denoted $\qIW_m(R)$ there coincides with $\qIW_m(R)_{(q-1)}^\complete$, at least under mild hypothesis (namely those of \cref{cor:qWittCompletion}); indeed, this follows from a simple comparison of universal properties. Since we want to develop our theory in a non-$(q-1)$-completed setting (which we'll need for the upcoming applications \cite{qWittHabiro,qHodge}), it seemed the right thing to change the notation, despite the confusion this may cause.}
	\end{defi}
%	\begin{rem}\label{rem:NotAsInMasterThesisI}
%		Beware that this definition is not consistent with \cite[Definition~\href{https://guests.mpim-bonn.mpg.de/ferdinand/q-deRham.pdf\#theorem.5.3}{5.3}]{MasterThesis}. The ring that was denoted $\qIW_m(R)$ there coincides with $\qIW_m(R)_{(q-1)}^\complete$, at least under mild hypothesis (namely those of \cref{cor:qWittCompletion}); indeed, this follows from a simple comparison of universal properties. Since we want to develop our theory in a non-$(q-1)$-completed setting (which we'll need for the upcoming applications \cite{qWittHabiro,qHodge}), it seemed the right thing to change the notation, despite the confusion this may cause.
%	\end{rem}
	\begin{rem}\label{rem:qWittNoqDeformation}
		Despite the name, $\qIW_m(R)$ is almost never a $q$-deformation of $\IW_m(R)$. Indeed, we have $V_{m/d}\circ F_{m/d}=m/d$ in the quotient $\qIW_m(R)/(q-1)$. This is usually not satisfied for ordinary Witt vectors, so $\IW_m(R)\rightarrow \qIW_m(R)/(q-1)$ fails to be injective. By contrast, $\IW_m(R)\rightarrow \qIW_m(R)$ is always injective, as we'll see in \cref{prop:WittToqWittInjective}. So enforcing the condition $V_{m/d}\circ F_{m/d}$ doesn't lose any information.
	\end{rem}
	\begin{proof}[Proof of \cref{lem:BigqWittUniversal}]
		If we can show that the $\IZ[q]$-linearly extended Frobenii and Verschiebungen $F_{m/d}\colon \IW_m(R)[q]\rightarrow \IW_d(R)[q]$ and $V_{m/d}\colon \IW_d(R)[q]\rightarrow \IW_m(R)[q]$ descend to maps between $\qIW_m(R)$ and $\qIW_d(R)$, then the claimed universal property will follow in a straightforward way from the definition. Furthermore, it's immediate from the definition of $\II_m$ that the Verschiebungen descend as required. So it remains to prove the same for the Frobenii. It'll be enough to show $F_p(\II_m)\subseteq \II_{m/p}$ for all prime factors $p\mid m$.
		
		Let's first consider generators of the form $(q^d-1)V_{m/d}x$ for $x\in \IW_d(R)$. Depending on whether $n\coloneqq m/d$ is coprime to $p$ or not, the relations from \cref{par:BigWitt} yield, respectively,
		\begin{equation*}
			F_p\bigl((q^d-1)V_{n}x\bigr)=(q^d-1)V_{n}(F_px)\quad\text{or}\quad F_p\bigl((q^d-1)V_{n}x\bigr)=p(q^d-1)V_{(m/p)/d}x\,.
		\end{equation*}
		In either case, we get an element of $\II_{m/p}$. Now let's consider the second type of generators of the form $[d/e]_{q^{e}}V_{m/d}x-V_{m/e}F_{d/e}x$ for some $x\in \IW_d(R)$. Once again we need to do a case distinction.		
		
		\emph{Case~1: $p$ divides both $m/d$ and $m/e$.} In this case we can use an easy computation as above to show that $F_p$ sends the element into $\II_{m/p}$.
		
		\emph{Case~2: $p$ is coprime to both $m/d$ and $m/e$.} Let's write $m_0\coloneqq m/p$, $d_0\coloneqq d/p$, and $e_0\coloneqq e/p$ for short. Using the relations from \cref{par:BigWitt}, we can compute
		\begin{multline*}
			F_p\bigl([d/e]_{q^{e}}V_{m/d}x-V_{m/e}F_{d/e}x\bigr)\\
			\begin{aligned}
				&=[d/e]_{q^{e}}V_{m_0/d_0}(F_px)-V_{m_0/e_0}F_{d_0/e_0}(F_px)\\
				&= [d_0/e_0]_{q^{e_0}}V_{m_0/d_0}(F_px)-V_{m_0/e_0}F_{d_0/e_0}(F_px)+\bigl([d/e]_{q^{e}}-[d_0/e_0]_{q^{e_0}}\bigr)V_{m_0/d_0}(F_px)
			\end{aligned}
		\end{multline*}
		The first summand is contained in $\II_{m/p}$ by definition. Regarding the second summand, observe that our assumptions imply that $p$ is coprime to $d/e=d_0/e_0$ and therefore the sequences $\{1,q^{e/p},(q^{e/p})^2\dotsc,(q^{e/p})^{d/e-1}\}$ and $\{1,q^{e},(q^{e})^2\dotsc,(q^{e})^{d/e-1}\}$ coincide modulo $q^{d/p}-1$ up to permutation. Thus $[d/e]_{q^{e}}-[d_0/e_0]_{q^{e_0}}$ is divisible by $q^{d/p}-1=q^{d_0}-1$ and so the second summand is also contained in $\II_{m/p}$.
		
		\emph{Case~3: $p$ is coprime to $m/d$, but not to $m/e$.} Put $m_0\coloneqq m/p$ and $d_0\coloneqq d/p$ again. Using the relations from \cref{par:BigWitt}, we can compute
		\begin{multline*}
			F_p\bigl([d/e]_{q^{e}}V_{m/d}x-V_{m/e}F_{d/e}x\bigr)\\
			\begin{aligned}
				&=[d/e]_{q^{e}}V_{m_0/d_0}(F_px)-pV_{m_0/e}F_{d/e}x\\
				&=p\bigl([d_0/e]_{q^{e}}V_{m_0/d_0}(F_px)-V_{m_0/e} F_{d_0/e}(F_px)\bigr)+\bigl([p]_{q^{d_0}}-p\bigr)[d_0/e]_{q^e}V_{m_0/d_0}(F_px)\,.
			\end{aligned}
		\end{multline*}
		The first summand is again contained in $\II_{m/p}$ by definition. Regarding the second summand, we observe $[p]_{q^{d_0}}\equiv p\mod q^{d_0}-1$ and so $[p]_{q^{d_0}}[d_0/e]_{q^{e}}\equiv p[d_0/e]_{q^e}\mod q^{d_0}-1$. Hence the second summand is contained in $(q^{d_0}-1)\im V_{m_0/d_0}$, which is in turn contained in $\II_{m/p}$ by definition. This finishes the proof.
	\end{proof}
	\begin{rem}\label{rem:TruncatedUniversalProperty}
		The proof of \cref{lem:BigqWittUniversal} shows that a similar universal property also holds for every truncated sequence: If $S\subseteq \IN$ is any truncation set (in the sense of \cref{par:BigWitt}), we define an \emph{$S$-truncated $q$-$FV$-system of rings} to be a system $(W_m)_{m\in S}$ equipped with the structure from \cref{def:qFVSystemOfRings}\cref{enum:qWittConditionA},~\cref{enum:qWittConditionB} for all $m\in S$. Then $(\qIW_{m}(R))_{m\in S}$ is initial among such systems. This observation will often be used, as its often easier to verify this \enquote{truncated} version of the universal property.
	\end{rem}
	\begin{numpar}[Ghost maps and Teichmüller lifts for $q$-Witt vectors.]\label{par:GhostMaps}
		We can construct analogues of the ghost maps for $q$-Witt vectors as follows: Recall that the classical ghost map $\gh_1\colon \IW_m(R)\rightarrow R$ can be identified with quotienting out the images of all Verschiebungen. For $q$-Witt vectors, we compute:
		\begin{align*}
			\qIW_m(R)/\left(\im V_p\ \middle|\ p\text{ prime factor of }m\right)&\cong \IW_m(R)[q]/\left(\II_m,\im V_p\ \middle|\ p\text{ prime factor of }m\right)\\
			&\cong R[q]/\bigl([p]_{q^{m/p}}\ \big|\ p\text{ prime factor of }m\bigr)\\
			&\cong R[q]/\Phi_m(q)\,.
		\end{align*}
		The isomorphism in the second line follows from $\IW_m(R)/\left(\im V_p\ \middle|\ p\text{ prime factor of }m\right)\cong R$ and the third isomorphism follows from \cref{lem:IdealGeneratedByPhi}. Therefore we obtain a canonical projection
		\begin{equation*}
			\gh_1\colon \qIW_m(R)\longrightarrow R[\zeta_m]\,,
		\end{equation*}
		where $R[\zeta_m]\coloneqq R[q]/\Phi_m(q)$ (so that $\zeta_m$ denotes an $m$\textsuperscript{th} root of unity). The map $\gh_1$ will be regarded as the first ghost map. In general, we define
		\begin{equation*}
			\gh_{m/d}\colon \qIW_m(R)\longrightarrow R[\zeta_d]
		\end{equation*}
		as the composition of $F_{m/d}\colon \qIW_m(R)\rightarrow \qIW_d(R)$ with $\gh_1\colon \qIW_d(R)\rightarrow R[\zeta_d]$. One immediately verifies that the ghost maps for $q$-Witt vectors are compatible with the ordinary ghost maps in the sense that
		%
		%\node (po) at (0,0) {\phantom{\textup{.}}};\draw (po.base) ++ (0,0.051ex) to ++(0,1.509ex) to ++(1.509ex,0);\node at (1.439ex,0) {\textup{.}};
		\begin{equation*}
			\begin{tikzcd}[column sep=large]
				\IW_m(R)\dar\rar["\gh_{m/d}"]&R\dar\\
				\qIW_m(R)\rar["\gh_{m/d}"]& R[\zeta_d]
			\end{tikzcd}
		\end{equation*}
		commutes. Furthermore, there is also a Teichmüller lift
		\begin{equation*}
			\tau_m(-)\colon R\longrightarrow \qIW_m(R)
		\end{equation*}
		given as the composition of $\tau_m(-)\colon R\rightarrow \IW_m(R)$ with the canonical map $\IW_m(R)\rightarrow \qIW_m(R)$.
	\end{numpar}
	\begin{numpar}[What about Restrictions?]
		Unfortunately, it turns out that the usual restriction maps $\operatorname{Res}_{m/d}\colon \IW_m(R)\rightarrow \IW_d(R)$ do not extend to $\IZ[q]$-algebra morphisms between $\qIW_m(R)$ and $\qIW_d(R)$. Indeed, such a morphism would necessarily commute with the Verschiebungen and thus induce a $\IZ[q]$-algebra morphism
		\begin{equation*}
			R[\zeta_m]= R[q]/\Phi_m(q)\longrightarrow R[q]/\Phi_d(q)=R[\zeta_d]\,,
		\end{equation*}
		which fails to exist even in very simple cases (e.g.\ $m=p^\alpha$ is a prime power, $R$ is not a ring of characteristic $p$). So it seems that there are no analogues of restrictions in our theory, and in particular, there is no ring $\qIW(R)=\limit_{m\in \IN,\,\operatorname{Res}_{m/d}}\qIW_m(R)$ of un-truncated $q$-Witt vectors.%However, there seems to be at least some use in considering the limit $\limit_{m\in \IN,\, F_{m/d}}\qIW_m(R)$ along the Frobenius maps; see \cref{rem:HabiroIsAinf?} below.
	\end{numpar}
	In the rest of this subsection, we'll show that various properties of ordinary Witt vectors carry over to $q$-Witt vectors. Our main technical tool will be the following proposition.
	\begin{prop}\label{prop:qWittKoszulExactSequence}
		Let $R$ a commutative but not necessarily unital ring and let $m$ be a positive integer. Let $p_1,\dotsc,p_r$ be the prime factors of $m$ \embrace{assumed to be distinct}. Then the following \enquote{augmented Koszul complex} is exact:
		\begin{equation*}
			\dotso\longrightarrow\bigoplus_{i<j}\qIW_{m/p_ip_j}(R)\xrightarrow{(V_{p_i}-V_{p_j})}\bigoplus_i\qIW_{m/p_i}(R) \xrightarrow{(V_{p_i})} \qIW_m(R)\xrightarrow{\gh_1} R[\zeta_m]\longrightarrow0\,.
		\end{equation*}
	\end{prop}
	\begin{rem}\label{rem:KoszulComplexClarity}
		For the sake of clarity, let us give a precise description of the \enquote{augmented Koszul complex} in \cref{prop:qWittKoszulExactSequence}. For every subset $S\subseteq\{1,\dotsc,r\}$, put $p_S=\prod_{i\in S}p_i$. Then the complex above is given by $\bigoplus_{\#S=i}\qIW_{m/p_S}(R)$ in homological degree $i-1$ (so that $R[\zeta_m]$ sits in homological degree $-2$). Furthermore, the differentials are determined as follows: For a subset $S\subseteq\{1,\dotsc,r\}$ and an element $j\notin S$, the component $\qIW_{m/p_{S\cup\{j\}}}(R)\rightarrow \qIW_{m/p_S}(R)$ of the differential is given by
		\begin{equation*}
			\pm V_{p_j}\colon \qIW_{m/p_Sp_j}(R)\longrightarrow\qIW_{m/p_S}(R)\,,
		\end{equation*}
		where the sign follows a \enquote{Koszul-like} sign rule. That is, the sign is $+1$ if $\#\left\{i\in S\ \middle|\ i<j\right\}$ is even and $-1$ if that number is odd.
		
		More succinctly, let $T$ be the set of all positive integers whose prime factors are a subset of $\{p_1,\dotsc,p_r\}$. Let $\qIW_T(R)\coloneqq \bigoplus_{t\in T}\qIW_t(R)$. The Verschiebungen $V_{p_i}$ can be viewed as endomorphisms of $\qIW_T(R)$ which respect the direct sum decomposition (up to an indexing shift). One can form the Koszul complex of the commuting endomorphisms $(V_{p_1},\dotsc,V_{p_r})$; furthermore, this comes with a canonical augmentation to $R[\zeta_T]\coloneqq \bigoplus_{t\in T}R[\zeta_t]$. The complex from \cref{prop:qWittKoszulExactSequence} is then a direct summand of this augmented Koszul complex.
	\end{rem}
	The proof of \cref{prop:qWittKoszulExactSequence} that we'll present avoids most calculations, at the cost of using some $\infty$-categorical trickery. We don't know if there is any direct proof that avoids this heavy machinery.  The first step will be to handle the case when $m$ is a prime power, which is fairly explicit and the only real calculation we'll have to do.
	\begin{lem}\label{lem:qWittpTypicalExactSequence}
		Let $R$ a commutative but not necessarily unital ring and let $m=p^\alpha$, $\alpha\geqslant 1$, be a prime power. Then the following sequence is exact:
		\begin{equation*}
			0\longrightarrow \qIW_{p^{\alpha-1}}(R)\overset{V_p}{\longrightarrow}\qIW_{p^\alpha}(R)\xrightarrow{\gh_1}R[\zeta_{p^{\alpha}}]\rightarrow 0\,.
		\end{equation*}
	\end{lem}
	\begin{proof}
		Exactness on the right is clear from the discussion in \cref{par:GhostMaps}, so it's enough to show that $V_p\colon \qIW_{p^{\alpha-1}}(R) \rightarrow \qIW_{p^\alpha}(R)$ is injective. Let $x\in \IW_{p^{\alpha-1}}(R)[q]$ and assume that the element $V_p(x)\in W_{p^\alpha}(R)[q]$ vanishes in $\qIW_{p^\alpha}(R)$, i.e., $V_p(x)$ is contained in $\II_{p^{\alpha}}$. Recall from \cref{def:BigqWitt} that the ideal $\II_{p^{\alpha}}$ has two kinds of generators: First the elements in the image of $(q^d-1)V_{m/d}$ for every divisor $d\mid m$; then $d$ must be of the form $d=p^i$ for $0\leqslant i\leqslant \alpha$. And second the elements in the image of $[d/e]_{q^e}V_{m/d}-V_{m/e}F_{d/e}$ for every chain of divisors $e\mid d\mid m$; then $d=p^i$ and $e=p^j$ for some $0\leqslant j\leqslant i\leqslant \alpha$. We may furthermore assume $j<i$, as the corresponding generators are $0$ in the case $e=d$. In total, we see that we can write
		\begin{equation*}%\label{eq:Vx=SomeStuff}
			V_p(x)=\sum_{0\leqslant i\leqslant \alpha}\bigl(q^{p^i}-1\bigr)V_{p^{\alpha-i}}(y_i)+\sum_{0\leqslant j<i\leqslant\alpha}\left([p^{i-j}]_{q^{p^j}}V_{p^{\alpha-i}}(z_{i,j})-V_{p^{\alpha-j}}F_{p^{i-j}}(z_{i,j})\right)
		\end{equation*}
		for some $y_i\in \IW_{p^i}(R)[q]$ and some $z_{i,j}\in \IW_{p^i}(R)[q]$. We're free to change $x$ by elements from $\II_{p^{\alpha-1}}$, so let's do that to simplify the equation above. If $0\leqslant i<\alpha$, then $p^i\mid p^{\alpha-1}$, hence $x$ and $x-(q^{p^i}-1)V_{p^{\alpha-1-i}}(y_i)$ agree modulo $\II_{p^{\alpha-1}}$. Replacing $x$ by the latter, the corresponding summand in the equation above cancels, so we may assume $y_i=0$ for all $0\leqslant i<\alpha$. Furthermore, if $0\leqslant j<i<\alpha$, then $p^j\mid p^i\mid p^{\alpha-1}$ is a chain of divisors. Consequently, we may replace $x$ by $x-[p^{i-j}]_{q^{p^j}}V_{p^{(\alpha-1)-i}}(z_{i,j})-V_{p^{(\alpha-1)-j}}F_{p^{i-j}}(z_{i,j})$ to assume $z_{i,j}=0$. Finally, if we replace $x$ by $x-\sum_{0\leqslant j<\alpha}([p^{(\alpha-1)-j}]_{q^{p^j}}F_p(z_{\alpha,j})-V_{p^{(\alpha-1)-j}}F_{p^{(\alpha-1)-j}}(F_pz_{\alpha,j}))$, then the summands corresponding to $z_{\alpha,j}$ won't quite cancel, but at least the equation above can be simplified to
		\begin{equation*}%\label{eq:Vx=SomeStuff2}
			V_p(x)=\bigl(q^{p^\alpha}-1\bigr)y+\Phi_{p^{\alpha-1}}(q)z-V_pF_p(z)\,,
		\end{equation*}
		where $y=y_\alpha$ and $z=\sum_{0\leqslant j<\alpha}[p^{(\alpha-1)-j}]_{q^{p^j}}z_{\alpha,j}$.
		
		This is now much easier to work with. We see that $V_p(x+F_p(z))\in \IW_{p^\alpha}(R)[q]$ is divisible by $\Phi_{p^{\alpha-1}}(q)$. The cokernel of $V_p\colon \IW_{p^{\alpha-1}}(R)[q]\rightarrow \IW_{p^\alpha}(R)[q]$ is isomorphic to $R[q]$ (using the well-known analogue of \cref{lem:qWittpTypicalExactSequence} for ordinary Witt vectors) and thus $\Phi_{p^{\alpha-1}}(q)$-torsion-free, since the latter is a monic polynomial. It follows that $x+F_p(z)=\Phi_{p^{\alpha-1}}(q)w$ for some $w\in \IW_{p^{\alpha-1}}(R)[q]$. Then 
		\begin{equation*}
			\Phi_{p^{\alpha-1}}(q)V_p(w)=V_p\bigl(x+F_p(z)\bigr)=\Phi_{p^{\alpha-1}}(q)\bigl((q^{p^{\alpha-1}}-1)y+z\bigr)\,.
		\end{equation*}
		Since the monic polynomial $\Phi_{p^{\alpha-1}}(q)$ is a nonzerodivisor in $\IW_{p^\alpha}(R)[q]$ as well, we get $V_p(w)=(q^{p^{\alpha-1}}-1)y+z$. Using that the ring $\qIW_{p^{\alpha-1}}(R)$ is $(q^{p^{\alpha-1}}-1)$-torsion, we obtain the following equations in $\qIW_{p^{\alpha-1}}(R)$:
		\begin{equation*}
			\Phi_{p^{\alpha-1}}(q)w=pw=F_pV_p(w)=\bigl(q^{p^{\alpha-1}}-1\bigr)F_p(y)+F_p(z)=F_p(z)\,.
		\end{equation*}
		This implies $x=0$ in $\qIW_{p^{\alpha-1}}(R)$, thus completing the proof that the Verschiebung map $V_p\colon \qIW_{p^{\alpha-1}}(R) \rightarrow \qIW_{p^\alpha}(R)$ is injective.
	\end{proof}
	To tackle the general case, we will interpret the \enquote{Koszul complex} from \cref{prop:qWittKoszulExactSequence} as an \enquote{$r$-dimensional pushout} in the derived $\infty$-category of abelian groups. This is possible thanks to the following technical lemma.
	
	\begin{lem}\label{lem:HigherPushout}
		Let $\square^r=\Pp(\{1,\dotsc,r\})^\op$ denote the \enquote{$r$-dimensional hypercube category}, that is, the poset of subsets of $\{1,\dotsc,r\}$, partially ordered by reverse inclusion \embrace{so that $\{1,\dotsc,r\}$ is the initial object and $\emptyset$ is the final object}. Furthermore, let $\pushoutsign^r=\square^r\smallsetminus\{\{1,\dotsc,r\}\}$ be the category obtained given by removing the final object of $\square^r$. Finally, let $X\colon \pushoutsign^r\rightarrow \Ch(\IZ)$ be a diagram indexed by $\pushoutsign^r$ and valued in the category of chain complexes.
		\begin{alphanumerate}
			\item[a_r] The colimit $\colimit_{S\in\pushoutsign^r}X(S)$, taken in the derived $\infty$-category $\Dd(\IZ)$, can be identified with the total complex of the \enquote{Koszul double complex}
			\begin{equation*}
				T^{(r)}(X)\coloneqq \left(X\bigl(\{1,\dotsc,r\}\bigr)\longrightarrow\bigoplus_{\abs{S}=r-1}X(S)\longrightarrow\dotso\rightarrow\bigoplus_{\abs{S}=2}X(S)\longrightarrow\bigoplus_{\abs{S}=1}X(S)\right)\,,
			\end{equation*}
			where each complex $X(S)_*$ sits in homological bi-degrees $(\abs{S}-1,*)$ and we employ the same \enquote{Koszul-like} sign rule as in \cref{rem:KoszulComplexClarity}. Furthermore, this identification of the colimit can be chosen in such a way that conditions \cref{enum:NaturalTransformation} and \cref{enum:MapToColimit} below are satisfied.\label{enum:ColimIsTX}
			\item[b_r] If $Y\colon \pushoutsign^r\rightarrow \Ch(\IZ)$ is another diagram and $\alpha\colon X\Rightarrow Y$ is a natural transformation \embrace{also valued in $\Ch(\IZ)$}, then the induced map $\colimit_{S\in\pushoutsign^r}X(S)\rightarrow \colimit_{S\in\pushoutsign^r}Y(S)$ is given by $\Tot T^{(r)}(X)\rightarrow \Tot T^{(r)}(Y)$.\label{enum:NaturalTransformation}
			\item[c_r] If $X^\triangleright\colon \square^r\rightarrow \Ch(\IZ)$ is another diagram such that $X^\triangleright|_{\pushoutsign^r}=X$, then the induced map $\colimit_{S\in\pushoutsign^r}X(S)\rightarrow X^\triangleright(\emptyset)$ is given by
			\begin{equation*}
				\Tot T^{(r)}(X)\longrightarrow \Tot \bigl(X^\triangleright(\emptyset)[0]\bigr)\cong X^\triangleright(\emptyset)\,.
			\end{equation*}
			Here $X^\triangleright(\emptyset)[0]$ denotes the double complex obtained by placing $X^\triangleright(\emptyset)_*$ in homological bi-degrees $(0,*)$ and the map $T^{(r)}(X)\rightarrow X^\triangleright(\emptyset)[0]$ is induced by $X(\{i\})=X^\triangleright(\{i\})\rightarrow X^\triangleright(\emptyset)$ for $i=1,\dotsc,r$ in bi-degrees $(0,*)$ and the zero map in all other bi-degrees. \label{enum:MapToColimit}
		\end{alphanumerate}
	\end{lem}
	\begin{proof}
		We prove all three assertions simultaneously using induction on $r$.
		The case $r=1$ is trivial. If $r=2$, then $X\colon \pushoutsign^2\rightarrow \Ch(\IZ)$ is a pushout diagram. By a well-known characterisation of pushouts in $\Dd(\IZ)$, we obtain a cofibre sequence
		\begin{equation*}
			X\bigl(\{1,2\}\bigr)\longrightarrow X\bigl(\{1\}\bigr)\oplus X\bigl(\{2\}\bigr)\longrightarrow\colimit_{S\in\pushoutsign^2}X(S)\,.
		\end{equation*}
		Thus, we may identify $\colimit_{S\in\pushoutsign^2}X(S)$ with the cone of the first map, which is (upon choosing the right sign convention) precisely $T^{(2)}(X)$. Then \embrace{\hyperref[enum:ColimIsTX]{$a_2$}}, \embrace{\hyperref[enum:NaturalTransformation]{$b_2$}}, and \embrace{\hyperref[enum:MapToColimit]{$c_2$}} are easily checked.
		
		Now let $r>2$ and assume that \embrace{\hyperref[enum:ColimIsTX]{$a_{r-1}$}}, \embrace{\hyperref[enum:NaturalTransformation]{$b_{r-1}$}}, and \embrace{\hyperref[enum:MapToColimit]{$c_{r-1}$}} are satisfied. We can write $\pushoutsign^r$ as a pushout $\pushoutsign^r=(\pushoutsign^{r-1}\times \Delta^1)\sqcup_{(\pushoutsign^{r-1}\times\{0\})}(\square^{r-1}\times\{0\})$, where we view $\square^{r-1}\times\{0\}$ as the full subcategory of $\square^r$ spanned by those $S\subseteq\{1,\dotsc,r\}$ such that $r\in S$, and likewise $\square^{r-1}\times\{1\}$ as the full subcategory spanned by those $S$ such that $r\notin S$. By \cite[Proposition~\HTTthm{4.4.2.2}]{HTT}, we may thus compute any $\pushoutsign^r$-indexed colimit in $\Dd(\IZ)$ as a pushout of the corresponding colimits indexed by $\pushoutsign^{r-1}\times\{0\}$, $\pushoutsign^{r-1}\times \Delta^1$, and $\square^{r-1}\times\{0\}$ respectively. Let's identify these one by one. Let $X_0=X|_{\pushoutsign^{r-1}\times\{0\}}$. Applying \embrace{\hyperref[enum:ColimIsTX]{$a_{r-1}$}}, we see that
		\begin{equation*}
			\colimit_{S\in\pushoutsign^{r-1}\times\{0\}}X(S)\simeq \Tot T^{(r-1)}(X_0)\,.
		\end{equation*}
		Next, observe that the inclusion $\pushoutsign^{r-1}\times\{1\}\subseteq \pushoutsign^{r-1}\times\Delta^1$ is coinitial (it is even right-anodyne, as the same is true for $\{1\}\subseteq \Delta^1$). Therefore, putting  $X_1=X|_{\pushoutsign^{r-1}\times\{1\}}$ and using \embrace{\hyperref[enum:ColimIsTX]{$a_{r-1}$}} again, we obtain
		\begin{equation*}
			\colimit_{S\in\pushoutsign^{r-1}\times\Delta^1}X(S)\simeq \Tot T^{(r-1)}(X_1)\,.
		\end{equation*}
		Moreover, \embrace{\hyperref[enum:NaturalTransformation]{$b_{r-1}$}} ensures that $\colimit_{S\in\pushoutsign^{r-1}\times\{0\}}X(S)\rightarrow \colimit_{S\in\pushoutsign^{r-1}\times\Delta^1}X(S)$ is identified with the canonical map $\Tot T^{(r-1)}(X_0)\rightarrow \Tot T^{(r-1)}(X_1)$. Finally, $\square^{r-1}\times \{0\}$ has a final object, and thus
		\begin{equation*}
			\colimit_{S\in\square^{r-1}\times\{0\}}X(S)\simeq X\bigl(\{r\}\bigr)\,;
		\end{equation*}
		furthermore, \embrace{\hyperref[enum:MapToColimit]{$c_{r-1}$}} implies that $\colimit_{S\in\pushoutsign^{r-1}\times\{0\}}X(S)\rightarrow \colimit_{S\in\square^{r-1}\times\{0\}}X(S)$ is identified with the canonical map $\Tot T^{(r-1)}(X_0)\rightarrow X(\{r\})$.
		Putting everything together and using the $r=2$ case, we conclude
		\begin{equation*}
			\colimit_{S\in\pushoutsign^r}X(S)\simeq \cone\left(\Tot T^{(r-1)}(X_0)\rightarrow \Tot T^{(r-1)}(X_0)\oplus X\bigl(\{r\}\bigr)\right)\,.
		\end{equation*}
		The right-hand side is precisely $T^{(r)}(X)$, which settles \cref{enum:ColimIsTX}. Assertion \cref{enum:NaturalTransformation} is an immediate consequence of \embrace{\hyperref[enum:NaturalTransformation]{$b_{r-1}$}} and the functoriality of cones. It remains to show that \cref{enum:MapToColimit} is true. Using \embrace{\hyperref[enum:MapToColimit]{$c_{r-1}$}}, we see that $\colimit_{S\in\pushoutsign^{r-1}\times\{0\}}X(S)\rightarrow X(\emptyset)$ is given by the canonical map $\Tot T^{(r-1)}(X_0)\rightarrow X(\emptyset)$ and likewise $\colimit_{S\in\pushoutsign^{r-1}\times\Delta^1}X(S)\rightarrow X(\emptyset)$ is given by the canonical map $\Tot T^{(r-1)}(X_1)\rightarrow X(\emptyset)$. Furthermore, the diagram
		\begin{equation*}
			\begin{tikzcd}
				\Tot T^{(r-1)}(X_0)\rar\dar & X\bigl(\{r\}\bigr)\dar\\
				\Tot T^{(r-1)}(X_1)\rar & X(\emptyset)
			\end{tikzcd}
		\end{equation*}
		commutes in $\Ch(\IZ)$. In particular, its commutativity in $\Dd(\IZ)$ is witnessed by a trivial homotopy, and so the map from the homotopy pushout to $X(\emptyset)$ is precisely the map $\Tot T^{(r)}(X)\rightarrow X(\emptyset)$ considered in \cref{enum:MapToColimit}. This finishes the induction.
	\end{proof}
	\begin{proof}[Proof of \cref{prop:qWittKoszulExactSequence}]
		Consider the diagram $\square^r\rightarrow \Ch(\IZ[q])$ that sends $\emptyset\neq S\subseteq\{1,\dotsc,r\}$ to $\qIW_{m/p_S}(R)$, where $p_S=\prod_{i\in S}p_i$, and sends $\emptyset$ to the complex $(\qIW_m(R)\rightarrow R[\zeta_m])$ concentrated in homological degrees $0$ and $-1$ (we will frequently use that this complex is quasi-isomorphic to $\ker(\qIW_m(R)\rightarrow R[\zeta_m])$). Morphisms in $\square^r$ are sent to the respective Verschiebungen. By \cref{lem:HigherPushout}, what we have to show is precisely that this diagram is a colimit diagram in $\Dd(\IZ)$, or equivalently, in $\Dd(\IZ[q])$.
		
		As a consequence of \cref{lem:DerivedBeauvilleLaszlo}, the following exact endofunctors of $\Dd(\IZ[q])$ are jointly conservative:
		\begin{equation*}
			(-)\left[\localise{p_1\dotsm p_r}\right]\,,\quad (-)_{(p_i,q^{m/p_j}-1)}^\complete\quad\text{for all $i\neq j$},\quad (-)\left[\localise{q^{m/p_j}-1}\ \middle|\ j\neq i\right] \quad\text{for all $i$}\,.
		\end{equation*}
		So it suffices to check that we get a colimit diagram after applying each of these functors.
		
		\emph{Proof after localisation at $p_1\dotsm p_r$.} After localising $p_1\dotsm p_r$, all occurring Verschiebungen become split injective, with $V_{p_i}$ having left-inverse $p_i^{-1}F_{p_i}$. In general, by \cite[Lemma~\HAthm{1.2.4.15}]{HA}, a diagram $X\colon \square^r\rightarrow \Dd(\IZ)$ is a colimit diagram if and only if the diagram $\square^{r-1}\rightarrow \Dd(\IZ)$ given by $S\mapsto \cofib(X(S)\rightarrow X(S\sqcup\{r\}))$ for all $S\subseteq \{1,\dotsc,r-1\}$ is a colimit diagram. In our situation, all these maps are split injective, hence the cofibres are just the ordinary quotients. Furthermore, the new diagram $\square^{r-1}\rightarrow \Dd(\IZ)$ still has the property that all transition maps are split injective, because $F_{p_i}$ commutes with $V_{p_j}$ for $i\neq j$ and so the splittings pass to cofibres. Iterating this argument $r$ times, we reduce to showing that
		\begin{equation*}
			\left(\ker\bigl(\qIW_m(R)\xrightarrow{\gh_1} R[\zeta_m]\bigr)/\bigl(\im V_{p_i}\ \big|\ i=1,\dotsc,r\bigr)\right)\left[\localise{p_1\dotsm p_r}\right]
		\end{equation*}
		is a colimit of the empty diagram $\pushoutsign^0\rightarrow \Dd(\IZ)$. This is clear since the quotient above is $0$ by our calculation in \cref{par:GhostMaps}.
		
		\emph{Proof after $(p_i,q^{m/p_j}-1)$-adic completion.} Put $p\coloneqq p_i$ and $\ell\coloneqq p_j$ for convenience. Let's first see what happens after $p$-adic completion. Note that $(-)_p^\complete\simeq ((-)_{(p)})_p^\complete$. After applying $(-)_{(p)}$, the Verschiebungen $V_\ell$ for $\ell\neq p$ become split injective. Applying \cite[Lemma~\HAthm{1.2.4.15}]{HA} once again, we can pass to cofibres $r-1$ times and therefore reduce our assertion to proving that
		\begin{equation*}
			\left(\qIW_{m/p}(R)/\left(\im V_\ell\ \middle|\ \ell\neq p\right)\right)_p^\complete\overset{V_p}{\longrightarrow} \Bigl(\qIW_{m}(R)/\left(\im V_\ell\ \middle|\ \ell\neq p\right)\Bigr)_p^\complete \longrightarrow R[\zeta_m]_p^\complete
		\end{equation*}
		is a cofibre sequence. In fact, we only need to show that this is a cofibre sequence after $(q^{m/\ell}-1)$-adic completion. To this end, note that the left and the middle term in the above sequence can be rewritten as
		\begin{align*}
			\cofib\left(\left(\qIW_{m/p\ell}(R)/(\im V_{\ell'}\ |\ \ell'\neq\ell,p)\right)_{p}^\complete\right.&\left.\overset{V_\ell}{\longrightarrow}\left(\qIW_{m/p}(R)/(\im V_{\ell'}\ |\ \ell'\neq\ell,p)\right)_{p}^\complete\right)\,,\\
			\cofib\left(\left(\qIW_{m/\ell}(R)/(\im V_{\ell'}\ |\ \ell'\neq\ell,p)\right)_{p}^\complete\right.&\left.\overset{V_\ell}{\longrightarrow}\bigl(\qIW_{m}(R)/(\im V_{\ell'}\ |\ \ell'\neq\ell,p)\bigr)_{p}^\complete\right)\,.
		\end{align*}
		After $(q^{m/\ell}-1)$-adic completion, these maps are not only split, but actually equivalences. Indeed, $V_\ell F_\ell$ is equal to $[\ell]_{q^{m/p\ell}}$ in the first case and $[\ell]_{q^{m/\ell}}$ in the second case, and both of them are units in $\IZ[q]_{(p,q^{m/\ell}-1)}^\complete$. So the first two terms in our would-be cofibre sequence vanish. Furthermore, we compute
		\begin{equation*}
			R[\zeta_m]_{(p,q^{m/\ell}-1)}^\complete\simeq \left(R[q]/^\L\Phi_m(q)\right)_{(p,q^{m/\ell}-1)}^\complete\simeq R[q]_{(p,q^{m/\ell}-1)}^\complete/^\L\Phi_m(q)
		\end{equation*}
		and the right-hand side is $0$ because $\Phi_m(q)$ divides the unit $[\ell]_{q^{m/\ell}}$ in $\IZ[q]_{(p,q^{m/\ell}-1)}^\complete$. So we obtain a cofibre sequence for trivial reasons.
		
		\emph{Proof after localisation at $(q^{m/p_j}-1)$ for all $j\neq i$.} Again, we put $p=p_i$ for convenience. Furthermore, the letter $\ell$ will be used to denote prime factors $\neq p$ of $m$. Note that almost all terms in our complex are $(q^{m/\ell}-1)$-torsion for some $\ell\neq p$ and thus die in our localisation. The only surviving terms are
		\begin{equation*}
			\qIW_{m/p}(R)\left[\localise{q^{m/\ell}-1}\ \middle|\ \ell\neq p\right]\overset{V_p}{\longrightarrow}\qIW_{m}(R)\left[\localise{q^{m/\ell}-1}\ \middle|\ \ell\neq p\right]\longrightarrow R[\zeta_m]\left[\localise{\zeta_m^{m/\ell}-1}\ \middle|\ \ell\neq p\right]
		\end{equation*}
		and we must show that this is a cofibre sequence in $\Dd(\IZ[q])$. Our strategy will be to show that this sequence is a flat base change of the sequence from \cref{lem:qWittpTypicalExactSequence}. To achieve this, write $m=p^\alpha n$, where $\alpha$ is the exponent of $p$ in the prime factorisation of $m$; we wish to show
		\begin{equation*}
			\qIW_{m}(R)\left[\localise{q^{m/\ell}-1}\ \middle|\ \ell\neq p\right]\cong \qIW_{p^\alpha}(R)\otimes_{\IZ[q],\psi^n}\IZ\left[q,\localise{q^{m/\ell}-1}\ \middle|\ \ell\neq p\right]\,,
		\end{equation*}
		where $\psi^n$ is the map that sends $q\mapsto q^n$. This follows by a comparison of universal properties: Let $T_m$ denote the truncation set of divisors of $m$. Then \cref{rem:TruncatedUniversalProperty}, together with the universal property of localisation, shows that $(\qIW_{d}(R)\bigl[1/(q^{m/\ell}-1)\ \big|\ \ell\neq p\bigr])_{d\in T_m}$ is universal among all system of rings $(W_d)_{d\in T_m}$ equipped with a $\IW_d(R)\bigl[q,1/(q^{m/\ell}-1)\ \big|\ \ell\neq p\bigr]/(q^d-1)$-structure on $W_d$ as well as Frobenii and Verschiebungen satisfying the conditions from \cref{def:qFVSystemOfRings}\cref{enum:qWittConditionB}. Consider such a system $(W_d)_{d\in T_m}$. Since $\IW_d(R)\bigl[q,1/(q^{m/\ell}-1)\ \big|\ \ell\neq p\bigr]/(q^d-1)$ vanishes unless $d$ is of the form $d=p^in$ for some $0\leqslant i\leqslant \alpha$, we see that only $W_n,W_{pn},\dotsc,W_{p^\alpha n}$ can be non-zero. Furthermore, compatibility with the usual Verschiebungen then shows that the $\IW_{p^in}(R)$-algebra structure on $W_{p^in}$ must factor over $\IW_{p^in}(R)/\left(\im V_\ell\ \middle|\ \ell\neq p\right)\cong \IW_{p^i}(R)$. Now \cref{rem:TruncatedUniversalProperty} again, together with the universal property of base change along $\psi^n$, shows that the sequence $(\qIW_{p^i}(R)\otimes_{\IZ[q],\psi^n}\IZ\bigl[q,1/(q^{m/\ell}-1)\ \big|\ \ell\neq p\bigr])_{i=0,\dotsc,\alpha}$ is universal among all sequences of rings $(W_{p^in})_{i=0,\dotsc,\alpha}$ equipped with a $\IW_{p^i}(R)\bigl[q,1/(q^{m/\ell}-1)\ \big|\ \ell\neq p\bigr]/(q^{p^in}-1)$-structure on $W_{p^in}$ as well as compatible Frobenii and Verschiebungen. This finishes the proof of the isomorphism claimed above, as it is now apparent that both sides satisfy the same universal property.
		
		So we've succeeded in writing the localisation $\qIW_{m}(R)\bigl[(q^{m/\ell}-1)^{-1}\ \big|\ \ell\neq p\bigr]$ as a base change of $\qIW_{p^\alpha}(R)$ along the flat map $\psi^n$. By the same argument, we can do the same for $\qIW_{m/p}(R)$. To reduce to \cref{lem:qWittpTypicalExactSequence}, it remains to see that
		\begin{equation*}
			R[q]/\Phi_m(q)\left[\localise{q^{m/\ell}-1}\ \middle|\ \ell\neq p\right]\cong R[q]/\Phi_{p^{\alpha}}(q)\otimes_{\IZ[q],\psi^n}\IZ\left[q,\localise{q^{m/\ell}-1}\ \middle|\ \ell\neq p\right]\,.
		\end{equation*}
		In other words, we must show that $\Phi_{p^\alpha}(q^n)$ and $\Phi_m(q)$ agree up to unit in the localisation $\IZ\bigl[q,(q^{m/\ell}-1)^{-1}\ \big|\ \ell\neq p\bigr]$. This follows inductively from
		\begin{equation*}
			\Phi_{d\ell}(q)=\begin{cases*}
				\Phi_{d}(q^\ell) & if $\ell\mid d$\\
				\Phi_{d}(q^\ell)/\Phi_{d}(q) & if $\ell\nmid d$
			\end{cases*}
		\end{equation*}
		for all $d\mid m$ and noting that $\Phi_{d}(q)$ is a unit in the second case, because it divides $q^{m/\ell}-1$.
	\end{proof}
	Having proved \cref{prop:qWittKoszulExactSequence}, we can now establish a bunch of pleasant properties of the $q$-Witt vector functors. We start with the fact that the Verschiebungen are injective, which---at least to the author---seems not at all trivial from \cref{def:BigqWitt}.
	\begin{cor}\label{cor:VerschiebungInjective}
		Let $R$ be a commutative, but not necessarily unital ring. For all positive integers $m$ and all divisors $d\mid m$, the Verschiebung $V_{m/d}\colon \qIW_d(R)\rightarrow \qIW_m(R)$ is injective.
	\end{cor}
	\begin{proof}
		We use induction on $m$. The case where $m$ is a prime power is clear from \cref{lem:qWittpTypicalExactSequence}. In the general case, we may assume without restriction that $m/d=p$ is a prime factor of $m$. Every element $x\in \ker(V_p\colon \qIW_{m/p}(R)\rightarrow \qIW_m(R))$ is $p$-torsion since $0=F_pV_p(x)=px$. Since the $p$-torsion part of $\qIW_{m/p}(R)$ maps bijectively to the the $p$-torsion part of $\qIW_{m/p}(R)_{(p)}$, it suffices to see that $V_p\colon \qIW_{m/p}(R)_{(p)}\rightarrow \qIW_m(R)_{(p)}$ is injective. Suppose this was wrong; then $\cofib(V_p)$ would acquire a nonzero $\H^{-1}$.
		
		Now let $p_1,\dotsc,p_r$ denote the prime factors of $m$, where $p_r=p$. Consider the diagram $X\colon \square^r\rightarrow \Ch(\IZ)$ sending $\emptyset\neq S\subseteq\{1,\dotsc,r\}$ to $\qIW_{m/p_S}(R)_{(p)}$, where $p_S=\prod_{i\in S}p_i$, and $\emptyset$ to $\ker(\qIW_m(R)_{(p)}\rightarrow R[\zeta_m]_{(p)})$; all morphisms in $\square^r$ are sent to the corresponding Verschiebungen. We know from the proof of \cref{prop:qWittKoszulExactSequence} that $X$ is a colimit diagram in $\Dd(\IZ)$. Furthermore, \cite[Lemma~\HAthm{1.2.4.15}]{HA} tells us that $X$ is a colimit diagram if and only if the diagram $X'\colon \square^{r-1}\rightarrow \Dd(\IZ)$ given by $S\mapsto \cofib(V_p\colon \qIW_{m/p_Sp}(R)_{(p)}\rightarrow \qIW_{m/p_S}(R)_{(p)})$ for all $S\subseteq \{1,\dotsc,r-1\}$ is a colimit diagram. From the induction hypothesis, we know that these cofibres are static, except for $\cofib(V_p\colon \qIW_{m/p}(R)_{(p)}\rightarrow \qIW_m(R)_{(p)})$, which we're assuming has a nonzero $\H^{-1}$. Furthermore, $X'$ maps every morphism in $\square^{r-1}$ to a split morphism, because $V_{p_i}$ for $p_i\neq p$ has a left inverse given by $p_i^{-1}F_{p_i}$ and these left inverses persist after taking cofibres, since $V_p$ and $F_{p_i}$ commute for $p_i\neq p$. Therefore, if we pass to cofibres again, then everything stays static, except for the nonzero $H_1$, which remains unchanged. Furthermore, the diagram $X''\colon \square^{r-2}\rightarrow \Dd(\IZ)$ obtained by passing to cofibres still has the property that all morphisms in $\square^{r-2}$ are sent to split morphisms, since again the splittings in $X'$ pass to cofibres. Iterating this argument $r$ times, we see that the nonzero $\H^{-1}$ persists, contradicting the fact that the original diagram $X$ is a colimit diagram. 
	\end{proof}
	
	Next, we will study the interaction with localisation.
	
	\begin{cor}\label{cor:qWittLocalisation}
		Let $R$ be a commutative, but not necessarily unital ring, and let $m$ be a positive integer.
		\begin{alphanumerate}
			\item For any multiplicative subset $U\subseteq \IZ$, we have $\qIW_m(R)[U^{-1}]\cong \qIW_m(R[U^{-1}])$.\label{enum:MultiplicativeSubsetOfZ}
			\item For any multiplicative subset $U\subseteq R$, we have $\qIW_m(R)[\tau_m(U)^{-1}]\cong \qIW_m(R[U^{-1}])$. Here $\tau_m(-)$ denotes the Teichmüller lift from \cref{par:GhostMaps}.\label{enum:MultiplicativeSubsetOfR}
		\end{alphanumerate}
	\end{cor}
	\begin{proof}
		First let $U\subseteq \IZ$ be a multiplicative subset. Using induction, \cref{prop:qWittKoszulExactSequence}, and the five lemma, it's clear that $\qIW_m(R[U^{-1}])$ is $U$-local as an abelian group, hence the canonical map $\qIW_m(R)\rightarrow \qIW_m(R[U^{-1}])$ can be extended to a map
		\begin{equation*}
			\qIW_m(R)[U^{-1}]\longrightarrow \qIW_m\bigl(R[U^{-1}]\bigr)\,.
		\end{equation*}
		That this map is an isomorphism follows again from induction, \cref{prop:qWittKoszulExactSequence}, and the five lemma. This proves \cref{enum:MultiplicativeSubsetOfZ}.
		
		Now let $U\subseteq R$ be a multiplicative subset. As the Teichmüller lift is multiplicative, it's clear that $\qIW_m(R[U^{-1}])$ is $\tau_m(U)$-local, hence we get a canonical map
		\begin{equation*}
			\qIW_m(R)\bigl[\tau_m(U)^{-1}\bigr]\longrightarrow \qIW_m\bigl(R[U^{-1}]\bigr)\,.
		\end{equation*}
		That this map is an isomorphism follows from induction, \cref{prop:qWittKoszulExactSequence}, and the five lemma. To make the induction work, implicitly we also use that the complex from \cref{prop:qWittKoszulExactSequence} is a complex of $\qIW_m(R)$-modules if we regard $\qIW_d(R)$ as an $\qIW_m(R)$-algebra via the Frobenius $F_{m/d}\colon \qIW_m(R)\rightarrow \qIW_d(R)$; furthermore, this $\qIW_m(R)$-module structure identifies $\qIW_d(R)[\tau_m(U)^{-1}]$ with $\qIW_d(R)[\tau_d(U)^{-1}]$, as $F_{m/d}(\tau_m(u))=\tau_d(u)^{m/d}$ for all $u\in U$. This finishes the proof of \cref{enum:MultiplicativeSubsetOfR}.
	\end{proof}
	Now we will study how $q$-Witt vectors interact with torsion. To this end, first we prove a technical lemma that will be used several times in the discussion to come.
	\begin{lem}\label{lem:TechnicalpTorsion}
		Let $R$ be a commutative, but not necessarily unital $\IZ_{(p)}$-algebra, let $m$ be a positive integer and let $p$ be a prime number. Assume $x\in \qIW_m(R)$ satisfies $\gh_1(x)=0$ and $F_\ell(x)=0$ for all prime factors $\ell\mid m$ such that $\ell\neq p$. Then $x=V_p(y)$ for some $y\in \qIW_{m/p}(R)$; in particular, $x=0$ if $p\nmid m$.
	\end{lem}
	\begin{proof}
		Note that since $R$ is a $\IZ_{(p)}$-algebra, every prime $\ell\neq p$ is invertible in $\qIW_m(R)$ by \cref{cor:qWittLocalisation}\cref{enum:MultiplicativeSubsetOfZ}. Hence every Verschiebung $V_\ell$ for $\ell\neq p$ is split injective, with left inverse given by $\ell^{-1}F_\ell$. Furthermore, $\ell^{-1}F_\ell$ commutes with $V_{\ell'}$ for $\ell\neq \ell'$. Now let $\ell_1,\dotsc,\ell_r$ be the prime factors $\neq p$ of $m$. By the previous considerations, we can write
		\begin{equation*}
			\qIW_m(R)\cong I_1\oplus \dotsb\oplus I_r\oplus \qIW_m(R)/\bigl(\im V_{\ell_1},\dotsc,\im V_{\ell_r}\bigr)\,,
		\end{equation*}
		where $I_i=\im (V_{\ell_i}\colon \qIW_{m/\ell_i}(R)/(\im V_{\ell_1},\dotsc,\im V_{\ell_{i-1}})\rightarrow \qIW_{m}(R)/(\im V_{\ell_1},\dotsc,\im V_{\ell_{i-1}}))$. Our assumption $\gh_1(x)=0$ implies $x\in (\im V_{\ell_1},\dotsc,\im V_{\ell_r},\im V_p)$ by \cref{par:GhostMaps} (we put $\im V_p=0$ in the case $p\nmid m$). On the other hand, $F_\ell(x)=0$ for $\ell\in\{\ell_1,\dotsc,\ell_r\}$ implies that the projection of $x$ to $I_i$ vanishes for all $i$. Hence $x=V_p(y)$ for some $y\in  \qIW_{m/p}(R)/(\im V_{\ell_1},\dotsc,\im V_{\ell_r})$. By the same argument as above, this quotient may be regarded as a direct summand of $\qIW_{m/p}(R)$ and so we may regard $y$ as an element of $\qIW_m(R)$ satisfying $x=V_p(y)$, as desired.
	\end{proof}
	\begin{cor}\label{cor:qWittpTorsion}
		Let $R$ be a commutative, but not necessarily unital ring. If $R$ is $p$-torsion-free, then so is $\qIW_m(R)$ for all positive integers $m$. Likewise, if $R$ has bounded $p^\infty$-torsion, then so has $\qIW_m(R)$.
	\end{cor}
	\begin{proof}
		We may equivalently show the same for the localisations $\qIW_m(R)_{(p)}$, as the $p^\infty$-torsion part is preserved under this localisation. By \cref{cor:qWittLocalisation}\cref{enum:MultiplicativeSubsetOfZ}, replacing $\qIW_m(R)$ by $\qIW_m(R_{(p)})$ amounts to replacing $R$ by $R_{(p)}$. Hence we may assume that $R$ is a $\IZ_{(p)}$-algebra. We show both assertions simultaneously using induction on $m$. The case $m=1$ is clear. Now let $m>1$ and suppose the assertion has been proved for smaller indices. Let $x\in\qIW_m(R)$ be a $p^\infty$-torsion element. By choosing common bounds for the $p^\infty$-torsion in $R[q]/\Phi_m(q)$ (which is a free $R$-module) and $\qIW_{m/\ell}(R)$ for all prime factors $\ell\mid m$ (including $\ell=p$), we find $N$ independent of $x$ such that $p^Nx$ vanishes under $\gh_1$ and under $F_\ell$ for all $\ell$. If $R$ is $p$-torsion free, we may choose $N=0$. Hence \cref{lem:TechnicalpTorsion} implies $p^Nx=V_p(y)$ for some $y\in \qIW_{m/p}(R)$. As $V_p$ is injective by \cref{cor:VerschiebungInjective}, it follows that $y\in\qIW_{m/p}(R)$ must be a $p^\infty$-torsion element. Thus $p^Ny=0$, which implies $p^{2N}x=0$. This shows that $\qIW_m(R)$ has $p^\infty$-torsion bounded by $2N$, and is $p$-torsion free if $R$ is, because in that case we may chose $N=0$.
	\end{proof}
	A related, but easier assertion is the following.
	\begin{lem}\label{lem:qWittGhostMapsJointlyInjective}
		Let $R$ be a commutative, but not necessarily unital ring, and let $m$ be a positive integer. If $R$ is $p$-torsion-free for all prime factors $p\mid m$, then the ghost maps $\gh_{m/d}\colon \qIW_m(R)\rightarrow R[\zeta_d]$ for $d\mid m$ are jointly injective.
	\end{lem}
	\begin{proof}
		Let $x\in \qIW_m(R)$ be nonzero; we wish to show that its image under some ghost map is nonzero. If $\gh_1(x)\neq 0$, we're done; otherwise, \cref{par:GhostMaps} tells us that $x=\sum_{d\mid m,\, d\neq m}V_{m/d}(x_d)$ for some $x_d\in \qIW_d(R)$. If $x_d\neq 0$ but $\gh_1(x_d)=0$, we may use \cref{par:GhostMaps} again to see that $x_d$ can be written as a sum $x_d=\sum_{e\mid d,\, e\neq d}V_{d/e}(y_e)$ for some $y_e\in \qIW_e(R)$. If we successively take the largest $d$ for which neither $x_d=0$ nor $\gh_1(x_d)\neq 0$ is satisfied and replace $x_d$ by such a sum, we will eventually arrive at an expression for $x$ in which every $x_d$ satisfies either $x_d=0$ or $\gh_1(x_d)\neq 0$.
		
		Now choose $x_d\neq 0$ such that $m/d$ is minimal. Then $\gh_{m/d}(V_{m/e}(x_e))=0$ for all $e\neq d$ as either $x_e=0$ or $m/d$ is not divisible by $m/e$ by minimality. Hence
		\begin{equation*}
			\gh_{m/d}(x)=\gh_{m/d}(V_{m/d}(x_d))=\frac{m}{d}\gh_1(x_d)\neq 0
		\end{equation*}
		as $\gh_1(x_d)\neq 0$ and $R[\zeta_d]$ is finite free over $R$, which is $m/d$-torsion-free by assumption.
	\end{proof}
	Another consequence of \cref{prop:qWittKoszulExactSequence} related to \cref{cor:qWittpTorsion} is the fact that $\qIW_m(-)$ interacts well with derived $p$-adic completion.
	\begin{cor}\label{cor:qWittpCompletion}
		Let $R$ be a commutative, but not necessarily unital ring. If the derived $p$-completion $\widehat{R}_p$ of $R$ is static, then $\qIW_m(R)_p^\complete\simeq \qIW_m(\widehat{R}_p)$ for all positive integers $m$. In particular, if $R$ is derived $p$-complete, then so is $\qIW_m(R)$.
	\end{cor}
	\begin{proof}
		We use induction on $m$. The case $m=1$ is trivial, as $\qIW_1(R)\cong R$. For $m>1$, we use \cref{prop:qWittKoszulExactSequence} and the induction hypothesis to see that $\qIW_m(\widehat{R}_p)$ sits inside an acyclic complex in which all other terms are derived $p$-complete (including $\widehat{R}_p[\zeta_m]$, as this is finite free over $\widehat{R}_p$). Hence $\qIW_m(\widehat{R}_p)$ is derived $p$-complete itself. Thus the canonical map $\qIW_m(R)\rightarrow\qIW_m(\widehat{R}_p)$ induces a map 
		\begin{equation*}
			\qIW_m(R)_p^\complete\longrightarrow \qIW_m(\widehat{R}_p)\,.
		\end{equation*}
		Morally, the way to see that this is an equivalence should be to apply the five lemma, but we have to be careful since the derived $p$-completion on the left is a priori only an object in $\Dd(\IZ)$. So instead, we use the proof of \cref{prop:qWittKoszulExactSequence}, the fact that derived $p$-completion preserves finite colimits, and the induction hypothesis, to compute
		\begin{align*}
			\fib\bigl(\qIW_m(R)_p^\complete\rightarrow R[\zeta_m]_p^\complete\bigr)&\simeq \colimit_{S\in\pushoutsign^r}\qIW_{m/p_S}(R)_p^\complete\\
			&\simeq \colimit_{S\in\pushoutsign^r}\qIW_{m/p_S}(\widehat{R}_p)\\
			&\simeq \fib\bigl(\qIW_m(R)_p^\complete\rightarrow \widehat{R}_p[\zeta_m]\bigr)\,;
		\end{align*}
		here $\{p_1,\dotsc,p_r\}$ are the prime factors of $m$ and $p_S=\prod_{i\in S}p_S$ for all subsets $S\subseteq\{1,\dotsc,r\}$. Together with $R[\zeta_m]_p^\complete\simeq \widehat{R}_p[\zeta_m]$ (as $R[\zeta_m]$ is finite free over $R$), this finishes the proof.
	\end{proof}
	\begin{rem}\label{rem:qWittpCompletionAnimated}
		\cref{cor:qWittpCompletion} is true without assuming that $\widehat{R}_p$ is static, if we interpret $\widehat{R}_p$ as an animated ring and $\qIW_m(\widehat{R}_p)$ as the animation of the $m$-truncated $q$-Witt vectors functor. Note that the animation of $\qIW_m(-)$ agrees with the un-animated version on static rings. Indeed, this follows via induction on $m$, using \cref{prop:qWittKoszulExactSequence} and the fact that $R[\zeta_m]\simeq R[q]/^\L\Phi_m(q)$ is already a derived quotient.
	\end{rem}
	
	To finish this subsection, we study completions of the form $\qIW_m(R)_{(q^n-1)}^\complete$.
	\begin{cor}\label{cor:qWittCompletion}
		Let $R$ be a commutative, but not necessarily unital ring, and let $m$, $n$ be positive integers. If the derived $p$-completions $\widehat{R}_p$ of $R$ are static for all prime factors $p\mid m$, then the derived $(q^n-1)$-adic completion $\qIW_m(R)_{(q^n-1)}^\complete$ is static too. Furthermore, if $R$ has bounded $p^\infty$-torsion for all $p\mid m$, then also the $(q^n-1)^\infty$-torsion of $\qIW_m(R)$ is bounded.
	\end{cor}
	To prove \cref{cor:qWittCompletion}, we must first show the corresponding staticness assertions for $R[\zeta_m]\cong R[q]/\Phi_m(q)$.
	\begin{lem}\label{lem:CyclotomicExtensionTorsion}
		Let $R$ be a commutative, but not necessarily unital ring, and let $m$, $n$ be positive integers.
		\begin{alphanumerate}
			\item If the quotient $m/\gcd(m,n)$ has at least two distinct prime factors, then $R[q]/\Phi_m(q)$ is $(q^n-1)$-torsion-free.\label{enum:MoreThanOnePrimeFactor}
			\item Suppose $m/\gcd(m,n)=p^\alpha$ is a prime power. If the derived $p$-completion $\widehat{R}_p$ of $R$ is static, then the derived $(q^n-1)$-completion of $R[q]/\Phi_m(q)$ is static too. Furthermore, if $R$ has bounded $p^\infty$-torsion, then $R[q]/\Phi_m(q)$ has bounded $(q^n-1)^\infty$-torsion.\label{enum:OnlyOnePrimeFactor}
		\end{alphanumerate}
	\end{lem}
	\begin{proof}
		Everything is $(q^m-1)$-torsion and $(q^m-1,q^n-1)=(q^{\gcd(m,n)}-1)$ holds in $\IZ[q]$, hence we may replace $n$ by $\gcd(m,n)$ and thus assume $n\mid m$. For \cref{enum:MoreThanOnePrimeFactor}, simply observe that $(\Phi_m(q),q^n-1)$ is the unit ideal in $\IZ[q]$, because it contains all prime factors of $m/n$ by the proof of \cref{lem:UnitIdeal}.
		
		For \cref{enum:OnlyOnePrimeFactor}, if $m=n$, then $R[q]/\Phi_m(q)$ is $(q^n-1)$-torsion, hence already $(q^n-1)$-complete, and nothing happens. So from now on, assume $m/n=p^\alpha$, where $\alpha\geqslant 1$. In this case $0=[p^\alpha]_{q^n}$ holds in $R[q]/\Phi_m(q)$. Observe that $[p^\alpha]_{q^n}=p^\alpha+(q^n-1)u=pv+(q^n-1)^{p^\alpha-1}$, where $u,v\in\IZ[q]$ are some polynomials. Hence the ideals generated by $p$ and by $q^n-1$ in $R[q]/\Phi_m(q)$ have the same radical. Thus, to show that the derived $(q^n-1)$-completion of $(R[q]/\Phi_m(q))_{(q^n-1)}^\complete$ is static, it's enough to show the same for the derived $p$-completion, which follows from our assumption that $\widehat{R}_p$ is static since $R[q]/\Phi_m(q)$ is a finite free $R$-module. For the torsion assertion, the same observation shows
		\begin{align*}
			\bigl(R[q]/\Phi_m(q)\bigr)\bigl[(q^n-1)^t\bigr]&\subseteq \bigl(R[q]/\Phi_m(q)\bigr)[p^{t\alpha}]\,,\\
			\bigl(R[q]/\Phi_m(q)\bigr)[p^t]&\subseteq \bigl(R[q]/\Phi_m(q)\bigr)\bigl[(q^n-1)^{t(p^\alpha-1)}\bigr]\,.
		\end{align*}
		for all $t\geqslant 1$, which immediately implies that $R[q]/\Phi_m(q)$ has bounded $(q^n-1)^\infty$-torsion if $R$ has bounded $p^\infty$-torsion.
	\end{proof}
	\begin{proof}[Proof of \cref{cor:qWittCompletion}]
		As in the proof of \cref{lem:CyclotomicExtensionTorsion}, we may assume without loss of generality that $n\mid m$. Assume $\widehat{R}_\ell$ is static for all prime factors $\ell\mid m$. We use induction to show that $\qIW_m(R)_{(q^n-1)}^\complete$ is static. The case $m=1$ is clear as $\qIW_1(R)\cong R$ is already $(q^n-1)$-torsion. Now let $m>1$. Recall (e.g.\ from \cite[\stackstag{0BKG}]{Stacks}) that
		\begin{equation*}
			 \H^{-1}\bigl(\qIW_m(S)_{(q^n-1)}^\complete\bigr)\cong\limit_{t\geqslant 1}\qIW_m(R)\bigl[(q^n-1)^t\bigr]
		\end{equation*}
		where the transition maps in the limit are multiplication by $(q^n-1)$. Let $x=(x_t)$ be an element of the right-hand side; we wish to show $x=0$. Using the inductive hypothesis, we see that $F_\ell(x_t)=0$ for all prime factors $\ell\mid m$ and all $t$; hence also $0=V_\ell F_\ell(x_t)=[\ell]_{q^{m/\ell}}x_t$. By \cref{lem:IdealGeneratedByPhi}, this implies $\Phi_m(q)x_t=0$.
		
		If $m/n$ has at least two distinct prime factors, then $(\Phi_m(q),q^n-1)$ is the unit ideal in $\IZ[q]$ (see the proof of \cref{lem:CyclotomicPolynomialsCoprime}). As every $x_t$ is both $\Phi_m(q)$-torsion and $(q^n-1)^t$-torsion, we obtain $x=0$, as required.
		
		So we're left to deal with the case where $m/n=p^\alpha$ is a prime power. In this case, the ideal $(\Phi_m(q),q^n-1)\subseteq \IZ[q]$ is not the unit ideal, but at least it contains $p$ (see again the proof of \cref{lem:CyclotomicPolynomialsCoprime}) and thus $x_t$ is a $p^t$-torsion element for all $t\geqslant 1$. We may therefore replace $R$ by its localisation $R_{(p)}$, which according to \cref{cor:qWittLocalisation}\cref{enum:MultiplicativeSubsetOfZ} amounts to replacing $\qIW_m(R)$ by $\qIW_m(R)_{(p)}$ and therefore doesn't change the $p^\infty$-torsion part. Also note that replacing $R$ by $R_{(p)}$ doesn't change the condition that the derived $\ell$-adic completions are static for all prime factors $\ell\mid m$, as these completions can only become $0$ for $\ell\neq p$. Using \cref{lem:CyclotomicExtensionTorsion}, we know that $\gh_1(x_t)=0$ for all $t$. Together with $F_\ell(x_t)=0$, we conclude $x_t=V_p(y_t)$ for some $y_t\in \qIW_{m/p}(R)$ by \cref{lem:TechnicalpTorsion}. As $V_p$ is injective by \cref{cor:VerschiebungInjective}, we see that $y_t$ is a $(q^n-1)^t$-torsion element and that $y_{t-1}=(q^n-1)y_t$. Hence $y=(y_t)$ defines an element in $\H^{-1}(\qIW_m(S)_{(q^n-1)}^\complete)$. Then the inductive hypothesis shows $y=0$ and thus $x=0$ as well, as desired. This finishes the proof that $\qIW_m(R)_{(q^n-1)}^\complete$ is static.
		
		Now assume that $R$ has bounded $\ell^\infty$-torsion for all prime factors $\ell\mid m$. Let $x\in\qIW_m(R)$ be a $(q^n-1)^\infty$-torsion element. By induction and \cref{lem:CyclotomicExtensionTorsion}, we may choose a common bound for the $(q^n-1)^\infty$-torsion in $R[q]/\Phi_m(q)$ and $\qIW_{m/\ell}(R)$ for all prime factors $\ell\mid m$. Hence we find a positive integer $N$, independent of $x$, such that $(q^n-1)^Nx$ vanishes under the ghost map $\gh_1$ and under the Frobenius maps $F_\ell$ for all $\ell\mid m$. Then $0=V_\ell F_\ell((q^n-1)^Nx)=[\ell]_{q^{m/\ell}}(q^n-1)^Nx$ for all $\ell\mid m$, which as before implies $0=\Phi_m(q)(q^n-1)^Nx$.
		
		If $m/n$ has at least two distinct prime factors, then $(q^n-1)^Nx$ being both $\Phi_m(q)$-torsion and $(q^n-1)^\infty$-torsion implies $(q^n-1)^Nx=0$ and we're done. If $m/n=p^\alpha$ is a prime power, we can only deduce that $(q^n-1)^Nx$ is $p^\infty$-torsion. But then we may once again replace $R$ by $R_{(p)}$ and apply \cref{lem:TechnicalpTorsion} to see that $(q^n-1)^Nx=V_p(y)$ for some $y\in\qIW_{m/p}(R)$. Then $V_p$ being injective by \cref{cor:VerschiebungInjective} implies that $y$ must be a $(q^n-1)^\infty$-torsion element too. Hence $(q^n-1)^Ny=0$. Thus $(q^n-1)^{2N}x=0$ and we're done.
	\end{proof}

	\subsection{Injectivity of \texorpdfstring{$\IW_m(R)\rightarrow \qIW_m(R)$}{W(R) -> q-W(R)}}\label{subsec:WittToqWitt}
	In this subsection we'll do as the title says and prove the following consequence of \cref{prop:qWittKoszulExactSequence}:
	\begin{prop}\label{prop:WittToqWittInjective}
		Let $R$ be a commutative, but not necessarily unital ring, and let $m$ be a positive integer. Then the natural map $\IW_m(R)\rightarrow \qIW_m(R)$ is injective.
	\end{prop}
	\cref{prop:WittToqWittInjective} won't be needed in the rest of this article, but it's perhaps rather satisfying to know that $\qIW_m(R)$ is really an extension of $\IW_m(R)$. To prove this, first we need one more simple corollary of \cref{prop:qWittKoszulExactSequence}.
	\begin{cor}\label{cor:qWittShortExactSequence}
		Let $R\twoheadrightarrow R'$ be a surjection of commutative, but not necessarily unital rings, and let $J$ be its kernel \embrace{which we again consider as a commutative, not necessarily unital ring}. Then the sequence
		\begin{equation*}
			0\longrightarrow\qIW_m(J)\longrightarrow\qIW_m(R)\longrightarrow\qIW_m(R')\longrightarrow 0
		\end{equation*}
		is exact for all positive integers $m$.
	\end{cor}
	\begin{proof}
		We use induction on $m$. The case $m=1$ is clear. For $m>1$, we use the proof of \cref{prop:qWittKoszulExactSequence} together with the fact that $\colimit_{S\in\pushoutsign^r}$ preserves cofibre sequences to see that $\fib(\qIW_m(J)\rightarrow J[\zeta_m])\rightarrow\fib(\qIW_m(R)\rightarrow R[\zeta_m])\rightarrow\fib(\qIW_m(R')\rightarrow R'[\zeta_m])$ is a cofibre sequence. So is $J[\zeta_m]\longrightarrow R[\zeta_m]\longrightarrow R'[\zeta_m]$, as it is the base change of $J\rightarrow R\rightarrow R'$ along the finite free ring map $\IZ\rightarrow \IZ[\zeta_m]$. This proves what we want.
	\end{proof}
	
	Furthermore, to prove \cref{prop:WittToqWittInjective}, we need the following lemma with a somewhat lengthy, but straightforward proof.
	\begin{lem}\label{lem:qWittOverZp}
		Let $R$ be a commutative, but not necessarily unital $\IZ_{(p)}$-algebra, and let $m=p^\alpha n$, where $\alpha$ is the exponent of $p$ in the prime factorisation of $m$. Then
		\begin{equation*}
			\qIW_m(R)\cong \prod_{d\mid n}\qIW_{p^\alpha}(R)\otimes_{\IZ_{(p)}[q],\psi^d}\IZ_{(p)}[q]/\Phi_d(q^{p^\alpha})\,,
		\end{equation*}
		where $\psi^d\colon \IZ_{(p)}[q]\rightarrow \IZ_{(p)}[q]$ is the map sending $q\mapsto q^d$.
	\end{lem}
	\begin{proof}
		Let's abbreviate the right-hand side by $\Pi_m$. We'll explain how to give $(\Pi_m)_{m\in\IN}$ the structure of a $q$-$FV$-system over $R$ as in \cref{def:qFVSystemOfRings}. To equip $\Pi_m$ with a $\IW_m(R)$-algebra structure, we must construct a ring map from $\IW_m(R)$ into each factor. On the $d$\textsuperscript{th} factor we construct the desired map as the composition
		\begin{equation*}
			\IW_m(R)=\IW_{p^\alpha n}(R)\xrightarrow{F_{n/d}}\IW_{p^\alpha d}(R)\xrightarrow{\operatorname{Res}_{d}}\IW_{p^\alpha}(R)\longrightarrow\qIW_{p^\alpha}(R)
		\end{equation*}
		This defines a map $\IW_m(R)\rightarrow \Pi_m$. To understand what this map is really doing, observe that $\IW_m(R)\cong \IW_{n}(\IW_{p^\alpha}(R))\cong \prod_{d\mid n}\IW_{p^\alpha}(R)$. Here the first isomorphism is \cite[Corollary~\href{https://arxiv.org/pdf/0801.1691\#subsection.5.4}{5.4}]{Borger}, the second is induced by the ghost maps for $\IW_n(-)$, using that $n$ is invertible on $\IW_{p^\alpha}(R)$ by our assumption that $R$ is a $\IZ_{(p)}$-algebra. Then the map $\IW_m(R)\rightarrow \Pi_m$ is simply given by matching up the factors. 
		
		This takes care of the $\IW_m(R)$-algebra structures. It remains to construct construct Frobenii $F_{\ell}\colon\Pi_m\rightarrow \Pi_{m/\ell}$ and Verschiebungen $V_\ell\colon \Pi_{m/\ell}\rightarrow \Pi_m$ for all prime factors $\ell\mid m$ and verify that they satisfy $V_\ell\circ F_\ell=[\ell]_{q^{m/\ell}}$ as well as $F_\ell\circ V_\ell=\ell$ (but the latter will be trivial, so we won't mention it). 
		
		\emph{Case~1: $\ell=p$.} Here we simply take the maps induced by the usual Frobenius and Verschiebung $F_p\colon \qIW_{p^\alpha}(R)\rightarrow \qIW_{p^{\alpha-1}}(R)$ and $V_p\colon \qIW_{p^{\alpha-1}}(R)\rightarrow \qIW_{p^\alpha}(R)$. To check $V_p\circ F_p=[p]_{q^{m/p}}$, we must check that
		\begin{equation*}
			\psi^d\bigl([p]_{q^{p^{\alpha}/p}}\bigr)\equiv [p]_{q^{m/p}}\mod \Phi_d(q^{p^\alpha})\,.
		\end{equation*}
		In fact, we even claim that $[p]_{q^{p^{\alpha}d/p}}\equiv [p]_{q^{p^{\alpha}n/p}}\mod (q^{p^\alpha d}-1)$. This is because the sequences $\{1,q^{p^{\alpha}d/p},(q^{p^{\alpha}d/p})^2,\dotsc,(q^{p^{\alpha}d/p})^{p-1}\}$ and $\{1,q^{p^{\alpha}n/p},(q^{p^{\alpha}n/p})^2,\dotsc,(q^{p^{\alpha}n/p})^{p-1}\}$ agree modulo $(q^{p^\alpha d}-1)$ up to permutation, similar to the argument in the proof of \cref{lem:BigqWittUniversal}.
		
		\emph{Case~2: $\ell\neq p$.} Here the Frobenius $F_\ell$ is simply given by the projection to those factors where $d$ divides $n/\ell$. The Verschiebung is given as follows: If $w=(w_e)_{e\mid (n/\ell)}$ is an element of $\Pi_m$, we let $V_\ell(w)=(V_\ell(w)_d)_{d\mid n}$, where
		\begin{equation*}
			V_\ell(w)_d=\begin{cases*}
				\ell w_e & if $d=e\mid (n/\ell)$\\
				0 & else
			\end{cases*}\,.
		\end{equation*}
		To check $V_\ell\circ F_\ell=[\ell]_{q^{m/\ell}}$, we have to verify $[\ell]_{q^{m/\ell}}$ is either $\ell$ or $0$ modulo $\Phi_d(q)\dotsm \Phi_{p^\alpha d}(q)$, depending on whether $d$ divides $n/\ell$ or not. This is straightforward.
		
		We've thus succeeded in equipping $(\Pi_m)_{m\in \IN}$ with the structure of a $q$-$FV$-system over $R$. It remains to verify that it is in fact the initial one. To do so, let $(W_m)_{m\in \IN}$ be an arbitrary $q$-$FV$ system. For $m=p^\alpha n$ as above and $d\mid n$, put $W_{m,d}\coloneqq W_{m}/\Phi_d(q^{p^\alpha})$. Then
		\begin{equation*}
			W_{m}\cong \prod_{d\mid n}W_{m,d}\,,
		\end{equation*}
		because $\IZ_{(p)}[q]/(q^m-1)\cong \prod_{d\mid n}\IZ_{(p)}[q]/\Phi_d(q^{p^\alpha})$ holds by \cref{lem:CyclotomicPolynomialsCoprime} and the Chinese remainder theorem. Furthermore, this decomposition of $W_m$ is respected by Frobenii and Verschiebungen, because it only depends on the $\IZ[q]$-algebra structure on $W_m$. For any $d\mid n$, the induced maps $F_{n/d}\colon W_{p^\alpha n,d}\rightarrow W_{p^\alpha d,d}$ and $V_{n/d}\colon W_{p^\alpha d,d}\rightarrow W_{p^\alpha n,d}$ are isomorphisms. Indeed, we have $F_{n/d}\circ V_{n/d}=n/d$, which is invertible in $\IZ_{(p)}$, and also $V_{n/d}\circ F_{n/d}=[n/d]_{q^{p^\alpha d}}=n/d$, because $W_{p^\alpha d,d}$ and $W_{p^\alpha n,d}$ are $(q^{p^\alpha d}-1)$-torsion.
		
		In particular, the $\IW_{p^\alpha n}(R)$-algebra structure on $W_{p^\alpha n,d}$ necessarily factors through the Frobenius $F_{n/d}\colon \IW_{p^\alpha n}(R)\rightarrow \IW_{p^\alpha d}(R)$. Furthermore, for all prime factors $\ell\mid d$, the diagram
		\begin{equation*}
			\begin{tikzcd}
				\IW_{p^\alpha d}(R)\rar & W_{p^\alpha d,d}\\
				\IW_{p^\alpha d/\ell}(R)\rar\uar & W_{p^\alpha d/\ell}/\Phi_d(q^{p^\alpha})\uar
			\end{tikzcd}
		\end{equation*}
		commutes. But the lower right corner vanishes by \cref{lem:CyclotomicPolynomialsCoprime}. Hence the $\IW_{p^\alpha d}(R)$-algebra structure on $W_{p^\alpha d,d}$ factors through the quotient $\IW_{p^\alpha d}(R)\twoheadrightarrow \IW_{p^\alpha d}(R)/\left(\im V_\ell\ |\ \ell\text{ prime factor of }d\right)$. This quotient map can be identified with $\operatorname{Res}_d\colon \IW_{p^\alpha d}(R)\rightarrow \IW_{p^\alpha}(R)$.
		
		To summarise, we've shown that the $\IW_{p^\alpha n}(R)$-algebra structure on $W_{p^\alpha n,d}$ really factors through $\IW_{p^\alpha}(R)$. Hence for fixed $n$ and $d$, the sequence $(W_{p^\alpha n,d})_{\alpha\geqslant 0}$ acquires the structure of a $\{1,p,p^2,\dotsc\}$-truncated $q^n$-$FV$-system in the sense of \cref{rem:TruncatedUniversalProperty} (that is, it satisfies the same axioms, but with $q$ replaced by $q^n$). Since we've checked $[p]_{q^{p^{\alpha}d/p}}\equiv [p]_{q^{p^{\alpha}n/p}}\mod q^{p^\alpha d}-1$, and $W_{p^\alpha,d}$ is a $(q^{p^\alpha d}-1)$-torsion ring, we obtain equivalently a $\{1,p,p^2,\dotsc\}$-truncated $q^d$-$FV$-system structure. The universal such system is clearly $(\qIW_{p^\alpha}(R)\otimes_{\IZ_{(p)}[q],\psi^d}\IZ_{(p)}[q])_{\alpha\geqslant 0}$ and so we obtain canonical morphisms
		\begin{equation*}
			\qIW_{p^\alpha}(R)\otimes_{\IZ_{(p)}[q],\psi^d}\IZ_{(p)}[q]/\Phi_d(q^{p^\alpha})\longrightarrow W_{p^\alpha n,d}\,,
		\end{equation*}
		witnessing the desired universal property for $(\Pi_m)_{m\in \IN}$.
	\end{proof}
	\begin{proof}[Proof of \cref{prop:WittToqWittInjective}]
		If $R$ is $p$-torsion free for all prime factors $p\mid m$, then the assertion is clear by \cref{lem:qWittGhostMapsJointlyInjective}. So let's next assume that $R$ is a $p$-torsion ring. Then $V_p\circ F_p=p$ holds on ordinary Witt vectors. In particular, $\qIW_{p^\alpha}(R)/(q-1)\cong \IW_{p^\alpha}(R)$, because then clearly both $(\qIW_{p^\alpha}(R)/(q-1))_{\alpha\geqslant 0}$ and $(\IW_{p^\alpha}(R))_{\alpha\geqslant 0}$ are universal among $\{1,p,p^2,\dotsc\}$-truncated $q$-$FV$-systems for which $q=1$. This proves that $\IW_{p^\alpha}(R)\rightarrow \qIW_{p^\alpha}(R)$ has a section and is therefore injective. For general $m$, write $m=p^\alpha n$, where $\alpha=v_p(m)$. We must show that
		\begin{equation*}
			\IW_m(R)\cong \prod_{d\mid n}\IW_{p^\alpha}(R)\longrightarrow \prod_{d\mid n}\qIW_{p^\alpha}(R)\otimes_{\IZ_{(p)}[q],\psi^d}\IZ_{(p)}[q]/\Phi_d(q^{p^\alpha})\cong \qIW_m(R)
		\end{equation*}
		is injective (see \cref{lem:qWittOverZp} and its proof). But the $d$\textsuperscript{th} factor on the right-hand side maps to $\qIW_{p^\alpha}(R)\otimes_{\IZ[q],\psi^d}\IZ_{(p)}[q]/\Phi_d(q)$. This is a $(q^d-1)$-torsion ring, hence it can be rewritten as $\qIW_{p^\alpha}(R)/(q-1)\otimes_{\IZ[q],\psi^d}\IZ_{(p)}[q]/\Phi_d(q)\cong \IW_{p^\alpha}(R)[\zeta_d]$. Now $\IW_{p^{\alpha}}(R)\rightarrow \IW_{p^\alpha}(R)[\zeta_d]$ is clearly injective and we're done in the case where $R$ is $p$-torsion.
		
		Next, let's assume $R$ is $p^\alpha$-torsion for some $\alpha\geqslant 1$. We use induction on $\alpha$; the case $\alpha=1$ was just done. For $\alpha\geqslant 2$, we have a short exact sequence $0\rightarrow R[p]\rightarrow R\rightarrow R/R[p]\rightarrow 0$ of non-unital rings. Using the inductive hypothesis for $R[p]$ and $R/R[p]$ together with \cref{cor:qWittShortExactSequence} and the four lemma, we conclude that $\IW_m(R)\rightarrow \qIW_m(R)$ is injective, as required. This also settles the case where $R$ is $p^\infty$-torsion, because then $R=\bigcup_{\alpha\geqslant 1}R[p^\alpha]$ and both $\IW_m(-)$ and $\qIW_m(-)$ commute with filtered colimits.
		
		Finally, let's do the general case. We have a short exact sequence of non-unital rings $0\rightarrow \bigoplus_{p\mid m}R[p^\infty]\rightarrow R\rightarrow \overline{R}\rightarrow 0$, where $\overline{R}$ is $p$-torsion free for all $p\mid m$. So we already know that the assertion is true for $\bigoplus_{p\mid m}R[p^\infty]$ and $\overline{R}$. Applying \cref{cor:qWittShortExactSequence} and the four lemma once again, we conclude that the assertion for $R$ is true as well.
	\end{proof}

	\subsection{\texorpdfstring{$q$}{q}-Witt vectors of \texorpdfstring{$\Lambda$}{Lambda}-rings}\label{subsec:qWittLambda}
	In general, $\qIW_m(R)$ can be quite far from $R[q]/(q^m-1)$. However, in the presence of a $\Lambda$-structure on $R$, there are certain maps between these rings. The purpose of this subsection is to describe these maps. As a consequence, we will see that $\qIW_m(\IZ)\cong \IZ[q]/(q^m-1)$.%(see \cref{cor:qWittOfPerfectLambdaRing} below).
	
	From now on, we will no longer consider non-unital rings; all rings in the following will be commutative and unital.
	
	\begin{numpar}[The trivial map.]\label{par:TrivialComparisonMap}
		Suppose $A$ is a $\Lambda$-ring. Then we get a section $A\rightarrow \IW_m(A)$ of $\gh_1\colon \IW_m(A)\rightarrow A$ as follows: The cofree $\Lambda$-ring under $A$ is the big Witt ring $\IW(A)$, hence we get a section $s\colon A\rightarrow \IW(A)$ of $\gh_1$. Composing with the restriction map $\operatorname{Res}\colon \IW(A)\rightarrow\IW_m(A)$ gives the desired section. We can now extend this section $\IZ[q]$-linearly to obtain a map
		\begin{equation*}
			s_m\colon A[q]/(q^m-1)\longrightarrow\qIW_m(A)\,,
		\end{equation*}
		whose composition with $\gh_1$ is the canonical map $A[q]/(q^m-1)\rightarrow A[q]/\Phi_m(q)$. More generally, if we write $s\colon A\rightarrow \IW(A)\cong A^\IN$ as $s(x)=(\delta_m(x))_{m\in\IN}$ for all $x\in A$, then
		\begin{equation*}
			\psi^m(x)\coloneqq \sum_{d\mid m}d\delta_d(x)^{m/d}=\gh_m\bigl(s(x)\bigr)
		\end{equation*}
		is the \emph{$m$\textsuperscript{th} Adams operation} $\psi^m\colon A\rightarrow A$ of the $\Lambda$-ring $A$. Clearly $\psi^m$ is a ring morphism. Hence for the trivial comparison map $s_m\colon A[q]/(q^m-1)\rightarrow \qIW_m(A)$ constructed above, the composition ${\gh_{m/d}}\circ s_m$ agrees with the canonical projection $A[q]/(q^m-1)\rightarrow A[q]/\Phi_d(q)$ followed by $\psi^{m/d}$, extended $\IZ[q]$-linearly.
	\end{numpar}
	\begin{rem}\label{rem:LambdaRingIsntNecessary}
		The construction from \cref{par:TrivialComparisonMap} as well as all other results in this subsection remain valid if we fix $m$ and only assume that $A$ is a \emph{$\Lambda_m$-ring}: that is, a $\Lambda_{\IZ,E}$-ring in the sense of \cite[{}\href{https://arxiv.org/pdf/0801.1691\#subsection.1.17}{1.17}]{Borger}, where $E=\left\{p\IZ\ \middle|\ p\text{ prime factor of }m\right\}$. The only necessary change will be to replace $\IW(A)$, the cofree $\Lambda$-ring under $A$, by $\IW_S(A)$, the cofree $\Lambda_m$-ring under $A$, where $S\subseteq \IN$ is the truncation set of all positive integers whose prime factors are also prime factors of $m$.
	\end{rem}
	The map from \cref{par:TrivialComparisonMap} is a little silly yet surprisingly useful (as we'll see). Nevertheless, it is seldom an isomorphism or even surjective. We'll now set out to construct another comparison map $c_m\colon \qIW_m(A)\rightarrow A[q]/(q^m-1)$, which is more suitable for computations. Similar maps have been found independently by Pridham \cite[Remark~\href{https://www.maths.ed.ac.uk/~jpridham/qDR.pdf\#page=23}{3.15}]{Pridham} and Molokov \cite[Proposition~\href{https://arxiv.org/pdf/2008.04956\#Th.3.1}{3.1}]{Molokov}.
	
	We've seen in \cref{par:TrivialComparisonMap} that the Adams operations on a $\Lambda$-ring $A$ can be expressed in terms of the ghost maps on $\IW(A)$ via the system of maps (of sets) $\delta_m\colon A\rightarrow A$. We'll show now that conversely, the ghost maps can be expressed in terms of Adams operations.
	\begin{lem}\label{lem:LambdaRingEpsilonOperations}
		For any $\Lambda$-ring $A$ with Adams operations $\psi^m\colon A\rightarrow A$, there are functorial maps \embrace{of sets} $\epsilon_m\colon \IW_m(A)\rightarrow A$ for all $m\in\IN$ such that
		\begin{equation*}
			\gh_m(x)=\sum_{d\mid m}d\psi^{m/d}\left(\epsilon_d\operatorname{Res}_{m/d}(x)\right)
		\end{equation*}
		for all $x\in\IW_m(A)$. Here $\operatorname{Res}_{m/d}\colon \IW_m(A)\rightarrow \IW_d(A)$ is used to denote the restriction map for ordinary Witt vectors.
	\end{lem}
	\begin{proof}
		For the purpose of this proof, let us abuse notation by denoting the composition of $s\colon A\rightarrow \IW(A)$ from \cref{par:TrivialComparisonMap} with $\operatorname{Res}\colon \IW(A)\rightarrow \IW_m(A)$ also by $s_m$. We claim that every $x\in \IW_m(A)$ can expressed as a sum $x=\sum_{d\mid m}V_{d}(s_{m/d}(x_{m/d}))$ for $x_d\in A$ in a unique way. Believing this claim, we can simply define $\epsilon_m(x)\coloneqq x_1$ and compute
		\begin{equation*}
			\gh_m(x)=\sum_{d\mid m}\gh_m\bigl(V_d\bigl(s_{m/d}(x_{m/d})\bigr)\bigr)=\sum_{d\mid m}d\gh_{m/d}\bigl(s_{m/d}(x_{m/d})\bigr)=\sum_{d\mid m}d\psi^{m/d}(x_{m/d})\,.
		\end{equation*}
		Since $\operatorname{Res}_{m/d}(x)=\sum_{e\mid d}V_e(s_{d/e}(x_{m/(d/e)}))$, our definition yields $\epsilon_d\operatorname{Res}_{m/d}(x)=x_{m/d}$, so the computation shows that the desired formula holds.
		
		To prove the claim, let's first show that every $x\in\IW_m(A)$ has such a representation. We use induction on $m$. The case $m=1$ is clear. For $m>1$, let $x_m\coloneqq \gh_1(x)$, then $\gh_1(x-s_m(x_m))=0$, hence $x-s_m(x_m)$ is contained in the ideal $(\im V_p\ |\ p\text{ prime factor of }m)$. Applying the inductive hypothesis for all $\IW_{m/p}(A)$, we get a representation $x=s_m(x_m)+\sum_{d\mid m,\, d\neq 1}V_d(s_{m/d}(x_{m/d}))$ as desired.
		
		We'll only prove uniqueness in the case where $A$ is flat over $\IZ$ and leave the general case to the reader. The flat case will be enough for our purposes, since it allows us to define $\epsilon_m$ for $\IZ$-flat $\Lambda$-rings, hence, in particular, for free $\Lambda$-rings (possibly in infinitely many generators). By a formal argument, there's then a unique way to extend the $\epsilon_m$ functorially to all $\Lambda$-rings: Just write an arbitrary $\Lambda$-ring as a reflective coequaliser of free $\Lambda$-rings and use that $\IW_m(-)$ commutes with reflective coequalisers. To show uniqueness in the flat case, first observe that in any representation $x=\sum_{d\mid m}V_d(s_{m/d}(x_{m/d}))$ the element $x_m$ is uniquely determined via $x_m=\gh_1(x)$. Next, for all prime factors $p\mid m$, the element $x_{m/p}$ is uniquely determined via $\gh_p(x)=\psi^p(x_m)+px_{m/p}$, since $A$ is $p$-torsion free by our flatness assumption. Continuing in this way, we find that all $x_{m/d}$ are uniquely determined.
	\end{proof}
	\begin{lem}\label{lem:WittToCyclicRing}
		For any $\Lambda$-ring $A$ with Adams operations $\psi^m\colon A\rightarrow A$, the map \embrace{of sets, a~priori} $c_m\colon \IW_m(A)\rightarrow A[q]/(q^m-1)$ given by
		\begin{equation*}
			c_m(x)\coloneqq\sum_{d\mid m}[d]_{q^{m/d}}\psi^{m/d}\left(\epsilon_d\operatorname{Res}_{m/d}(x)\right)
		\end{equation*}
	 	is a morphism of rings.
	\end{lem}
	\begin{proof}
		We can always find a surjection $A'\twoheadrightarrow R$ from a $\IZ$-flat $\Lambda$-ring $A'$, hence we may assume that $A$ is $\IZ$-flat itself. Then $A[q]/(q^m-1)\rightarrow \prod_{d\mid m}A[q]/\Phi_d(q)$ is injective. So it suffices to check that $c_m$ is a ring morphism modulo $\Phi_d(q)$ for all $d\mid m$. For this, note that
		\begin{equation*}
			[e]_{q^{m/e}}\equiv\begin{cases*}
				e & if $d\mid \frac me$\\
				0 & if $d\nmid \frac me$
			\end{cases*}\mod \Phi_d(q)
		\end{equation*}
		and then a straightforward calculation shows $c_m\equiv \psi^d\circ \gh_{m/d}\mod \Phi_d(q)$. This is clearly a ring morphism, hence we're done.
	\end{proof}
	\begin{cor}\label{cor:qWittToCyclicRing}
		Let $R$ be a $\Lambda$-ring. Then the ring morphism $c_m\colon \IW_m(A)\rightarrow A[q]/(q^m-1)$ from \cref{lem:WittToCyclicRing} extends uniquely to a functorial ring morphism
		\begin{equation*}
			c_m\colon \qIW_m(A)\longrightarrow A[q]/(q^m-1)
		\end{equation*}
		such that the following diagrams commute for all $d\mid m$:
		\begin{equation*}
			\begin{tikzcd}
				\qIW_m(A)\rar["c_m"]\dar["F_{m/d}"']& A[q]/(q^m-1)\dar\\
				\qIW_d(A)\rar["c_d"]& A[q]/(q^d-1)
			\end{tikzcd}\quad\text{and}\quad
			\begin{tikzcd}
				\qIW_m(A)\rar["c_m"]& A[q]/(q^m-1)\\
				\qIW_d(A)\rar["c_d"]\uar["V_{m/d}"]& A[q]/(q^d-1)\uar["{[m/d]_{q^d}}"']
			\end{tikzcd}
		\end{equation*}
		Furthermore, the composition $c_m\circ s_m\colon A[q]/(q^m-1)\rightarrow A[q]/(q^m-1)$ is the $\IZ[q]$-linear extension of the $m$\textsuperscript{th} Adams operation $\psi^m\colon A\rightarrow A$.
	\end{cor}
	\begin{proof}
		First one checks that the analogous diagrams with $\IW_m(A)$ and $\IW_d(A)$ commute. As in the proof of \cref{lem:WittToCyclicRing}, it's enough to check this modulo all the cyclotomic polynomials, which leads to a straightforward calculation. Having checked this, it's now clear that the system $(A[q]/(q^m-1))_{d\mid m}$ satisfies the conditions \cref{enum:qWittConditionA} and~\cref{enum:qWittConditionB} of \cref{def:qFVSystemOfRings} if we equip $A[q]/(q^d-1)$ with the $\IW_d(A)[q]/(q^d-1)$-algebra structure obtained via $c_d$ and extend Frobenius and Verschiebung in the indicated way. By the universal property of $(\qIW_m(A))_{m\in\IN}$, this provides us with the desired maps $c_m\colon \qIW_m(A)\rightarrow A[q]/(q^m-1)$.
		
		By surjecting from a $\IZ$-flat $\Lambda$-ring, it's again enough to check that $c_m\circ s_m=\psi^m$ holds modulo $\Phi_d(q)$ for every $d\mid m$. But we've seen in \cref{par:TrivialComparisonMap} that ${\gh_{m/d}}\circ s_m=\psi^d$ and we've seen in the proof of \cref{lem:WittToCyclicRing} that $c_m\equiv \psi^d\circ\gh_{m/d}\mod \Phi_d(q)$. Since $\psi^m=\psi^d\circ\psi^{m/d}$, we win.
	\end{proof}
	Using the comparison map from \cref{cor:qWittToCyclicRing}, we can now compute $q$-Witt vectors in some cases.
	\begin{prop}\label{prop:qWittOfLambdaRing}
		Let $A$ be a $\Lambda$-ring such that all Adams operations $\psi^m\colon A\rightarrow A$ are injective. Then the ring morphism $c_m\colon \qIW_m(A)\rightarrow A[q]/(q^m-1)$ from \cref{cor:qWittToCyclicRing} is an isomorphism onto the subring
		\begin{equation*}
			\sum_{d\mid m}[d]_{q^{m/d}}\psi^{m/d}(A)[q]/(q^m-1)\subseteq A[q]/(q^m-1)\,.
		\end{equation*}
	\end{prop}
	\begin{proof}
		For all primes $p$, $A_{(p)}$ is a $\delta$-ring with injective Frobenius $\psi^p\colon A_{(p)}\rightarrow A_{(p)}$, hence \cite[Lemma~\href{https://arxiv.org/pdf/1905.08229v4\#theorem.2.28}{2.28}]{Prismatic} shows that $A$ is $p$-torsion free for all primes $p$. Hence the projections $A[q]/(q^m-1)\rightarrow A[q]/\Phi_d(q)$ for $d\mid m$ are jointly injective. The composition of $c_m$ with the $d$\textsuperscript{th} such projection is $\psi^d\circ\gh_{m/d}$. Since $\psi^d$ is injective and the ghost maps on $\qIW_m(A)$ are jointly injective by \cref{lem:qWittGhostMapsJointlyInjective}, we deduce that $c_m\colon \qIW_m(A)\rightarrow A[q]/(q^m-1)$ is injective.
		
		It's clear from the construction that the image of $c_m$ is contained in the indicated subring. To show surjectivity, recall the construction of the maps $\epsilon_m\colon \IW_m(A)\rightarrow A$ from the proof of \cref{lem:LambdaRingEpsilonOperations}. For every $a\in A$ it immediately follows that $c_m(V_d(s_{m/d}(a)))=[d]_{q^{m/d}}\psi^{m/d}(a)$, hence the image of $c_m$ must contain the whole subring above.
	\end{proof}
	\begin{cor}\label{cor:qWittOfPerfectLambdaRing}
		If $A$ is a perfect $\Lambda$-ring \embrace{e.g.\ $A=\IZ$, $A=\IZ_p$, or $A=\mathrm{A}_\inf(R)$ for some perfectoid ring $R$}, then the comparison maps from \cref{par:TrivialComparisonMap} and \cref{cor:qWittToCyclicRing} are both isomorphisms:
		\begin{equation*}
			s_m\colon A[q]/(q^m-1)\overset{\cong}{\longrightarrow} \qIW_m(A)\quad\text{and}\quad c_m\colon \qIW_m(A)\overset{\cong}{\longrightarrow} A[q]/(q^m-1)\,.
		\end{equation*}
	\end{cor}
	\begin{proof}
		For $c_m$ this is immediate from \cref{prop:qWittOfLambdaRing}; for $s_m$ we use  $c_m\circ s_m=\psi^m$ by \cref{cor:qWittToCyclicRing}, which is an isomorphism since $A$ is perfect.
	\end{proof}
	\begin{exm}\label{exm:GhostMapsIsosForTrivialLambdaRing}
		As indicated in \cref{rem:LambdaRingIsntNecessary}, \cref{cor:qWittOfPerfectLambdaRing} is still true, for fixed $m$, if $A$ is only a perfect $\Lambda_m$-ring. Now assume that $R$ is any ring such that $m$ is invertible in $R$. We can equip $R$ with the trivial $\Lambda_m$-structure, where all Adams operations are the identity. This is clearly perfect, hence $c_m\colon \qIW_m(R)\rightarrow R[q]/(q^m-1)$ is an isomorphism. On the other hand, if $p_1,\dotsc,p_r$ are invertible in $R$, then the cyclotomic polynomials $\Phi_d(q)$ for $d\mid m$ are pairwise coprime in $R[q]$ by \cref{lem:CyclotomicPolynomialsCoprime}, hence $R[q]/(q^m-1)\cong \prod_{d\mid m}R[q]/\Phi_d(q)$ by the Chinese remainder theorem. Now recall from the proof of \cref{lem:WittToCyclicRing} that $c_m\equiv \psi^d\circ \gh_{m/d}\mod \Phi_d(q)$ and $\psi^d$ is the identity on $R$. So an equivalent way of stating \cref{cor:qWittOfPerfectLambdaRing} in our case is that
		\begin{equation*}
			\bigl(\gh_{m/d}\bigr)_{d\mid m}\colon \qIW_m(R)\overset{\cong}{\longrightarrow}\prod_{d\mid m}R[q]/\Phi_d(q)
		\end{equation*}
		is an isomorphism.
	\end{exm}
	\begin{cor}\label{cor:qWittNoetherian}
		If $R$ is of finite type over $\IZ$, then $\qIW_m(R)$ is of finite type over $\IZ[q]/(q^m-1)$. In particular, it is noetherian.
	\end{cor}
	\begin{proof}
		If $P\twoheadrightarrow R$ is a surjection from a polynomial ring, then $\qIW_m(P)\rightarrow \qIW_m(R)$ is surjective too. Hence it suffices to consider the case where $R=\IZ[T_1,\dotsc,T_n]$ is a polynomial ring. Equip $\IZ[T_1,\dotsc,T_n]$ with the unique $\Lambda$-structure in which $\psi^p(T_i)=T_i^p$ for all primes~$p$ and all $i=1,\dotsc,n$. In this case, \cref{prop:qWittOfLambdaRing} tells us that $\qIW_m(R)$ isomorphic to the subing $B_m\coloneqq \sum_{d\mid m}[d]_{q^{m/d}}\IZ[T_1^{m/d},\dotsc,T_n^{m/d},q]/(q^{m}-1)$ sitting inside a chain of inclusions
		\begin{equation*}
			\IZ[T_1^m,\dotsc,T_k^m,q]/(q^m-1)\subseteq B_m\subseteq \IZ[T_1,\dotsc,T_k,q]/(q^m-1)\,.
		\end{equation*}
		Since $\IZ[T_1,\dotsc,T_k,q]/(q^m-1)$ is finite over the noetherian ring $\IZ[T_1^m,\dotsc,T_k^m,q]/(q^m-1)$, it follows that $B_m$ must be finite over $\IZ[T_1^m,\dotsc,T_k^m,q]/(q^m-1)$ as well. This proves that $B_m$ has finite type over $\IZ[q]/(q^m-1)$, as desired.
	\end{proof}
	
	\subsection{Relative \texorpdfstring{$q$}{q}-Witt vectors}\label{subsec:RelativeqWitt}
	Using the comparison map $c_m$ from \cref{subsec:qWittLambda}, one can develop a theory of $q$-Witt vectors relative to a fixed $\Lambda$-ring $A$, in such a way that all structure maps are $A[q]$-linear. This is necessary since we would like to formulate our eventual applications \cite{qWittHabiro,qHodge} in a relative setting. First we introduce the following relative variant of \cref{def:qFVSystemOfRings}.
	\begin{defi}\label{def:RelativeqFVSystemOfRings}
		Let $A$ be a $\Lambda$-ring and let $R$ be an $A$-algebra. A \emph{$q$-$FV$-system of $A$-algebras over $R$} is a system of $A[q]$-algebras $(W_{m})_{m\in \IN}$, together with the following structure:
		\begin{alphanumerate}
			\item For all $m\in \IN$, an $A[q]$-algebra map $\qIW_m(R)\otimes_{\qIW_m(A)}A[q]/(q^m-1)\rightarrow W_{m}$. Here the tensor product is taken along the map $c_m$ from \cref{lem:WittToCyclicRing}.\label{enum:RelativeqWittConditionA}
			\item For all divisors $d\mid m$, an $A[q]$-algebra morphism $F_{m/d}\colon W_{m}\rightarrow W_{d}$ and a $A[q]$-module morphism $V_{m/d}\colon W_{d}\rightarrow W_{m}$. These must be compatible with the Frobenii and Verschiebungen on $q$-Witt vectors (via the morphisms from \cref{enum:qWittGeneratorsI}) and satisfy\label{enum:RelativeqWittConditionB}
			\begin{equation*}
				F_{m/d}\circ V_{m/d}=m/d\quad\text{and}\quad V_{m/d}\circ 	F_{m/d}=[m/d]_{q^{d}}\,.
			\end{equation*}
		\end{alphanumerate}
		These objects form an obvious category, which we denote $\cat{CRing}_{R/A}^{\q FV}$.
	\end{defi}
	\begin{lem}\label{lem:RelativeqWitt}
		Let $A$ be a $\Lambda$-ring and $R$ an $A$-algebra. The category $\cat{CRing}_{R/A}^{\q FV}$ has an initial object $(\qIW_m(R/A))_{m\in\IN}$. Explicitly, $\qIW_m(R/A)$ is the quotient
		\begin{equation*}
			\qIW_m(R/A)\cong \bigl(\qIW_m(R)\otimes_{\qIW_m(A)}A[q]/(q^m-1)\bigr)/\IU_{m}\,,
		\end{equation*}
		where $\IU_{m}$ is the ideal generated by $V_{m/d}(xy)\otimes 1-V_{m/d}(x)\otimes c_d(y)$ for all divisors $d\mid m$, all $x\in \qIW_d(R)$, and all $y\in\qIW_d(A)$.
	\end{lem}
	\begin{defi}\label{def:RelativeqWitt}
		Let $A$ be a $\Lambda$-ring and $R$ an $A$-algebra. We call $\qIW_m(R/A)$ the \emph{ring of $m$-truncated big $q$-Witt vectors of $R$ relative $A$}.
	\end{defi}
	\begin{proof}[Proof of \cref{lem:RelativeqWitt}]
		First we should remark that $V_{m/d}(x)\otimes c_d(y)$ is a well-defined element of $\qIW_m(R)\otimes_{\qIW_m(A),c_m}A[q]/(q^m-1)$, even though, a priori, $c_d(y)$ is only an element of $A[q]/(q^d-1)$. But $V_{m/d}(x)\in \qIW_m(R)$ is a $(q^d-1)$-torsion element, so it doesn't matter how we lift $c_d(y)$ to an element of $A[q]/(q^m-1)$.
		
		To show that $\qIW_m(R/A)$, defined as the quotient above, is indeed the desired initial object, we only need to check condition~\cref{enum:RelativeqWittConditionB} from \cref{def:RelativeqFVSystemOfRings}. The ideal $\IU_{m}$ is constructed in such a way that we get a well-defined $A[q]$-linear map $\qIW_d(R)\otimes_{\qIW_d(A)}A[q]/(q^d-1)\rightarrow \qIW_m(R/A)$ by sending $x\otimes a\mapsto V_{m/d}(x)\otimes a$ for all $x\in\qIW_d(R)$, $a\in A[q]/(q^d-1)$. Clearly, this map kills $\IU_{d}$, hence we get our desired Verschiebung $V_{m/d}\colon \qIW_d(R/A)\rightarrow \qIW_m(R/A)$.
		
		To construct the Frobenii, it's enough to construct $F_p\colon \qIW_m(R/A)\rightarrow \qIW_{m/p}(R/A)$ for all prime factors $p\mid m$. The Frobenius $F_p\colon \qIW_m(R)\rightarrow \qIW_{m/p}(R)$ and the canonical projection $A[q]/(q^m-1)\rightarrow A[q]/(q^{m/p}-1)$, which are compatible by \cref{cor:qWittToCyclicRing}, induce an $A[q]$-algebra morphism
		\begin{equation*}
			F_p\colon \qIW_m(R)\otimes_{\qIW_m(A)}A[q]/(q^m-1)\longrightarrow \qIW_{m/p}(R)\otimes_{\qIW_{m/p}(A)}A[q]/(q^{m/p}-1)\,.
		\end{equation*}
		To finish the proof, we must check $F_p(\IU_{m})\subseteq \IU_{m/p}$. So let's consider a generator of the form $V_{m/d}(xy)\otimes 1-V_{m/d}(x)\otimes c_d(y)$. Depending on whether $p$ divides $n\coloneqq m/d$ or not, the element $F_pV_{m/d}(xy)\otimes 1-F_pV_{m/d}(x)\otimes c_d(y)$ can be evaluated to either
		\begin{equation*}
			p\bigl(V_{(m/p)/d}(xy)\otimes 1-V_{(m/p)/d}(x)\otimes c_d(y)\bigr)\quad\text{or}\quad V_{n}\bigl(F_p(x)F_p(y)\bigr)\otimes 1-V_{n}\bigl(F_p(x)\bigr)\otimes c_d(y)\,.
		\end{equation*}
		In the first case, we visibly get an element of $\IU_{m/p}$. In the second case, recall from \cref{cor:qWittToCyclicRing} that the image of $c_d(y)$ in $A[q]/(q^{d_0}-1)$ is precisely $c_{d_0}(F_p(y))$. Hence the element above is again contained in $\IU_{m/p}$.
	\end{proof}
	\begin{rem}\label{rem:TruncatedUniversalPropertyRelative}
		As in \cref{rem:TruncatedUniversalProperty}, for every truncation set $S\subseteq \IN$, the truncated sequence $(\qIW_m(R/A))_{m\in S}$ satisfies a similar universal property. In the special case where $S=T_m$ is the set of divisors of $m$, the construction doesn't need a full $\Lambda$-structure on $A$; instead, a $\Lambda_m$-structure in the sense of \cref{rem:LambdaRingIsntNecessary} will be enough.
	\end{rem}
	Throughout the rest of this article, we'll exclusively work in the relative setting, since our applications work in the relative setting and the relative case isn't really more difficult. It is, however, a little heavier on the notation.
	
	In the rest of this subsection, we'll show that (under mild assumptions) most of our results so far can be carried over to the relative setting. Let's begin with a few canonical constructions.
	
	\begin{numpar}[Ghost maps for relative $q$-Witt vectors.]\label{par:RelativeGhostMaps}
		For all divisors $d\mid m$, we get a relative ghost map
		\begin{equation*}
			\gh_{m/d}\colon \qIW_m(R/A)\longrightarrow R\otimes_{A,\psi^d}A[q]/\Phi_d(q)\,.
		\end{equation*}
		This map can be constructed by tensoring the usual ghost map $\gh_{m/d}\colon \qIW_m(R)\rightarrow R[q]/\Phi_d(q)$ with the natural projection $A[q]/(q^m-1)\rightarrow A[q]/\Phi_d(q)$ and checking that the ideal $\IU_{m}$ from \cref{lem:RelativeqWitt} is sent to $0$. It's also straightforward to check that $\gh_{m/d}={\gh_{d/d}}\circ F_{m/d}$ and that $\gh_{m/m}\colon \qIW_m(R/A)\rightarrow R\otimes_{A,\psi^m}A[q]/\Phi_m(q)$ can be identified with the quotient of $\qIW_m(R/A)$ by the ideal generated by the images of all Verschiebungen $V_{m/d}$ for $d\mid m$, $d\neq m$.
	\end{numpar}
	\begin{numpar}[Relative comparison maps.]\label{par:RelativeComparisonMaps}
		Suppose $R$ is a $\Lambda$-$A$-algebra. Using the universal property of $(\qIW_m(-/A))_{m\in\IN}$, we see that the map $\qIW_m(R)\rightarrow R[q]/(q^m-1)$ from \cref{cor:qWittToCyclicRing} extends to an $A[q]$-algebra morphism
		\begin{equation*}
			c_{m/A}\colon \qIW_m(R/A)\longrightarrow R[q]/(q^m-1)\,.
		\end{equation*}
		As we'll see in the proof of \cref{lem:RelativeqWittBaseChange} and in \cref{rem:FaithfullyFlatCoverByPerfectLambdaRing}, this map is often injective and its image can be pinned down.
		
		Furthermore, using $c_m\circ s_m=\psi^m$ by \cref{cor:qWittToCyclicRing}, we see that the trivial comparison map $s_m\colon A[q]/(q^m-1)\rightarrow \qIW_m(A)$ extends to an $A[q]$-algebra morphism
		\begin{equation*}
			s_{m/A}\colon R\otimes_{A,\psi^m}A[q]/(q^m-1)\longrightarrow\qIW_m(R/A)\,.
		\end{equation*}
		The composition $c_{m/A}\circ s_{m/A}$ is given by the linearised Adams operation $\psi^m_{/A}\colon R\otimes_{A,\psi^m}A\rightarrow R$.
	\end{numpar}
	%Next, we'll show an important base change result for relative $q$-Witt vectors.
	\begin{lem}\label{lem:RelativeqWittBaseChange}
		If $A\rightarrow A'$ is a morphism of $\Lambda$-rings and $R$ is an $A$-algebra, then for all $m\in\IN$ the canonical map is an isomorphism
		\begin{equation*}
			\qIW_m(R/A)\otimes_AA'\overset{\cong}{\longrightarrow}\qIW_m(R\otimes_AA'/A')\,.
		\end{equation*}
	\end{lem}
	\begin{proof}
		The statement might seem like an exercise in universal properties, but it's not; the problem with such an approach is to construct a $\IW_m(R\otimes_AA')$-algebra structure on $\qIW_m(R/A)\otimes_AA'$. So instead, our proof will be somewhat indirect. It's enough to prove the case where $R\cong A[\{T_i\}_{i\in I}]$ is a polynomial ring over $A$ (possibly in infinitely many variables). Indeed, using the universal property, it's straightforward to check that $\qIW_m(-/A)$ commutes with reflective coequalisers and every $A$-algebra can be written as a reflective coequaliser of polynomial rings over $A$.
		
		To prove the polynomial ring case, equip $A[\{T_i\}_{i\in I}]$ with a $\Lambda$-$A$-algebra structure via $\psi^p(T_i)\coloneqq T_i^p$. The comparison map $c_{m/A}\colon \qIW_m(A[\{T_i\}_{i\in I}]/A)\rightarrow A[\{T_i\}_{i\in I},q]/(q^m-1)$ from \cref{par:RelativeComparisonMaps} has its image contained in the subring
		\begin{equation*}
			B_{m}\coloneqq \sum_{d\mid m}[d]_{q^{m/d}}A\bigl[\bigl\{T_i^{m/d}\bigr\}_{i\in I},q\bigr]/(q^m-1)
		\end{equation*}
		But we also have a canonical map $\pi\colon \qIW_m(\IZ[\{T_i\}_{i\in I}])\otimes_\IZ A\rightarrow \qIW_m(A[\{T_i\}_{i\in I}]/A)$. Using \cref{prop:qWittOfLambdaRing} for $\IZ[\{T_i\}_{i\in I}]$, it's clear that $c_m\circ \pi\colon \qIW_m(\IZ[\{T_i\}_{i\in I}])\otimes_\IZ A\rightarrow B_{m}$ is an isomorphism. We claim that the composition $\iota\coloneqq \pi\circ (c_m\circ \pi)^{-1}\circ c_m$ is the identity on $\qIW_m(A[\{T_i\}_{i\in I}]/A)$. Believing this for the moment, we're done. Indeed, if $\iota$ is the identity, then $\pi$ yields an isomorphism $\qIW_m(\IZ[\{T_i\}_{i\in I}])\otimes_\IZ A\cong \qIW_m(A[\{T_i\}_{i\in I}]/A)$, an analogous isomorphism holds for $A'$, and then the desired base change property is immediate.
		
		To prove that $\iota$ is the identity, recall from \cref{lem:RelativeqWitt} that $\qIW_m(A[\{T_i\}_{i\in I}]/A)$ can be written as a quotient of $\qIW_m(A[\{T_i\}_{i\in I}])\otimes_{\qIW_m(A)}A[q]/(q^m-1)$. Furthermore, it follows from the proof \cref{lem:LambdaRingEpsilonOperations} that $\qIW_m(A[\{T_i\}_{i\in I}])$ is generated as a $\IZ[q]$-module by elements of the form $V_d(s_{m/d}(fa))$, where $f\in \IZ[\{T_i\}_{i\in I}]$ is a polynomial with integral coefficients and $a\in A$. So we only need to check that $\iota$ fixes elements of the form $V_d(s_{m/d}(fa))\otimes a'$ for $f$ and $a$ as above and $a'\in A[q]/(q^m-1)$. By construction, $c_m$ sends such an element to $[d]_{q^{m/d}}\psi^{m/d}(fa)a'$. Under the isomorphism $(c_m\circ \pi)^{-1}$, this is sent to $V_d(s_{m/d}(f))\otimes \psi^{m/d}(a)a'$. But, again by construction, we have $\psi^{m/d}(a)=c_{m/d}(s_{m/d}(a))$. Hence $V_d(s_{m/d}(fa))\otimes a'-V_d(s_{m/d}(f))\otimes \psi^{m/d}(a)a'$ is contained in the ideal $\IU_{m}$ from \cref{lem:RelativeqWitt}, which proves that $\iota$ indeed sends the element $V_d(s_{m/d}(f))\otimes \psi^{m/d}(a)a'$ to itself. We're done.
	\end{proof}
	\begin{rem}\label{rem:FaithfullyFlatCoverByPerfectLambdaRing}
		If $A_\infty$ is a perfect $\Lambda$-ring, then $\qIW_m(R)\cong \qIW_m(R/A_\infty)$ holds for all $A_\infty$-algebras $R$ by  \cref{cor:qWittOfPerfectLambdaRing}. In general, if $A$ is a $\Lambda$-ring for which the map $A\rightarrow A_\infty$ into its colimit perfection is faithfully flat%
		\footnote{In fact, if \emph{any} faithfully flat morphism of $\Lambda$-rings $A\rightarrow A_\infty$ into a perfect $\Lambda$-ring exists, then the Adams operations $\psi^m\colon A\rightarrow A$ are faithfully flat (and so the map from $A$ into its colimit perfection is faithfully flat as well). Indeed, whether $-\otimes_{A,\psi^m}A$ is exact can be checked after the faithfully flat base change along $A\rightarrow A_\infty$. But\label{footnote:FaithfullyFlatMapOfLambdaRings}
		\begin{equation*}
			\left(-\otimes_{A,\psi^m}A\right)\otimes_AA_\infty\cong -\otimes_{A,\psi^m}A_\infty\cong -\otimes_AA_\infty
		\end{equation*}
		as $A_\infty$ is perfect, so we can conclude since $A\rightarrow A_\infty$ is flat. This shows that $\psi^m\colon A\rightarrow A$ is flat. The same argument shows faithfulness.}%
		, then all the nice properties we proved about $\qIW_m(-)$ in \cref{subsec:qWitt} will also hold for $\qIW_m(-/A)$, since we can deduce them via \cref{lem:RelativeqWittBaseChange} and faithfully flat descent. For example, it will be true that the Verschiebungen $V_{m/d}\colon \qIW_d(R/A)\rightarrow \qIW_m(R/A)$ are injective, the analogue of \cref{prop:qWittKoszulExactSequence} is true, and if $R$ is \emph{relatively perfect over $A$}, meaning that the linearised Adams operations $\psi_{/A}^p\colon R\otimes_{A,\psi^p}A\rightarrow R$ are isomorphisms for all $p$, then the comparison maps $s_{m/A}$ and $c_{m/A}$ from \cref{par:RelativeComparisonMaps} are isomorphisms.
		
		$\Lambda$-rings with the property that $A\rightarrow A_\infty$ is faithfully flat will be called \emph{perfectly covered}. In most real life situations, the Adams operations $\psi^m\colon A\rightarrow A$ will be faithfully flat, hence $A$ will be perfectly covered. $A$ being perfectly covered will also be a crucial assumption in our eventual applications \cite{qHodge,qWittHabiro}. Still, it seems believable that even without this assumption the analogue of \cref{prop:qWittKoszulExactSequence} is true (from which all other desired properties could easily be deduced). To prove this, the crucial step would be to show injectivity of the Verschiebung $V_p\colon \qIW_{p^{\alpha-1}}(R/A)\rightarrow \qIW_{p^\alpha}(R/A)$. But it's not clear (at least to the author) how the proof of \cref{lem:qWittpTypicalExactSequence} could be adapted.
	\end{rem}

	\subsection{\texorpdfstring{$q$}{q}-Witt vectors and étale morphisms}\label{subsec:qWittEtale}
	The goal of this subsection is to prove the following proposition, which is a $q$-Witt vector analogue of results obtained by van der Kallen \cite[Theorem~(2.4)]{VanDerKallen}, Langer--Zink \cite[Corollary~\href{https://www.math.uni-bielefeld.de/~zink/dRW.pdf\#page=104}{A.18}]{LangerZink}, and Borger \cite[Theorem~\href{https://arxiv.org/pdf/0801.1691\#subsection.9.2}{9.2}]{Borger}.
	\begin{prop}\label{prop:vanDerKallen}
		Let $A$ be a $\Lambda$-ring, let $R\rightarrow R'$ be an étale morphism of $A$-algebras, and let $m$ be a positive integer. Then $\qIW_m(R/A)\rightarrow \qIW_m(R'/A)$ is étale again. Furthermore, if $d\mid m$, then
		\begin{equation*}
			\qIW_m(R'/A)\otimes_{\qIW_m(R/A)}\qIW_d(R/A)\overset{\cong}{\longrightarrow} \qIW_d(R'/A)
		\end{equation*}
		is an isomorphism, where the tensor product is taken with respect to the Frobenius map $F_{m/d}\colon \qIW_m(R/A)\rightarrow \qIW_{m/d}(R/A)$.
	\end{prop}
	\begin{rem}\label{rem:FrobeniusPushoutNotInLiterature}
		Let us indicate how the ordinary Witt vector analogue of \cref{prop:vanDerKallen} follows from the literature. For the étaleness of $\qIW_m(R)\rightarrow \qIW_m(R')$, this is clear, but the assertion that
		\begin{equation*}
			\IW_m(R')\otimes_{\IW_m(R),F_{m/d}}\IW_d(R)\overset{\cong}{\longrightarrow}\IW_d(R')
		\end{equation*}
		is either stated only for $p$-typical Witt vectors (under the assumption that $R$ and $R'$ are \emph{$F$-finite}) or only for the tensor product with respect to the restriction map $\operatorname{Res}_{m/d}\colon \IW_m(R)\rightarrow \IW_d(R)$. Nevertheless, the general case is true and can be deduced as follows. It's enough to consider the case where $m/d=p$ is a prime. Write $m=p^\alpha n$, where $\alpha=v_p(m)$. By the $p$-typical case, as stated in \cite[Corollary~\href{https://www.math.uni-bielefeld.de/~zink/dRW.pdf\#page=104}{A.18}]{LangerZink}, the diagram
		\begin{equation*}
			\begin{tikzcd}
				\IW_{p^\alpha}(R)\rar\drar[pushout]\dar["F_p"']& \IW_{p^\alpha}(R')\dar["F_p"]\\
				\IW_{p^\alpha}(R)\rar & \IW_{p^\alpha}(R')
			\end{tikzcd}
		\end{equation*}
		is a pushout diagram of rings, provided  that $R$ and $R'$ are \emph{$F$-finite}. By writing $R\rightarrow R'$ as a filtered colimit of étale morphism between rings of finite type over $\IZ$ (which are $F$-finite), we see that the diagram above is a pushout in general. Furthermore, \cite[Theorem~\href{https://arxiv.org/pdf/0801.1691\#subsection.9.2}{9.2}]{Borger} shows that the horizontal arrows in the pushout diagram are étale (in the $F$-finite case, this is also proved by Langer--Zink). Now \cite[Corollary~\href{https://arxiv.org/pdf/0801.1691\#subsection.5.4}{5.4}]{Borger} allows us to write $\IW_m(-)\cong \IW_n(\IW_{p^{\alpha}}(-))$ and \cite[Corollary~\href{https://arxiv.org/pdf/0801.1691\#subsection.9.4}{9.4}]{Borger} shows that the functor $\IW_n(-)$ preserves pushouts in which one leg is étale. This proves what we want.
	\end{rem}
	The crucial ingredient in the proof of \cref{prop:vanDerKallen} is the following.
	\begin{lem}\label{lem:qWittEtaleBaseChange}
		Let $A$ be a $\Lambda$-ring, let $R\rightarrow R'$ be an étale morphism of $A$-algebras, and let $m$ be a positive integer. Then we get a canonical isomorphism
		\begin{equation*}
			\IW_m(R')\otimes_{\IW_m(R)}\qIW_m(R/A)\overset{\cong}{\longrightarrow}\qIW_m(R'/A)\,.
		\end{equation*}
	\end{lem}
	\begin{proof}
		We start with the case $A=\IZ$ (that is, the case of absolute $q$-Witt vectors). We define
		\begin{equation*}
			M\coloneqq \bigoplus_{d\mid m}\IW_d(R)[q]\quad\text{and}\quad N\coloneqq \bigoplus_{e\mid d\mid m}\IW_d(R)[q]\,.
		\end{equation*}
		By \cref{def:BigqWitt}, we can write $\qIW_m(R)\cong \coker(M\oplus N\rightarrow \IW_m(R)[q])$, where the map in question is given as follows: For a divisor $d\mid m$, the $d$\textsuperscript{th} component of $M\rightarrow \IW_m(R)[q]$ is given by $(q^d-1)V_{m/d}$, and for a chain of divisors $e\mid d\mid m$, the $(e,d)$\textsuperscript{th} component of $N\rightarrow \IW_m(R)[q]$ is given by $[d/e]_{q^e}V_{m/d}-V_{m/e}F_{d/e}$. Note that all of these are morphisms of $\IW_m(R)[q]$-modules, if we equip $\IW_{m/d}(R)[q]$ with the module structure obtained through the Frobenius $F_{m/d}\colon \IW_m(R)[q]\rightarrow \IW_{d}(R)[q]$.
		
		Similarly, $\IW_m(R')\cong \coker(M'\oplus N'\rightarrow \IW_m(R')[q]$), where $M'$ and $N'$ are defined as above, but with $R$ replaced by $R'$. The discussion in \cref{rem:FrobeniusPushoutNotInLiterature} shows that $M'\cong \IW_m(R')\otimes_{\IW_m(R)}M$ and $N'\cong \IW_m(R')\otimes_{\IW_m(R)}N$, which immediately yields $\IW_m(R')\otimes_{\IW_m(R)}\qIW_m(R)\cong \qIW_m(R')$, as claimed.
		
		The proof in the relative case is analogous. Let
		\begin{equation*}
			K\coloneqq \bigoplus_{d\mid m}\qIW_d(R)\otimes_{\IZ[q]} \qIW_m(A)\otimes_{\IZ[q]} A[q]\,.
		\end{equation*}
		Then \cref{lem:RelativeqWitt} shows that $\qIW_m(R/A)\cong \coker(K\rightarrow \qIW_m(R)\otimes_{\qIW_m(A)}A[q]/(q^m-1))$, where the map is given is given as follows: On the $d$\textsuperscript{th} component, we send $x\otimes y\otimes a$ to $V_{m/d}(xy)\otimes a-V_{m/d}(x)\otimes c_m(y)a$. Similarly, we can describe $\qIW_m(R'/A)$ as a cokernel $\coker(K'\rightarrow \qIW_m(R')\otimes_{\qIW_m(A)}A[q]/(q^m-1))$. Since we've already proved the absolute case, we find $K'\cong \IW_m(R')\otimes_{\IW_m(R)}K$ and the claim follows.
	\end{proof}
	\begin{proof}[Proof of \cref{prop:vanDerKallen}]
		Both assertions follow immediately from \cref{lem:qWittEtaleBaseChange} plus the analogous assertions for ordinary Witt vectors, which hold true as explained in \cref{rem:FrobeniusPushoutNotInLiterature}.
	\end{proof}
	We'll present two applications of \cref{prop:vanDerKallen}. The first one is a similar pushout result for ghost maps.
	\begin{cor}\label{cor:qWittGhostMapsPushout}
		If $A$ is a $\Lambda$-ring, $R\rightarrow R'$ is an étale map of $A$-algebras, and $m$ is a positive integer, then
		\begin{equation*}
			\begin{tikzcd}
				\qIW_{m}(R/A)\rar\drar[pushout]\dar["\gh_{m/d}"']& \qIW_{m}(R'/A)\dar["\gh_{m/d}"]\\
				R\otimes_{A,\psi^d}A[\zeta_d]\rar & R'\otimes_{A,\psi^d}A[\zeta_d]
			\end{tikzcd}
		\end{equation*}
		is a pushout diagram of rings \embrace{both in the derived and in the underived sense} for all $d\mid m$.
	\end{cor}
	\begin{proof}
		Using $\gh_{m/d}=\gh_{d/d}\circ F_{m/d}$ and \cref{prop:vanDerKallen}, we may assume $m=d$. Then $\gh_{m/m}\colon \qIW_m(R/A)\rightarrow (R\otimes_{A,\psi^m}A)[q]/\Phi_m(q)$ is identified with the projection map 
		\begin{equation*}
			\qIW_m(R/A)\longrightarrow \coker\bigl(M\rightarrow \qIW_m(R/A)\bigr)\,,
		\end{equation*}
		where $M\coloneqq \bigoplus_{d\mid m}\qIW_d(R/A)$ and the map is given by $(V_{m/d})_{d\mid m}$. Likewise, the ghost map for $R'$ is given by a similar  projection map $\qIW_m(R'/A)\rightarrow \coker(M'\rightarrow \qIW_m(R'/A))$. \cref{prop:vanDerKallen} implies $M'\cong \qIW_m(R'/A)\otimes_{\qIW_m(R/A)}M$ and we're done. 
	\end{proof}
	As a second application, we prove a partial inverse of \cref{cor:qWittOfPerfectLambdaRing}. This result won't be needed again, but it's pretty convenient as a sanity check: It often appears on first glance that our $q$-Witt vectors (or later our $q$-de Rham--Witt complexes) are trivial in the sense of $\qIW_m(R/A)\cong R[q]/(q^m-1)$. The following result shows that already when $A=\IZ$ and $R$ is étale, this is not at all the case.
	\begin{cor}\label{cor:qWittTrivialisationGivesFrobeniusLift}
		Let $p$ be a prime and let $R$ be an étale $\IZ$-algebra such that $R\rightarrow \widehat{R}_p$ is injective \embrace{equivalently, $p$ is not invertible on any connected component of $\Spec R$}. If a $\IZ[q]$-algebra isomorphism
		\begin{equation*}
			\psi\colon \qIW_m(R)\overset{\cong}{\longrightarrow} R[q]/(q^m-1)
		\end{equation*}
		exists for some positive integer $m$ divisible by $p$, then the unique Frobenius lift $\phi_p\colon \widehat{R}_p\rightarrow \widehat{R}_p$ restricts to a morphism $\phi_p\colon R\rightarrow R$. Furthermore, the $\phi_p$ commute for different $p$ and $R$ can be equipped with a $\Lambda_m$-structure.
	\end{cor}
	\begin{proof}
		Recall from \cref{cor:qWittOfPerfectLambdaRing} that $\qIW_m(\IZ)\cong \IZ[q]/(q^m-1)$; furthermore, by the commutative diagrams from \cref{cor:qWittToCyclicRing}, this isomorphism identifies the Frobenius $F_{m/p}$ with the canonical projection $\IZ[q]/(q^m-1)\rightarrow \IZ[q]/(q^p-1)$. Then \cref{prop:vanDerKallen} implies $\qIW_p(R)\cong \qIW_m(R)/(q^p-1)$. In particular, if an isomorphism $\psi$ as above exists, then there exists an isomorphism $\qIW_p(R)\cong R[q]/(q^p-1)$ too. So we may as well assume $m=p$.
		
		The isomorphism $\qIW_p(\IZ)\cong \IZ[q]/(q^p-1)$ from \cref{cor:qWittOfPerfectLambdaRing} identifies the ghost maps $\gh_1$ and $\gh_p$ with the canonical projections $\IZ[q]/(q^p-1)\rightarrow \IZ[\zeta_p]$ and $\IZ[q]/(q^p-1)\rightarrow \IZ[\zeta_1]=\IZ$, respectively. The pushouts of $\IZ[q]/(q^p-1)\rightarrow R[q]/(q^p-1)$ along these maps are $R[\zeta_p]$ and $R$, respectively. But \cref{cor:qWittGhostMapsPushout} tells us that these pushouts can also be identified with the ghost maps for $\qIW_p(R)$. We thus obtain unique $\IZ[q]$-algebra automorphisms
		\begin{equation*}
			\psi_1\colon R[\zeta_p]\overset{\cong}{\longrightarrow}R[\zeta_p]\quad\text{and}\quad \psi_p\colon R\overset{\cong}{\longrightarrow}R
		\end{equation*}
		such that the following diagrams commute:
		\begin{equation*}
			\begin{tikzcd}
				\qIW_p(R)\rar["\gh_1"]\dar["\psi","\cong"'] & R[\zeta_p]\dar["\psi_1","\cong"']\\
				R[q]/(q^p-1)\rar & R[\zeta_p]
			\end{tikzcd}\quad\text{and}\quad
			\begin{tikzcd}
				\qIW_p(R)\rar["\gh_p"]\dar["\psi","\cong"'] & R\dar["\psi_p","\cong"']\\
				R[q]/(q^p-1)\rar & R
			\end{tikzcd}
		\end{equation*}
		By composing $\psi$ with $\psi_p^{-1}\otimes \operatorname{id}\colon R[q]/(q^p-1)\cong R\otimes_\IZ\IZ[q]/(q^p-1)\rightarrow R\otimes_\IZ\IZ[q]/(q^p-1)$ we may assume $\psi_p=\operatorname{id}$. Now let $x=(x_1,x_p)\in \IW_p(R)$ be an element written in Witt vector coordinates; we view $x$ as an element of $\qIW_p(R)$ as well. Then the left commutative diagram shows $\psi(x)\equiv \psi_1\gh_1(x)\equiv \psi_1(x_1)\mod \Phi_p(q)$ and similarly the right one shows $\psi(x)\equiv \psi_p\gh_p(x)\equiv x_1^p+px_p\mod (q-1)$ since we assume $\psi_p=\operatorname{id}$. Thus
		\begin{equation*}
			\psi_1(x_1)\equiv x_1^p\mod (\zeta_p-1)\,;
		\end{equation*}
		in other words, $\psi_1$ induces the Frobenius on $R[\zeta_p]/(\zeta_p-1)\cong R/p$. Since $R$ is étale over $\IZ$, this property, together with $\IZ[q]$-linearity, uniquely determines the morphism induced by $\psi_1$ on the $(\zeta_p-1)$-adic completion $R[\zeta_p]_{(\zeta_p-1)}^\complete$. This completion coincides with $\widehat{R}_p[\zeta_p]\cong \widehat{R}_p\otimes_\IZ\IZ[\zeta_p]$. Now 
		\begin{equation*}
			\phi_p\otimes \operatorname{id}\colon \widehat{R}_p\otimes_\IZ\IZ[\zeta_p]\longrightarrow \widehat{R}_p\otimes_\IZ\IZ[\zeta_p]
		\end{equation*}
		also restricts to the Frobenius modulo $(\zeta_p-1)$. Hence it coincides with $\psi_1$ and must therefore map the subring $R\otimes_\IZ\IZ[\zeta_p]\subseteq \widehat{R}_p\otimes_\IZ\IZ[\zeta_p]$ into itself. But $\phi_p\otimes\operatorname{id}$ also respects the decomposition $\widehat{R}_p\otimes_\IZ\IZ[\zeta_p]\cong \bigoplus_{i=0}^{p-2}\zeta_p^i\widehat{R}_p$, hence $\phi_p$ must restrict to an endomorphism of $R$, as claimed.
		
		Now let $\ell\neq p$ be another prime factor of $m$ such that $R\rightarrow \widehat{R}_\ell$ is injective. Then $\phi_\ell\colon R\rightarrow R$ induces an endomorphism of $\widehat{R}_p$ as well, and it's enough to show $\phi_p\circ \phi_\ell=\phi_\ell\circ \phi_p$ as endomorphisms of $\widehat{R}_p$. By $p$-complete étaleness, we can further reduce to checking this on $R/p$. But then everything becomes obvious because any ring endomorphism of $R/p$ commutes with the Frobenius.
		
		To prove that $R$ can be equipped with a $\Lambda_m$-structure, observe that the above construction allows us to define commuting Frobenius lifts on $R$ for all primes $p\mid m$. Indeed, on those components of $\Spec R$ where $p$ is not invertible, we can use the construction above, and on the other components we can simply take the identity.
	\end{proof}

	\newpage
	\section{\texorpdfstring{$q$}{q}-de Rham--Witt complexes}\label{sec:qDRW}
	There are several objects that people call \emph{de Rham--Witt \embrace{pro-}complex}: The original construction \cite{Illusie} due to Illusie, building on work of Bloch, Deligne, and Lubkin, defines a pro-complex $(W_n\Omega_R^*)_{n\geqslant 1}$ for any $\IF_p$-algebra $R$. Langer--Zink \cite{LangerZink} define a relative de Rham--Witt pro-complex $(W_n\Omega_{R/A}^*)_{n\geqslant 1}$ for any map $A\rightarrow R$ of $\IZ_{(p)}$-algebras. Finally, Hesselholt--Madsen \cite{HesselholtMadsenKTheoryLocalFields,HesselholtBigDeRhamWitt} define an absolute big de Rham--Witt complex for any ring $R$. 
	
	The goal of this section is to study a system of (strictly) graded-commutative differential-graded $A[q]$-algebras $(\qIW_m\Omega_{R/A}^*)_{m\in\IN}$, which we call \emph{truncated $q$-de Rham--Witt complexes of $R$}, for any $\Lambda$-ring $A$ and any $A$-algebra $R$. Even though our construction naturally works with big Witt vectors, even in the case $A=\IZ$ it's a much closer analogue of \cite{LangerZink} than of Hesselholt--Madsen's absolute big de Rham--Witt complex. It seems possible that by dropping the $\IZ[q]$-linearity of the differentials, one can obtain a $q$-analogue $\qIW_m\Omega_R^*$ of Hesselholt--Madsen's absolute construction, but we haven't pursued this so far.
	
	Throughout this section, we work relative to a fixed $\Lambda$-ring $A$%, for which we assume there exists a faithfully flat morphism of $\Lambda$-rings $A\rightarrow A_\infty$ into a perfect $\Lambda$-ring (see \cref{rem:FaithfullyFlatCoverByPerfectLambdaRing} for a discussion)
	. The most important case is $A=\IZ$, the additional generality is only needed for our eventual applications \cite{qWittHabiro,qHodge}.
	
	%This section is divided into two subsections. In \cref{subsec:qDeRhamWittDefinitions} we'll define $(\qIW_m\Omega_R^*)_{m\in\IN}$ and discuss several equivalent constructions. In \cref{subsec:qDeRhamWittSmooth}, we'll construct a quasi-isomorphism
	%\begin{equation*}
	%	\left(\qIW_m\Omega_R^*\right)_p^\complete\simeq \left(\Omega_R^*\right)_p^\complete\otimes_\IZ\IZ[q]/(q^m-1)
	%\end{equation*}
	%for any smooth $\IZ$-algebra $R$ and any prime $p$. 
	
	\subsection{Three non-equivalent categories}\label{subsec:qDeRhamWittDefinitions}
	Even though Langer--Zink's construction generalises Illusie's, they proceed in a slightly different way: For Langer--Zink, the Frobenius operators are part of the definition of $(W_n\Omega_{R/A})_{n\geqslant 1}$, whereas Illusie only constructs them a posteriori. Similarly, we have the choice of whether or not to include Frobenii in our definition of $(\qIW_m\Omega_{R/A}^*)_{m\in\IN}$. Both definitions turn out to be equivalent, or rather they lead to two non-equivalent categories (\cref{def:qVSystemOfCDGA,def:qFVSystemOfCDGA}) which happen to have the same initial object (as we'll see in \cref{prop:qDRWHasFrobenii}). What makes things even more confusing is that in the case where $R$ is smooth over $\IZ$, there is a third category (\cref{def:qVSystemTorsionFree}) in which $(\qIW_m\Omega_{R/A}^*)_{m\in\IN}$ is initial (as we'll see in \cref{prop:qDRWisTorsionFreeForSmoothRings}).
	
	To alleviate this confusion, let us first carefully introduce these three different categories. We begin with the variant without Frobenii.
	\begin{defi}\label{def:qVSystemOfCDGA}
		Fix an $A$-algebra $R$. A \emph{$q$-$V$-system of differential-graded $A$-algebras over $R$} is a system $(P_m^*)_{m\in\IN}$ of commutative differential-graded $A[q]$-algebras, equipped with the following additional structure:
		\begin{alphanumerate}
			\item For all $m\in \IN$, an $A[q]$-algebra morphism $\qIW_m(R/A)\rightarrow P_m^0$.\label{enum:qDeRhamWittConditionA}
			\item For all divisors $d\mid m$, a morphism $V_{m/d}\colon P_d^*\rightarrow P_m^*$ of graded $A[q]$-modules. These are required to be compatible with the Verschiebungen on relative $q$-Witt vectors (via the morphisms from \cref{enum:qDeRhamWittConditionA}) and must satisfy $V_{m/e}=V_{m/d}\circ V_{d/e}$ for all chains of divisors $e\mid d\mid m$ as well as $V_{m/d}(\omega\d\eta)=V_{m/d}(\omega)\d V_{m/d}(\eta)$ for all $\omega\in P_d^{i+1}$, $\eta\in P_d^i$.\label{enum:qDeRhamWittConditionB}
		\end{alphanumerate}
		Furthermore, we require that the following \emph{$V$-Teichmüller condition} is satisfied:
		\begin{alphanumerate}
			\item[\tau_V] For all $d\mid m$ and all $\omega\in P_d^*$, $r\in R$, one has\label{enum:TeichmuellerV}
			\begin{equation*}
				V_{m/d}(\omega)\d\tau_m(r)=V_{m/d}\bigl(\omega\tau_d(r)^{m/d-1}\bigr)\d V_{m/d}\tau_d(r)\,.
			\end{equation*}
			Here $\tau_m(r)\in\qIW_m(R/A)$ and $\tau_d(r)\in\qIW_d(R/A)$ denote the respective Teichmüller lifts, which we also implicitly identify with their images in $P_m^0$ and $P_d^0$, respectively.
		\end{alphanumerate}
		There is an obvious category of $q$-$V$-systems of differential-graded $A$-algebras over $R$, which we denote $\cat{CDGAlg}_{R/A}^{\q V}$.
	\end{defi}
	\begin{lem}\label{lem:dAfterV}
		Let $R$ be an $A$-algebra. In any $q$-$V$-system $(P_m^*)_{m\in\IN}$ over $R$, the Verschiebungen satisfy the relation
		\begin{equation*}
			V_{n}\circ \d=n(\d \circ V_n)\,.
		\end{equation*}
	\end{lem}
	\begin{proof}
		 We write $n=m/d$ and use the second condition from \cref{def:qVSystemOfCDGA}\cref{enum:qDeRhamWittConditionB} to obtain $V_{m/d}(\d\omega)=V_{m/d}(1)\d V_{m/d}(\omega)=[m/d]_{q^d}\d V_{m/d}(\omega)$ for all $\omega\in P_d^*$. Here we also used that $V_{m/d}(1)=[m/d]_{q^d}$ holds in $\qIW_m(R/A)$. But $\omega$ and thus $\d V_{m/d}(\omega)$ are $(q^d-1)$-torsion elements, hence multiplication by $[m/d]_{q^d}$ agrees with multiplication by $m/d$.
	\end{proof}
	\begin{numpar}[$V$-Divided powers.]
		In classical Witt vector theory, one often considers a divided power structure on the ideal generated by the Verschiebung. In general, this can't be done for our $q$-Witt vectors $\qIW_m(R)$, since we're working with big rather than $p$-typical Witt vectors and we don't assume that $R$ is a $\IZ_{(p)}$-algebra. However, if $p$ is a prime factor of $m$, then there's still a well-defined map $\gamma_p\colon \im V_p\rightarrow \im V_p$ sending $V_p(x)$ to $p^{p-2}V_p(x^p)$ (here we use that $V_p$ is injective by \cref{cor:VerschiebungInjective}% and \cref{rem:FaithfullyFlatCoverByPerfectLambdaRing}
		); it satisfies $p\gamma_p(v)=v^{p}$ for all $v\in \im V_p$. We then say that a derivation $\d\colon \qIW_m(R)\rightarrow M$ is a \emph{$V$-PD-derivation} if $\d\gamma_p(v)=v^{p-1}\d v$.
	\end{numpar}
	\begin{lem}\label{lem:VPDderivation}
		Let $R$ be an $A$-algebra. For any $q$-$V$-system $(P_m^*)_{m\in\IN}$ over $R$ and all $m$, the composition $\qIW_m(R)\rightarrow\qIW_m(R/A)\rightarrow P_m^0\rightarrow P_m^1$ is a $V$-PD-derivation. In fact, for any prime factor $p\mid m$ and all $x\in \qIW_{m/p}(R/A)$, we have the stronger condition
		\begin{equation*}
			\d V_p(x^p)=V_p(x^{p-1})\d V_p(x)\,.
		\end{equation*}
	\end{lem}
	\begin{proof}
		The proof is mostly the same as in \cite[Lemma~\href{https://www.math.uni-bielefeld.de/~zink/dRW.pdf\#page=20}{1.5}]{LangerZink}. We use induction on $m$, the case $m=1$ being trivial. So let $m>1$. We proceed in three steps: Step~1 is to prove the relation in the case $x=a\tau_{m/p}(r)$ for some $a\in A[q]$ and some $r\in R$. Step~2 is to prove that if the relation is satisfied for $x=x_1$ and $x=x_2$, then it's satisfied for $x=x_1+x_2$ as well. Step~3 is prove the relation in the case $x=V_\ell(y)$ for some prime factor $\ell\mid m$ (including $\ell=p$) and some $y\in \qIW_{m/p}(R)$.
		
		For the first two steps, we can copy Langer--Zink's proof, except that in the first step, Langer and Zink use the $F$-Teichmüller condition (see \cref{enum:TeichmuellerF} below), but the argument works equally well with the $V$-Teichmüller condition \cref{enum:TeichmuellerV}. For the third step, apply $\ell^{p-2}V_\ell$ to both sides of $\d V_p(y^p)=V_p(y^{p-1})\d V_p(y)$ (which we know from the inductive hypothesis). Using $V_\ell\circ \d=\ell(\d\circ V_\ell)$ by \cref{lem:dAfterV} and $V_\ell(y)^p=\ell^{p-1}V_\ell(y)$, the left-hand side becomes
		\begin{equation*}
			\ell^{p-2}V_\ell\bigl(\d V_p(y^p)\bigr)=\ell^{p-1}\d V_\ell V_p(y^p)=\d V_p\left(\ell^{p-1}V_\ell(y^p)\right)=\d V_p\bigl(V_\ell(y)^p\bigr)\,.
		\end{equation*}
		In a similar way, the right-hand side becomes
		\begin{equation*}
			\ell^{p-2}V_\ell\left(V_p(y^{p-1})\d V_p(y)\right)=\ell^{p-2}V_\ell\left(V_p(y^{p-1})\right)\d V_\ell V_p(y)=V_p\left(V_\ell(y)^{p-1}\right)\d V_pV_\ell(y)\,.
		\end{equation*}
		This finishes Step~3, the induction, and the proof.
	\end{proof}
	\begin{rem}
		It turns out that an even stronger version of \cref{lem:VPDderivation} is true: If $d\mid m$ is any divisor, and $x\in \qIW_d(R/A)$, then
		\begin{equation*}
			\d V_{m/d}\bigl(x^{m/d}\bigr)=V_{m/d}\bigl(x^{m/d-1}\bigr)\d V_{m/d}(x)\,.
		\end{equation*}
		The author doesn't know how to generalise the proof of \cref{lem:VPDderivation}. Instead, one can argue as follows: It suffices to prove this for the universal $q$-$V$-system over $R$. We'll see that such a thing exists in \cref{prop:qDRWExists}; furthermore, this proposition shows that we can reduce to the case where $R$ is a polynomial ring over $A$ (possibly in infinitely many variables). Furthermore, it will follow from \cref{prop:qDRWisTorsionFreeForSmoothRings} and passing to filtered colimits that the universal $q$-$V$-system over a polynomial ring is degree-wise $\IZ$-torsion-free. But the relation in question is easily seen to be true after multiplication with $(m/d)^{m/d-1}$, so we're done. Note that we couldn't have used this argument in the first place, since we'll need \cref{lem:VPDderivation} (in its weak form) to prove \cref{prop:qDRWisTorsionFreeForSmoothRings}.
	\end{rem}
	Next we define the variant that has Frobenii.
	\begin{defi}\label{def:qFVSystemOfCDGA}
		Fix an $A$-algebra $R$. A \emph{$q$-$FV$-system of differential-graded $A$-algebras over $R$} is a $q$-$V$-system $(P_m^*)_{m\in\IN}$ as in \cref{def:qVSystemOfCDGA} together with the following additional structure:
		\begin{alphanumerate}
			\item[c] For all $d\mid m$, a morphism $F_{m/d}\colon P_m^*\rightarrow P_d^*$ of graded $A[q]$-algebras.\label{enum:qDeRhamWittConditionC} These are required to be compatible with the Frobenius maps on $q$-Witt vectors (via the morphisms from \cref{def:qVSystemOfCDGA}\cref{enum:qDeRhamWittConditionA}) and must satisfy $F_{m/e}=F_{d/e}\circ F_{m/d}$ for all chains of divisors $e\mid d\mid m$. Furthermore, they must interact with the Verschiebungen in the following way:
			\begin{equation*}
				F_{m/d}\circ \d \circ V_{m/d}=\d\quad\text{and}\quad V_{m/d}\bigl(\omega F_{m/d}(\eta)\bigr)=V_{m/d}(\omega)\eta
			\end{equation*}
			for all $\omega\in P_d^*$, $\eta\in P_m^*$. Moreover, $F_n$ must commute with $V_k$ whenever $n$, $k$ are coprime. Finally, we must have the familiar relations
			\begin{equation*}
				F_{m/d}\circ V_{m/d}=m/d\quad\text{and}\quad V_{m/d}\circ 	F_{m/d}=[m/d]_{q^{d}}\,.
			\end{equation*}
		\end{alphanumerate}
		Last but not least, we require that the following \emph{$F$-Teichmüller condition} is satisfied:
		\begin{alphanumerate}
			\item[\tau_F] For all $d\mid m$ and all $r\in R$, one has\label{enum:TeichmuellerF}
			\begin{equation*}
				F_{m/d}\bigl(\d\tau_{m}(r)\bigr)=\tau_d(r)^{m/d-1}\d\tau_d(r)\,.
			\end{equation*}
			Here $\tau_m(r)\in\qIW_m(R/A)$ and $\tau_d(r)\in\qIW_d(R/A)$ denote the respective Teichmüller lifts, which we also implicitly identify with their images in $P_m^0$ and $P_d^0$, respectively.
		\end{alphanumerate}
		There is an obvious category of $q$-$FV$-systems of differential-graded algebras over $R$, which we denote $\cat{CDGAlg}_{R/A}^{\q FV}$, and an obvious forgetful functor $\cat{CDGAlg}_{R/A}^{\q FV}\rightarrow \cat{CDGAlg}_{R/A}^{\q V}$.
	\end{defi}
	\begin{lem}\label{lem:dAfterF}
		Let $R$ be an $A$-algebra. In any $q$-$FV$-system $(P_m^*)_{m\in\IN}$ over $R$, the Frobenii satisfy the relation
		\begin{equation*}
			\d \circ F_n=n(\d\circ V_n)\,.
		\end{equation*}
	\end{lem}
	\begin{proof}
		We write $n=m/d$ and use the conditions from \cref{def:qFVSystemOfCDGA}\cref{enum:qDeRhamWittConditionC} to compute that $\d F_{m/d}(\omega)=F_{m/d}(\d V_{m/d}F_{m/d}(\omega))=[m/d]_{q^d}F_{m/d}(\d\omega)$ for all $\omega\in P_m^*$. But this computation takes place in $P_d^*$, which is $(q^d-1)$-torsion, so multiplication by $[m/d]_{q^d}$ and by $m/d$ agree.
	\end{proof}
	\begin{rem}\label{rem:Redundancies}
		In \cref{def:qFVSystemOfCDGA}, the condition that $(P_m^*)$ be a $q$-$V$-system was added for simplicity, but it's partially redundant. As explained after \cite[Definition~\href{https://www.math.uni-bielefeld.de/~zink/dRW.pdf\#page=19}{1.4}]{LangerZink}, the condition $V_{m/d}(\omega \d\eta)=V_{m/d}(\omega)\d V_{m/d}(\eta)$ from \cref{def:qVSystemOfCDGA}\cref{enum:qDeRhamWittConditionB} is easily implied by the conditions from \cref{def:qFVSystemOfCDGA}\cref{enum:qDeRhamWittConditionC}, and the $V$-Teichmüller condition \cref{enum:TeichmuellerV} follows easily from this and the $F$-Teichmüller condition \cref{enum:TeichmuellerF}.
	\end{rem}
	Finally, we introduce the variant that only becomes relevant for smooth $\IZ$-algebras.
	\begin{defi}\label{def:qVSystemTorsionFree}
		Fix an $A$-algebra $R$. A \emph{torsion-free $q$-$V$-system of differential-graded $A$-algebras over $R$} is a system $(P_m^*)$ of degree-wise $\IZ$-torsion-free differential-graded $A[q]$-algebras, equipped with the additional structure from \cref{def:qVSystemOfCDGA}\cref{enum:qDeRhamWittConditionA} and~\cref{enum:qDeRhamWittConditionB}. The corresponding category will be denoted $(\cat{CDGAlg}_{R/A}^{\q V})^{\mathrm{tors\mhyph free}}$.
	\end{defi}
	\begin{rem}\label{def:TorsionFreeMakesStuffRedundant}
		Note that every torsion-free $q$-$V$-system is also a $q$-$V$-system. Indeed, in general, the $V$-Teichmüller condition \cref{enum:TeichmuellerV} always holds up to $(m/d)^{m/d-1}$-torsion, so it's automatically true in the $\IZ$-torsion-free case. In particular, there is a fully faithful forgetful functor $(\cat{CDGAlg}_{R/A}^{\q V})^{\mathrm{tors\mhyph free}}\rightarrow \cat{CDGAlg}_{R/A}^{\q V}$.
	\end{rem}
	\begin{numpar}[A theory without restrictions (again).]\label{par:WhoNeedsRestrictions}
		Observe that we do not include any restriction maps in \cref{def:qVSystemOfCDGA,def:qFVSystemOfCDGA,def:qVSystemTorsionFree}.\footnote{Which is also why we can't use the term \emph{$FV$-pro-complex} as in \cite[Definition~1.4]{LangerZink}---at best, in the presence of Frobenii, we get a pro-system of graded $A[q]$-algebras, but never of complexes.} This is of course necessitated by the fact that there are no restrictions for $q$-Witt vectors. 
		
		Surprisingly though, restrictions are also not needed for Langer--Zink's construction! Indeed, let $R\rightarrow S$ be any map of $\IZ_{(p)}$-algebras and let $\cat{CDGAlg}_{S/R}^{FV,\,(p)}$ be the category of \emph{$FV$-pro-complexes $(P_n^*)_{n\geqslant 1}$ over the $R$-algebra $S$} as in \cite[Definition~\href{https://www.math.uni-bielefeld.de/~zink/dRW.pdf\#page=18}{1.4}]{LangerZink}, but without the restriction maps $P_{n+1}^*\rightarrow P_n^*$ (this is of course an informal definition; we leave it to the reader to formalise it). Then the de Rham--Witt pro-complex $(W_n\Omega_{S/R}^*)_{n\geqslant 1}$ is still initial in the category $\cat{CDGAlg}_{S/R}^{FV,\,(p)}$. To see this, just skim through \cite[\S\href{https://www.math.uni-bielefeld.de/~zink/dRW.pdf\#page=22}{1.3}]{LangerZink} and note that compatibility with the restrictions is never enforced. We'll see in \cref{rem:DRWtoqDRW} below that this observation provides us with a map from Langer--Zink's de Rham--Witt complexes to our $q$-de Rham--Witt complexes. 
	\end{numpar}
	
	\subsection{Construction of \texorpdfstring{$q$}{q}-de Rham--Witt complexes}
	%Next, we're going to state the main propositions for this subsection. \cref{prop:qDRWExists} will be proved at the end of this subsection, \cref{prop:qDRWHasFrobenii} will be proved in \cref{subsec:ConstructionOfFrobenii}, and \cref{prop:qDRWisTorsionFreeForSmoothRings} in \cref{subsec:qDeRhamWittSmooth}.
	In this subsection we'll construct $(\qIW_m\Omega_{R/A}^*)_{m\in\IN}$ and derive some first properties. Our goal is to prove the following proposition.
	\begin{prop}\label{prop:qDRWExists}
		Let $R$ be an $A$-algebra. The category $\cat{CDGAlg}_{R/A}^{\q V}$ has an initial object $(\qIW_m\Omega_{R/A}^*)_{m\in\IN}$, which has the following properties:
		\begin{alphanumerate}
			\item For all $m\in \IN$, the canonical map $\Omega_{\qIW_m(R/A)/A[q]}^*\rightarrow \qIW_m\Omega_{R/A}^*$ is surjective. For $m=1$, it induces an isomorphism $\Omega_{R/A}^*\cong \qIW_1\Omega_{R/A}^*$.\label{enum:dRtoqDRWisSurjective}
			\item For all $m\in\IN$, the structure map from \cref{def:qVSystemOfCDGA}\cref{enum:qDeRhamWittConditionA} induces an isomorphism $\qIW_m(R/A)\cong \qIW_m\Omega_{R/A}^0$.\label{enum:qDRWisqWinDegree0}
		\end{alphanumerate}
	\end{prop}
	\begin{defi}\label{def:qDRW}
		Let $R$ be an $A$-algebra. For all $m\in\IN$, the differential-graded $A[q]$-algebra $\qIW_m\Omega_{R/A}^*$ from \cref{prop:qDRWExists} is called the \emph{$m$-truncated $q$-de Rham--Witt complex of $R$ relative to $A$}.%
		\footnote{Again, we've deviated again from the terminology of \cite[Definition~\href{https://guests.mpim-bonn.mpg.de/ferdinand/q-deRham.pdf\#theorem.5.17}{5.17}]{MasterThesis}, as we're not working in a $(q-1)$-complete setting and also $\IZ$-torsion-freeness is not assumed. It will be apparent from \cref{prop:qDRWisTorsionFreeForSmoothRings} (plus \cref{cor:qWittNoetherian} to ensure that completions behave nicely) that the $q$-de Rham--Witt complexes defined in \cite[Definition~\href{https://guests.mpim-bonn.mpg.de/ferdinand/q-deRham.pdf\#theorem.5.17}{5.17}]{MasterThesis} coincide with $(\qIW_m\Omega_R^*)_{(q-1)}^\complete$, where the completion is taken degree-wise (and it doesn't matter whether we take the underived or the derived completion).}
	\end{defi}
%	\begin{rem}
%		Similar to \cref{rem:NotAsInMasterThesisI}, we've deviated again from the terminology of \cite[Definition~\href{https://guests.mpim-bonn.mpg.de/ferdinand/q-deRham.pdf\#theorem.5.17}{5.17}]{MasterThesis}, as we're not working in a $(q-1)$-complete setting and also $\IZ$-torsion-freeness is not assumed. It will be apparent from \cref{prop:qDRWisTorsionFreeForSmoothRings} (plus \cref{cor:qWittNoetherian} to ensure that completions behave nicely) that the $q$-de Rham--Witt complexes defined in \cite[Definition~5.17]{MasterThesis} coincide with $(\qIW_m\Omega_R^*)_{(q-1)}^\complete$, where the completion is taken degree-wise (and it doesn't matter whether we take the underived or the derived completion).
%	\end{rem}
	\begin{proof}[Proof of \cref{prop:qDRWExists}]
		We proceed inductively. For $m=1$, we put $\qIW_1\Omega_{R/A}^*\coloneqq\Omega_{R/A}^*$. Now let $m>1$ and assume that we've already constructed $\qIW_d\Omega_{R/A}^*$ for all divisors $d\mid m$, $d\neq m$, satisfying \cref{enum:dRtoqDRWisSurjective} and \cref{enum:qDRWisqWinDegree0}, along with Verschiebungen $V_{d/e}$ for all $e\mid d$ satisfying the conditions from \cref{def:qVSystemOfCDGA}. Now let $\IJ_{m}^*\subseteq \Omega_{\qIW_m(R/A)/A[q]}^*$ be the smallest differential-graded ideal satisfying the following conditions for all divisors $d\mid m$, $d\neq m$:
		\begin{alphanumerate}
			\item[V_d] For\label{enum:GeneratorsOfVType} all $j\geqslant 1$, all finite indexing sets $I$, and all sequences $(w_i,x_{i,1},\dotsc,x_{i,j})_{i\in I}$ of elements of $\qIW_d(R/A)$ such that $0=\sum_{i\in I}w_i\d x_{i,1}\wedge \dotsm\wedge \d x_{i,j}$ holds in $\qIW_d\Omega_{R/A}^j$ (which is a quotient of $\Omega_{\qIW_d(R/A)/A[q]}^j$ by \cref{enum:dRtoqDRWisSurjective}, so the sum makes sense), the following homogeneous degree-$j$ element is contained in $\IJ_{m}^*$:
			\begin{equation*}
				\xi\coloneqq\sum_{i\in I}V_{m/d}(w_i)\d V_{m/d}(x_{i,1})\wedge \dotsm\wedge \d V_{m/d}(x_{i,j})\,.
			\end{equation*}
			\item[\tau_d] For\label{enum:GeneratorsOfTeichmuellerType} all $x\in \qIW_d(R/A)$ and all $r\in R$ (so that $V_{m/d}(x)$ and $V_{m/d}(x\tau_d(r)^{m/d-1})$ are already defined), the following homogeneous degree-$1$ element is contained in $\IJ_{m}^*$:
			\begin{equation*}
				\eta\coloneqq V_{m/d}(x)\d\tau_m(r)-V_{m/d}\bigl(x\tau_d(r)^{{m/d}-1}\bigr)\d V_{m/d}\tau_d(r)\,.
			\end{equation*}
		\end{alphanumerate}
		Explicitly, $\IJ_{m}^*$ is the graded ideal generated by all $\xi$, $\d\xi$, $\eta$, and $\d\eta$, where $\xi$ and $\eta$ are as in \cref{enum:GeneratorsOfVType} and \cref{enum:GeneratorsOfTeichmuellerType}, respectively. We put $\qIW_m\Omega_{R/A}^*\coloneqq \Omega_{\qIW_m(R)/A[q]}^*/\IJ_{m}^*$. It's a differential-graded $A[q]$-algebra by construction. Condition~\cref{enum:GeneratorsOfVType} makes sure that there's a well-defined map $V_{m/d}\colon \qIW_d\Omega_{R/A}^*\rightarrow \qIW_m\Omega_{R/A}^*$ given by the formula
		\begin{equation*}
			V_{m/d}\left(w\d x_1\wedge\dotsm\wedge\d x_j\right)=V_{m/d}(w)\d V_{m/d}(x_1)\wedge\dotsm\wedge\d V_{m/d}(x_j)\,.
		\end{equation*}
		This automatically satisfies the condition from \cref{def:qVSystemOfCDGA}\cref{enum:qDeRhamWittConditionB}. Furthermore, \cref{enum:GeneratorsOfTeichmuellerType} ensures that $\qIW_m\Omega_R^*$ satisfies the $V$-Teichmüller condition \cref{enum:TeichmuellerV} for $\omega=x$ in degree $0$, which easily implies the general case. Finally, it's clear from the construction that conditions~\cref{enum:dRtoqDRWisSurjective} and~\cref{enum:qDRWisqWinDegree0} are true. This finishes the inductive step. It's straightforward to see that $(\qIW_m\Omega_{R/A}^*)_{m\in\IN}$ is really initial in $\cat{CDGAlg}_{R/A}^{\q V}$.
	\end{proof}
	\begin{rem}\label{rem:TruncatedUniversalProperty2}
		As in \cref{rem:TruncatedUniversalPropertyRelative}, the proof of \cref{prop:qDRWExists} shows that a similar universal property also holds for every truncated system: If $S\subseteq \IN$ is any truncation set (in the sense of \cref{par:BigWitt}), we define an \emph{$S$-truncated $q$-$V$-system of differential-graded $A$-algebras over $R$} to be a systems of differential-graded $\IZ[q]$-algebras $(P_m^*)_{m\in S}$ equipped with the structure from \cref{def:qVSystemOfCDGA}\cref{enum:qDeRhamWittConditionA}, \cref{enum:qDeRhamWittConditionB} for all $m\in S$ as well as satisfying the $V$-Teichmüller condition \cref{enum:TeichmuellerV} for all $m\in S$. Then $(\qIW_{m}\Omega_{R/A}^*)_{m\in S}$ is the initial $S$-truncated $q$-$V$-system.
		
		In the case where $S=T_m$ is the set of divisors of $m$, we only need a $\Lambda_m$-structure on $A$ %(plus the assumption that there exists a faithfully flat cover $A\rightarrow A_\infty$ by a perfect $\Lambda_m$-ring) 
		to define $(\qIW_d\Omega_{R/A}^*)_{d\in T_m}$.
	\end{rem}
	\begin{numpar}[Ghost maps.]\label{par:qdRWGhostMaps}
		It turns out that $(\qIW_m\Omega_{R/A}^*)$ comes equipped with maps of differential-graded $A[q]$-algebras
		\begin{equation*}
			\gh_{m/d}\colon \qIW_m\Omega_{R/A}^*\longrightarrow \Omega_{R/A}^*\otimes_{A,\psi^d}A[\zeta_d]
		\end{equation*}
		for all $d\mid m$, generalising the ghost maps for relative $q$-Witt vectors. To construct these maps, it's enough to equip $\bigl(\prod_{d\mid m}(\Omega_{R/A}^*\otimes_{A,\psi^d}A[\zeta_d])\bigr)_{m\in \IN}$ with the structure of a $q$-$V$-system over $R$. 
		
		According to \cref{def:qVSystemOfCDGA}\cref{enum:qDeRhamWittConditionA}, the first piece of structure we must provide are ring maps $\qIW_m(R/A)\rightarrow\prod_{d\mid m}(R\otimes_{A,\psi^d}A[\zeta_d])$ for all $m$. But we can simply take them to be the product $(\gh_{m/d})_{d\mid m}$ of all relative $q$-Witt vector ghost maps; see \cref{par:RelativeGhostMaps}. Furthermore, we have to define Verschiebungen
		\begin{equation*}
			V_{m/n}\colon \prod_{e\mid n}\bigl(\Omega_{R/A}^*\otimes_{A,\psi^e}A[\zeta_e]\bigl)\longrightarrow \prod_{d\mid m}\bigl(\Omega_{R/A}^*\otimes_{A,\psi^d}A[\zeta_d]\bigr)
		\end{equation*}
		for all divisors $n\mid m$. We do this as follows: If $\omega=(\omega_e)_{e\mid n}$ is homogeneous of degree $i$, we let $V_{m/n}(\omega)\coloneqq(V_{m/n}(\omega)_d)_{d\mid m}$, where
		\begin{equation*}
			V_{m/n}(\omega)_d\coloneqq \begin{cases*}
				(m/n)^{i+1}\omega_e & if $d=e\mid n$\\
				0 & else
			\end{cases*}\,.
		\end{equation*}
		This is compatible with $V_{m/n}\colon \qIW_n(R/A)\rightarrow \qIW_m(R/A)$ because of how Witt vector Verschiebungen interact with ghost maps. The other conditions from \cref{def:qVSystemOfCDGA}\cref{enum:qDeRhamWittConditionB} as well as the $V$-Teichmüller condition \cref{enum:TeichmuellerV} are straightforward to check. This finishes the construction.% of the $q$-$V$-system structure on $\bigl(\prod_{d\mid m}\Omega_{R[\zeta_d]/\IZ[\zeta_d]}^*\bigr)_{m\in \IN}$.
	\end{numpar}
	Finally, our relative $q$-de Rham--Witt complexes enjoy a similar base change property as in \cref{lem:RelativeqWittBaseChange}.
	\begin{lem}\label{lem:RelativeqDRWBaseChange}
		If $A\rightarrow A'$ is a morphism of $\Lambda$-rings and $R$ is an $A$-algebra, then for all $m\in\IN$ the canonical map is an isomorphism
		\begin{equation*}
			\qIW_m\Omega_{R/A}^*\otimes_AA'\overset{\cong}{\longrightarrow}\qIW_m\Omega_{R\otimes_AA'/A'}^*\,.
		\end{equation*}
	\end{lem}
	\begin{proof}
		It's straightforward to verify the desired universal property for $\qIW_m\Omega_{R/A}^*\otimes_AA'$. The only non-obvious property is the condition from \cref{def:qFVSystemOfCDGA}\cref{enum:qDeRhamWittConditionA}, that is, the existence of an $A[q]$-algebra map $\qIW_m(R\otimes_AA'/A')\rightarrow \qIW_m(R/A)\otimes_AA'$. But this was taken care of in \cref{lem:RelativeqWittBaseChange}.
	\end{proof}
	\subsection{Construction of Frobenii}\label{subsec:ConstructionOfFrobenii}
	In this subsection we'll prove the following proposition:
	\begin{prop}\label{prop:qDRWHasFrobenii}
		Let $R$ be an $A$-algebra. There is a unique choice of Frobenius operators on $(\qIW_m\Omega_{R/A}^*)_{m\in\IN}$, making it into a $q$-$FV$-system. Moreover, this exhibits $(\qIW_m\Omega_{R/A}^*)_{m\in\IN}$ as an initial object of the category of $q$-FV-systems.
	\end{prop}
	\begin{rem}\label{rem:DRWtoqDRW}
		As a consequence of \cref{par:WhoNeedsRestrictions} and \cref{prop:qDRWHasFrobenii}, we get a comparison map between ordinary and $q$-de Rham--Witt complexes in the case where $A$ a $\IZ_{(p)}$-algebra. Indeed, in this case there's a forgetful functor
		\begin{equation*}
		\cat{CDGAlg}_{R/A}^{\q FV}\longrightarrow \cat{CDGAlg}_{R/\IZ_{(p)}}^{FV,\,(p)}
		\end{equation*}
		sending $(P_m^*)_{m\in\IN}$ to $(P_{p^{n-1}}^*)_{n\geqslant 1}$. By the universal property of usual de Rham--Witt complexes (with the modification from~\cref{par:WhoNeedsRestrictions}), this induces morphisms
		\begin{equation*}
			W_{\alpha+1}\Omega_{R/A}^*\longrightarrow \qIW_{p^\alpha}\Omega_{R/A}^*
		\end{equation*}
		for all $\alpha\geqslant 0$, compatible with Frobenii and Verschiebungen.
	\end{rem}
	\begin{numpar}[Battle plan.]
		Unfortunately, the proof of \cref{prop:qDRWHasFrobenii} will be rather laborious. We construct the Frobenii $F_{m/d}\colon\qIW_m\Omega_{R/A}^*\rightarrow \qIW_d\Omega_{R/A}^*$ using induction on $m$. For $m=1$, there's nothing to do. For the rest of this subsection, let $m>1$ and assume that $F_{d/e}$ has been constructed for all $e\mid d\mid m$, $d\neq m$, in such a way that the conditions from \cref{def:qFVSystemOfCDGA}\cref{enum:qDeRhamWittConditionC} and~\cref{enum:TeichmuellerF} are satisfied. It then suffices to construct $F_p\colon \qIW_m\Omega_{R/A}^*\rightarrow \qIW_{m/p}\Omega_{R/A}^*$ for any prime factor $p\mid m$.
		
		To construct $F_p$ for some fixed prime~$p$, we will proceed as follows. First, we'll restrict to the case $A=\IZ$ and construct a $\IZ[q]$-linear derivation
		\begin{equation*}
			F_p\d\colon \qIW_m(R)\longrightarrow \qIW_{m/p}\Omega_{R/\IZ}^1
		\end{equation*}
		(\cref{con:Fpd} and \crefrange{lem:FrobeniusWelldefinedOnV}{lem:FpdDerivation}). A posteriori, it will turn out that $F_p\d=F_p\circ \d$, but we haven't constructed $F_p$ yet, so for the moment the notation $F_p\d$ has no intrinsic meaning.
		Once $F_p\d$ is constructed, we'll allow $A$ to be an arbitrary $\Lambda$-ring again, construct the desired graded $A[q]$-algebra map $F_p\colon \qIW_m\Omega_{R/A}^*\rightarrow\qIW_{m/p}\Omega_{R/A}^*$ (\cref{con:FrobeniusPreliminary}), and painstakingly verify that it has all desired properties (\crefrange{lem:Fpxi=0}{lem:FpSatisfiesAllProperties}).
	\end{numpar}
	
	\begin{numpar}[The derivation $F_p\d$.]\label{par:PropertiesOfFpd}
		Since we wish to have $F_p\d=F_p\circ \d$ eventually, \cref{def:qFVSystemOfCDGA} already tells us the values of $F_p\d$ in certain cases:
		\begin{alphanumerate}
			\item On Teichmüller lifts, the values of $F_p\d$ are prescribed by \cref{def:qFVSystemOfCDGA}\cref{enum:TeichmuellerF}: We must have\label{enum:FpdOnTeichmueller}
			\begin{equation*}
				F_p\d \tau_m(r)=\tau_{m/p}(r)^{p-1}\d\tau_{m/p}(r)\,.
			\end{equation*}
			\item On elements in $\im V_\ell$, where $\ell$ is any prime factor of $m$, the values of $F_p\d$ are prescribed by \cref{def:qFVSystemOfCDGA}\cref{enum:qDeRhamWittConditionC}: If $\ell=p$, we immediately get\label{enum:FpdOnVerschiebung}
			\begin{equation*}
				F_p\d V_p(x)=\d x\,.
			\end{equation*}
			If $\ell\neq p$, note that $F_p\d V_\ell(x)$ is uniquely determined by $pF_p\d V_\ell(x)$ and $\ell F_p\d V_\ell(x)$, as $\ell$ and $p$ are coprime. Using \cref{lem:dAfterV,lem:dAfterF}, we see that necessarily
			\begin{equation*}
				pF_p\d V_\ell(x)=\d F_pV_\ell(x)\quad\text{and}\quad \ell F_p\d V_\ell(x)=F_p(V_\ell\d x)=V_\ell(F_p\d x)\,.
			\end{equation*}
			Now $V_\ell(F_p\d x)$ is already defined by induction, and $\d F_pV_\ell(x)$ is already defined because $F_p$ should be just the usual $q$-Witt vector Frobenius in degree $0$.
		\end{alphanumerate}
		Note that there can be at most one $\IZ[q]$-linear map $F_p\d\colon \qIW_m(R)\rightarrow \qIW_{m/p}\Omega_{R/\IZ}^1$ satisfying \cref{enum:FpdOnTeichmueller} and~\cref{enum:FpdOnVerschiebung}, so the only non-trivial task is to show existence.
	\end{numpar}
	\begin{lem}\label{lem:FrobeniusWelldefinedOnV}
		Let $\IV_m\coloneqq \left(\im V_\ell\ \middle|\ \ell\text{ prime factor of }m\right)\subseteq \qIW_m(R)$. Then there is a well-defined $\IZ[q]$-linear map $F_p\d\colon \IV_m\rightarrow \qIW_{m/p}\Omega_{R/\IZ}^1$ given as in \cref{par:PropertiesOfFpd}\cref{enum:FpdOnVerschiebung}.
	\end{lem}
	\begin{proof}
		From \cref{prop:qWittKoszulExactSequence}, we get an exact sequence
		\begin{equation*}
			\bigoplus_{\ell_1\neq\ell_2}\qIW_{m/\ell_1\ell_2}(R)\xrightarrow{(V_{\ell_1}-V_{\ell_2})}\bigoplus_\ell\qIW_{m/\ell}(R)\xrightarrow{(V_\ell)}\IV_m\longrightarrow 0\,;
		\end{equation*}
		here $\ell$, $\ell_1$, and $\ell_2$ range over all prime factors of $m$. Condition~\cref{enum:FpdOnVerschiebung} above defines a unique $\IZ[q]$-linear map $\bigoplus_\ell\qIW_{m/\ell}(R)\rightarrow \qIW_{m/p}\Omega_{R/\IZ}^1$ (here we also use injectivity of the Verschiebungen, see \cref{cor:VerschiebungInjective}) and we only have to check that $\bigoplus_{\ell_1\neq\ell_2}\qIW_{m/\ell_1\ell_2}(R)$ maps into its kernel. We can do this one summand at a time. So fix prime factors $\ell_1\neq \ell_2$. We distinguish two cases:
		
		\emph{Case~1: $p\notin\{\ell_1,\ell_2\}$.} In this case, it's enough to check that $pF_p\d$ and $\ell_1\ell_2F_p\d$ are well-defined, since $p$ and $\ell_1\ell_2$ are coprime. For $p F_p\d$, we have to check that $\d F_pV_{\ell_1}(V_{\ell_2}(x))=\d F_pV_{\ell_2}(V_{\ell_1}(x))$ holds for all $x\in \qIW_{m/\ell_1\ell_2}(R)$, which is clear. For $\ell_1\ell_2F_p\d$, we have to check the condition $\ell_2 V_{\ell_1}(F_p\d V_{\ell_2}(x))=\ell_1 V_{\ell_2}(F_p\d V_{\ell_1}(x))$, which again is clear as both sides can be transformed into $V_{\ell_1\ell_2}(F_p\d x)$, using \cref{lem:dAfterV}.
		
		\emph{Case~2: $p\in\{\ell_1,\ell_2\}$.} Without restriction let $p=\ell_1$ and $\ell=\ell_2$. This time we check that $pF_p\d$ and $\ell F_p\d$ are well-defined. For $pF_p\d$, we must check that $\d F_pV_\ell(V_p(x))=p\d V_\ell(x)$ holds for all $x\in \qIW_{m/\ell_1\ell_2}(R)$, which follows from $F_p\circ V_\ell \circ V_p=F_p\circ V_p\circ V_\ell=pV_\ell$. For $\ell F_p\d$, we must check that $V_\ell(F_p\d V_p(x))=\ell\d V_\ell(x)$. But this follows from the inductive hypothesis, which ensures that $F_p\d V_p(x)=\d x$, and \cref{lem:dAfterV}.
	\end{proof}
	\begin{con}\label{con:Fpd}
		We construct a well-defined map $F_p\d\colon \qIW_m(R)\rightarrow \qIW_{m/p}\Omega_{R/\IZ}^1$ as follows: By \cref{par:GhostMaps}, every $x\in \qIW_m(R)$ can be uniquely written as 
		\begin{equation*}
			x=\sum_{i=0}^{\varphi(m)-1}q^i\tau_m(r_i)+v\,,
		\end{equation*}
		where $\varphi(m)=\deg\Phi_m(q)$ denotes Euler's $\varphi$-function, $r_i\in R$, and $v\in\IV_m$. We then define
		\begin{equation*}
			F_p\d x=\sum_{i=0}^{\varphi(m)-1}q^i\tau_{m/p}(r_i)^{p-1}\d\tau_{m/p}(r_i)+F_p\d v\,,
		\end{equation*}
		where $F_p\d v$ is constructed as in \cref{lem:FrobeniusWelldefinedOnV}.
	\end{con}
	\begin{lem}\label{lem:FpdAdditive}
		The map $F_p\d\colon \qIW_m(R)\rightarrow \qIW_{m/p}\Omega_{R/\IZ}^*$ from \cref{con:Fpd} is additive.
	\end{lem}
	\begin{proof}
		It's straightforward to see from \cref{def:qDRW} that $q$-de Rham--Witt complexes commute with filtered colimits in $R$. By writing $R$ as a filtered colimit of finite type $\IZ$-algebras, we may thus assume that $R$ itself is of this form. In this case, \cref{cor:qWittNoetherian} shows that $\qIW_{m/p}(R)$ is of finite type over $\IZ$ too and then \cref{prop:qDRWExists}\cref{enum:dRtoqDRWisSurjective} shows that $\qIW_{m/p}\Omega_{R/\IZ}^1$ is a finitely generated module over the noetherian ring $\qIW_{m/p}(R)$. For any finitely generated module $M$ over a noetherian $\IZ[q]$-algebra, the natural map
		\begin{equation*}
			M\longrightarrow M\left[\localise{p}\right]\times\prod_{\ell\neq p}\widehat{M}_{(p,q^{m/\ell}-1)}\times M\left[\localise{q^{m/\ell}-1}\ \middle|\ \ell\neq p\right]
		\end{equation*}
		is injective; here $\ell$ ranges over all prime factors $\neq p$ of $m$.%
		\footnote{Here's the technical argument: It's clear that the map is injective after applying each of the functors $(-)[1/p]$, $(-)_{(p,q^{m/\ell}-1)}^\complete$, and $(-)\bigl[1/(q^{m/\ell}-1)\ \big|\ \ell\neq p\bigr]$. Thus, if $K$ denotes the kernel of the map above, then $M\rightarrow M/K$ will become injective after applying each of these functors, because then an injective map factors through it. But all of these functors preserve exactness of the sequence $0\rightarrow K\rightarrow M\rightarrow M/K\rightarrow 0$: For the localisations, this is clear, for the completions, we appeal to the fact that we're working with finitely generated modules over a noetherian ring. Hence $K$ vanishes after each of these functors and \cref{lem:DerivedBeauvilleLaszlo} shows $K\cong 0$.}		
		In particular, we see that it suffices to show additivity of our would-be derivation $F_p\d$ after applying each of the functors $(-)[1/p]$, $(-)_{(p,q^{m/\ell}-1)}^\complete$, and $(-)\bigl[1/(q^{m/\ell}-1)\ \big|\ \ell\neq p\bigr]$.
		
		\emph{Proof after localisation at $p$.} Since $pF_p\d=\d F_p$ holds by construction, it's clear that $F_p\d$ is additive after inverting $p$.
		
		\emph{Proof after $(p,q^{m/\ell}-1)$-adic completion.} It suffices to show additivity of $\ell F_p\d$, since $\ell$ becomes invertible after $(p,q^{m/\ell}-1)$-adic completion. Furthermore, $[\ell]_{q^{m/p\ell}}$ becomes invertible too. Hence the Frobenius $F_\ell\colon \qIW_{m/p}\Omega_{R/\IZ}^*\rightarrow \qIW_{m/p\ell}\Omega_{R/\IZ}^*$, which we've already constructed, induces an isomorphism (with inverse $[\ell]_{q^{m/p\ell}}^{-1}V_{\ell}$)
		\begin{equation*}
			F_\ell\colon \bigl(\qIW_{m/p}\Omega_{R/\IZ}^1\bigr)_{(p,q^{m/\ell}-1)}^\complete\overset{\cong}{\longrightarrow}\bigl(\qIW_{m/p\ell}\Omega_{R/\IZ}^1\bigr)_{(p,q^{m/\ell}-1)}^\complete\,.
		\end{equation*}
		Therefore it suffices to show that $\ell F_\ell(F_p\d)$ is additive. But we'll show in \cref{lem:lFlFpd} that $\ell F_\ell(F_p\d)=F_p\d F_\ell$, where $F_p$ on the right-hand side refers to $F_p\colon \qIW_{m/\ell}\Omega_{R/\IZ}^*\rightarrow \qIW_{m/p\ell}\Omega_{R/\IZ}^*$, which we've already constructed. Now it's clear that $F_p\d F_\ell$ is additive.
		
		\emph{Proof after localisation at $(q^{m/\ell}-1)$ for all $\ell\neq p$.} Let $m=p^\alpha n$, where $\alpha=v_p(m)$. Observe that
		\begin{equation*}
			\qIW_{m/p}\Omega_{R/\IZ}^*\left[\localise{q^{m/\ell}-1}\ \middle|\ \ell\neq p\right]\cong \qIW_{p^{\alpha-1}}\Omega_{R/\IZ}^*\otimes_{\IZ[q],\psi^n}\IZ\left[q,\localise{q^{m/\ell}-1}\ \middle|\ \ell\neq p\right]\,,
		\end{equation*}
		where $\psi^n$ is the map that sends $q\mapsto q^n$. Indeed, this can be shown as in the proof of \cref{prop:qWittKoszulExactSequence} by comparing universal properties (more precisely, by comparing the truncated universal properties from \cref{rem:TruncatedUniversalProperty2}). So we can reduce to the case where $m=p^\alpha$. In this case, the calculation from the proof of \cite[Proposition~\href{https://www.math.uni-bielefeld.de/~zink/dRW.pdf\#page=16}{1.3}]{LangerZink} can be carried over to our situation (note that Langer--Zink's calculation needs \cref{lem:VPDderivation}).
	\end{proof}
	\begin{lem}\label{lem:lFlFpd}
		If $\ell\neq p$ is another prime factor of $m$, then the map $F_p\d$ from \cref{con:Fpd} satisfies
		\begin{equation*}
			\ell F_\ell(F_p\d x)=F_p\d F_\ell(x)
		\end{equation*}
		for all $x\in \qIW_m(R)$, where $F_\ell$ on the left-hand side refers to $F_\ell\colon \qIW_{m/p}\Omega_{R/\IZ}^*\rightarrow \qIW_{m/p\ell}\Omega_{R/\IZ}^*$, which has already been constructed by induction, and $F_\ell$ on the right-hand side refers to the $q$-Witt vector Frobenius, which we also already know how to construct.
	\end{lem}
	\begin{proof}
		By \cref{con:Fpd}, it suffices to prove $\ell F_\ell(F_p\d x)=F_p\d F_\ell(x)$ in the following three special cases:
		
		\emph{Case~1: $x=q^i\tau_m(r)$ for some $0\leqslant i<\varphi(m)$ and some $r\in R$.} %In this case we have $F_p\d (q^i\tau_m(r))=q^i\tau_{m/p}(r)^{p-1}\d\tau_{m/p}(r)$. 
		Using \cref{lem:dAfterF}, we can transform the left-hand side as follows:
		\begin{equation*}
			\ell F_\ell\bigl(F_p\d\bigl(q^i\tau_{m/p}(r)\bigr)\bigr)=q^i\ell F_\ell\left(\tau_{m/p}(r)^{p-1}\d\tau_m(r)\right)=q^i\tau_{m/p\ell}(r)^{\ell(p-1)}\d F_\ell\tau_{m/p}(r)\,.
		\end{equation*}
		On the right-hand side, we use $F_\ell(q^i\tau_m(r))=q^i\tau_{m/\ell}(r^\ell)$, and then the $F$-Teichmüller condition \cref{enum:TeichmuellerF} shows
		\begin{equation*}
			F_p\d\bigl(q^i\tau_{m/\ell}(r^\ell)\bigr)=q^i\tau_{m/p\ell}(r^\ell)^{p-1}\d\tau_{m/p\ell}(r^\ell)=q^i\tau_{m/p\ell}(r)^{\ell(p-1)}\d F_\ell\tau_{m/p}(r)\,.
		\end{equation*}
		This finishes the proof in the first case. 
		
		\emph{Case~2: $x=V_p(y)$ for some $y\in \qIW_{m/p}(R)$.} In this case we have $F_p\d V_p(y)=\d y$. So the left-hand side simply becomes $\ell F_\ell\d y=\d F_\ell(y)$. On the right-hand side we use that $F_\ell$ commutes with $V_p$ to obtain $F_p\d F_\ell V_p(y)=F_p\d V_pF_\ell(y)=\d F_\ell(y)$, as required.
		
		\emph{Case~3: $x=V_{\ell_0}(z)$ for some prime factor $\ell_0\neq p$ of $m$ \embrace{$\ell_0=\ell$ is allowed} and some $z\in \qIW_{m/\ell'}(R)$.} In this case, we prove $\ell F_\ell(F_p\d x)=F_p\d F_\ell(x)$ after multiplication by $p$ and after multiplication by $\ell_0$, which is enough as $p$ and $\ell_0$ are coprime. After multiplication by $p$, the left-hand side becomes $p\ell F_\ell(F_p\d V_{\ell_0}(z))=\ell F_\ell(\d F_pV_{\ell_0}(z))=\d F_pF_{\ell}V_{\ell_0}(z)$, whereas the right-hand side becomes $pF_p\d F_\ell V_{\ell_0}(z)=\d F_pF_\ell V_{\ell_0}(z)$. These two are the same, since we already know $F_p\circ F_\ell=F_\ell\circ F_p$ as maps $\qIW_m(R)\rightarrow \qIW_{m/p\ell}(R)$.
		
		Now let's see what happens after multiplication by $\ell_0$. By \cref{par:PropertiesOfFpd}\cref{enum:FpdOnVerschiebung}, $\ell_0F_p\d V_{\ell_0}(z)=V_{\ell_0}(F_p\d z)$. Plugging this into the left-hand side, we obtain
		\begin{equation*}
			\ell_0 \ell F_\ell\bigl(F_p\d V_{\ell_0}(z)\bigr)=\ell F_\ell V_{\ell_0}(F_p\d z)=\begin{cases*}
				\ell^2 F_p\d z & if $\ell=\ell_0$\\
				\ell V_{\ell_0}F_\ell(F_p\d z) & if $\ell\neq \ell_0$
			\end{cases*}\,.
		\end{equation*}
		If $\ell=\ell_0$, then also $\ell_0F_p\d F_\ell V_{\ell_0}(z)=\ell F_p\d (\ell z)=\ell^2F_p\d z$, in agreement with the formula above. In the case $\ell\neq \ell_0$, we can do the following calculation (each step will be justified below):
		\begin{equation*}
			\ell_0 F_p\d F_\ell V_{\ell_0}(z)=F_p\bigl(\ell_0 \d V_{\ell_0}F_\ell(z)\bigr)= V_{\ell_0} F_p\bigl(\d F_\ell(z)\bigr)=\ell V_{\ell_0}F_p(F_\ell \d z)=\ell V_{\ell_0}F_\ell(F_p\d z)\,,
		\end{equation*}
		This again agrees with the formula above. In the first step, we use $F_\ell\circ V_{\ell_0}=V_{\ell_0}\circ F_\ell$ as maps $\qIW_{m/\ell_0}(R)\rightarrow \qIW_{m/\ell}(R)$. In the second step, we use \cref{lem:dAfterV} together with the fact that $V_{\ell_0}\circ F_p=F_p\circ V_{\ell_0}$ holds as maps $\qIW_{m/\ell_0\ell}\Omega_R^*\rightarrow \qIW_{m/p\ell}\Omega_R^*$ (we already know this by induction). In the third step, we use \cref{lem:dAfterV}. In the last step, we use $F_p\circ F_\ell=F_\ell\circ F_p$ as maps $\qIW_{m/\ell'}\Omega_{R/\IZ}^*\rightarrow \qIW_{m/p\ell'\ell}\Omega_{R/\IZ}^*$ (again by induction). This finishes the proof.
	\end{proof}
	\begin{lem}\label{lem:FpdZqLinear}
		The map $F_p\d\colon \qIW_m(R)\rightarrow \qIW_{m/p}\Omega_{R/\IZ}^*$ from \cref{con:Fpd}, which we know to be additive by \cref{lem:FpdAdditive}, is $\IZ[q]$-linear.
	\end{lem}
	\begin{proof}
		We must show $F_p\d(qx)=qF_p\d x$ for all $x\in \qIW_m(R)$. By \cref{lem:FrobeniusWelldefinedOnV}, we already know this if $x$ is contained in the ideal $\IV\subseteq \qIW_m(R)$ generated by the images of all Verschiebungen, so it's enough to check the case $x=q^i\tau_m(r)$ for some $r\in R$ and some $0\leqslant i<\varphi(m)$. If $i<\varphi(m)-1$, then everything is clear from \cref{con:Fpd}. In the case $i=\varphi(m)-1$ we may then equivalently check that $F_p\d (\Phi_m(q)\tau_m(r))=\Phi_m(q)F_p\d\tau_m(r)=\Phi_m(q)\tau_{m/p}(r)^{p-1}\d\tau_{m/p}(r)$. Using \cref{lem:IdealGeneratedByPhi}, we can write $\Phi_m(q)$ as a $\IZ[q]$-linear combination of $[\ell]_{q^{m/\ell}}$, where $\ell$ ranges through all prime factors of $m$  (including $\ell=p$). Since $[\ell]_{q^{m/\ell}}\tau_m(r)=V_\ell F_\ell(\tau_m(r))$ is contained in $\IV$, where we already know $\IZ[q]$-linearity, it suffices to show
		\begin{equation*}
			F_p\d \bigl(V_\ell F_\ell\bigl(\tau_m(r)\bigr)\bigr)=[\ell]_{q^{m/\ell}}\tau_{m/p}(r)^{p-1}\d\tau_{m/p}(r)
		\end{equation*}
		for all $\ell$. This requires once again a case distinction.
		
		\emph{Case~1: $\ell=p$.} In this case we have $[p]_{q^{m/p}}\tau_{m/p}(r)^{p-1}\d\tau_{m/p}(r)=p\tau_{m/p}(r)^{p-1}\d\tau_{m/p}(r)$, because both sides live in $\qIW_{m/p}\Omega_{R/\IZ}^*$, which is $(q^{m/p}-1)$-torsion. Also, according to \cref{par:PropertiesOfFpd}\cref{enum:FpdOnVerschiebung}, we have $F_p\d V_p(\tau_{m/p}(r)^p)=\d(\tau_{m/p}(r)^p)$. Using that $\d$ is a derivation, we're done. 
		
		\emph{Case~2: $\ell\neq p$.} We use our standard trick and show the desired equation after multiplication by $p$ and by $\ell$. After multiplication by $p$, we obtain
		\begin{equation*}
			p F_p\bigl(\d V_\ell F_\ell\bigl(\tau_{m}(r)\bigr)\bigr)=\d F_pV_\ell F_\ell\bigl(\tau_m(r)\bigr)=[\ell]_{q^{m/\ell}}\d F_p\tau_m(r)=p[\ell]_{q^{m/\ell}}\tau_{m/p}(r)^{p-1}\d\tau_{m/p}(r)\,,
		\end{equation*}
		as required. After multiplication by $\ell$, \cref{par:PropertiesOfFpd}\cref{enum:FpdOnVerschiebung} allows us to compute
		\begin{align*}
			\ell F_p\bigl(\d V_\ell F_\ell\bigl(\tau_{m}(r)\bigr)\bigr)=V_\ell\bigl(F_p\d \tau_{m/\ell}(r^\ell)\bigr)&=V_\ell\bigl(\tau_{m/p\ell}(r)^{\ell(p-1)}\d\tau_{m/p\ell}(r^\ell)\bigr)\\
			&=\ell V_\ell\bigl(\tau_{m/p\ell}(r)^{\ell(p-1)}\tau_{m/p\ell}(r)^{\ell-1}\bigr)\d V_\ell\tau_{m/p\ell}(r)\\
			&=\ell V_\ell\bigl(\tau_{m/p\ell}(r)^{\ell(p-1)}\bigr)\d\tau_{m/p}(r)\\
			&=\ell[\ell]_{q^{m/\ell}}\tau_{m/p}(r)^{p-1}\d\tau_{m/p}(r)\,.
		\end{align*}
		In the first line, we used the $F$-Teichmüller condition \cref{enum:TeichmuellerF}, which we already know for $F_p\colon \qIW_{m/\ell}\Omega_R^*\rightarrow \qIW_{m/p\ell}\Omega_R^*$. In the second line, we used that $\d$ is a derivation together with the last condition from \cref{def:qVSystemOfCDGA}\cref{enum:qDeRhamWittConditionB}. In the third line, we applied the $V$-Teichmüller condition \cref{enum:TeichmuellerV} for $\omega=\tau_{m/p\ell}(r)^{\ell(p-1)}$. Finally, in the fourth line we used the fact that $V_\ell(\tau_{m/p\ell}(r)^{\ell(p-1)})=V_\ell F_\ell(\tau_{m/p}(r)^{p-1})=[\ell]_{q^{m/\ell}}\tau_{m/p}(r)^{p-1}$. We're done.
	\end{proof}
	\begin{lem}\label{lem:FpdDerivation}
		The map $F_p\d\colon \qIW_m(R)\rightarrow \qIW_{m/p}\Omega_{R/\IZ}^1$ from \cref{con:Fpd} is a $\IZ[q]$-linear derivation. If $R$ is an $A$-algebra for some $\Lambda$-ring $A$, then $F_p\d$ extends uniquely to an $A[q]$-linear derivation $F_p\d\colon \qIW_m(R/A)\rightarrow \qIW_{m/p}\Omega_{R/A}^1$.
	\end{lem}
	\begin{proof}
		To show that $F_p\d\colon \qIW_m(R)\rightarrow \qIW_{m/p}\Omega_{R/\IZ}^1$ is a derivation, we use the same method as in the proof of \cref{lem:FpdAdditive}: By compatibility with filtered colimits, reduce to the case where $R$ is of finite type over $\IZ$. Then apply $(-)[1/p]$, $(-)_{(p,q^{m/\ell}-1)}^\complete$, or $(-)\bigl[1/(q^{m/\ell}-1)\ \big|\ \ell\neq p\bigr]$: After localisation at $p$, the fact that $F_p\d$ is a derivation follows from $p F_p\d=\d F_p$. After $(p,q^{m/\ell}-1)$-adic completion, we can use induction again. After localisation at $(q^{m/\ell}-1)$ for all $\ell\neq p$, we can reduce to the case $m=p^\alpha$ again. In this case, we can adapt the proof of \cite[Proposition~\href{https://www.math.uni-bielefeld.de/~zink/dRW.pdf\#page=16}{1.3}]{LangerZink}; to make the adaptation, one needs to use that $F_p\d$ is $\IZ[q]$-linear, which we know from \cref{lem:FpdZqLinear}.
		
		Now assume $R$ is an $A$-algebra. First observe that the composition
		\begin{equation*}
			\qIW_m(R)\xrightarrow{F_p\d}\qIW_{m/p}\Omega_{R/\IZ}^1\longrightarrow \qIW_{m/p}\Omega_{R/A}^1 
		\end{equation*}
		kills all elements in the image of $\qIW_m(A)\rightarrow \qIW_m(R)$. Indeed, according to the description in \cref{par:PropertiesOfFpd}, we only have to show $\d x=0$ for all $x\in \qIW_{m/p}(A)$ as well as $V_\ell(F_p\d y)=0$ for all $y\in \qIW_{m/\ell}(A)$ and all $\ell\neq p$. The latter follows from the inductive hypothesis, whereas the former is ensured by the fact that $\qIW_{m/p}\Omega_{R/A}^1$ is a quotient of $\Omega_{\qIW_{m/p}(R/A)/A[q]}^1$ by \cref{prop:qDRWExists}\cref{enum:dRtoqDRWisSurjective}.
		
		Thus, $F_p\d$ can be extended to an $A[q]$-linear derivation
		\begin{equation*}
			\qIW_m(R)\otimes_{\qIW_m(A)}A[q]/(q^m-1)\longrightarrow \qIW_{m/p}\Omega_{R/A}^1\,.
		\end{equation*}
		It remains to show that the ideal $\IU_m$ from \cref{lem:RelativeqWitt} is killed. This is another straightforward check. It's enough to consider generators of the form $V_\ell(xy)\otimes 1-V_\ell(x)\otimes c_{m/\ell}(y)$ for $\ell\mid m$ a prime factor, $x\in \qIW_{m/\ell}(R)$, and $y\in \qIW_{m/\ell}(A)$. If $\ell=p$, then the map above sends this generator to $\d(xy)-c_{m/\ell}(y)\d x=0$, using that the differentials of $\qIW_{m/p}\Omega_{R/A}^*$ are $A[q]$-linear. If $\ell\neq p$, we have multiply by $p$ and by $\ell$ once again. After multiplication by $p$, we obtain $\d F_p(V_\ell(xy)\otimes 1-V_\ell(x)\otimes c_{m/\ell}(y))$, which vanishes because $V_\ell(xy)\otimes 1-V_\ell(x)\otimes c_{m/\ell}(y)=0$ holds already in $\qIW_m(R/A)$. After multiplication by $\ell$, we get $V_\ell (F_p\d (xy)-c_{m/p\ell}(y)F_p\d x)$, which vanishes by the inductive hypothesis.
	\end{proof}
	\begin{con}\label{con:FrobeniusPreliminary}
		From \cref{lem:FpdDerivation} and the universal property of Kähler differentials we get a $\qIW_m(R/A)$-module map $F_p\colon \Omega_{\qIW_m(R/A)/A[q]}^1\rightarrow \qIW_{m/p}\Omega_{R/A}^1$. By the universal property of exterior algebras, this extends uniquely to a map
		\begin{equation*}
			F_p\colon \Omega_{\qIW_m(R/A)/A[q]}^*\longrightarrow \qIW_{m/p}\Omega_{R/A}^*
		\end{equation*}
		of graded $\qIW_m(R/A)$-algebras. We wish to show that this map factors uniquely over $\qIW_m\Omega_{R/A}^*$. By revisiting the explicit construction, we see that we must check $F_p(\xi)=0$, $F_p(\d\xi)=0$, $F_p(\eta)=0$, and $F_p(\d\eta)=0$, where $\xi$ and $\eta$ are as in the proof of \cref{prop:qDRWExists}. This will be proved in \cref{lem:Fpxi=0,lem:Fpeta=0} below.
	\end{con}
	\begin{lem}\label{lem:Fpxi=0}
		Fix a divisor $d\mid m$ such that $d\neq m$. Let $j\geqslant 1$, let $I$ be a finite indexing set, and let $(w_i,x_{i,1},\dotsc,x_{i,j})_{i\in I}$ be a sequence of elements of $\qIW_d(R/A)$ such that $0=\sum_{i\in I}w_i\d x_{i,1}\wedge \dotsm\wedge \d x_{i,j}$ holds in $\qIW_d\Omega_{R/A}^j$. Put
		\begin{equation*}
			\xi\coloneqq\sum_{i\in I}V_{m/d}(w_i)\d V_{m/d}(x_{i,1})\wedge \dotsm\wedge \d V_{m/d}(x_{i,j})\,.
		\end{equation*}
		Then $F_p(\xi)=0$ and $F_p(\d\xi)=0$.
	\end{lem}
	\begin{proof}
		Note that \cref{par:PropertiesOfFpd} tells us how to compute $F_p(\xi)$, but in order to do that, we need to distinguish whether or not $p$ divides $m/d$. If $p$ does divide $m/d$, we get
		\begin{align*}
			F_p(\xi)&=\sum_{i\in I}pV_{m/pd}(w_i)\d V_{m/pd}(x_{i,1})\wedge\dotsb \wedge \d V_{m/pd}(x_{i,j})\\
			&=pV_{m/pd}\left(\sum_{i\in I}w_i\d x_{i,1}\wedge\dotsb\wedge \d x_{i,j}\right)\,,
		\end{align*}
		which vanishes because $\sum_{i\in I}w_i\d x_{i,1}\wedge\dotsb\wedge \d x_{i,j}=0$ by assumption. If $p$ doesn't divide $m/d$, we'll show $p^jF_p(\xi)=0$ and $(m/d)^jF_p(\xi)=0$ instead. For the first one, we compute
		\begin{align*}
			p^jF_p(\xi)&=\sum_{i\in I}F_pV_{m/d}(w_i)\d F_pV_{m/d}(x_{i,1})\wedge\dotsb\wedge \d F_pV_{m/d}(x_{i,j})\\
			&=V_{m/d}\left(\sum_{i\in I}F_p(w_i)\d F_p(x_{i,1})\wedge\dotsb\wedge\d F_p(x_{i,j})\right)\\
			&=p^jV_{m/d}F_p\left(\sum_{i\in I}w_i\d x_{i,1}\wedge\dotsb\wedge \d x_{i,j}\right)\,,
		\end{align*}
		which once again vanishes by our assumption $\sum_{i\in I}w_i\d x_{i,1}\wedge\dotsb\wedge \d x_{i,j}=0$ plus the fact that we already know $F_p\colon \qIW_d\Omega_{R/A}^*\rightarrow \qIW_{d/p}\Omega_{R/A}^*$ to be well-defined. Similarly,
		\begin{align*}
			\left(\frac{m}{d}\right)^jF_p(\xi)&=\sum_{i\in I}F_pV_{m/d}(w_i)V_{m/d}(F_p\d x_{i,1})\wedge\dotsb\wedge V_{m/d}(F_p\d x_{i,j})\\
			&=\left(\frac{m}{d}\right)^jV_{m/d}F_p\left(\sum_{i\in I}w_i\d x_{i,1}\wedge\dotsb\wedge \d x_{i,j}\right)\,.
		\end{align*}
		This vanishes because of our assumption again, plus the fact that we already know the Frobenius $F_p\colon \qIW_d\Omega_{R/A}^*\rightarrow \qIW_{d/p}\Omega_{R/A}^*$ to be well-defined. This finishes the proof that $F_p(\xi)=0$. The proof that $F_p(\d\xi)=0$ is completely analogous and we'll leave it to the reader.
	\end{proof}
	\begin{lem}\label{lem:Fpeta=0}
		Fix a divisor $d\mid m$ such that $d\neq m$. Let $x\in \qIW_d(R/A)$ and $r\in R$. Put
		\begin{equation*}
			\eta\coloneqq V_{m/d}(x)\d\tau_m(r)-V_{m/d}\bigl(x\tau_d(r)^{{m/d}-1}\bigr)\d V_{m/d}\tau_d(r)\,.
		\end{equation*}
		Then $F_p(\eta)=0$ and $F_p(\d\eta)=0$.
	\end{lem}
	\begin{proof}
		Again, we have to distinguish whether or not $m/d$ is divisible by $p$. If it is, we get
		\begin{align*}
			F_p(\eta)&=pV_{m/pd}(x)\tau_{m/p}(r)^{p-1}\d\tau_{m/p}(r)-pV_{m/pd}\bigl(x\tau_d(r)^{m/d-1}\bigr)\d V_{m/pd}\tau_d(r)\\
			&=pV_{m/pd}\bigl(x\tau_{m/d}(r)^{(p-1)m/pd}\bigr)\d\tau_{m/p}(r)-pV_{m/pd}\bigl(x\tau_d(r)^{m/d-1}\bigr)\d V_{m/pd}\tau_d(r)\,.
		\end{align*} 
		In the second line we used $V_{m/pd}(x)\tau_{m/p}(r)^{p-1}=V_{m/pd}(x F_{m/pd}(\tau_{m/p}(r)^{p-1}))$, as we already know that the conditions from \cref{def:qFVSystemOfCDGA}\cref{enum:qDeRhamWittConditionC} are true for $F_{m/pd}\colon \qIW_{m/p}\Omega_{R/A}^*\rightarrow \qIW_{d}\Omega_{R/A}^*$, and $F_{m/pd}(\tau_{m/p}(r)^{p-1})=\tau_{m/d}(r)^{(p-1)m/pd}$. Now $F_p(\eta)$ vanishes because the last line is precisely $p$ times the $V$-Teichmüller condition \cref{enum:TeichmuellerV} for $x\tau_{m/d}(r)^{(p-1)m/pd}$ and $r$.
		
		Now assume $p$ doesn't divide $m/d$. Since $\eta$ is automatically a $(m/d)^{m/d-1}$-torsion element in $\Omega_{\qIW_m(R/A)/A[q]}^1$, it suffices to show $pF_p(\eta)=0$ in this case. Note that $p$ not dividing $m/d$ implies that $p$ must divide $d$. We put $m_0\coloneqq m/p$ and $d_0\coloneqq d/p$ for short and compute
		\begin{align*}
			pF_p(\eta)&=F_pV_{m/d}(x)\d F_p\tau_d(r)-F_pV_{m/d}\bigl(x\tau_d(r)^{m/d-1}\bigr)\d F_pV_{m/d}\tau_d(r)\\
			&=V_{m_0/d_0}\bigl(F_p(x)\bigr)\d\tau_{d_0}(r^p)-V_{m_0/d_0}\bigl(F_p(x)\tau_{d_0}(r^p)^{m_0/d_0-1}\bigr)\d V_{m_0/d_0}\tau_{d_0}(r^p)\,.
		\end{align*}
		Now $pF_p(\eta)=0$ follows because the last line is precisely the $V$-Teichmüller condition \cref{enum:TeichmuellerV} for $F_p(x)$ and $r^p$. This proves $F_p(\eta)=0$. The proof of $F_p(\d\eta)=0$ is similar and we leave it to the reader once again.
	\end{proof}
	With \cref{lem:Fpxi=0,lem:Fpeta=0} proved, \cref{con:FrobeniusPreliminary} finally gives a complete construction of the Frobenius $F_p$. It remains to check that it has all necessary properties.
	\begin{lem}\label{lem:FpSatisfiesAllProperties}
		The map $F_p\colon \qIW_m\Omega_{R/A}^*\rightarrow \qIW_{m/p}\Omega_{R/A}^*$ from \cref{con:FrobeniusPreliminary} satisfies all properties from \cref{def:qFVSystemOfCDGA}\cref{enum:qDeRhamWittConditionC}.
	\end{lem}
	\begin{proof}
		Compatibility with the $q$-Witt vector Frobenius holds by construction. For the chain condition, we must check $F_p\circ F_\ell=F_\ell\circ F_p$ for all prime factors $\ell\neq p$ of $m$ (in both compositions, the left factor is defined via \cref{con:FrobeniusPreliminary} and the right factor is defined by induction). This can be checked after multiplication by $p$ and $\ell$ and then \cref{lem:lFlFpd} takes care of the essential case. The condition $F_p\circ \d\circ V_p=\d$ holds again by construction.
		
		It remains to prove $V_p(\omega F_p(\eta))=V_p(\omega)\eta$ for all $\omega\in \qIW_{m/p}\Omega_R^*$ and all $\eta\in \qIW_m\Omega_R^*$ (which also implies $V_p\circ F_p=[p]_{q^{m/p}}$). This is easily reduced to checking $V_p(x F_p(\d y))=V_p(x)\d y$ for all $x\in \qIW_{m/p}(R)$ and $y\in \qIW_m(R)$ (that is, it suffices to do the case $\omega=x$ and $\eta=\d y$). Furthermore, it suffices to treat the three cases $y=a\tau_m(r)$ for some $a\in A[q]$ and some $r\in R$, $y=V_p(z)$ for some $z\in \qIW_{m/p}(R)$, and $y=V_\ell(w)$ for some prime factor $\ell\neq p$ of $m$ and some $w\in \qIW_{m/\ell}(R)$.
		
		The case $y=a\tau_m(r)$ follows immediately from the $V$-Teichmüller condition \cref{enum:TeichmuellerV}. For $y=V_p(z)$, we compute $V_p(xF_p\d V_p(z))=V_p(x\d z)=V_p(x)\d V_p(z)$, as required. Finally, to handle the case $y=V_\ell(y)$, we have to multiply both sides by $p$ and $\ell$ for one last time. After multiplication by $p$, we obtain
		\begin{equation*}
			p V_p\bigl(x F_p\d V_\ell(w)\bigr)=V_p\bigl(x\d F_pV_\ell(w)\bigr)=V_p(x)\d V_pF_pV_\ell(w)=[p]_{q^{m/p}}V_p(x)\d V_\ell(w)\,.
		\end{equation*}
		But $V_p(x)$ is $(q^{m/p}-1)$-torsion, so $[p]_{q^{m/p}}V_p(x)\d V_\ell(w)=pV_p(x)\d V_\ell(w)$, as required. After multiplication by $\ell$, we compute
		\begin{equation*}
			\ell V_p\bigl(x F_p\d V_\ell(w)\bigr)=V_p\bigl(xV_\ell\bigl(F_p(\d w)\bigr)\bigr)=V_pV_\ell\bigl(F_\ell(x)F_p(\d w)\bigr)=V_\ell V_p\bigl(F_\ell(x)F_p(\d w)\bigr)\,.
		\end{equation*}
		In the second equality we used the property for $F_\ell\colon \qIW_{m/p}\Omega_{R/A}^*\rightarrow \qIW_{m/p\ell}\Omega_{R/A}^*$, which we already know by induction. Furthermore, applying the inductive hypothesis to the Frobenius $F_p\colon \qIW_{m/\ell}\Omega_{R/A}^*\rightarrow \qIW_{m/p\ell}\Omega_{R/A}^*$, we get 
		\begin{equation*}
			V_\ell V_p\bigl(F_\ell(x)F_p(\d w)\bigr)=V_\ell\bigl(V_pF_\ell (x)\d w\bigr)=V_\ell F_\ell V_p(x)\d V_\ell(w)=[\ell]_{q^{m/\ell}}V_p(x)\d V_\ell(w)\,.
		\end{equation*}
		But $\d V_\ell(w)$ is $(q^{m/\ell}-1)$-torsion, so $[\ell]_{q^{m/\ell}}V_p(x)\d V_\ell(w)=\ell V_p(x)\d V_\ell(w)$, as required.
	\end{proof}
	
	\begin{proof}[Proof of \cref{prop:qDRWHasFrobenii}]
		It follows from \cref{lem:FpSatisfiesAllProperties} that for all prime factors $p\mid m$ there exists a Frobenius $F_p\colon \qIW_m\Omega_{R/A}^*\rightarrow \qIW_{m/p}\Omega_{R/A}^*$ satisfying all properties from \cref{def:qFVSystemOfCDGA}. Furthermore, $F_p$ satisfies the $F$-Teichmüller condition \cref{enum:TeichmuellerF} by construction. This finishes the inductive construction of Frobenii on $(\qIW_m\Omega_{R/A}^*)_{m\in\IN}$, thus making it a $q$-$FV$-system of differential-graded algebras over $R$.
		
		For initiality, let $(P_m^*)_{m\in \IN}$ be an arbitrary $q$-$FV$-system. By definition of $(\qIW_m\Omega_{R/A}^*)_{m\in\IN}$, there is a unique morphism
		\begin{equation*}
			\bigl(\qIW_m\Omega_{R/A}^*\bigr)_{m\in\IN}\longrightarrow (P_m^*)
		\end{equation*}
		of $q$-$V$-systems, so there can be at most one morphism of $q$-$FV$-systems, depending on whether the above is compatible with the Frobenii. Let's show that this is always the case. Since the graded $A[q]$-algebra $\qIW_m\Omega_{R/A}^*$ is generated by elements in degree~$0$ and~$1$ (because it is a quotient of $\Omega_{\qIW_{m}(R/A)/A[q]}^*$ by \cref{prop:qDRWExists}\cref{enum:dRtoqDRWisSurjective}), it's enough to check compatibility with the Frobenii in degrees~$0$ and~$1$. In degree~$0$, this follows immediately from \cref{def:qFVSystemOfCDGA}\cref{enum:qDeRhamWittConditionC}. In degree~$1$, consider the diagram
		\begin{equation*}
			\begin{tikzcd}
				\qIW_m(R/A)\rar["\d"]\dar & \qIW_m\Omega_{R/A}^1\rar["F_p"]\dar & \qIW_{m/p}\Omega_{R/A}^1\dar\\
				P_m^0\rar["\d"] & P_m^1\rar["F_p"] & P_{m/p}^1
			\end{tikzcd}
		\end{equation*}
		The left square commutes by construction and we must show that the right square commutes too. It will be enough to show that the outer rectangle commutes, because then both ways of walking around the diagram will determine the same $A[q]$-linear derivation $\qIW_m(R/A)\rightarrow P_{m/p}^1$, hence the same map $\Omega_{\qIW_m(R/A)/A[q]}^1\rightarrow P_{m/p}^1$, and $\qIW_{m}\Omega_{R/A}^1$ is a quotient of $\Omega_{\qIW_m(R/A)/A[q]}^1$ by \cref{prop:qDRWExists}\cref{enum:dRtoqDRWisSurjective}.
		
		Commutativity of the outer rectangle can be checked ona set of $A[q]$-linear generators of $\qIW_m(R/A)$, so we can reduce to the special cases of \cref{par:PropertiesOfFpd}\cref{enum:FpdOnTeichmueller} and~\cref{enum:FpdOnVerschiebung}, where everything is clear. This finishes the proof that $(\qIW_m\Omega_{R/A}^*)_{m\in\IN}$ is initial in $\cat{CDGAlg}_{R/A}^{\q FV}$ too.
	\end{proof}
	\subsection{Étale base change}
	The goal of this subsection is to prove the following:
	\begin{prop}\label{prop:EtaleBaseChange}
		Let $R\rightarrow R'$ be an étale morphism of $A$-algebras. Then the canonical morphism $(\qIW_m\Omega_{R/A}^*)_{m\in\IN}\rightarrow (\qIW_m\Omega_{R'/A}^*)_{m\in\IN}$ induces isomorphisms of differential-graded $\qIW_m(R'/A)$-algebras
		\begin{equation*}
			\qIW_m(R'/A)\otimes_{\qIW_m(R/A)}\qIW_m\Omega_{R/A}^*\overset{\cong}{\longrightarrow}\qIW_m\Omega_{R'/A}^*\,.
		\end{equation*}
	\end{prop}
	To prove this, first we have to construct the differential-graded algebra structures on $\qIW_m(R'/A)\otimes_{\qIW_m(R/A)}\qIW_m\Omega_{R/A}^*$. This is achieved by the following lemma.
	\begin{lem}\label{lem:EtaleBaseChangeOfCDGA}
		Let $P^*$ be a differential-graded $A[q]$-algebra concentrated in nonnegative \embrace{cohomological} degrees and let $P^0\rightarrow S$ be an étale morphism of rings. Then the graded $A[q]$-algebra $S\otimes_{P^0}P^*$ admits a unique differential-graded $A[q]$-algebra structure compatible with the one on $P^*$. Furthermore, this exhibits $S\otimes_{P^0}P^*$ as an initial object among all differential-graded $A[q]$-algebras $P^*$-algebras $Q^*$ equipped with a ring map $S\rightarrow Q^0$.
	\end{lem}
	\begin{proof}
		This elegant proof is taken from \cite[Proposition~\href{https://www.math.ias.edu/~lurie/papers/Crystalline.pdf\#theorem.5.3.2}{5.3.2}]{SaturatedDeRhamWitt}. Since $P^0\rightarrow S$ is étale, we obtain $S\otimes_{P^0}\Omega_{P^0/A[q]}^*\cong \Omega_{S/A[q]}^*$ as graded rings. Then
		\begin{equation*}
			S\otimes_{P^0}P^*\cong \Omega_{S/A[q]}^*\otimes_{\Omega_{P^0/A[q]}^*}P^*\,,
		\end{equation*}
		where $\Omega_{P^0/A[q]}^*\rightarrow P^*$ is the differential-graded morphism induced by the universal property of the algebraic de Rham complex. Now the tensor product on the right-hand side of the isomorphism above carries an obvious differential-graded structure, which also clearly satisfies the desired universal property.
	\end{proof}
	\begin{proof}[Proof of \cref{prop:EtaleBaseChange}]
		We use induction on $m$ and work with the truncated universal properties from \cref{rem:TruncatedUniversalProperty2}, applied to the truncation set $T_m$ of positive divisors of $m$. The case $m=1$ is trivial as $\qIW_1\Omega_{R/A}^*\cong \Omega_{R/A}^*$ by \cref{prop:qDRWExists}\cref{enum:dRtoqDRWisSurjective} and same for $R'$. So let $m>1$ and assume that the base change formula is true for all divisors $d\neq m$ of $m$. We equip $\qIW_m(R'/A)\otimes_{\qIW_m(R/A)}\qIW_m\Omega_{R/A}^*$ with the differential-graded $A[q]$-algebra structure from \cref{lem:EtaleBaseChangeOfCDGA}, using that $\qIW_m(R/A)\rightarrow \qIW_m(R'/A)$ is étale by \cref{prop:vanDerKallen}. Furthermore, $\qIW_d(R'/A)\otimes_{\qIW_d(R/A)}\qIW_d\Omega_{R/A}^*\cong \qIW_m(R'/A)\otimes_{\qIW_m(R/A)}\qIW_d\Omega_{R/A}^*$ holds by the second part of \cref{prop:vanDerKallen}. Consequently, we can define
		\begin{equation*}
			V_{m/d}'\colon \qIW_d(R'/A)\otimes_{\qIW_d(R/A)}\qIW_d\Omega_{R/A}^*\longrightarrow \qIW_m(R'/A)\otimes_{\qIW_m(R/A)}\qIW_m\Omega_{R/A}^*
		\end{equation*}
		to be the $\qIW_m(R'/A)$-linear extension of the Verschiebung $V_{m/d}\colon \qIW_d\Omega_{R/A}^*\rightarrow \qIW_m\Omega_{R/A}^*$. If we can show that these $V_{m/d}'$ satisfy the conditions from \cref{def:qVSystemOfCDGA}, then combining the universal property of $(\qIW_d\Omega_{R/A}^*)_{d\in T_m}$ with the universal property of the differential-graded structure on $\qIW_m(R'/A)\otimes_{\qIW_m(R/A)}\qIW_m\Omega_{R/A}^*$ obtained from \cref{lem:EtaleBaseChangeOfCDGA} will show that $(\qIW_d(R'/A)\otimes_{\qIW_d(R/A)}\qIW_d\Omega_{R/A}^*)_{d\in T_m}$ satisfies the universal property from \cref{rem:TruncatedUniversalProperty2}. In particular, it will immediately show $\qIW_m(R'/A)\otimes_{\qIW_m(R/A)}\qIW_m\Omega_{R/A}^*\cong\qIW_m\Omega_{R'/A}^*$, thus finishing the induction.
		
		Most conditions from \cref{def:qVSystemOfCDGA} are straightforward to check, except for two tricky ones: $V'_{m/d}(\omega\d\eta)=V'_{m/d}(\omega)\d V'_{m/d}(\eta)$ and the $V$-Teichmüller condition \cref{enum:TeichmuellerV}. Nevertheless, these can be checked without doing any calculations. Let
		\begin{equation*}
			F_{m/d}'\colon \qIW_m(R'/A)\otimes_{\qIW_m(R/A)}\qIW_m\Omega_{R/A}^*\longrightarrow \qIW_d(R'/A)\otimes_{\qIW_d(R/A)}\qIW_d\Omega_{R/A}^*
		\end{equation*}
		be the $\qIW_m(R'/A)$-linear extension of the Frobenius $F_{m/d}$. As noted in \cref{rem:Redundancies}, to show the conditions for $V_{m/d}'$, it's enough to check that $F_{m/d}'$ satisfies the conditions from \cref{def:qFVSystemOfCDGA}\cref{enum:qDeRhamWittConditionC}, which are clear except for $F'_{m/d}\circ \d\circ V'_{m/d}=\d$, and the $F$-Teichmüller condition \cref{enum:TeichmuellerF}. Both of these are assertions about
		\begin{equation*}
			F'_{m/d}\circ \d\colon \qIW_m(R'/A)\rightarrow \qIW_d(R'/A)\otimes_{\qIW_d(R/A)}\qIW_d\Omega_{R/A}^*\,.
		\end{equation*}
		It would certainly be enough to show that $F_{m/d}'\circ\d$ agrees with the derivation $F_{m/d}\circ \d\colon \qIW_m(R'/A)\rightarrow \qIW_d\Omega_{R'/A}^1$ from the actual $q$-de Rham Witt complexes over $R'$. But $F_{m/d}'\circ\d$ is derivation as well, because $\d$ is derivation and $F_{m/d}'$ is a map of graded $\qIW_m(R'/A)$-algebras. Now whether two derivations on $\qIW_m(R'/A)$ agree can be checked after restriction along the étale morphism $\qIW_m(R/A)\rightarrow \qIW_m(R'/A)$. By construction, $F_{m/d}\circ \d$ and $F_{m/d}'\circ \d$ agree on $\qIW_m(R/A)$, so we're done.
	\end{proof}
	\begin{cor}\label{cor:qDRWEtaleSheaf}
		For all positive integers $m$, the functor
		\begin{equation*}
			\qIW_m\Omega_{-/A}\colon \cat{CRing}_A\longrightarrow \CAlg\bigl(\Dd\bigl(A[q]\bigr)\bigr)
		\end{equation*}
		sending an $A$-algebra $R$ to the $\IE_\infty$-$A[q]$-algebra underlying the differential-graded $A[q]$-algebra $\qIW_m\Omega_{R/A}^*$, is an étale sheaf.
	\end{cor}
	\begin{proof}
		It suffices to show that the underlying functor $\cat{CRing}_A\rightarrow \Dd(\IZ[q])$ is an étale sheaf. By writing $\qIW_m\Omega_{R/A}$ as a (derived) limit over its stupid truncations $\qIW_m\Omega_{R/A}^{\leqslant i}$ and passing to graded pieces, it's enough to show that $\qIW_m\Omega_{-/A}^i$ is an étale sheaf for every $i\geqslant 0$. Using \cref{prop:EtaleBaseChange}, this will be a consequence of the following assertion:
		\begin{alphanumerate}
			\item[\boxtimes] \itshape Let $R$ be an $A$-algebra and let $M\in\Dd(\qIW_m(R/A))$ be a bounded below complex. Then the functor $R'\mapsto \qIW_m(R'/A)\lotimes_{\qIW_m(R/A)}M$ defines a $\Dd(\IZ[q])$-valued sheaf on the small étale site of $R$.\label{enum:QuasiCoherentSheaf}
		\end{alphanumerate}
		To prove \cref{enum:QuasiCoherentSheaf}, first note that $\qIW_m(R'/A)\simeq \qIW_m(R')\lotimes_{\qIW_m(R)}\qIW_m(R/A)$ follows from \cref{lem:qWittEtaleBaseChange}. So the functor under consideration agrees with $R'\mapsto \qIW_m(R')\lotimes_{\qIW_m(R)}M$. We use induction on $m$. For $m=1$, we can write $R'\lotimes_RM\simeq \lim_{i\geqslant 0}R'\lotimes_R\tau_{\leqslant i}M$ due to connectivity reasons to reduce to the case where $M$ is bounded, and then further to the case where $M$ is concentrated in a single degree. In this case we simply obtain a quasi-coherent sheaf on the small étale site of $R$, which has vanishing higher cohomology \cite[Corollaire~\href{https://www.cmls.polytechnique.fr/perso/laszlo/sga4/SGA4-2/sga42.pdf\#page=245}{VII.4.4}]{sga4.2} and is therefore also a sheaf with values in the $\infty$-category $\Dd(\IZ[q])$.
		
		For $m>1$, we use \cref{prop:qWittKoszulExactSequence}: It's enough to prove that $R'\mapsto R'[\zeta_m]\lotimes_{\qIW_m(R)}M$ and $R'\mapsto\qIW_d(R')\lotimes_{\qIW_m(R)}M$, for $d\neq m$ a divisor of $m$, are $\Dd(\IZ[q])$-valued sheaves. For the latter, we can simply apply the inductive hypothesis to $\qIW_d(R)\lotimes_{\qIW_m(R)}M\in \Dd(\qIW_d(R))$. To see that the former functor constitutes a sheaf, we can use a similar argument as in the $m=1$ case, applied to $R[\zeta_m]\lotimes_{\qIW_m(R)}M\in \Dd(R[\zeta_m])$.
	\end{proof}
	\begin{cor}\label{cor:qDRWTrivialAfterLocalisation}
		For any $A$-algebra $R$ and any positive integer $m$, the ghost maps from \cref{par:qdRWGhostMaps} induce isomorphisms of differential-graded-$A[q]$-algebras
		\begin{equation*}
			\qIW_m\Omega_{R/A}^*\left[\localise{m}\right]\overset{\cong}{\longrightarrow} \prod_{d\mid m}\left(\Omega_{R/A}^*\otimes_{A,\psi^d}A\left[\localise{m},\zeta_d\right]\right)%\cong \Omega_{R/A}^*\otimes_AA\left[\localise{m},q\right]/(q^m-1)\,.
		\end{equation*}
		%The same holds true if $R$ has finite type over $\IZ$ and we replace $\qIW_m\Omega_R^*$, $\Omega_{R[\zeta_d]/\IZ[\zeta_d]}^*$, and $\Omega_R^*$ by their degree-wise \embrace{derived or underived} $p$-completions for any prime $p$.
	\end{cor}
	\begin{proof}
		We'll compare universal properties. Using \cref{cor:qWittLocalisation}\cref{enum:MultiplicativeSubsetOfZ} and the universal property from \cref{def:RelativeqWitt}, it's clear that $\qIW_m(R[1/m]/A)\cong \qIW_m(R/A)[1/m]$. Combined with \cref{prop:EtaleBaseChange}, we obtain
		\begin{equation*}
			\qIW_d\Omega_{R/A}^*\left[\localise{m}\right]\cong \qIW_d\left(R\left[\localise{m}\right]/A\right)\otimes_{\qIW_d(R/A)}\qIW_d\Omega_{R/A}^*\cong \qIW_d\Omega_{R[1/m]/A}^*
		\end{equation*}
		for all $d\mid m$. Thus, if $T_m$ denotes the truncation set of positive divisors of $m$, then $(\qIW_d\Omega_{R/A}^*[1/m])_{d\in T_m}$ is the initial $T_m$-truncated $q$-$V$-system over $R[1/m]$ in the sense of \cref{rem:TruncatedUniversalProperty2}. We will show that $\bigl(\prod_{e\mid d}(\Omega_{R/A}^*\otimes_{A,\psi^e}A[1/m,\zeta_e])\bigr)_{d\in T_m}$, with its $T_m$-truncated $q$-$V$-system structure from \cref{par:qdRWGhostMaps}, is initial too. 
		
		To prove this, first observe that the relative $q$-Witt vector ghost maps induce isomorphisms
		\begin{equation*}
			(\gh_{d/e})_{e\mid d}\colon \qIW_d\left(R\left[\localise{m}\right]/A\right)\overset{\cong}{\longrightarrow}\prod_{e\mid d}\Bigl(R\otimes_{A,\psi^e}A\left[\localise{m},\zeta_e\right]\Bigr)
		\end{equation*}
		Indeed, the canonical map $\qIW_m(R[1/m])\otimes_{\qIW_m(A)}A[q]/(q^m-1)\rightarrow \qIW_d(R[1/m]/A)$ is surjective by \cref{lem:RelativeqWitt}, but according to \cref{exm:GhostMapsIsosForTrivialLambdaRing} we also have an isomorphism $\qIW_d(R[1/m])\otimes_{\qIW_m(A)}A[q]/(q^m-1)\cong \prod_{e\mid d}(R\otimes_{A,\psi^e}A)[1/m,\zeta_e]$, and so it is a left inverse of $(\gh_{d/e})_{d\mid e}$.
		
		Now let $(P_d^*)_{d\in T_m}$ be an arbitrary $T_m$-truncated $q$-$V$-system over $R[1/m]$. Observe that for all $d\mid m$, the decomposition $A[1/m,q]/(q^d-1)\cong \prod_{e\mid d}A[1/m,\zeta_e]$  induces a similar decomposition $P_d^*\cong \prod_{e\mid d}P_{d,e}^*$, where $P_{d,e}^*$ is a differential-graded $A[\zeta_e]$-algebra. Furthermore, the $\qIW_d(R[1/m]/A)$-algebra structure on $P_d^0$ plus the ghost map isomorphism above induce $(R\otimes_{A,\psi^e}A)[1/m,\zeta_e]$-algebra structures on $P_{d,e}^0$ for all $e\mid d$. Thus, we obtain canonical morphisms
		\begin{equation*}
			\prod_{e\mid d}\Bigl(\Omega_{R/A}^*\otimes_{A,\psi^e}A\left[\localise{m},\zeta_e\right]\Bigr)\longrightarrow\prod_{e\mid d}P_{d,e}^*\,.
		\end{equation*}
		By \cref{lem:dAfterV}, these are automatically compatible with the Verschiebungen. This proves that $\bigl(\prod_{e\mid d}(\Omega_{R/A}^*\otimes_{A,\psi^e}A[1/m,\zeta_e])\bigr)_{d\in T_m}$ satisfies the required universal property.
		%
		%Now assume $R$ is of finite type over $\IZ$ and fix some degree $i\geqslant0$. \cref{cor:qWittNoetherian} and \cref{prop:qDRWExists}\cref{enum:dRtoqDRWisSurjective} show that $\qIW_m\Omega_R^i\rightarrow \prod_{d\mid m}\Omega_{R[\zeta_d]/\IZ[\zeta_d]}^i$ is a map between finitely generated modules over the noetherian ring $\qIW_m(R)$. In particular, the kernel and cokernel must be finitely generated again. By the above, both the kernel and cokernel must be annihilated by some power of $m$. Then the same remains true after $p$-completion. Since $p$-completion is exact on finitely generated $\qIW_m(R)$-modules, we conclude that $(\qIW_m\Omega_R^i)_p^\complete\rightarrow \prod_{d\mid m}(\Omega_{R[\zeta_d]/\IZ[\zeta_d]}^i)_p^\complete$ becomes an isomorphism after localisation at $m$.
	\end{proof}
	\newpage

	\section{\texorpdfstring{$q$}{q}-de Rham Witt complexes in the smooth case}\label{sec:qDeRhamWittSmooth}
	Fix a $\Lambda$-ring $A$ which is \emph{perfectly covered} in the sense of \cref{rem:FaithfullyFlatCoverByPerfectLambdaRing}. The most important special case is $A=\IZ$. In this section we'll study $\qIW_m\Omega_{R/A}^*$ for smooth $A$-algebras $R$. Our goal will be to prove \cref{thm:qDeRhamWittGlobalIntro}---and in fact, the more general version \cref{thm:qDeRhamWittqHodge}---as well as the following two propositions.
	\begin{prop}\label{prop:qDRWisTorsionFreeForSmoothRings}
		Let $R$ be smooth over $A$. Then $\qIW_m\Omega_{R/A}^*$ is degree-wise $\IZ$-torsion-free for all $m\in \IN$. In particular, $(\qIW_m\Omega_{R/A}^*)_{m\in\IN}$ is also an initial object of $(\cat{CDGAlg}_{R/A}^{\q V})^{\mathrm{tors\mhyph free}}$.
	\end{prop}
	\begin{prop}\label{prop:qDRWTrivialAfterpCompletion}
		Let $R$ be smooth over $A$ and let $p$ be a prime. Then for every exponent~$\alpha$ there exists an equivalence of $p$-complete $\IE_\infty$-$A[q]$-algebras
		\begin{equation*}
			\bigl(\Omega_{R/A}\lotimes_{A,\psi^{p^\alpha}}A[q]/(q^{p^\alpha}-1)\bigr)_p^\complete\overset{\simeq}{\longrightarrow}\left(\qIW_{p^\alpha}\Omega_{R/A}\right)_p^\complete\,,
		\end{equation*}
		functorial in smooth $A$-algebras $R$. More generally, if $m=p^\alpha n$, where $n$ is coprime to $p$, then there exists an equivalence of $p$-complete $\IE_\infty$-$A[q]$-algebras
		\begin{equation*}
			\prod_{d\mid n}\Bigl(\Omega_{R/A}\lotimes_{A,\psi^{p^\alpha d}}A[q]/\Phi_d(q^{p^\alpha})\Bigr)_p^\complete\overset{\simeq}{\longrightarrow}\left(\qIW_m\Omega_{R/A}\right)_p^\complete\,,
		\end{equation*}
		functorial in smooth $A$-algebras $R$.
	\end{prop}
	\begin{numpar}[Battle plan.]\label{par:OutlineOfStrategy}
		Let us explain the logical structure of this section, since it'll be not entirely obvious. We'll prove \cref{thm:qDeRhamWittGlobalIntro}, \cref{prop:qDRWisTorsionFreeForSmoothRings}, and \cref{prop:qDRWTrivialAfterpCompletion} first in the case where $m=p^\alpha$ is a prime power; the general case can be reduced to this by standard arguments (as we'll see). To handle the special case $m=p^\alpha$, we'll prove all three results at once using an induction on $\alpha$. More precisely, we'll show the following four assertions using induction on $\alpha$:
		\begin{alphanumerate}
			\item[a_\alpha] If $R$ is smooth over $A$, then the differential-graded $A[q]$-algebra $\qIW_{p^\alpha}\Omega_R^*$ is degree-wise $p$-torsion-free.\label{enum:alphaA}
			\item[b_\alpha] \cref{thm:qDeRhamWittGlobalIntro}, and more generally \cref{thm:qDeRhamWittqHodge} that we'll state below, are true after $p$-completion for $m=p^\alpha$.\label{enum:alphaB}
			\item[c_\alpha] \cref{prop:qDRWTrivialAfterpCompletion} is true for $m=p^\alpha$.\label{enum:alphaC}
			\item[d_\alpha] Suppose $R\cong A[T_1,\dotsc,T_n]$ is a polynomial ring over $A$. If $\xi\in \qIW_{p^\alpha}\Omega_{R/A}^i$ satisfies $\d\xi\equiv 0\mod p$, then there exist $\omega\in \qIW_{p^{\alpha+1}}\Omega_{R/A}^i$ and $\eta\in \qIW_{p^\alpha}\Omega_{R/A}^i$ satisfying\label{enum:alphaD}
			\begin{equation*}
				\xi=F_p(\omega)+p\eta\,.
			\end{equation*}
		\end{alphanumerate}
		Assertion (\hyperref[enum:alphaD]{$d_{-1}$}) is vacuously true. To carry out the inductive step, we'll prove the implications (\hyperref[enum:alphaC]{$d_{\alpha-1}$}) $\Rightarrow$ \cref{enum:alphaA} $\Rightarrow$ \cref{enum:alphaB} $\Rightarrow$ \cref{enum:alphaC} $\Rightarrow$ \cref{enum:alphaD}. The implication (\hyperref[enum:alphaC]{$d_{\alpha-1}$}) $\Rightarrow$ \cref{enum:alphaA} will be shown in \cref{subsec:pTorsionFree}. After introducing the $q$-Hodge complex and proving some first properties in \crefrange{subsec:qHodgeAdditive}{subsec:qHodgeMultiplicative}, we'll prove the implications \cref{enum:alphaB} $\Rightarrow$ \cref{enum:alphaC} $\Rightarrow$ \cref{enum:alphaD} in \cref{subsec:pTypical}. Finally, in \cref{subsec:Global}, we'll deduce the global cases of \cref{thm:qDeRhamWittGlobalIntro,thm:qDeRhamWittqHodge} as well as \cref{prop:qDRWisTorsionFreeForSmoothRings,prop:qDRWTrivialAfterpCompletion}.
		
		For ease of notation, throughout the induction we'll denote the $p$\textsuperscript{th} Adams operation $\psi^p\colon A\rightarrow A$ for the fixed prime $p$ instead by $\phi\colon A\rightarrow A$.
	\end{numpar}
	%Before we dive into the proofs, let us make two philosophical remarks.
	\begin{rem}\label{rem:qDRWNotTrivial}
		Suppose $A=\IZ$ and $R$ is smooth over~$\mathbb Z$. On first glance, it seems \cref{cor:qDRWTrivialAfterLocalisation} and \cref{prop:qDRWTrivialAfterpCompletion} could be combined to show $\qIW_m\Omega_{R/\IZ}\simeq \Omega_{R/\IZ}\lotimes_\IZ\IZ[q]/(q^m-1)$ as $\IE_\infty$-$\IZ[q]$-algebras. Indeed, from \cref{cor:qDRWTrivialAfterLocalisation} and
		$\IZ[1/m,q]/(q^m-1)\cong \prod_{d\mid m}\IZ[1/m,q]/\Phi_d(q)$ we get
		\begin{equation*}
			\qIW_m\Omega_{R/\IZ}^*\left[\localise{m}\right]\cong \Omega_{R/\IZ}^*\otimes_\IZ\IZ\left[\localise{m},q\right]/(q^m-1)\,.
		\end{equation*}
		Similarly, in \cref{prop:qDRWTrivialAfterpCompletion} the Adams operations $\psi^{p^\alpha d}$ become trivial and we obtain
		\begin{equation*}
			\bigl(\qIW_m\Omega_{R/\IZ}\bigr)_p^\complete\simeq \bigl(\Omega_{R/\IZ}\lotimes_{\IZ}\IZ[q]/(q^m-1)\bigr)_p^\complete\,.
		\end{equation*}
		But these two equivalences are usually not compatible! We can already see this in the case where $R$ is étale over $\IZ$: By \cref{cor:qWittTrivialisationGivesFrobeniusLift}, the existence of an equivalence $\IE_\infty$-$\IZ[q]$-algebra equivalence between $\qIW_m\Omega_R\simeq \qIW_m(R)$ and $\Omega_{R/\IZ}\lotimes_\IZ\IZ[q]/(q^m-1)\simeq R[q]/(q^m-1)$ is obstructed by the existence of a $\Lambda_m$-structure on $R$.
		
		Upon closer examination, this also explains why \cref{cor:qDRWTrivialAfterLocalisation} and \cref{prop:qDRWTrivialAfterpCompletion} cannot be combined to construct an isomorphism $\qIW_m(R)\cong R[q]/(q^m-1)$ in the case where $R$ is étale over $\IZ$: In this case, $\qIW_m(R)_p^\complete\simeq \widehat{R}_p[q]/(q^m-1)$ comes from  \cref{cor:qWittOfPerfectLambdaRing}, noticing that $\qIW_m(R)_p^\complete\simeq \qIW_m(\widehat{R}_p)$ by \cref{cor:qWittpCompletion} and that $\widehat{R}_p$ carries a unique Frobenius lift $\phi_p\colon \widehat{R}_p\rightarrow \widehat{R}_p$, which can be trivially upgraded to a perfect $\Lambda$-structure by declaring the other Adams operations to be the identity. On the other hand, as explained in \cref{exm:GhostMapsIsosForTrivialLambdaRing}, the isomorphism $\qIW_m(R)[1/m]\cong R[1/m,q]/(q^m-1)$ comes from the trivial perfect $\Lambda_m$-structure on $R[1/m]$, in which all Adams operations are the identity. So the two isomorphisms are incompatible, unless $\phi_p\colon \widehat{R}_p\rightarrow \widehat{R}_p$ happens to be restrict to a Frobenius lift on $R$.
		
		In \cref{cor:qDRWArithmeticFractureSquare} we'll continue these considerations for arbitrary perfectly covered $\Lambda$-rings $A$ and arbitrary smooth $A$-algebras $R$. 
	\end{rem}
%	\begin{rem}
%		If $R$ is étale over $A$, \cref{prop:qDRWTrivialAfterpCompletion} can be viewed as a higher-dimensional analogue of the equivalence $\qIW_m(R)_p^\complete\simeq \widehat{R}_p[q]/(q^m-1)$ that was explained in \cref{rem:qDRWNotTrivial}. If $R$ is smooth over $\IZ$ of relative dimension $\geqslant 1$, then it will be no longer true that the $p$-completion of $R$ carries a canonical Frobenius lift. Instead, there's a canonical Frobenius lift on the $\IE_\infty$-algebra $(\Omega_{R/\IZ})_p^\complete$; this Frobenius lift can be identified with the crystalline Frobenius under the equivalence
%		\begin{equation*}
%			(\Omega_{R/\IZ})_p^\complete\simeq \R\Gamma_\crys\bigl((R/p)/\IZ_p\bigr)\,.
%		\end{equation*}
%	\end{rem}
	
	\subsection{\texorpdfstring{$p$}{p}-Torsion freeness of \texorpdfstring{$q$}{q}-de Rham--Witt complexes}\label{subsec:pTorsionFree}
	In this subsection we'll prove the implication \embrace{\hyperref[enum:alphaD]{$d_{\alpha-1}$}} $\Rightarrow$ \cref{enum:alphaA} of our battle plan~\cref{par:OutlineOfStrategy}. This needs a preparatory lemma.
	\begin{lem}\label{lem:qDRWGhostMap}
		The subset $\IV_{p^\alpha}^*\coloneqq\im V_p+\im \d V_p\subseteq \qIW_{p^\alpha}\Omega_{R/A}^*$ is a differential-graded ideal and the ghost map $\gh_1$ from \cref{par:qdRWGhostMaps} induces a functorial isomorphism
		\begin{equation*}
			\gh_1\colon \qIW_{p^\alpha}\Omega_{R/A}^*/\IV_{p^\alpha}^*\overset{\cong }{\longrightarrow}\Omega_{R/A}^*\otimes_{A,\phi^\alpha}A[\zeta_{p^\alpha}]\,.
		\end{equation*}
	\end{lem}
	\begin{proof}
		It's clear that $\IV_{p^\alpha}^*$ is closed under $\d$. To show that it is a graded ideal, choose homogeneous elements $\omega\in \qIW_{p^{\alpha-1}}\Omega_{R/A}^i$ and $\eta\in \qIW_{p^\alpha}\Omega_{R/A}^j$. We compute $V_p(\omega)\eta=V(\omega F_p(\eta))$ and
		\begin{equation*}
			\d V_p(\omega)\eta=\d\bigl(V_p(\omega)\eta\bigr)-(-1)^iV_p(\omega)\d\eta=\d V_p\bigl(\omega F_p(\eta)\bigr)-(-1)^iV\bigl(\omega F_p(\d\eta)\bigr)\,,
		\end{equation*}
		using the condition from \cref{def:qFVSystemOfCDGA}\cref{enum:qDeRhamWittConditionC}. This proves that $\IV_{p^\alpha}^*$ is a graded ideal.
		
		To prove the second assertion, note that $(0,\dotsc,0,\qIW_{p^\alpha}\Omega_{R/A}^*/\IV_{p^\alpha}^*)$ is initial among all $\{1,p,p^2,\dotsc,p^\alpha\}$-truncated $q$-$V$-systems $(P_{1}^*,P_{p}^*,P_{p^2}^*,\dotsc,P_{p^\alpha}^*)$ satisfying $P_{p^i}^*=0$ for all $i<\alpha$. By inspection, such a system is nothing else but a differential-graded $A[q]$-algebra $P_{p^\alpha}^*$ together with a $\qIW_{p^\alpha}(R)/\im V_{p^\alpha}$-algebra structure on $P_{p^\alpha}^0$; all the extra structure and conditions become trivial. Now $\qIW_{p^\alpha}(R)/\im V_p\cong R\otimes_{A,\phi^\alpha}A[\zeta_{p^\alpha}]$ by \cref{par:RelativeGhostMaps}, hence, according to the universal property of the de Rham complex, the initial $\{1,p,p^2,\dotsc,p^\alpha\}$-truncated $q$-$V$-system is also given by $(0,\dotsc,0,\Omega_{R/A}^*\otimes_{A,\phi^\alpha}A[\zeta_{p^\alpha}])$. This finishes the proof.
	\end{proof}
	\begin{proof}[Proof of \embrace{\hyperref[enum:alphaD]{$d_{\alpha-1}$}} $\Rightarrow$ \cref{enum:alphaA}]
		Assume first that $R\cong A[T_1,\dotsc,T_n]$ is a polynomial ring. Suppose $\xi\in \qIW_{p^\alpha}\Omega_{R/A}^i$ satisfies $p\xi=0$. By \cref{lem:qDRWGhostMap}, the quotient $\qIW_{p^\alpha}\Omega_R^*/\IV_{p^\alpha}^*$ is isomorphic to $\Omega_{R/A}^*\otimes_{A,\phi^\alpha}A[\zeta_{p^\alpha}]$, which is degree-wise $p$-torsion-free. Indeed, our assumptions imply that $A$ is $p$-torsion free (because its faithfully flat cover $A_\infty$ is $p$-torsion free, as is any perfect $\Lambda$-ring) and then $\Omega_{R/A}^*\otimes_{A,\phi^\alpha}A$ is degree-wise projective over $R\otimes_{A,\phi^\alpha}A$, which is smooth over $A$ and thus $p$-torsion free as well.
		
		Hence $p\xi=0$ implies $\xi\in \IV_{p^\alpha}^*$. So write $\xi=V_p(\xi_0)+\d V_p(\xi_1)$ for some $\xi_0\in \qIW_{p^{\alpha-1}}\Omega_{R/A}^i$ and $\xi_1\in \qIW_{p^{\alpha-1}}\Omega_{R/A}^{i-1}$. Since $V_p\circ \d=p(\d\circ V_p)$ by \cref{lem:dAfterV}, we can rewrite our assumption $p\xi=0$ as $V_p(p\xi_0+\d\xi_1)=0$. Note that $V_p\colon \qIW_{p^{\alpha-1}}\Omega_{R/A}^*\rightarrow \qIW_{p^\alpha}\Omega_{R/A}^*$ is injective, because $F_p\circ V_p=p$ and $\qIW_{p^{\alpha-1}}\Omega_{R/A}^*$ is degree-wise $p$-torsion-free by the inductive hypothesis. Thus $p\xi_0+\d\xi_1=0$.
		
		Applying \embrace{\hyperref[enum:alphaC]{$c_{\alpha-1}$}} shows that we can write $\xi_1=F_p(\omega)+p\eta$ for some $\omega\in \qIW_{p^\alpha}\Omega_{R/A}^{i-1}$ and $\eta\in \qIW_{p^{\alpha-1}}\Omega_{R/A}^{i-1}$. Then $\d\xi_1=pF_p(\d\omega)+p\d\eta$ and so our assumption $p\xi_0+d\xi_1=0$ implies $\xi_0=-F_p(\d\omega)-\d\eta$ by $p$-torsion-freeness of $\qIW_{p^{\alpha-1}}\Omega_{R/A}^*$. We conclude
		\begin{align*}
			\xi=V_p(\xi_0)+\d V_p(\xi_1)&=V_p\bigl(-F_p(\d\omega)-\d\eta\bigr)+\d V_p\bigl(F_p(\omega)+p\eta\bigr)\\
			&=-\Phi_{p^\alpha}(q)\d\omega -V_p(\d\eta)+\Phi_{p^\alpha}(q)\d\omega+p\d V_p(\eta)\\
			&=0\,,
		\end{align*}
		using $V_p\circ \d=p(\d\circ V_p)$, which holds by \cref{lem:dAfterV}, as well as $V_p\circ F_p=\Phi_{p^\alpha}(q)$, which holds by \cref{def:qFVSystemOfCDGA}\cref{enum:qDeRhamWittConditionC}. This finishes the proof in the polynomial ring case.
		
		Now let $R$ be an arbitrary smooth $A$-algebra. Fix a degree $i$; we've seen in the proof of \cref{cor:qDRWEtaleSheaf} that $\qIW_m\Omega_{-/A}^i$ is an étale sheaf with values in the $\infty$-category $\Dd(A[q])$. Then it's also a sheaf in the ordinary category of $A[q]$-modules. Hence the $p$-torsion part is an étale sheaf as well, since it can be written as the kernel of the multiplication map $p\colon \qIW_m\Omega_{-/A}^i\rightarrow \qIW_m\Omega_{-/A}^i$. Since smooth $A$-algebras are étale-locally polynomial rings, we conclude that the $p$-torsion part of $\qIW_m\Omega_{R/A}^i$ must be trivial, as desired.
	\end{proof}
	In a similar way, one can show the following technical result.
	\begin{lem}\label{lem:qDRWDegreewiseBoundedq-1Torsion}
		For all smooth $A$-algebras $R$ and all $m\in\IN$, the relative $q$-de Rham--Witt complex $\qIW_m\Omega_{R/A}^*$ has degree-wise bounded $(q-1)^\infty$-torsion. %In fact, the bound on the $(q-1)^\infty$-torsion only depends on the relative dimension of $R$ over $A$, but not on $R$ or $A$.
		In particular, for every degree $i$, the underived and derived $(q-1)$-completions of $\qIW_m\Omega_{R/A}^i$ agree.
	\end{lem}
	\begin{proof}
		If $A=\IZ$ and $R=\IZ[T_1,\dotsc,T_n]$, then \cref{cor:qWittNoetherian} and \cref{prop:qDRWExists}\cref{enum:dRtoqDRWisSurjective} imply that $\qIW_m\Omega_{R/\IZ}^*$ is a complex of finitely generated modules over the noetherian ring $\qIW_m(R)$ and the assertion is clear. If $A$ is an arbitrary $\Lambda$-ring (with the assumption that the morphism $A\rightarrow A_\infty$ into its colimit perfection is faithfully flat), and $R=A[T_1,\dotsc,T_n]$, then
		\begin{equation*}
			\qIW_m\Omega_{A[T_1,\dotsc,T_n]/A}^*\cong \qIW_m\Omega_{\IZ[T_1,\dotsc,T_n]/\IZ}^*\otimes_\IZ A
		\end{equation*}
		holds by \cref{lem:RelativeqDRWBaseChange}. Since our assumption implies that $A$ is flat over $\IZ$, we get bounded $(q-1)^\infty$-torsion in this case as well (with the same bound as in the case $A=\IZ$). Finally, using étale descent as in the proof of \embrace{\hyperref[enum:alphaD]{$d_{\alpha-1}$}} $\Rightarrow$ \cref{enum:alphaA} above, we get that $\qIW_m\Omega_{R/A}^*$ has bounded $(q-1)^\infty$-torsion for arbitrary smooth $A$-algebras $R$ (still with the same bound as in the case $A=\IZ$, $R=\IZ[T_1,\dotsc,T_n]$).
	\end{proof}
	
	\subsection{The \texorpdfstring{$q$}{q}-Hodge complex I: Additive Structure}\label{subsec:qHodgeAdditive}
	In the introduction \cref{sec:Intro}, we've only sketched the construction of the $q$-de Rham complex and the $q$-Hodge complex of a framed smooth $\IZ$-algebra $(R,\square)$. So let's do that now again, both in more detail and in the relative setting.
	\begin{numpar}[The $q$-de Rham and the $q$-Hodge complex.]\label{par:qDeRhamqHodge}
		Let $R$ be smooth over $A$ and suppose there exists an étale morphism $\square\colon A[T_1,\dotsc,T_n]\rightarrow R$. For all $i=1,\dotsc,n$ let $\gamma_i\colon A[T_1,\dotsc,T_n]\qpower\rightarrow A[T_1,\dotsc,T_n]\qpower$,  be the $A$-algebra morphism that sends $T_i\mapsto qT_i$ and leaves the other variables fixed. Observe that $\gamma_i$ is the identity modulo $q-1$, that $\square$ induces a $(q-1)$-completely étale morphism $A[T_1,\dotsc,T_n]\qpower\rightarrow R\qpower$ (see \cref{par:Notation} for the terminology), and that $R\qpower\rightarrow R$ is a $(q-1)$-complete pro-infinitesimal thickening. Hence there exists a unique dashed lift in the solid diagram
		\begin{equation*}
			\begin{tikzcd}
				A[T_1,\dotsc,T_n]\qpower\dar["\square"']\rar["\gamma_i"] & R\qpower\dar\\
				R\qpower\rar\urar[dashed, "\exists!"'] & R
			\end{tikzcd}
		\end{equation*}
		This lift will also be denoted $\gamma_i$. By lifting against $R\qpower/(q-1)T_i\rightarrow R$ instead, which is still a $(q-1)$-complete pro-infinitesimal thickening, we see that $\gamma_i$ is not only congruent to the identity modulo $q-1$, but also modulo $(q-1)T_i$. This allow us define algebraic versions $\q\partial_i\colon R\qpower\rightarrow R\qpower$ of Jackson's $q$-derivatives from \cref{par:qDeRhamComplex} using the formula
		\begin{equation*}
			\q\partial_i f\coloneqq \frac{\gamma_i(f)-f}{qT_i-T_i}
		\end{equation*}
		for $i=1,\dotsc,n$. Note that $\q\partial_i$ and $\q\partial_j$ commute for all $i$ and $j$. Indeed, this reduces to the same assertion for $\gamma_i$ and $\gamma_j$, which follows once again by an infinitesimal lifting argument. We may thus construct the \emph{$q$-de Rham complex of $(R,\square)$} as the Koszul complex of the commuting $A\qpower$-module endomorphisms $\q\partial_1,\dotsc,\q\partial_n$:
		\begin{equation*}
			\qOmega_{R/A,\square}^*\coloneqq\left(R\qpower \xrightarrow{\q\nabla}\Omega_{R/A}^1\qpower\xrightarrow{\q\nabla} \dotsb \xrightarrow{\q\nabla} \Omega^n_{R/A}\qpower \right)\,.
		\end{equation*}
		Similarly, the \emph{$q$-Hodge complex of $(R,\square)$} is the Koszul complex of $(q-1)\q\partial_1,\dotsc,(q-1)\q\partial_n$:
		\begin{equation*}
			\qHodge_{R/A,\square}^*\coloneqq\left(R\qpower \xrightarrow{(q-1)\q\nabla}\Omega_{R/A}^1\qpower\xrightarrow{(q-1)\q\nabla} \dotsb \xrightarrow{(q-1)\q\nabla} \Omega^n_{R/A}\qpower \right)\,.
		\end{equation*}
	\end{numpar}
	\begin{numpar}[Non-commutative multiplicative structure.]\label{par:MultiplicativeStructure}
		It's straightforward to check that the partial $q$-derivative $\q\partial_i$ satisfies the \emph{$q$-Leibniz rule} $\q \partial_i(fg)=f\q\partial_i g+\gamma_i(g)\q\partial_if$. This allows us to equip the $q$-de Rham complex with a non-commutative differential-graded algebra structures as follows: For homogeneous elements $\omega=f\d T_{i_1}\wedge\dotsb\wedge\d T_{i_k}\in \qOmega_{R/A,\square}^k$ and $\eta=g\d T_{j_1}\wedge\dotsb\wedge\d T_{j_\ell}\in \Omega_{R/A,\square}^\ell$ we put
		\begin{equation*}
			\omega\wedge\eta\coloneqq f\gamma_{i_1}\big(\gamma_{i_2}(\dotsb \gamma_{i_k}(g)\dotsb )\big)\d T_{i_1}\wedge\dotsb\wedge\d T_{i_k}\wedge \d T_{j_1}\wedge\dotsb\wedge\d T_{j_\ell}\,.
		\end{equation*}
		More succinctly, we use the good old wedge product and impose the additional non-commutative rule $\d T_i\wedge f\coloneqq \gamma_i(f)\wedge\d T_i$ for all $f\in R\qpower$ and all $i=1,\dotsc,d$. From the $q$-Leibniz rule, we easily get $\q\nabla(\omega\wedge\eta)=\q\nabla(\omega)\wedge\eta+(-1)^k\omega\wedge\q\nabla(\eta)$, so this multiplication does indeed define a differential-graded $A\qpower$-algebra structure on $\qOmega_{R/A,\square}^*$. The same definition also works for the $q$-Hodge complex $\qHodge_{R/A,\square}^*$.
	\end{numpar}
	Our eventual goal is to compute the cohomology $\H^*(\qHodge_{R/A,\square}^*/(q^m-1))$, including its multiplicative structure coming from \cref{par:MultiplicativeStructure}. As a preparation, we'll now determine the additive structure of $\H^*((\qHodge_{R/A,\square}^*)_p^\complete/(q^{p^\alpha}-1))$ as an $\widehat{A}_p\qpower$-module.%
	\footnote{Observe that our assumption on the existence of a faithfully flat map $A\rightarrow A_\infty$ into a perfect $\Lambda$-ring implies that $A$ and $R$ are $p$-torsionfree, so it doesn't matter whether we interpret $\widehat{A}_p$ and $\widehat{R}_p$ as the derived or the underived completions. Similarly, $\qHodge_{R/A,\square}^*$ is degree-wise $p$-torsion free, so it's derived $p$-completion agrees with the degree-wise underived completion.}
	Our strategy will be to construct a certain decomposition of $(\qHodge^*_{R,\square})_p^\complete/(q^{p^n}-1)$ according to a Frobenius lift on $\widehat{R}_p$. This is an old trick, going back (at least) to Katz's proof of the Cartier isomorphism in \cite[Theorem~(\href{https://web.math.princeton.edu/~nmk/old/nilpconn.pdf\#page=27}{7.2})]{Katz}.
	
	\begin{numpar}[The Frobenius lift.]\label{par:FrobeniusLift}
		The $p$-completion $\widehat{A}_p\langle T_1,\dotsc,T_n\rangle\coloneqq (A[T_1,\dotsc,T_n])_p^\complete$ can be equipped with a $\delta$-structure in which the Frobenius $\phi_\square$ is given by the Frobenius $\phi=\psi^p$ on $A$ and $\phi_\square(T_i)\coloneqq T_i^p$. Since $\square\colon \widehat{A}_p\langle T_1,\dotsc,T_n\rangle \rightarrow \widehat{R}_p$ is $p$-completely étale, $\phi_\square$ admits a unique extension to a Frobenius lift on $\widehat{R}_p$, which we still denote $\phi_\square\colon \widehat{R}_p\rightarrow \widehat{R}_p$ by abuse of notation. We observe that $\phi_\square $ is injective. Indeed, $\phi_\square$ is injective modulo $p$, since $R/p$ is reduced: It's a smooth $A/p$-algebra, and $A/p$ must be reduced because it admits a faithfully flat cover $A/p\rightarrow A_\infty/p$ by a perfect $\IF_p$-algebra. Hence every $x\in \widehat{R}_p$ with $\phi_\square (x)=0$ must be divisible by $p$; say $x=px'$. But then $0=\phi_\square(px')=p\phi_\square(x')$ implies $\phi_\square(x')=0$. Iterating this argument shows that $x$ is divisible by $p$ arbitrarily many times. But $\widehat{R}_p$ is $p$-complete and thus $p$-adically separated, so $x=0$, as required.
	\end{numpar}
	\begin{numpar}[The Frobenius decomposition I.]\label{par:FrobeniusDecompositions}
		
		For every $\alpha\geqslant 0$, let
		\begin{equation*}
			\widehat{R}_p^{(\alpha)}\coloneqq \left(\phi_\square^\alpha(\widehat{R}_p)\otimes_{\phi^\alpha(\widehat{A}_p)}\widehat{A}_p\right)_p^\complete
		\end{equation*}
		(as usual, it doesn't matter whether we complete in the derived or underived sense, as can be seen by base change to $A_\infty$). We claim that the canonical map $\widehat{R}_p^{(\alpha)}\rightarrow \widehat{R}_p$ exhibits $\widehat{R}_p$ as a free module over $\widehat{R}_p^{(\alpha)}$, with a basis given by $T_1^{v_1}\dotsm T_n^{v_n}$ for all multi-indices $v=(v_1,\dotsc,v_n)$ satisfying $0\leqslant v_i\leqslant p^\alpha-1$. To see why this is true, first observe that the corresponding assertion for $\widehat{A}_p\langle T_1,\dotsc,T_n\rangle$ is true for obvious reasons. Hence it suffices to show that
		\begin{equation*}
			\begin{tikzcd}
				\widehat{A}_p\langle T_1,\dotsc,T_n\rangle \rar\drar[pushout]\dar["\phi_\square^\alpha"'] & \widehat{R}_p\dar["\phi_\square^\alpha"]\\
				\widehat{A}_p\langle T_1,\dotsc,T_n\rangle \rar & \widehat{R}_p
			\end{tikzcd}
		\end{equation*}
		is a derived pushout square of rings. But both the derived pushout and $\widehat{R}_p$ are derived $p$-complete, so by the derived Nakayama lemma it's enough to check that we get a derived pushout square after applying $-\lotimes_{\IZ_p}\IF_p$ everywhere. This is proved in \cite[\stackstag{0EBS}]{Stacks}.

		Now, for every multi-index $v=(v_1,\dotsc,v_n)$ satisfying $0\leqslant v_i\leqslant p^\alpha-1$ for all $i$, we let $(\qHodge_{R/A,\square}^{*,v})_p^\complete\subseteq (\qHodge_{R/A,\square}^*)_p^\complete$ be the free graded $\widehat{R}_p^{(\alpha)}\qpower$-module with basis the elements
		\begin{equation*}
			\prod_{i\in I}T_i^{v_i}\bigwedge_{j\in J}T_j^{v_j-1}\d T_j
		\end{equation*}
		for all disjoint decompositions $I\sqcup J=\{1,\dotsc,n\}$. If $v_j=0$ for some $j$, we use the convention that $T_j^{v_j-1}\d T_j\coloneqq T_j^{p^\alpha-1}\d T_j$. Then we obtain a decomposition
		\begin{equation*}
			\bigl(\qHodge_{R/A,\square}^*\bigr)_p^\complete\cong \bigoplus_v\bigl(\qHodge_{R/A,\square}^{*,v}\bigr)_p^\complete
		\end{equation*}
		as graded $\widehat{R}_p^{(\alpha)}\qpower$-modules.%For brevity, the summand corresponding to $\alpha=(0,\dotsc,0)$ will be denoted $\widehat{\Omega}_R^{*,0}$, $\qhatOmega_{R,\square}^{*,0}$, and $\qHatge_{R,\square}^{*,0}$, respectively.
	\end{numpar}
	
	\begin{lem}\label{lem:DecompRespectsqDifferential}
		The decomposition from \cref{par:FrobeniusDecompositions} is not just a decomposition underlying graded modules, but a decomposition of complexes. Furthermore, for the induced decomposition on $(\qHodge_{R/A,\square}^*)/(q^{p^\alpha}-1)$, each piece $(\qHodge_{R/A,\square}^{*,v})_p^\complete/(q^{p^\alpha}-1)$ is a complex of $\widehat{R}_p^{(\alpha)}\qpower$-modules.
	\end{lem}
	\begin{proof}
		To show that the differentials $(q-1)\q\nabla$ of $(\qHodge_{R/A,\square})_p^\complete$ respect the decomposition, it will be enough to show that $\gamma_i\colon \widehat{R}_p\qpower\rightarrow \widehat{R}_p\qpower$ respects the decomposition of $\widehat{R}_p\qpower$ as a free $\widehat{R}_p^{(\alpha)}\qpower$-module with basis $T_1^{v_1}\dotsm T_n^{v_n}$ for all multi-indices $v=(v_1,\dotsc,v_n)$ as above. To show this, extend $\phi_\square$ to a Frobenius lift $\phi_\square\colon \widehat{R}_p\qpower\rightarrow \widehat{R}_p\qpower$ by putting $\phi_\square(q)\coloneqq q^p$. It will certainly be enough to show that $\phi_\square$ and $\gamma_i$ commute. It's straightforward to check that they commute when restricted $\widehat{A}_p\langle T_1,\dotsc,T_n\rangle\qpower$. Furthermore, they commute modulo $(p,q-1)$, because $\gamma_i$ is the identity modulo $q-1$. Hence the desired commutativity follows from  uniqueness of infinitesimal lifting for $(p,q-1)$-completely étale morphisms.
		
		To show that each piece $(\qHodge_{R,\square}^{*,v})_p^\complete/(q^{p^\alpha}-1)$ is $\widehat{R}_p^{(\alpha)}\qpower$-linear, it will be enough to show that 
		\begin{equation*}
			(q-1)\q\nabla\colon \widehat{R}_p\qpower\longrightarrow (\Omega_{R/A}^1)_p^\complete\qpower
		\end{equation*}
		is divisible by $q^{p^\alpha}-1$ when restricted to $\widehat{R}_p^{(\alpha)}\qpower$. This follows easily from the fact that $\gamma_i$ and $\phi_\square$ commute, as we've observed above.
	\end{proof}
	\begin{numpar}[The Frobenius decomposition II.]\label{par:FrobeniusDecompositionsII}
		We 'll now further simplify the Frobenius decomposition from \cref{par:FrobeniusDecompositions}, adapting the arguments in the proof of \cite[Theorem~(\href{https://web.math.princeton.edu/~nmk/old/nilpconn.pdf\#page=27}{7.2})]{Katz}. Using \cref{lem:DecompRespectsqDifferential}, we can write
		\begin{equation*}
			\bigl(\qHodge_{R/A,\square}^{*,v}\bigr)_p^\complete/(q^{p^\alpha}-1)\cong \widehat{R}_p^{(\alpha)}\qpower\otimes_{\IZ_p\qpower}K_\alpha^{*,v}(n)\,,
		\end{equation*}
		where $K_\alpha^{*,v}(n)$ is the complex of free $\IZ_p\qpower/(q^{p^\alpha}-1)$-modules with basis given by the elements from \cref{par:FrobeniusDecompositions} and differentials given by $(q-1)\q\nabla$. The complex $K_\alpha^{*,v}(n)$ can be decomposed into a tensor product $K_\alpha^{*,v}(n)\cong K_\alpha^{*,v_1}(1)\otimes_{\IZ_p\qpower}\dotsb\otimes_{\IZ_p\qpower}K_\alpha^{*,v_d}(1)$, where $K_\alpha^{*,v_i}(1)$ is the complex
		\begin{equation*}
			K_\alpha^{*,v_i}(1)\coloneqq \left(T_i^{v_i}\cdot \IZ_p\qpower/(q^{p^\alpha}-1)\xrightarrow{(q-1)\q\nabla} T_i^{v_i-1}\d T_i\cdot \IZ_p\qpower/(q^{p^\alpha}-1)\right)
		\end{equation*}
		concentrated in degrees $0$ and $1$.  As in \cref{par:FrobeniusDecompositions}, we use the convention that $T_i^{v_i-1}\d T_i\coloneqq T_i^{p^\alpha-1}\d T_i$ if $v_i=0$.  If $v_i\geqslant 1$, then we can write $v_i=p^{e}v_i'$, where $e$ is the exponent of $p$ in the prime factorisation of $v_i$. The differential $(q-1)\q\nabla$ of $K^{*,v_i}(1)$ sends the generator $T_i^{v_i}$ in degree zero to 
		\begin{equation*}
			(q-1)\q\nabla(T_i^{v_i})=(q^{v_i}-1)T_i^{v_i-1}\d T_i=[v_i']_{q^{p^e}}(q^{p^e}-1)T_i^{v_i-1}\d T_i\,.
		\end{equation*}
		Now observe that $[v_i']_{q^{p^e}}$ is a unit in $\IZ_p\qpower/(q^{p^n}-1)$. Indeed, it can be written as a sum of $v_i'$, which is a unit, and a multiple of the topologically nilpotent element $q-1$. Hence $K^{*,v_i}(1)$ is isomorphic to the complex $K^*_{\alpha,e}$ given by
		\begin{equation*}
			K^*_{\alpha,e}\coloneqq \left(\IZ_p\qpower/(q^{p^\alpha}-1)\xrightarrow{(q^{p^e}-1)} \IZ_p\qpower/(q^{p^\alpha}-1)\right)\,,
		\end{equation*}
		again concentrated in degrees $0$ and $1$. If $v_i=0$, then similarly $K^{*,0}(1)\cong K_{\alpha,\alpha}^*$, where the differential of $K_{\alpha,\alpha}^*$ is multiplication with $q^{p^\alpha}-1$, hence zero.
		
		Summarising, we see that $K^{*,v}(n)$ can be written as a tensor product of complexes of the form $K_{n,e_i}^*$ for some $0\leqslant e_1,\dotsc,e_n \leqslant \alpha$. Fortunately, such a tensor product is easy to compute:
	\end{numpar}
	\begin{lem}\label{lem:Ke1xKe2}
		If $e_1\geqslant e_2\geqslant 0$, then there is an isomorphism of complexes of  $\IZ_p\qpower$-modules
		\begin{equation*}
			K_{\alpha,e_1}^*\otimes_{\IZ_p\qpower}K_{\alpha,e_2}^*\cong K_{\alpha,e_2}^*[-1]\oplus K_{\alpha,e_2}^*\,.
		\end{equation*}
	\end{lem}
	\begin{proof}
		An explicit isomorphism $K_{\alpha,e_1}^*\otimes_{\IZ_p\qpower}K_{\alpha,e_2}^*\overset{\cong }{\longrightarrow}K_{\alpha,e_2}^*[-1]\oplus K_{\alpha,e_2}^*$ is given by the following commutative diagram:
		\begin{equation*}
			\begin{tikzcd}[column sep= large, row sep=large]
				\IZ_p\qpower/(q^{p^\alpha}-1)\rar["\!\!\biggl(\begin{matrix}q^{p^{e_1}}-1\\q^{p^{e_2}}-1\end{matrix}\biggr)"]\eqar[d] &[-0.35em] \bigl(\IZ_p\qpower/(q^{p^\alpha}-1)\bigr)^{\oplus 2}\dar\rar["{\left(-(q^{p^{e_2}}-1),\,q^{p^{e_1}}-1\right)}"] &[3.9em] \IZ_p\qpower/(q^{p^\alpha}-1)\eqar[d]\\
				\IZ_p\qpower/(q^{p^\alpha}-1)\rar["\!\!\biggl(\begin{matrix}0\\q^{p^{e_2}}-1\end{matrix}\biggr)"] & \bigl(\IZ_p\qpower/(q^{p^\alpha}-1)\bigr)^{\oplus 2} \rar["{\left(-(q^{p^{e_2}}-1),\,0\right)}"] & \IZ_p\qpower/(q^{p^\alpha}-1)
			\end{tikzcd}
		\end{equation*}
		Here the vertical arrow in the middle sends $(a,b)\mapsto \left(a-\frac{q^{p^{e_1}}-1}{q^{p^{e_2}}-1}b,b\right)$.
	\end{proof}
	\begin{lem}\label{lem:qHatgeAdditive}
		Let $v=(v_1,\dotsc,v_n)$ be a multi-index as in \cref{par:FrobeniusDecompositions}, and let's write $v_i=p^{e_i}v_i'$, where $e_i$ is the exponent of $p$ in the prime factorisation of $v_i$ as in \cref{par:FrobeniusDecompositionsII} \embrace{with the convention that $e_i\coloneqq \alpha$ in the case $v_i=0$}. If $e\coloneqq \min\{e_1,\dotsc,e_d\}$, then there is an isomorphism of complexes of  $\widehat{A}_p\qpower$-modules
		\begin{equation*}
			\bigl(\qHodge_{R/A,\square}^{*,v}\bigr)_p^\complete/(q^{p^\alpha}-1)\cong \widehat{R}_p^{(\alpha)}\qpower\otimes_{\IZ_p\qpower}\left(\bigoplus_{k=0}^{n-1}\big(K_{\alpha,e}^*[-k]\big)^{\oplus \binom{n-1}k}\right)\,.
		\end{equation*}
	\end{lem}
	\begin{proof}
		Use \cref{par:FrobeniusDecompositionsII}, \cref{lem:Ke1xKe2}, and induction on $n$.
	\end{proof}
	\begin{cor}\label{cor:H*pTorsionFree}
		For all $\alpha\geqslant 0$, the cohomology groups $\H^*((\qHodge_{R/A,\square}^*)_p^\complete/(q^{p^\alpha}-1))$ are $p$-torsion free.
	\end{cor}
	\begin{proof}
		By \cref{lem:qHatgeAdditive}, each cohomology group of $(\qHodge_{R/A,\square}^*)_p^\complete/(q^{p^\alpha}-1)$ is a direct sum of terms of the form 
		\begin{equation*}
			\widehat{R}_p^{(\alpha)}\qpower\otimes_{\IZ_p\qpower}\H^0(K^*_{\alpha,e})\quad\text{or}\quad\widehat{R}_p^{(\alpha)}\qpower\otimes_{\IZ_p\qpower}\H^1(K^*_{\alpha,e})
		\end{equation*}
		for some $e\geqslant 1$. But $\H^0(K^*_{\alpha,e})=[p^{\alpha-e}]_{q^{p^e}}\IZ_p\qpower/(q^{p^\alpha}-1)\cong \IZ_p\qpower/(q^{p^e}-1)$ and also $\H^1(K^*_{\alpha,e})\cong \IZ_p\qpower/(q^{p^e}-1)$, so everything is indeed $p$-torsion free.
	\end{proof}

	\subsection{The \texorpdfstring{$q$}{q}-Hodge complex II: Multiplicative Structure}\label{subsec:qHodgeMultiplicative}
	Our goal in this subsection is to equip the cohomologies $(\H^*(\qHodge_{R/A,\square}^*/(q^m-1)))_{m\in\IN}$ with the structure of a $q$-$FV$-system of differential-graded $A$-algebras over $R$ as defined in \cref{def:qFVSystemOfCDGA}. This will allow us to formulate a relative version of \cref{thm:qDeRhamWittGlobalIntro}, which will eventually be proved in \crefrange{subsec:pTypical}{subsec:Global}. Throughout, we fix a framed smooth $A$-algebra $(R,\square)$ and use the shorthand $\H_{R/A,\square}^*(m)\coloneqq \H^*(\qHodge_{R/A,\square}^*/(q^m-1))$. Let's start constructing the various pieces of structure.
	
	\begin{numpar}[Differential-graded algebra structure.]\label{par:BocksteinDifferential}
		We know from \cref{par:MultiplicativeStructure} that $\qHodge_{R/A,\square}^*$ can be equipped with a non-commutative differential-graded $A[q]$-algebra structure. This induces a graded algebra structure on $\H_{R/A,\square}^*(m)$, which turns out to be commutative as we'll see in \cref{lem:HWithBocksteinIsCDGA} below.
		
		This leaves us with the question of how to define the differentials. We'll use the Bockstein differentials: Our assumptions on $A$ and $R$ guarantee that $q^{m}-1$ is a nonzerodivisor in $R\qpower$, hence we have a short exact sequence of complexes
		\begin{equation*}
			0\longrightarrow \qHodge_{R/A,\square}^*/(q^m-1)\xrightarrow{\!(q^m-1)\!}\qHodge_{R/A,\square}^*/(q^{m}-1)^2\longrightarrow \qHodge_{R/A,\square}^*/(q^m-1)\longrightarrow 0\,.
		\end{equation*}
		The associated connecting morphisms $\beta_m\colon \H_{R/A,\square}^*(m)\rightarrow \H_{R/A,\square}^{*+1}(m)$ are called \emph{Bockstein differentials}. As the name suggests, $\beta_m$ turns the graded $A\qpower$-module $\H_{R/A,\square}^*(m)$ into a cochain complex (see \cite[\stackstag{0F7N}]{Stacks} for example).
	\end{numpar}
	\begin{lem}\label{lem:HWithBocksteinIsCDGA}
		The graded algebra structure from \cref{par:MultiplicativeStructure} and the Bockstein differentials constructed in \cref{par:BocksteinDifferential} make $(\H_{R/A,\square}^*(m),\beta_m)$ a commutative differential-graded $A\qpower$-algebra.
	\end{lem}
	\begin{proof}
		To show commutativity, let $\omega\in \qHodge_{R/A,\square}^k$ be a $k$-form representing an element in $\H_{R/A,\square}^k(m)$. It will be enough to show $\d T_i\wedge \omega\equiv (-1)^k\omega\wedge\d T_i\mod q^m-1$ for all $i=1,\dotsc,n$. To see this, write
		\begin{equation*}
			\omega=\sum_{j\in J}f_j\d T_{j_1}\wedge\dotsb\wedge \d T_{j_k}
		\end{equation*}
		for some finite indexing set $J$. Our assumption on $\omega$ reads $0\equiv (q-1)\q\nabla(\omega)\mod q^m-1$. In particular, it implies that
		\begin{equation*}
			0\equiv \sum_{j\in J}\bigl(\gamma_i(f_j)-f\bigr)\d T_i\wedge \d T_{j_1}\wedge\dotsb\wedge \d T_{j_k}\mod (q^m-1)\,.
		\end{equation*} 
		But $\d T_i\wedge f_j \d T_i\wedge \d T_{j_1}\wedge\dotsb\wedge \d T_{j_k}=\gamma_i(f)\d T_i\wedge \d T_{j_1}\wedge\dotsb\wedge \d T_{j_k}$ holds by definition of the multiplication on $\qHodge_{R/A,\square}^*$. So the congruence above is exactly what we need.
		
		To show the graded Leibniz rule, we'll only verify that $\beta_m\colon \H_{R/A,\square}^0(m)\rightarrow \H_{R/A,\square}^1(m)$ is a derivation; the arguments in higher degrees are similar. Let $f,g\in R\qpower$ be elements whose images modulo $(q^m-1)$ are contained in $\H_{R/A,\square}^0(m)$. Then $(q-1)\q\nabla(f)\in \Omega_{R/A}^1\qpower$ is divisible by $(q^m-1)$, so that $\q\nabla(f)$ is divisible by $[m]_q$. A quick unravelling then shows that $\beta_m(f)$ is the image of 
		\begin{equation*}
			\frac{(q-1)\q\nabla(f)}{(q^m-1)}=\frac{\q\nabla(f)}{[m]_q}
		\end{equation*}
		in $\H_{R/A,\square}^1(m)$, and likewise for $\beta_m(g)$. Furthermore, if $\q\nabla f$ is divisible by $[m]_q$, then $\gamma_i(f)-f$ must be divisible by $q^m-1$ for all $i=1,\dotsc,n$. Thus, by the $q$-Leibniz rule,
		\begin{equation*}
			\frac{\q\partial_{i}(fg)}{[m]_q}= \gamma_i(f)\frac{\q\partial_ig}{[m]_q}+g\frac{\q\partial_if}{[m]_q}\equiv f\frac{\q\partial_ig}{[m]_q}+g\frac{\q\partial_if}{[m]_q}\mod (q^m-1)\,.
		\end{equation*}
		This shows $\beta_m(fg)=f\beta_m(g)+g\beta_m(f)$, which means that the Bockstein differential $\beta_m$ satisfies the Leibniz rule, as desired.
	\end{proof}
	Next, we need to construct $A[q]$-algebra maps $\qIW_m(R/A)\rightarrow \H^0_{R/A,\square}(m)$ for all $m$. If $R=A[T_1,\dotsc,T_n]$ and $\square$ is just the identity, then these maps are easily defined: We equip $R$ with the $\Lambda$-$A$-algebra structure in which $\psi^p(T_i)=T_i^p$ and consider the composition
	\begin{equation*}
		\qIW_m(R/A)\xrightarrow{c_{m/A}} R[q]/(q^m-1)\longrightarrow R\qpower/(q^m-1)\,,
	\end{equation*}
	where $c_{m/A}$ is the comparison map from \cref{par:RelativeComparisonMaps}. It's straightforward to see (but we don't need it at this point) that this map induces an isomorphism $\qIW_m(R/A)_{(q-1)}^\complete\rightarrow \H_{R/A,\square}^0(m)$.%
	\footnote{One can use the $\Lambda$-structure on $R$ to decompose $\qHodge_{R/A,\square}^*/(q^m-1)$ as in \cref{par:FrobeniusDecompositions}. Then $\H_{R/A,\square}^0(m)$ can be read off and it matches up with the $(q-1)$-completion of the description of $\qIW_m(R/A)$ from the proof of \cref{lem:RelativeqWittBaseChange}.}
	For general $R$, however, the $\Lambda$-structure on $A[T_1,\dotsc,T_n]$ doesn't extend along the étale morphism $\square\colon A[T_1,\dotsc,T_n]\rightarrow R$. Instead, we only get $\delta$-structures on the $p$-completions $\widehat{R}_p$ for all primes~$p$. So instead our strategy will be to use an \emph{arithmetic fracture pullback square}.
	\begin{numpar}[Arithmetic fracture pullback squares.]\label{par:DerivedBeauvilleLaszlo}
		For any $M\in\Dd(\IZ)$, and any integer $N\neq 0$, the canonical commutative square
		\begin{equation*}
			\begin{tikzcd}
				M\rar\dar\drar[pullback] & \prod_{p\mid N}\widehat{M}_p\dar\\
				M\bigl[\localise{N}\bigr]\rar & \prod_{p\mid N}\widehat{M}_p\bigl[\localise{p}\bigr]
			\end{tikzcd}
		\end{equation*}
		is a pullback square in the derived $\infty$-category $\Dd(\IZ)$. Indeed, this is obvious after applying any of the functors $(-)[1/N]$ or $(-)_{p}^\complete$ for $p\mid N$, and these are jointly conservative by \cref{lem:DerivedBeauvilleLaszlo}.
	\end{numpar}
	
	We'll now study how these squares appear in the situation at hand.
	
	\begin{lem}\label{lem:qWittBeauvilleLaszloSquare}
		Let $N\neq 0$ be divisible by $m$. For all primes~$p$ let $\alpha_p\coloneqq v_p(m)$ denote the exponent of $p$ in the prime factorisation of $m$. Then the derived $(q-1)$-completion of the arithmetic fracture pullback square from \cref{par:DerivedBeauvilleLaszlo} for $\qIW_m(R/A)$ takes the form
		\begin{equation*}
			\begin{tikzcd}[column sep=huge]
				\qIW_m(R/A)_{(q-1)}^\complete\rar["\left(F_{m/p^{\smash{\alpha_p}}}\right)_p"]\dar["\gh_m"']\drar[pullback] & \prod_{p\mid N}\qIW_{p^{\smash{v_p(m)}}}(\widehat{R}_p/A)_{(p,q-1)}^\complete\dar["\left(\gh_{p^{\smash{\alpha_p}}}\right)_p"]\\
				R\bigl[\localise{N}\bigr]\rar & \prod_{p\mid N}\widehat{R}_p\bigl[\localise{p}\bigr]
			\end{tikzcd}
		\end{equation*}
	\end{lem}
	\begin{proof}
		We start with the top right corner. The Frobenius and Verschiebung
		\begin{equation*}
			F_{m/p^{\alpha_p}}\colon \qIW_m(R/A)\longrightarrow \qIW_{p^{\alpha_p}}(R/A)\quad\text{and}\quad V_{m/p^{\alpha_p}}\colon \qIW_{p^{\alpha_p}}(R/A)\longrightarrow \qIW_m(R/A)
		\end{equation*}
		become isomorphisms after derived $(p,q-1)$-completion because $F_{m/p^{\alpha_p}}\circ V_{m/p^{\alpha_p}}=m/p^{\alpha_p}$ and $V_{m/p^{\alpha_p}}\circ F_{m/p^{\alpha_p}}=[m/p^{\alpha_p}]_{q^{p^{\alpha_p}}}$ are invertible in $\IZ[q]_{(p,q-1)}^\complete$. This explains why $\qIW_{p^{\alpha_p}}(-/A)$ shows up in the diagram instead of $\qIW_m(-/A)$. Furthermore, the canonical morphism $\qIW_{p^{\alpha_p}}(R/A)\rightarrow \qIW_{p^{\alpha_p}}(\widehat{R}_p/A)$ becomes an isomorphism after $p$-completion. Indeed, this can be checked after $p$-completed base change along the faithfully flat map $A\rightarrow A_\infty$. Via \cref{lem:RelativeqWittBaseChange}, we're then reduced to a question about absolute $q$-Witt vectors, which was addressed in \cref{cor:qWittpCompletion}.
		
		The bottom corners are similar. Since $m$ and $[m]_q$ are invertible in $\IZ[1/N,q]_{(q-1)}^\complete$, the same argument as above shows that the ghost map $\gh_m\colon \qIW_m(R/A)\rightarrow R$, or equivalently, the Frobenius $F_m$, becomes an isomorphism after $(-)[1/N]_{(q-1)}^\complete$. Also note that $R[1/N]$ doesn't need to be $(q-1)$-completed because it's already $(q-1)$-torsion. This takes care of the bottom left corner; the argument for the bottom right corner is analogous.
	\end{proof}
	\begin{rem}
		Observe that all factors of the pullback from \cref{lem:qWittBeauvilleLaszloSquare} are static (in the sense of \cref{par:Notation}), so we don't just get a pullback in the derived $\infty$-category, but also an honest pullback of $A[q]$-modules. To see this, it's enough to check that $\qIW_m(R/A)$ and $\qIW_{p^{\alpha_p}}(\widehat{R}_p/A)$ have bounded $(q-1)^\infty$ and $p^\infty$-torsion. This can be done after base change along $A\rightarrow A_\infty$, where it reduces via \cref{lem:RelativeqWittBaseChange} to questions about absolute $q$-Witt vectors that were addressed in \cref{cor:qWittpTorsion,cor:qWittCompletion}. The same conclusion is true (but for easier reasons) for the pullback square in \cref{lem:OtherBeauvilleLaszloSquare} below.
	\end{rem}
	\begin{lem}\label{lem:OtherBeauvilleLaszloSquare}
		With notation as in \cref{lem:qWittBeauvilleLaszloSquare}, the derived $(q-1)$-completion of the arithmetic fracture pullback square from \cref{par:DerivedBeauvilleLaszlo} for $R\qpower/(q^m-1)$ takes the form
		\begin{equation*}
			\begin{tikzcd}
				R\qpower/(q^m-1)\rar\dar\drar[pullback] & \prod_{p\mid N}\widehat{R}_p\qpower/(q^{p^{\alpha_p}}-1)\dar\\
				R[1/N]\rar & \prod_{p\mid N}\widehat{R}_p\bigl[\localise{p}\bigr]
			\end{tikzcd}
		\end{equation*}
	\end{lem}
	\begin{proof}
		One can argue as in \cref{lem:qWittBeauvilleLaszloSquare}, but with the Frobenius and Verschiebung replaced by the canonical projections $R\qpower/(q^m-1)\rightarrow R\qpower/(q^{p^{\alpha_p}}-1)$ and the multiplication maps $[m/p^{\alpha_p}]_{q^{p^{\alpha_p}}}\colon R\qpower/(q^{p^{\alpha_p}}-1)\rightarrow R\qpower/(q^m-1)$.
	\end{proof}
	
	\begin{con}\label{con:qWittToH0}
		Fix a prime $p$. We've already seen in \cref{par:FrobeniusLift} that the étale framing $\square\colon A[T_1,\dotsc,T_n]\rightarrow R$ determines a Frobenius lift on the $p$-completion $\widehat{R}_p$, which turns $\widehat{R}_p$ into a $\delta$-$A$-algebra by $p$-torsion freeness. By \cref{rem:LambdaRingIsntNecessary}, $\qIW_{p^{\alpha_p}}(\widehat{R}_p/A)$ only depends on the $\Lambda_p$-structure on $A$, that is, on the $\delta$-structure, and the relative comparison map
		\begin{equation*}
			c_{p^{\alpha_p}/A}\colon \qIW_{p^{\alpha_p}}(\widehat{R}_p/A)\longrightarrow \widehat{R}_p[q]/(q^{p^{\alpha_p}}-1)
		\end{equation*}
		from \cref{par:RelativeComparisonMaps} can be defined using only a $\delta$-$A$-algebra structure on $\widehat{R}_p$. These comparison maps for all $p\mid N$ induce a morphism between the pullback squares from \cref{lem:qWittBeauvilleLaszloSquare,lem:OtherBeauvilleLaszloSquare} and hence a morphism
		\begin{equation*}
			c_{m,\square}\colon \qIW_m(R/A)_{(q-1)}^\complete\longrightarrow R\qpower/(q^m-1)\,,
		\end{equation*}
		even though there's usually no $\Lambda$-$A$-algebra structure on $R$. Furthermore, it's easy to see that $c_{m,\square}$ lands in $\H^0_{R/A,\square}(m)$. Indeed, we can write the free $R\qpower/(q^m-1)$-module $\qHodge_{R/A,\square}^1$ as a similar pullback as in \cref{lem:OtherBeauvilleLaszloSquare} and then it suffices to check that we get $0$ in each factor. For the bottom factors this is trivial, since the differentials of $\qHodge_{R/A,\square}^*$ vanish modulo $q-1$. So it remains to show that the composition
		\begin{equation*}
			\qIW_{p^{\alpha_p}}(\widehat{R}_p/A)\longrightarrow \widehat{R}_p\qpower/(q^{p^{\alpha_p}}-1)\xrightarrow{(q-1)\q\nabla}(\Omega_{R/A}^1)_p^\complete\qpower/(q^{p^{\alpha_p}}-1)
		\end{equation*}
		vanishes, or in other words, that map $c_{p^{\alpha_p}/A}$ lands in $\H^0((\qHodge_{R/A,\square}^*)_p^\complete/(q^{p^{\alpha_p}}-1))$ for all $p\mid N$. But thanks to \cref{par:FrobeniusDecompositions} and \cref{par:FrobeniusDecompositionsII}, we can determine that cohomology group. If $v=(v_1,\dotsc,v_n)$ is a multi-index and $e$ is defined as in \cref{lem:qHatgeAdditive}, then the $v$\textsuperscript{th} component of our desired $\H^0$ is given by
		\begin{equation*}
			\H^0\bigl((\qHodge_{R/A,\square}^{*,v})_p^\complete/(q^{p^{\alpha_p}}-1)\bigr)\cong [p^{\alpha_p-e}]_{q^{p^e}}T_1^{v_1}\dotsm T_n^{v_n}\widehat{R}_p^{(\alpha_p)}\qpower/(q^{p^{\alpha_p}}-1)\,.
		\end{equation*}
		By unravelling the definitions, it's clear that $c_{p^{\alpha_p}/A}$ really lands the direct sum of over all $v$ of these groups. This finishes the construction of a $\qIW_m(R/A)$-algebra structure on $\H^0_{R/A,\square}(m)$.
	\end{con}
	\begin{numpar}[Frobenius and Verschiebung.]\label{par:FrobeniusVerschiebung}
		For $d\mid m$, we define the Frobenius and the Verschiebung
		\begin{equation*}
			F_{m/d}\colon \H^*_{R/A,\square}(m)\longrightarrow \H^*_{R/A,\square}(d)\quad\text{and}\quad V_{m/d}\colon  \H^*_{R/A,\square}(d)\longrightarrow \H^*_{R/A,\square}(m)
		\end{equation*}
		to be the maps induced by the canonical projection $\qHodge_{R/A,\square}^*/(q^m-1)\rightarrow \qHodge_{R/A,\square}^*/(q^d-1)$ and the multiplication map $[m/d]_{q^d}\colon \qHodge_{R/A,\square}^*/(q^d-1)\rightarrow\qHodge_{R/A,\square}^*/(q^m-1)$, respectively. It's clear that $F_{m/d}$ is a map of graded $A[q]$-algebras and $V_{m/d}$ is a map of graded $A[q]$-modules. In fact, $V_{m/d}$ is a map of graded $\H_{R/A,\square}^*(m)$-modules if we equip $\H_{R/A,\square}^*(d)$ with the module structure induced from $F_{m/d}$.
		
		These maps are compatible with the Frobenii and Verschiebungen on relative $q$-Witt vectors under the maps constructed in \cref{con:qWittToH0}, which is straightforward to check from the definition and \cref{cor:qWittToCyclicRing}. Furthermore, we clearly have $F_{m/e}=F_{d/e}\circ F_{m/d}$ and $V_{m/e}=V_{m/d}\circ V_{d/e}$ for all chains of divisors $e\mid d\mid m$. The condition $V_{m/d}(\omega F_{m/d}(\eta))=V_{m/d}(\omega)\eta$ follows from our observation that $V_{m/d}$ is a map of graded $\H_{R/A,\square}^*(m)$-modules.
	\end{numpar}
	According to \cref{rem:Redundancies}, we've thus equipped $(\H_{R/A,\square}^*(m))_{m\in\IN}$ with all the structure from \cref{def:qVSystemOfCDGA}\cref{enum:qDeRhamWittConditionA} and~\cref{enum:qDeRhamWittConditionB} as well as \cref{def:qFVSystemOfCDGA}\cref{enum:qDeRhamWittConditionC}, and it only remains to check the $F$-Teichmüller condition~\cref{enum:TeichmuellerF}. It's easy to see that the $F$-Teichmüller condition always holds up to $(m/d)^{m/d-1}$-torsion, so it will be enough to show the following lemma.
	\begin{lem}\label{lem:H*pTorsionFree}
		For all $m\in\IN$ and all primes $p$, the cohomology $\H_{R/A,\square}^*(m)$ is $p$-torsion free in every degree.
	\end{lem}
	For the proof, we need a technical lemma.
	\begin{lem}\label{lem:TechnicalCompletionLemma}
		Fix a prime $p$ and consider the following three conditions on a cochain complex $M^*$ of $\IZ[q]$-modules:
		\begin{alphanumerate}\itshape
			\item In every degree, $M^*$ and $\H^*(M^*)$ have bounded $(q-1)^\infty$-torsion.\label{enum:ConditionABoundedq-1InftyTorsion}
			\item In every degree, $M^*$ and $\H^*(M^*)$ are $p$-torsion free.\label{enum:ConditionBpTorsionFree}
			\item In every degree, $M^*/p$ and $\H^*(M^*)/p$ have bounded $(q-1)^\infty$-torsion.\label{enum:ConditionCpq-1Torsion}
		\end{alphanumerate}
		If $M$ satisfies these conditions, then the derived $(q-1)$-completion of $M$ can be computed as the degree-wise underived $(q-1)$-completion, and it also satisfies conditions~\cref{enum:ConditionABoundedq-1InftyTorsion},~\cref{enum:ConditionBpTorsionFree}, and~\cref{enum:ConditionCpq-1Torsion}. Moreover, taking cohomology of $M^*$ commutes with derived $(q-1)$-completion and with derived $(p,q-1)$-completion.
	\end{lem}
	\begin{proof}
		Condition~\cref{enum:ConditionABoundedq-1InftyTorsion} implies that the derived $(q-1)$-completion of $M^*$ can be computed as the degree-wise underived $(q-1)$-completion. Furthermore, via the spectral sequence from \cite[\stackstag{0BKE}]{Stacks}, it implies that $(q-1)$-completion (derived or underived doesn't matter) commutes with taking cohomology. In formulas,
		\begin{equation*}
			\H^*\bigl(\widehat{M}_{(q-1)}^*\bigr)\cong \H^*(M)_{(q-1)}^\complete\,.
		\end{equation*}
		It is now clear that \cref{enum:ConditionABoundedq-1InftyTorsion} is still satisfied for $\widehat{M}_{(q-1)}^*$. For conditions~\cref{enum:ConditionBpTorsionFree} and~\cref{enum:ConditionCpq-1Torsion}, let $N$ be any $p$-torsion free $\IZ[q]$-module. Then
		\begin{equation*}
			N\overset{p}{\longrightarrow}N\longrightarrow N/p
		\end{equation*}
		is a cofibre sequence of static objects in $\Dd(\IZ[q])$. If both $N$ and $N/p$ have bounded $(q-1)^\infty$-torsion, then applying $(-)_{(q-1)}^\complete$ to this cofibre sequence still yields a cofibre sequence of static objects. This shows that $N_{(q-1)}^\complete$ is still $p$-torsion free and that $N_{(q-1)}^\complete/p$ is the derived or underived $(q-1)$-completion of $N/p$. Applied to $M^*$, we conclude that the degree-wise $(q-1)$-completion $\widehat{M}_{(q-1)}^*$ still satisfies all three conditions. In particular, its cohomology is $p$-torsion free and so taking cohomology commutes with derived $p$-completion. In formulas,
		\begin{equation*}
			\H^*\bigl(\widehat{M}_{(p,q-1)}^*\bigr)\cong \H^*\bigl(\widehat{M}_{(q-1)}^*\bigr)_p^\complete\cong \H^*(M)_{(p,q-1)}^\complete\,,
		\end{equation*}
		which is the last assertion we had to show.
	\end{proof}
	\begin{proof}[Proof of \cref{lem:H*pTorsionFree}]
		We'll show that the conditions of \cref{lem:TechnicalCompletionLemma} are satisfied for the complex $M^*=\qHodge_{R/A,\square}^*/(q^m-1)$. Let's first consider the case $A=\IZ$. Thanks to \cref{con:qWittToH0}, we know that $\qHodge_{R/\IZ,\square}^*/(q^m-1)$ is a complex of $\qIW_m(R)_{(q-1)}^\complete$-modules. We also know that $\qIW_m(R)_{(q-1)}^\complete$ is noetherian using \cref{cor:qWittNoetherian}. Furthermore, the proof of said corollary shows that $R\qpower/(q^m-1)$ is finite over $\qIW_m(R)_{(q-1)}^\complete$. So $\qHodge_{R/\IZ,\square}^*/(q^m-1)$ is a complex of finitely generated modules over a noetherian ring. Hence conditions~\cref{enum:ConditionABoundedq-1InftyTorsion} and~\cref{enum:ConditionCpq-1Torsion} from \cref{lem:TechnicalCompletionLemma} become obvious. It remains to check condition~\cref{enum:ConditionBpTorsionFree}. It's clear that $\qHodge_{R/\IZ,\square}^*/(q^m-1)$ is degree-wise $p$-torsion free. For complexes of finitely generated modules over a noetherian ring, $p$-completion (derived or underived doesn't matter) commutes with cohomology. Furthermore, an abelian group is $p$-torsion free if and only if its derived $p$-completion is static and $p$-torsion free. Now the derived $p$-completion of $\qHodge_{R/\IZ,\square}^*/(q^m-1)$ is computed by $(\qHodge_{R/\IZ,\square}^*)_p^\complete/(q^{p^\alpha}-1)$, where $\alpha$ is the exponent of $p$ in the prime factorisation of $m$. Since we know this complex has $p$-torsion free cohomology by \cref{cor:H*pTorsionFree}, it follows that $\H^*_{R/\IZ,\square}(m)$ must be $p$-torsion free as well, as desired.
		
		Now let $A$ be an arbitrary $\Lambda$-ring (except that we still impose the condition that a faithfully flat morphism of $\Lambda$-rings $A\rightarrow A_\infty$ into a perfect $\Lambda$-ring exists). Let $R_0\coloneqq \IZ[T_1,\dotsc,T_n]$. We claim that
		\begin{equation*}
			\qHodge_{R/A,\square}^*/(q^m-1)\cong \left(\qHodge_{R_0/\IZ,\square}^*/(q^m-1)\otimes_{\qIW_m(R_0)} \qIW_m(R/A)\right)_{(q-1)}^\complete\,,
		\end{equation*}
		where the completion is the degree-wise underived $(q-1)$-completion, but it computes the derived $(q-1)$-completion (as we'll see in a moment). To see this, first observe that the base change along $\qIW_m(R_0)\rightarrow \qIW_m(R/A)$ is flat. Indeed, our assumptions imply that $A$ is flat over $\IZ$, hence $\qIW_m(R_0)\rightarrow \qIW_m(R_0\otimes_\IZ A/A)$ is flat by \cref{lem:RelativeqWittBaseChange}. Furthermore, $\square\colon A[T_1,\dotsc,T_n]\cong R_0\otimes_\IZ A\rightarrow R$ is étale, hence so is $\qIW_m(R_0\otimes_\IZ A/A)\rightarrow \qIW_m(R/A)$ by \cref{prop:vanDerKallen}. The fact that $\qHodge_{R_0/\IZ,\square}^*/(q^m-1)\otimes_{\qIW_m(R_0)} \qIW_m(R/A)$ is a flat base change of $\qHodge_{R_0/\IZ,\square}^*/(q^m-1)$ immediately implies that it satisfies the conditions \cref{enum:ConditionABoundedq-1InftyTorsion},~\cref{enum:ConditionBpTorsionFree}, and~\cref{enum:ConditionCpq-1Torsion} from \cref{lem:TechnicalCompletionLemma}, and so the degree-wise underived $(q-1)$-completion indeed agrees with the derived $(q-1)$-completion. This also shows that once we've proved the isomorphism above, we'll be immediately done by \cref{lem:TechnicalCompletionLemma}.
		
		To prove the claimed isomorphism, it's enough to show that that the natural map
		\begin{equation*}
			 \left(R_0\qpower/(q^m-1)\lotimes_{\qIW_m(R_0)}\qIW_m(R/A)\right)_{(q-1)}^\complete\overset{\simeq}{\longrightarrow}R\qpower/(q^m-1)
		\end{equation*}
		is an equivalence, as $\qHodge_{R_0/\IZ,\square}^*/(q^m-1)$ and $\qHodge_{R/A,\square}^*/(q^m-1)$ are degree-wise free modules over $R_0\qpower/(q^m-1)$ and $R\qpower/(q^m-1)$, respectively, with compatible bases. By the derived Nakayama lemma \cite[\stackstag{0G1U}]{Stacks} such an equivalence can be checked after applying $R\lotimes_{R_0\qpower/(q^m-1)}-$. The right-hand side then becomes $R$. The left-hand side becomes
		\begin{equation*}
			R_0\lotimes_{\gh_m,\,\qIW_m(R_0)}\qIW_m(R/A)\simeq (R_0\otimes_\IZ A)\lotimes_{\gh_m,\,\qIW_m(R_0\otimes_\IZ A/A)}\qIW_m(R/A)\,,
		\end{equation*}
		where we've used that $c_m\colon \qIW_m(R_0)\rightarrow R\qpower/(q^m-1)$ intertwines the canonical projection to $R_0$ with the $m$\textsuperscript{th} ghost map. The claim now follows from \cref{cor:qWittGhostMapsPushout}. 
	\end{proof}
	\begin{con}\label{con:qDRWtoqHodge}
		This finishes the construction of a $q$-$FV$-system of differential-graded $A$-algebras over $A$ on $(\H^*_{R/A,\square}(m))_{m\in\IN}$. By \cref{prop:qDRWHasFrobenii}, this induces a unique morphism of $q$-$FV$-systems $(\qIW_m\Omega_{R/A}^*)_{m\in\IN}\rightarrow (\H^*_{R/A,\square}(m))_{m\in\IN}$. Since $\H^*_{R/A,\square}(m)$ is derived $(q-1)$-complete in every degree, this morphism of $q$-$FV$ systems factors through the degree-wise derived $(q-1)$-completions. By \cref{lem:qDRWDegreewiseBoundedq-1Torsion}, we can equally well take the degree-wise underived $(q-1)$-completions, and so we obtain morphisms
		\begin{equation*}
			\bigl(\qIW_m\Omega_{R/A}^*\bigr)_{(q-1)}^\complete\longrightarrow \H_{R/A,\square}^*(m)
		\end{equation*}
		which finally allow us to state the relative version of \cref{thm:qDeRhamWittGlobalIntro}.
	\end{con}
	\begin{thm}\label{thm:qDeRhamWittqHodge}
		Let $R$ be a smooth algebra over the perfectly covered $\Lambda$-ring $A$. Then the canonical morphism from \cref{con:qDRWtoqHodge} is an isomorphism
		\begin{equation*}
			\bigl(\qIW_m\Omega_{R/A}^*\bigr)_{(q-1)}^\complete\overset{\cong}{\longrightarrow}\H^*\bigl(\qHodge_{R/A,\square}^*/(q^m-1)\bigr)
		\end{equation*}
		for all positive integers $m\in\IN$.
	\end{thm}

	\subsection{Proof of the main results I: The \texorpdfstring{$p$}{p}-typical case}\label{subsec:pTypical}
	After our lengthy digression, we return to the induction outlined in our battle plan~\cref{par:OutlineOfStrategy}. The goal of this subsection is to prove the remaining three implications, starting with \cref{enum:alphaA} $\Rightarrow$ \cref{enum:alphaB}. We keep the shorthand $\H_{R/A,\square}^*(p^\alpha)\coloneqq \H^*(\qHodge_{R/A,\square}^*/(q^{p^\alpha}-1))$ as in \cref{subsec:qHodgeMultiplicative} and note that the degree-wise $p$-completion $\H_{R/A,\square}^*(p^\alpha)_p^\complete$ computes the cohomology of $(\qHodge_{R/A,\square}^*)_p^\complete/(q^{p^\alpha}-1)$ by \cref{lem:H*pTorsionFree}.
	
	\begin{numpar}[The Frobenius on the $q$-Hodge complex.]\label{par:FrobeniusqHodge}
		As in the proof of \cref{lem:DecompRespectsqDifferential} we extend  $\phi_\square$ to a Frobenius lift $\phi_\square\colon \widehat{R}_p\qpower\rightarrow \widehat{R}_p\qpower$ by putting $\phi_{\square}(q)\coloneqq q^p$. We can further extend this to an endomorphism of complexes
		\begin{equation*}
			\phi_\square\colon \bigl(\qHodge_{R/A,\square}^*\bigr)_p^\complete\longrightarrow \bigl(\qHodge_{R/A,\square}^*\bigr)_p^\complete
		\end{equation*}
		by putting $\phi_{\square}(\d T_i)\coloneqq T_i^{p-1}\d T_i$. Indeed, we've checked in the proof of \cref{lem:DecompRespectsqDifferential} that $\phi_\square$ and $\gamma_i$ commute and so to check that $\phi_\square$ respects the differential, it's enough to compute
		\begin{align*}
			\phi_\square\bigl((q-1)\q\partial_i f\bigr)=\phi_\square\left(\frac{\gamma_i(f)-f}{T_i}\d T_i\right)&=\frac{\phi_\square\bigl(\gamma_i(f)\bigr)-\phi_{\square}(f)}{T_i^p}T_i^{p-1}\d T_i\\
			&=\frac{\gamma_i\bigl(\phi_\square(f)\bigr)-\phi_{\square}(f)}{T_i}\d T_i\\
			&=(q-1)\q\partial_i\bigl(\phi_{\square}(f)\bigr)\,.
		\end{align*}
		The Frobenius lift on $A$ can be extended to a Frobenius lift $\phi\colon A[q]\rightarrow A[q]$ by putting $\phi(q)\coloneqq q^p$. Then the endomorphism $\phi_\square$ induces a relative Frobenius
		\begin{equation*}
			\phi_{\square/A}^\alpha\colon \left(\qHodge_{R/A,\square}^*\otimes_{A[q],\phi^\alpha}A[q]\right)_{(p,q-1)}^\complete\longrightarrow \bigl(\qHodge_{R/A,\square}^*\bigr)_p^\complete\,.
		\end{equation*}
		Following our conventions from \cref{par:Notation}, here we take the degree-wise underived completions, but they coincide with the derived completion, because $\qHodge_{R/A,\square}^*$ is degree-wise $p$- and $(q-1)$-torsion free and $\phi\colon A[q]\rightarrow A[q]$ is flat (using footnote~\cref{footnote:FaithfullyFlatMapOfLambdaRings} in \cref{rem:FaithfullyFlatCoverByPerfectLambdaRing}). Also recall from \cref{par:FrobeniusDecompositions} that we have a decomposition
		\begin{equation*}
			\bigl(\qHodge_{R/A,\square}^*\bigr)_p^\complete\cong \bigoplus_v\bigl(\qHodge_{R/A,\square}^{*,v}\bigr)_p^\complete\,.
		\end{equation*}
		Let us denote by $(\qHodge_{R/A,\square}^{*,0})_p^\complete$ the direct summand for $v=(0,\dotsc,0)$ and by $\H_{R/A,\square}^{*,0}(p^\alpha)_p^\complete$ the corresponding direct summand of $\H_{R/A,\square}^*(p^\alpha)_p^\complete$.
	\end{numpar}
	\begin{lem}\label{lem:H*0DeRham}
		The relative Frobenius $\phi_{\square/A}^\alpha$ is an isomorphism onto $(\qHodge_{R/A,\square}^{*,0})_p^\complete$. In particular, we obtain an isomorphism of differential-graded $A[q]$-algebras
		\begin{equation*}
			\H_{R/A,\square}^{*,0}(p^\alpha)_p^\complete\cong \bigl(\Omega_{R/A}^*\otimes_{A,\phi^\alpha}A[q]/(q^{p^{\alpha}}-1)\bigr)_p^\complete
		\end{equation*}
	\end{lem}
	\begin{proof}
		To see that we get an isomorphism onto $(\qHodge_{R/A,\square}^{*,0})_p^\complete$, just observe that
		\begin{equation*}
			\phi_{\square/A}^\alpha\colon\left(R\qpower\otimes_{A[q],\phi^\alpha}A[q]\right)_{(p,q-1)}\overset{\simeq}{\longrightarrow}\widehat{R}_p^{(\alpha)}\qpower
		\end{equation*}
		is an equivalence essentially by definition of $\widehat{R}_p^{(\alpha)}$ and then check that $\phi_{\square/A}^\alpha$ sends the bases from \cref{par:FrobeniusDecompositions} onto each other.
		
		To show the second assertion, we deduce from the first one that $(\qHodge_{R/A,\square}^{*,0})_p^\complete/(q^{p^\alpha}-1)$ is a flat base change of $\qHodge_{R/A,\square}^*/(q-1)$ up to $(p,q-1)$-completion (or just $p$-completion because we're already in a $(q^{p^\alpha}-1)$-torsion situation). It straightforward to check that the cohomology $\H^*(\qHodge_{R/A,\square}^*/(q-1))$, equipped with the Bockstein differential, is the de Rham complex $\Omega_{R/A}^*$. It remains to check that $p$-completion commutes with taking cohomology, but this is clear since we've just checked that
		\begin{equation*}
			\H^*\left(\qHodge_{R/A,\square}/(q-1)\otimes_{A[q],\phi^\alpha}A[q]\right)\cong \Omega_{R/A}^*\otimes_{A,\phi^\alpha}A[q]/(q^{p^\alpha}-1)
		\end{equation*}
		is degree-wise $p$-torsion free.
	\end{proof}
	\begin{lem}\label{lem:VerschiebungSequence}
		Let $\ov{\IV}_{p^\alpha}^*\subseteq \H_{R/A,\square}^*(p^\alpha)_p^\complete$ be the $p$-complete graded sub-$A[q]$-module generated by the images of $V_p$ and $\d V_p$. Then there's a short exact sequence
		\begin{equation*}
			0\longrightarrow \ov\IV_{p^\alpha}^*\longrightarrow \H_{R/A,\square}^*(p^\alpha)_p^\complete\longrightarrow \H_{R/A,\square}^{*,0}(p^\alpha)_p^\complete/\Phi_{p^\alpha}(q)\longrightarrow 0\,,
		\end{equation*}
		in which the second arrow is the canonical inclusion and the third arrow is the projection to the direct summand $\H_{R/A,\square}^{*,0}(p^\alpha)_p^\complete$.
	\end{lem}
	\begin{proof}
		It's clear from the construction in \cref{par:FrobeniusVerschiebung} that the kernel $K^*$ of the canonical projection $\H_{R/A,\square}^*(p^\alpha)_p^\complete\rightarrow \H_{R/A,\square}^{*,0}(p^\alpha)_p^\complete/\Phi_{p^\alpha}(q)$ contains the image of $V_p$, hence $K^*$ also contains the $p$-complete differential-graded ideal generated by the image of $V_p$. This shows $\ov\IV_{p^\alpha}^*\subseteq K^*$.
		
		It remains to show that $\ov\IV_{p^\alpha}^*\rightarrow K^*$ is surjective. We would like to use the decomposition from \cref{lem:qHatgeAdditive} to show this. Alas, this decomposition is not compatible with the Bockstein differential. But we can use a trick: It's enough to show that $\ov\IV_{p^\alpha}^*\rightarrow K^*$ is surjective modulo~$p$, as both sides are $p$-complete. But modulo~$p$, we have a surjection 
		\begin{equation*}
			\qHodge_{R/A,\square}^*/\bigl(p,q^{p^{\alpha+1}}-1\bigr)\longtwoheadrightarrow\qHodge_{R/A,\square}^*/\bigl(p,(q^{p^\alpha}-1)^2\bigr)\,,
		\end{equation*}
		hence the decomposition of $\qHodge_{R/A,\square}^*/(p,q^{p^{\alpha+1}}-1)$ is, in fact, compatible with the Bockstein differential modulo $(p,q^{p^\alpha}-1)$. So we'll use this decomposition instead. We'll also use that modding out~$p$ commutes with taking cohomology for the complex $\qHodge_{R/A,\square}^*/(q^{p^\alpha}-1)$ thanks to \cref{lem:H*pTorsionFree}.
		
		Consider a two-term complex $K_{\alpha+1,e}^*$ as in \cref{par:FrobeniusDecompositionsII}, where $0\leqslant e\leqslant \alpha+1$. Put $\ov e\coloneqq \min\{\alpha,e\}$ and $\ov K_{\alpha+1,e}^*\coloneqq K_{\alpha+1,e}^*/(p,q^{p^\alpha}-1)$. Then $\ov K_{\alpha+1,e}^*$ is the two-term complex
		\begin{equation*}
			\ov K_{\alpha+1,e}^*\cong\left(\IF_p[q]/(q-1)^{p^\alpha}\xrightarrow{(q-1)^{p^e}}\IF_p[q]/(q-1)^{p^\alpha}\right)
		\end{equation*}
		with cohomology $\H^0(\ov K_{\alpha+1,e}^*)\cong (q-1)^{p^{\alpha}-p^{\ov{e}}}\IF_p[q]/(q-1)^{p^\alpha}$, $\H^1(\ov K_{\alpha+1,e}^*)\cong \IF_p[q]/(q-1)^{p^{\ov{e}}}$. By a simple unravelling, the Bockstein differential $\ov\beta_{p^\alpha}\colon \H^0(\ov K_{\alpha+1,e}^*)\rightarrow \H^1(\ov K_{\alpha+1,e}^*)$ coming from the surjection $K_{\alpha+1,e}^*/(p,(q-1)^{2p^\alpha})\twoheadrightarrow \ov K_{\alpha+1,e}^*$ is given by
		\begin{equation*}
			\ov\beta_{p^\alpha}\bigl((q-1)^{p^{\alpha}-p^{\ov e}}\omega\bigr)=(q-1)^{p^e-p^{\ov e}}\omega
		\end{equation*}
		for all $\omega\in \IF_p[q]/(q^{p^\alpha}-1)$. For $e\leqslant\alpha-1$, we conclude that every class in $\H^0(\ov K_{\alpha+1,e}^*)$ is contained in the image of $V_p$, since they're all divisible by $(q-1)^{p^\alpha-p^{\alpha-1}}\equiv \Phi_{p^\alpha}(q)\mod p$, and similarly that every class in $\H^1(\ov K_{\alpha+1,e}^*)$ is contained in the image of $\d V_p$. This shows that $\ov\IV_{p^\alpha}\rightarrow \H_{R/A,\square}^*(p^\alpha)_p^\complete$ is a surjection in every direct summands except possibly $\H_{R/A,\square}^{*,0}(p^\alpha)_p^\complete$. To analyse the situation for the latter, just observe that the kernel of 
		\begin{equation*}
			\H_{R/A,\square}^{*,0}(p^\alpha)_p^\complete\longrightarrow\H_{R/A,\square}^{*,0}(p^\alpha)_p^\complete/\Phi_{p^\alpha}(q)
		\end{equation*}
		is given by the classes divisible by $\Phi_{p^\alpha}(q)$. But every such class is in the image of $V_p$ due to the relation $V_p\circ F_p=\Phi_{p^\alpha}(q)$.
	\end{proof}
	\begin{proof}[Proof of \cref{enum:alphaA} $\Rightarrow$ \cref{enum:alphaB}]
		We wish to show that after $p$-completion, the canonical map from \cref{thm:qDeRhamWittqHodge} for $m=p^\alpha$ becomes an isomorphism
		\begin{equation*}
			\bigl(\qIW_{p^\alpha}\Omega_{R/A}^*\bigr)_p^\complete\overset{\cong}{\longrightarrow}\H_{R/A,\square}^*(p^\alpha)_p^\complete
		\end{equation*}
		(the left-hand side doesn't need to be $(q-1)$-completed since it is already $p$-complete and $(q^{p^\alpha}-1)$-torsion). To do so, let $\IV_{p^\alpha}^*\subseteq (\qIW_{p^\alpha}\Omega_{R/A}^*)_p^\complete$ be the $p$-complete graded sub-$A[q]$-module generated by the image of $V_p$ and $\d V_p$. Now consider the diagram
		\begin{equation*}
			\begin{tikzcd}
				0\rar & \IV_{p^\alpha}^*\dar\rar & \bigl(\qIW_{p^\alpha}\Omega_{R/A}^*\bigr)_p^\complete\dar\rar & \bigl(\Omega_{R/A}^*\otimes_{A,\phi^\alpha}A[\zeta_{p^\alpha}]\bigr)_p^\complete\rar\eqar[d] & 0\\
				0\rar & \ov\IV_{p^\alpha}^*\rar & \H_{R/A,\square}^*(p^\alpha)_p^\complete\rar & \bigl(\Omega_{R/A}^*\otimes_{A,\phi^\alpha}A[\zeta_{p^\alpha}]\bigr)_p^\complete\rar & 0
			\end{tikzcd}
		\end{equation*}
		The bottom row is exact by \cref{lem:H*0DeRham,lem:VerschiebungSequence}. The top row is exact by passing to degree-wise $p$-completions in \cref{lem:qDRWGhostMap}; this preserves exactness since \cref{enum:alphaA} implies that all $p$-completions agree with the corresponding derived $p$-completions. By the five lemma, it will be enough to show that the left and the right vertical maps are isomorphisms.
		
		Let's begin with the right one and verify that it really is the identity, as indicated. By the universal property of de Rham complexes it's enough to check this in degree~$0$. Now the diagram
		\begin{equation*}
			\begin{tikzcd}
				\qIW_{p^\alpha}(\widehat{R}_p/A)\rar["c_{p^\alpha/A}"]\dar["\gh_1"'] & \widehat{R}_p[q]/(q^{p^\alpha}-1)\dar\\
				\left(R\otimes_{A,\phi^\alpha}A\right)_p^\complete[\zeta_{p^\alpha}] \rar["\phi_{\square/A}^\alpha"] & \widehat{R}_p[\zeta_{p^{\alpha}}]
			\end{tikzcd}
		\end{equation*}
		commutes and the bottom row is injective, as we've checked in~\cref{par:FrobeniusLift} and~\cref{par:FrobeniusDecompositions}. This shows that we really get the identity in degree~$0$.
		
		To complete the proof, let's shows that the left vertical arrow is an isomorphism. Using the inductive hypothesis \embrace{\hyperref[enum:alphaB]{$b_{\alpha-1}$}} as well as injectivity of the Verschiebungen $V_p$ (which follows from $F_p\circ V_p=p$ combined with our $p$-torsion freeness results \cref{enum:alphaA} and \cref{lem:H*pTorsionFree}), we see that we get an isomorphism when restricted to the respective images of $V_p$. This immediately implies that $\IV_{p^\alpha}^*\rightarrow \ov\IV_{p^\alpha}^*$ must be surjective. For injectivity, observe that $p\IV_{p^\alpha}^*$ and $p\ov\IV_{p^\alpha}^*$ are contained in the respective images of $V_p$ since $p\d V_p=V_p\circ \d$ by \cref{lem:dAfterV}. So $p\IV_{p^\alpha}^*\rightarrow p\ov\IV_{p^\alpha}^*$ must be injective. By $p$-torsion freeness we conclude that $\IV_{p^\alpha}^*\rightarrow \ov\IV_{p^\alpha}^*$ must be injective as well.
	\end{proof}
	
	Next we set out to prove the implication \cref{enum:alphaB} $\Rightarrow$ \cref{enum:alphaC}. From now on, $R$ is an arbitrary smooth $A$-algebra; the existence of an étale framing $\square\colon A[T_1,\dotsc,T_n]\rightarrow R$ is no longer assumed.
	\begin{numpar}[A non-canonical quasi-isomorphism I.]\label{par:NonCanonicalQuasiIso}
		Contrary to what we just said, assume that $R$ admits an étale framing $\square\colon A[T_1,\dotsc,T_n]\rightarrow R$ and fix one such choice. We'll use it to construct a quasi-isomorphism as in \cref{prop:qDRWTrivialAfterpCompletion}, albeit a priori a non-canonical one (and compatibility with the $\IE_\infty$-$A[q]$-algebra structures is also not a priori clear). To do that, we consider the following diagram:
		\begin{equation*}
			\begin{tikzcd}
				\left(\qOmega_{R/A,\square}^*\otimes_{A[q],\phi^\alpha}A[q]/(q^{p^\alpha}-1)\right)_p^\complete \rar["\phi_{\square/A}^\alpha"]\dar[dashed] & \eta_{(\Phi_p(q)\Phi_{p^2}(q)\dotsm \Phi_{p^\alpha}(q))}\left(\bigl(\qOmega_{R/A,\square}^*\bigr)_p^\complete\right)/(q^{p^\alpha}-1)\dar["\cong"]\\
				\H^*\left(\bigl(\qHodge_{R/A,\square}^*\bigr)_p^\complete/(q^{p^\alpha}-1)\right) & \eta_{(q^{p^{\alpha}}-1)}\left(\bigl(\qHodge_{R/A,\square}^*\bigr)_p^\complete\right)/(q^{p^\alpha}-1)\lar["\simeq"']
			\end{tikzcd}
		\end{equation*}
		Let us explain what happens here: The right vertical arrow is an isomorphism for obvious reasons, and the bottom arrow is the quasi-isomorphism from \cite[\stackstag{0F7T}]{Stacks}. The top row comes from the map of complexes
		\begin{equation*}
			\phi_{\square}\colon \bigl(\qOmega_{R/A,\square}^*\bigr)_{p}^\complete\longrightarrow \bigl(\qOmega_{R/A,\square}^*\bigr)_p^\complete
		\end{equation*}
		which is constructed as in \cref{par:FrobeniusqHodge} except that we put $\phi_{\square}(\d T_i)\coloneqq \Phi_p(q)T_i^{p-1}\d T_i$.%
		\footnote{Under the identification $\eta_{\Phi_p(q)}(\qOmega_{R/A,\square}^*)\cong \eta_{(q^p-1)}(\qHodge_{R/A,\square}^*)$, this agrees with applying $\eta_{(q-1)}$ to the Frobenius morphism from \cref{par:FrobeniusqHodge} (which sends $q\mapsto q^p$, so $\eta_{(q-1)}$ on the left-hand side corresponds to $\eta_{(q^p-1)}$ on the right-hand side).}
		By construction, the image of $\phi_\square$ is contained in the subcomplex $\eta_{\Phi_p(q)}((\qOmega_{R/A,\square}^*)_p^\complete)$, hence the top row $\phi_{\square/A}^\alpha$ of the diagram has the indicated form.
	\end{numpar}
	\begin{lem}
		The morphism $\phi_{\square/A}^\alpha$ from the diagram from \cref{par:NonCanonicalQuasiIso} is a quasi-isomorphism.
	\end{lem}
	\begin{proof}
		This is a general feature of prismatic cohomology. If we consider $S\coloneqq \widehat{R}_p[\zeta_p]$ and the $q$-de Rham prism $(\widehat{A}_p\qpower,(\Phi_p(q)))$, then $(\qOmega_{R/A,\square}^*)_p^\complete\simeq \Prism_{S/\widehat{A}_p\qpower}$ by \cite[Theorem~\href{https://arxiv.org/pdf/1905.08229v4\#theorem.16.22}{16.22}]{Prismatic}. As we'll check below, this identifies $\phi_\square$ with the primsatic Frobenius. Now for any prism $(B,J)$, \cite[Theorem~\href{https://arxiv.org/pdf/1905.08229v4\#theorem.15.3}{15.3}]{Prismatic} shows that the relative Frobenius
		\begin{equation*}
			\bigl(\Prism_{S/B}\lotimes_{B,\phi_B^\alpha}B\bigr)_{(p,J)}^\complete\overset{\simeq}{\longrightarrow} \L\eta_{(J\phi_B(J)\dotsm\phi_B^{\alpha-1}(J))}\Prism_{S/B}\,.
		\end{equation*}
		is a quasi-isomorphism, which is what we want.
		
		To check that $\phi_\square$ agrees with the prismatic Frobenius, first observe that with the same definition, one can define an endomorphism $\phi_\square$ for any $q$-PD de Rham complex as in \cite[Construction~\href{https://arxiv.org/pdf/1905.08229v4\#theorem.16.20}{16.20}]{Prismatic}. We claim that each of them computes the prismatic Frobenius. To show this, consider the cosimplicial complex $M^{\bullet,*}$ from the proof of \cite[Theorem~\href{https://arxiv.org/pdf/1905.08229v4\#theorem.16.22}{16.22}]{Prismatic}. Applying our construction of $\phi_\square$ for $q$-PD de Rham complexes yields an endomorphism $\phi_\square\colon M^{\bullet,*}\rightarrow M^{\bullet,*}$. Restricted to $M^{0,*}$, this is the endomorphism we started with. Restricted to the cosimplicial $\widehat{A}_p\qpower$-module $M^{\bullet,0}$, we get the prismatic Frobenius, since in degree $0$ the endomorphism $\phi_\square$ is induced by the Frobenii of the $q$-PD envelopes involved (which are $\delta$-rings).
	\end{proof}
	\begin{numpar}[A non-canonical isomorphism II.]\label{par:NonCanonicalQuasiIsoII}
		Since $\phi^\alpha\colon A[q]\rightarrow A[q]/(q^{p^\alpha}-1)$ factors through the projection $A[q]\rightarrow A$, the top left corner in the diagram from \cref{par:NonCanonicalQuasiIso} is isomorphic to $(\Omega_{R/A}^*\otimes_{A,\phi^\alpha}A[q]/(q^{p^\alpha}-1))_p^\complete$. By \cref{enum:alphaB}, the bottom left corner is isomorphic to $(\qIW_{p^\alpha}\Omega_{R/A}^*)_p^\complete$. This yields a quasi-isomorphism
		\begin{equation*}
			s_{\square/A}\colon \bigl(\Omega_{R/A}\lotimes_{A,\phi^\alpha}A[q]/(q^{p^\alpha}-1)\bigr)_p^\complete\overset{\simeq}{\longrightarrow} \bigl(\qIW_{p^\alpha}\Omega_{R/A}\bigr)_p^\complete\,,
		\end{equation*}
		which has the desired form from \cref{prop:qDRWTrivialAfterpCompletion}. As we've already mentioned above, this quasi-isomorphism is a priori non-canonical and compatibility with the $\IE_\infty$-$A[q]$-algebra structures is far from clear.
	\end{numpar}
	As we'll see now, the map $s_{\square/A}$ from \cref{par:NonCanonicalQuasiIsoII} is, in fact, canonical.
	\begin{numpar}[A functorial comparison map.]\label{par:CanonicalComparisonMap}
		Let $\pi\colon P \twoheadrightarrow \widehat{R}_p$ be any surjection from a $p$-completely ind-smooth $\delta$-$\widehat{A}_p$-algebra (see \cref{par:Notation} for the terminology). Let $D(\pi)$ denote the divided power envelope of $\ker\pi$. Then $D(\pi)$ has a canonical $\delta$-structure by \cite[Corollary~\href{https://arxiv.org/pdf/1905.08229v4\#theorem.2.39}{2.39}]{Prismatic}. We let $\Omega_{D(\pi)/A}^*$ and $\breve{\Omega}_{D(\pi)/A}^*\cong \Omega_{P/A}^*\otimes_PD(\pi)$ denote the ordinary and the PD-de Rham complex of $D(\pi)$, respectively. Note that the crystalline Poincaré lemma \cite[\stackstag{07LG}]{Stacks} shows that the canonical map $(\breve{\Omega}_{D(\pi)/A}^*)_p^\complete\rightarrow (\Omega_{R/A}^*)_p^\complete$ is an equivalence on underlying $\IE_\infty$-algebras.
		
		Using \cref{par:RelativeComparisonMaps} (for which we only need a $\Lambda_p$-$A$-algebra structure on $D(\pi)$, so the given $\delta$-structure does it) we get a comparison map $s_{p^\alpha/A}\colon D(\pi)\otimes_{A,\phi^\alpha}A[q]/(q^{p^\alpha}-1)\rightarrow \qIW_{p^\alpha}(D(\pi)/A)$ which induces a morphism of differential-graded $A[q]$-algebras
		\begin{equation*}
			\Omega_{D(\pi)/A}^*\otimes_{A,\phi^\alpha}A[q]/(q^{p^\alpha}-1)\longrightarrow \qIW_{p^\alpha}\Omega_{D(\pi)/A}^*\,.
		\end{equation*}
		Now let $\qIW_{p^\alpha}\ov\Omega_{D(\pi)/A}^*\coloneqq \qIW_{p^\alpha}\Omega_{D(\pi)/A}^*/(p^\infty\text{-torsion})$. Since $\qIW_{p^\alpha}\overline\Omega_{D(\pi)/A}^*$ is $p$-local (because $D(\pi)$ is, see \cref{cor:qWittLocalisation}\cref{enum:MultiplicativeSubsetOfZ}) and $p$-torsion-free, its differentials must be PD-derivations. Therefore, we get an induced map $\breve{\Omega}_{D(\pi)/A}^*\rightarrow \qIW_{p^\alpha}\ov\Omega_{D(\pi)/A}^*$ and hence also a map
		\begin{equation*}
			\bigl(\breve{\Omega}_{D(\pi)/A}^*\otimes_{A,\phi^\alpha}A[q]/(q^{p^\alpha}-1)\bigr)_p^\complete\longrightarrow \bigl(\qIW_{p^\alpha}\Omega_{D(\pi)/A}^*\bigr)_p^\complete\,,
		\end{equation*}
		where the completion is the degree-wise underived $p$-completion (following our convention from~\cref{par:Notation}). Finally, since $\qIW_{p^\alpha}\Omega_{R/A}^*$ is degree-wise $p$-torsion-free by \cref{enum:alphaA}, the natural map $\qIW_{p^\alpha}\Omega_{D(\pi)/A}^*\rightarrow \qIW_{p^\alpha}\Omega_{R/A}^*$ induces a morphism $(\qIW_{p^\alpha}\ov\Omega_{D(\pi)/A}^*)_p^\complete\rightarrow (\qIW_{p^\alpha}\Omega_{R/A}^*)_p^\complete$ of differential-graded $A[q]$-algebras. Summarising, we obtain the following diagram of $\IE_\infty$-$A[q]$-algebras
		\begin{equation*}
			\begin{tikzcd}
				\bigl(\breve{\Omega}_{D(\pi)/A}\lotimes_{A,\phi^\alpha}A[q]/(q^{p^\alpha}-1)\bigr)_p^\complete\rar\dar["\simeq"']& \bigl(\qIW_{p^\alpha}\ov\Omega_{D(\pi)/A}\bigr)_p^\complete\dar\\
				\bigl(\Omega_{R/A}\lotimes_{A,\phi^\alpha}A[q]/(q^{p^\alpha}-1)\bigr)_p^\complete\rar[dashed,"{s_\pi}"] & \bigl(\qIW_{p^\alpha}\Omega_{R/A}\bigr)_p^\complete
			\end{tikzcd}
		\end{equation*}
		Since the left vertical arrow is an equivalence, the bottom dashed arrow $s_\pi$ exists uniquely up to contractible choice.
		
		We claim that $s_\pi$ doesn't depend on the choice of $\pi\colon P \twoheadrightarrow \widehat{R}_p$ (up to equivalence in the $\infty$-category $\CAlg(\widehat{\Dd}_p(A[q]))$ of $p$-complete $\IE_\infty$-$A[q]$-algebras). Indeed, let $\pi'\colon P' \twoheadrightarrow\widehat{R}_p$ be another surjection from a $p$-completely ind-smooth $\delta$-$\widehat{A}_p$-algebra. If there exists a $\delta$-$\widehat{A}_p$-algebra map $f\colon P\rightarrow P'$ such that $\pi=\pi'\circ f$, then it's clear that $s_\pi\simeq s_{\pi'}$ since the the whole construction is functorial with respect to $\delta$-$\widehat{A}_p$-algebra maps. In general, let $\Cc_R$ denote the category of surjections $(\pi\colon P \twoheadrightarrow \widehat{R}_p)$ as above, with morphisms $\pi\rightarrow \pi'$ given by $\delta$-$\widehat{A}_p$-algebra morphisms $f\colon P\rightarrow P'$ such that $\pi=\pi'\circ f$. Then $\Cc_R$ has coproducts given by $\pi\sqcup \pi'=(\pi\otimes\pi'\colon (P\otimes_{A}P')_p^\complete\twoheadrightarrow \widehat{R}_p)$, hence it is weakly contractible.%
		\footnote{If $\Cc$ is any $\infty$-category with coproducts, the diagonal $\Delta\colon\Cc\rightarrow \Cc\times\Cc$ has a left adjoint, which forces $\abs{\Delta}\colon\abs{\Cc}\rightarrow\abs{\Cc}\times\abs{\Cc}$ to be an equivalence. Then all $\pi_*\abs{\Cc}$ must be singletons.} Therefore, for arbitrary elements in $\Cc_R$, there's an essentially unique way to compare them, proving that they all give rise to the same map $s_{\pi}$ up to equivalence.
		
		Furthermore, this map can be made functorial in $R$. The easiest way to do so is to simply choose a functorial surjection $\pi_R\colon P_R \twoheadrightarrow\widehat{R}_p$ from a $p$-completely ind-smooth $\delta$-ring; for example, one can take $P_R\coloneqq \widehat{A}_p\{W(\widehat{R}_p)\}_p^\complete$ to be the $p$-complete free $\delta$-$\widehat{A}_p$-algebra on the set $W(\widehat{R}_p)$ of $p$-typical Witt vectors, together with its canonical surjection
		\begin{equation*}
			\pi_R\colon P_R \longtwoheadrightarrow W(\widehat{R}_p) \longtwoheadrightarrow \widehat{R}_p
		\end{equation*}
		Here $W(\widehat{R}_p)$ is equipped with its $\widehat{A}_p$-algebra structure induced via $\widehat{A}_p\rightarrow W(\widehat{A}_p)\rightarrow W(\widehat{R}_p)$, using the $\delta$-structure on $\widehat{A}_p$. This yields the desired functoriality.
	\end{numpar}
	\begin{rem}\label{rem:ConceptualConstruction}
		 A conceptually nicer way would be to consider the category $\Cc$ of all pairs $(R,\pi\colon P\twoheadrightarrow \widehat{R}_p)$ together with its forgetful functor $U\colon \Cc\rightarrow \cat{Sm}_A$ into the category of smooth $A$-algebras. In \cref{par:CanonicalComparisonMap} we've constructed a natural transformation
		\begin{equation*}
			s_{(-)}\colon \bigl(\breve{\Omega}_{D(-)/A}\lotimes_{A,\phi^\alpha}A[q]/(q^{p^\alpha}-1)\bigr)_p^\complete\Longrightarrow \bigl(\qIW_{p^\alpha}\Omega_{-/A}\bigr)_p^\complete\circ U\,.
		\end{equation*}
		One can show that $(\Omega_{-/A}\lotimes_{A,\phi^\alpha}A[q]/(q^{p^\alpha}-1))_p^\complete$ is the left Kan extension of the functor on the left-hand side along $U$.%
		\footnote{For any $R\in\cat{Sm}_A$, the value of the left Kan extension at $R$ is given by $\colimit(\Cc_{/R}\rightarrow \Cc\rightarrow \CAlg(\widehat{\Dd}_p(A[q])))$. Using that $\Cc_R$ is weakly contractible, it will thus be enough to check that $\Cc_R\rightarrow \Cc_{/R}$ is coinitial. The same argument as in \cref{par:CanonicalComparisonMap} shows that $\Cc_{/R}$ has coproducts. Thus, if we choose any $\pi_0\in \Cc_R$, then the slice category projection $(\Cc_{/R})_{\pi_0/}\rightarrow \Cc_{/R}$ will be a right adjoint, with left adjoint given by $-\sqcup \pi_0$. In particular, $(\Cc_{/R})_{\pi_0/}\rightarrow \Cc_{/R}$ is coinitial. Now let $(\Cc_{/R})^{\mathrm{surj}}\subseteq \Cc_{/R}$ be the full subcategory spanned by those $((R',\pi')\in \Cc, (R'\rightarrow R)\in (\cat{Sm}_A)_{/R})$ for which $R'\rightarrow R$ is surjective. By inspection, the image of $(\Cc_{/R})_{\pi_0/}\rightarrow \Cc_{/R}$ lands in $(\Cc_{/R})^\mathrm{surj}$. The same adjointness argument then shows that $(\Cc_{/R})_{\pi_0/}\rightarrow (\Cc_{/R})^\mathrm{surj}$ is coinitial too, hence $(\Cc_{/R})^\mathrm{surj}\rightarrow \Cc_{/R}$ must also be coinitial. Finally, $\Cc_R\rightarrow (\Cc_{/R})^\mathrm{surj}$ is a right adjoint, hence coinitial: The left adjoint simply sends a surjection $\pi'\colon P\twoheadrightarrow \widehat{R}_p'$ to its composition with $\widehat{R}_p'\twoheadrightarrow \widehat{R}_p$.}
		Then the universal property of left Kan extension provides the desired natural transformation.
	\end{rem}
	Armed with a functorial comparison map, we'll now prove the remaining two implications, thus finishing the induction outlined in~\cref{par:OutlineOfStrategy}.
	\begin{proof}[Proof of \cref{enum:alphaB} $\Rightarrow$ \cref{enum:alphaC}]
		We need to prove that the natural transformation
		\begin{equation*}
			\bigl(\Omega_{-/A}\lotimes_{A,\phi^\alpha}A[q]/(q^{p^\alpha}-1)\bigr)_p^\complete\Longrightarrow \bigl(\qIW_{p^\alpha}\Omega_{-/A}\bigr)_p^\complete
		\end{equation*}
		constructed in \cref{par:CanonicalComparisonMap} and \cref{rem:ConceptualConstruction} is an equivalence. As both sides are étale sheaves by \cref{cor:qDRWEtaleSheaf}, it's enough to do this in the case where there exists an étale map $\square\colon A[T_1,\dotsc,T_n]\rightarrow R$. As in \cref{par:FrobeniusLift}, we get a $\delta$-$\widehat{A}_p$-algebra structure on $\widehat{R}_p$, thus making $(\id\colon \widehat{R}_p\rightarrow \widehat{R}_p)$ into an object of $\Cc_R$. The corresponding divided power envelope is just $\widehat{R}_p$ itself and so, with notation as in \cref{par:CanonicalComparisonMap}, $s_{\id}$ is a morphism of differential-graded $A[q]$-algebras
		\begin{equation*}
			\bigl(\Omega_{R/A}^*\otimes_{A,\phi^\alpha}A[q]/(q^{p^\alpha}-1)\bigr)_p^\complete\longrightarrow\bigl(\qIW_{p^\alpha}\Omega_{R/A}^*\bigr)_p^\complete\,.
		\end{equation*}
		We claim that this map coincides with the quasi-isomorphism $s_{\square/A}$ from \cref{par:NonCanonicalQuasiIso}, which would finish the proof.
		
		Both $s_{\id}$ and $s_{\square/A}$ are given as explicit maps of differential-graded $A[q]$-algebras. By the universal property of de Rham complexes, it is thus enough to check that $s_{\id}$ and $s_{\square/A}$ agree in degree $0$. In fact, it's enough to check this after postcomposition with the comparison map $c_{p^\alpha/A}\colon \qIW_{p^\alpha}(\widehat{R}_p/A)\rightarrow \widehat{R}_p[q]/(q^{p^\alpha}-1)$, which we know to be injective as a consequence of \cref{enum:alphaB}. By construction, $s_{\id}$ is given by the comparison map
		\begin{equation*}
			s_{p^\alpha/A}\colon \widehat{R}_p\otimes_{A,\phi^\alpha}A[q]/(q^{p^\alpha}-1)\longrightarrow \qIW_{p^\alpha}(\widehat{R}_p/A)
		\end{equation*}
		from \cref{par:RelativeComparisonMaps} in degree $0$, using the $\delta$-$A$-algebra structure on $\widehat{R}_p$. We've noted in \cref{par:RelativeComparisonMaps} that $c_{p^\alpha/A}\circ s_{p^\alpha/A}$ is given by the linearised $(p^\alpha)$\textsuperscript{th} Adams operation of $\widehat{R}_p$. But that's just $\phi_{\square/A}^\alpha$! By unravelling \cref{par:NonCanonicalQuasiIso}, we see that $s_{\square/A}$ in degree~$0$, postcomposed with $c_{p^\alpha/A}$, is also given by $\phi_{\square/A}^\alpha$. This finishes the proof.
	\end{proof}
	\begin{proof}[Proof of \cref{enum:alphaC} $\Rightarrow$ \cref{enum:alphaD}]
		The idea is to combine the equivalence from \cref{prop:qDRWTrivialAfterpCompletion} with the Cartier isomorphism. Let $R\cong A[T_1,\dotsc,T_n]$ be a polynomial ring; we equip $\widehat{R}_p$ with the identity $p$-completely étale framing $\square\colon \widehat{A}_p\langle T_1,\dotsc,T_n\rangle\rightarrow \widehat{R}_p$ and the corresponding $\delta$-$\widehat{A}_p$-algebra structure given by $\delta(T_i)=0$. Furthermore, let
		\begin{equation*}
			s_{\square/A}\colon \bigl(\Omega_{R/A}^*\otimes_{A,\phi^\alpha}A[q]/(q^{p^\alpha}-1)\bigr)_p^\complete\overset{\simeq}{\longrightarrow}\bigl(\qIW_{p^\alpha}\Omega_{R/A}^*\bigr)_p^\complete
		\end{equation*}
		be the explicit quasi-isomorphism from \cref{par:NonCanonicalQuasiIso}. We note that the left-hand side can be rewritten as $(\Omega_{R_0/\IZ}\otimes_\IZ A[q]/(q^{p^\alpha}-1))_p^\complete$, where $R_0\coloneqq \IZ[T_1,\dotsc,T_n]$.
		
		Now assume $\xi \in \qIW_{p^\alpha}\Omega_{R/A}^i$ satisfies $\d\xi\equiv 0\mod p$. Letting $\overline{\xi}$ denote the image of $\xi$ in $\qIW_{p^\alpha}\Omega_{R/A}^*/p$; we see that $\overline{\xi}$ is a cycle. Since both source and target of $s_{\square/A}$ are $p$-torsion-free (using \cref{enum:alphaA}, which we already know), it induces a quasi-isomorphism
		\begin{equation*}
			\overline{s}_{\square/A}\colon \Omega_{(R_0/p)/\IF_p}^*\otimes_{\IF_p}(A/p)[q]/\bigl(q^{p^\alpha}-1\bigr)\overset{\simeq}{\longrightarrow}\qIW_{p^\alpha}\Omega_{R/A}^*/p\,.
		\end{equation*}
		Consequently, we can write $\overline{\xi}=\overline{s}_{\square/A}(\overline{\vartheta})+\d\overline{\xi}_0$, where $\overline{\vartheta}$ is a cycle in the base changed de Rham complex $\Omega_{(R_0/p)/\IF_p}^i\otimes_{\IF_p}(A/p)[q]/(q^{p^\alpha}-1)$. But cycles in the de Rham complex of a polynomial ring over $\IF_p$ can be very explicitly described using the Cartier isomorphism, or rather the ideas that lead to it. Namely, we can write $\overline{\vartheta}=\overline{\vartheta}_0+\d\overline{\vartheta}_1$, where $\overline{\vartheta}_0$ is an $A[q]$-linear combination of terms of the form $T_1^{pv_1}\dotsm T_n^{p v_n} (T_{n_1}^{p-1}\d T_{n_1})\dotsb (T_{n_i}^{p-1}\d T_{n_i})$, where $v_1,\dotsc,v_d\geqslant 0$ and $1\leqslant n_1<n_2<\dotso<n_i\leqslant n$. Now choose lifts $\xi_0$, $\vartheta_0$ and $\vartheta_1$ of $\overline{\xi}_0$, $\overline{\vartheta}_0$, and $\overline{\vartheta}_1$, respectively. Then $\xi\equiv s_{\square/A}(\vartheta_0)+\d s_{\square/A}(\vartheta_1)+\d\xi_0\mod p$. Both $\d s_{\square/A}(\vartheta_1)$ and $\d\xi_0$ are in the image of $F_p\colon \qIW_{p^{\alpha+1}}\Omega_{R/A}^*\rightarrow \qIW_{p^\alpha}\Omega_{R/A}^*$ since $F_p\circ \d \circ V_p=\d$. Furthermore, we've seen in the proof of \cref{enum:alphaB} $\Rightarrow$ \cref{enum:alphaC} above that $s_{\square/A}$ is induced by the comparison map $s_{p^\alpha/A}$ from \cref{par:RelativeComparisonMaps}. This map sends $T_i$ to its Teichmüller lift $\tau_{p^\alpha}(T_i)$ since $\delta(T_i)=0$. Hence $s_{\square/A}(\vartheta_0)$ is an $A[q]$-linear combination of terms of the form
		\begin{equation*}
			\tau_{p^\alpha}(T_1)^{pv_1}\dotsb\tau_{p^\alpha}(T_n)^{pv_n}\bigl(\tau_{p^\alpha}(T_{n_1})^{p-1}\d\tau_{p^\alpha}(T_{n_1})\bigr)\dotsm \bigl(\tau_{p^\alpha}(T_{n_i})^{p-1}\d\tau_{p^\alpha}(T_{n_i})\bigr)\,,
		\end{equation*}
		which are also in the image of $F_p$. This finishes the proof.
	\end{proof}
	
	\subsection{Proof of the main results II: The global case}\label{subsec:Global}
	In the previous section we've carried out the induction outlined in our battle plan~\cref{par:OutlineOfStrategy}. It remains to prove the global cases of \cref{prop:qDRWisTorsionFreeForSmoothRings,prop:qDRWTrivialAfterpCompletion} as wellas \cref{thm:qDeRhamWittqHodge}. Fortunately, these are all easily reduced to the $p$-typical cases. 
	\begin{proof}[Proof of \cref{prop:qDRWisTorsionFreeForSmoothRings}]
		We show that $\qIW_m\Omega_R^*$ is degree-wise $p$-torsion-free using induction on $m$. The case $m=1$ is covered by \embrace{\hyperref[enum:alphaA]{$a_0$}}. Now let $m>1$. Using étale descent as in the proofs of \cref{lem:qDRWGhostMap,lem:qDRWDegreewiseBoundedq-1Torsion}, we can reduce to the case where $R$ is a polynomial ring, and then by base change and \cref{lem:RelativeqDRWBaseChange} we can reduce to the case $A=\IZ$. 
		
		By \cref{cor:qWittNoetherian} and \cref{prop:qDRWExists}\cref{enum:dRtoqDRWisSurjective}, we see that $\qIW_m\Omega_{R/\IZ}^*$ is degree-wise finitely generated over the noetherian ring $\qIW_m(R)$. By the same argument as in the proof of \cref{lem:FpdAdditive}, it's therefore enough to show $p$-torsion freeness after applying each of the functors
		\begin{equation*}
			(-)\bigl[\localise{p}\bigr]\,,\quad (-)_{(p,q^{m/\ell}-1)}^\complete\,,\quad\text{and}\quad (-)\left[\localise{q^{m/\ell}-1}\ \middle|\ \ell\neq p\right]\,,
		\end{equation*}
		where $\ell$ ranges over all prime factors $\neq p$ of $m$. After localisation at $p$, the $p$-torsion freeness is trivial. After $(p,q^{m/\ell}-1)$-adic completion, both $\ell$ and $[\ell]_{q^{m/\ell}}$ become units and so
		\begin{equation*}
			F_\ell\colon \bigl(\qIW_m\Omega_{R/\IZ}^*\bigr)_{(p,q^{m/\ell}-1)}^\complete\overset{\cong}{\longrightarrow} \bigl(\qIW_{m/\ell}\Omega_{R/\IZ}^*\bigr)_{(p,q^{m/\ell}-1)}^\complete
		\end{equation*}
		is a graded isomorphism with inverse $[\ell]_{q^{m/\ell}}^{-1}V_\ell$. By the inductive hypothesis, the right-hand side is $p$-torsion-free, hence so is the left-hand side. Finally, if $\alpha$ is the exponent of $p$ in the prime factorisation of $m$, then $\qIW_m\Omega_{R/\IZ}^*\bigl[1/(q^{m/\ell}-1)\ \big|\ \ell\neq p\bigr]$ is isomorphic to a flat base change of $\qIW_{p^\alpha}\Omega_{R/\IZ}^*$, as was argued in the proof of \cref{lem:FpdAdditive}, so we're done by \cref{enum:alphaA}.
	\end{proof}
	To prove \cref{prop:qDRWTrivialAfterpCompletion}, we first need a $q$-de Rham--Witt analogue of \cref{lem:qWittOverZp}:
	\begin{lem}\label{lem:qDRWpLocalDecomposition}
		Let $R$ be any $A$-algebra and let $m=p^\alpha n$ be an integer, where $\alpha=v_p(m)$ is the exponent of $p$ in the prime factorisation of $m$. Then there's an isomorphism of differential-graded $A_{(p)}[q]$-algebras
		\begin{equation*}
			\bigl(\qIW_m\Omega_{R/A}^*\bigr)_{(p)}\cong \prod_{d\mid n}\left(\qIW_{p^\alpha}\Omega_{R/A}^*\otimes_{A[q],\psi^d} A_{(p)}[q]/\Phi_d(q^{p^\alpha})\right)\,,
		\end{equation*}
		where the map $\psi^d\colon A[q]\rightarrow A_{(p)}[q]/\Phi_d(q^{p^\alpha})$ is given by the Adams operation on the $\Lambda$-ring $A$ and $\psi^d(q)\coloneqq q^d$.
	\end{lem}
	\begin{proof}[Proof sketch]
		Let us abbreviate the right-hand side by $\Pi_m^*$. We'll show that $(\Pi_m^*)_{m\in\IN}$ exhibits the same universal property as $((\qIW_m\Omega_{R/A}^*)_{(p)})_{m\in\IN}$. To do so, one must first construct the structure of a $q$-$FV$-system of differential-graded $A$-algebras over $R$ on $(\Pi_m^*)_{m\in\IN}$. Let us first explain how to equip each $\Pi_m^0$ with a $\qIW_m(R)\otimes_{\qIW_m(A),c_m}A_{(p)}[q]/(q^m-1)$-algebra structure. To do so, we use
		\begin{equation*}
			\begin{tikzcd}
				\qIW_{m}(A)_{(p)}\rar["\cong","{(\labelcref{lem:qWittOverZp})}"']\dar["c_m"']& \prod_{d\mid n}\qIW_{p^\alpha}(A)\otimes_{\IZ[q],\psi^d}\IZ_{(p)}[q]/\Phi_d(q^{p^\alpha})\dar\\
				A_{(p)}[q]/(q^m-1)\rar["\cong"] & \prod_{d\mid n}A_{(p)}[q]/\Phi_d(q^{p^\alpha})
			\end{tikzcd}
		\end{equation*}
		The top isomorphism is \cref{lem:qWittOverZp} and the bottom isomorphism is the Chinese remainder theorem. The arrow on the right to make the diagram commute is given as follows: In the $d$\textsuperscript{th} factor, take the composition
		\begin{equation*}
			\qIW_{p^\alpha}(A)\overset{c_{p^\alpha}}{\longrightarrow}A[q]/(q^{p^\alpha}-1)\overset{\psi^d}{\longrightarrow}A_{(p)}[q]/\Phi_d(q^{p^\alpha})
		\end{equation*}
		and extend it linearly along the map $\psi^d\colon \IZ[q]\rightarrow \IZ_{(p)}[q]/\Phi_d(q^{p^\alpha})$. To check that this really makes the diagram commute can be done on ghost components, where it becomes a straightforward but tedious unravelling of definitions.
		
		Now to construct the desired $\qIW_m(R)\otimes_{\qIW_m(A),c_m}A_{(p)}[q]/(q^m-1)$-algebra structure on $\Pi_m^0\cong \prod_{d\mid n}\qIW_{p^\alpha}(R/A)\otimes_{A[q],\psi^d}A_{(p)}[q]/\Phi_d(q^{p^\alpha})$, we observe that the relative $q$-Witt vector ring $\qIW_{p^\alpha}(R/A)$ is an algebra over $\qIW_{p^\alpha}(R)\otimes_{\qIW_{p^\alpha}(A),c_{p^\alpha}}A[q]/(q^{p^\alpha}-1)$ and then use \cref{lem:qWittOverZp} together with the diagram above.
		
		To construct Frobenii and Verschiebungen on $(\Pi_m^*)_{m\in\IN}$, we proceed in the exact same way as in the proof of \cref{lem:qWittOverZp}. It's straightforward to verify that these satisfy the conditions from \cref{def:qVSystemOfCDGA}\cref{enum:qDeRhamWittConditionB}, \cref{def:qFVSystemOfCDGA}\cref{enum:qDeRhamWittConditionC}, and the Teichmüller conditions~\cref{enum:TeichmuellerV} and~\cref{enum:TeichmuellerF}. The existence of $A[q]$-linear Verschiebungen implies that the $\qIW_m(R)\otimes_{\qIW_m(A),c_m}A_{(p)}[q]/(q^m-1)$-algebra structure on $\Pi_m^0$ factors through a $\qIW_m(R/A)_{(p)}$-algebra structure.
		
		This finishes the construction of the desired structure on $(\Pi_m^*)_{m\in\IN}$. To prove universality, one can use the same argument as in the proof of \cref{lem:qWittOverZp}.
	\end{proof}
	
	\begin{proof}[Proof of \cref{prop:qDRWTrivialAfterpCompletion}]
		The morphisms $\psi^d\colon A[q]/(q^{p^\alpha}-1)\rightarrow A_{(p)}[q]/(\Phi_d(q)\dotsm \Phi_{p^\alpha d}(q))$ are flat by our assumptions on $A$. Therefore, the tensor products in \cref{lem:qDRWpLocalDecomposition} can be replaced by derived tensor products and we can reduce the general case of \cref{prop:qDRWTrivialAfterpCompletion} to the case $m=p^\alpha$, which we've already proved in \cref{enum:alphaC}.
	\end{proof}
	%Finally, we can prove the general case of \cref{thm:qDeRhamWittqHodge}.
	\begin{proof}[Proof of \cref{thm:qDeRhamWittqHodge}]
		By \cref{lem:DerivedBeauvilleLaszlo}, to check whether 
		\begin{equation*}
			\bigl(\qIW_m\Omega_{R/A}^*\bigr)_{(q-1)}^\complete\longrightarrow \H^*\bigl(\qHodge_{R/A,\square}^*/(q^m-1)\bigr)
		\end{equation*}
		is an isomorphism, it's enough to do so after degree-wise application of the functors $(-)[1/N]$ and $(-)_p^\complete$ for $p\mid N$, where $N\neq 0$ is divisible by $m$. Furthermore, as both sides are degree-wise derived $(q-1)$-complete, the localisation can be replaced by $(-)[1/N]_{(q-1)}^\complete$ instead. We've seen in the proof of \cref{lem:H*pTorsionFree} that the complex $\qHodge_{R/A,\square}^*/(q^m-1)$ satisfies the conditions of \cref{lem:TechnicalCompletionLemma}. Hence the functors $(-)_p^\complete$ and $(-)[1/N]_{(q-1)}^\complete$ commute with taking cohomology, and all completions can be computed as underived completions.
		
		Let's consider $p$-completions first. Let $\alpha\coloneqq v_p(m)$ be the exponent of $p$ in the prime factorisation of $m$. Both $F_{m/p^\alpha}\circ V_{m/p^\alpha}=m/p^\alpha$ and $V_{m/p^\alpha}\circ F_{m/p^\alpha}=[m/p^\alpha]_{q^{p^\alpha}}$ are invertible over $\IZ_p\qpower$, hence
		\begin{equation*}
			F_{m/p^\alpha}\colon \bigl(\qIW_m\Omega_{R/A}^*\bigr)_{(p,q-1)}^\complete\overset{\cong}{\longrightarrow} \bigl(\qIW_{p^\alpha}\Omega_{R/A}^*\bigr)_{(p,q-1)}^\complete
		\end{equation*}
		(where we take the degree-wise derived completion) is an isomorphism. The same conclusion holds for $\H^*(\qHodge_{R/A,\square}^*/(q^m-1))_p^\complete\rightarrow\H^*(\qHodge_{R/A,\square}^*/(q^{p^\alpha}-1))_p^\complete$. So we're reduced the assertion to \cref{enum:alphaB}, which we already know.
		
		The argument for $(-)_{(q-1)}^\complete$ is similar: Both $F_m\circ V_m=m$ and $V_m\circ F_m=[m]_q$ are invertible in $\IZ[1/N]\qpower$, hence we only need to check that
		\begin{equation*}
			\qIW_1\Omega_{R/A}^*\bigl[\localise{N}\bigr]_{(q-1)}^\complete\longrightarrow \H^*\bigl(\qHodge_{R/A,\square}^*/(q-1)\bigr)\bigl[\localise{N}\bigr]_{(q-1)}^\complete
		\end{equation*}
		is an isomorphism. By \cref{prop:qDRWExists}\cref{enum:dRtoqDRWisSurjective}, the left-hand side is just $\Omega_{R/A}^*[1/N]$, as is the right-hand side by inspection. The map between them is clearly the identity on $R[1/N]$ in degree~$0$ and compatible with the differential-graded algebra structures by construction, hence it must be the identity on $\Omega_{R/A}^*[1/N]$ by the universal property of de Rham complexes.
	\end{proof}
	
	\subsection{The arithmetic fracture square for \texorpdfstring{$q$}{q}-de Rham--Witt complexes}
	
	To finish this section, we'll prove two corollaries that allow us to identify  the arithmetic fracture square (in the sense of \cref{par:DerivedBeauvilleLaszlo}) for $\qIW_m\Omega_{R/A}$. These results won't play any role in the present article, but they will be used quite a lot in the sequel \cite{qWittHabiro}. 
	\begin{cor}\label{cor:qDRWArithmeticFractureSquare}
		Let $R$ be a smooth $A$-algebra, let $m\in\IN$, and let $N\neq 0$ be divisible by $m$. For every prime $p\mid N$ and every divisor $d\mid m$ write $m=p^{v_p(m)}m_p$ and $d=p^{v_p(d)}d_p$. Let also
		\begin{equation*}
			\phi_{p/A}\colon \Omega_{R/A}\lotimes_{A,\psi^p}A\longrightarrow \left(\Omega_{R/A}\right)_p^\complete
		\end{equation*}
		denote the relative Frobenius coming from the identification $(\Omega_{R/A})_p^\complete\simeq \R\Gamma_\crys((R/p)/\widehat{A}_p)$. Then there exists a functorial pullback diagram
		\begin{equation*}
			\begin{tikzcd}
				\qIW_m\Omega_{R/A}\rar\dar["(\gh_{m/d})_{d\mid m}"']\drar[pullback] & \prod_{p\mid N}\prod_{d_p\mid m_p}\left(\Omega_{R/A}\lotimes_{A,\psi^{p^{\smash{v_p(m)}}d_p}}A[q]\right)_p^\complete/\Phi_{d_p}\bigl(q^{p^{v_p(m)}}\bigr)\dar["\left(\phi_{p/A}^{v_p(m/d)}\right)_{p\mid N,\,d\mid m}"]\\
				\prod_{d\mid m}\Omega_{R/A}\lotimes_{A,\psi^d}A\bigl[\localise{N},q\bigr]/\Phi_d(q)\rar& \prod_{p\mid N}\prod_{d\mid m}\left(\Omega_{R/A}\lotimes_{A,\psi^d}A[q]\right)_p^\complete\bigl[\localise{p}\bigr]/\Phi_d(q)
			\end{tikzcd}
		\end{equation*}
	\end{cor}
	\begin{proof}
		Using \cref{cor:qDRWTrivialAfterLocalisation} and \cref{prop:qDRWTrivialAfterpCompletion}, we can identify the bottom left and top right corner with $\qIW_m\Omega_{R/A}[1/N]$ and $\prod_{p\mid N}(\qIW_m\Omega_{R/A})_p^\complete$, respectively. In view of the arithmetic fracture square from \cref{par:DerivedBeauvilleLaszlo}, it only remains to show the following:
		\begin{alphanumerate}\itshape
			\item[\boxtimes] \label{claim:FrobeniusIdentification} The following diagram commutes:
			\begin{equation*}
				\begin{tikzcd}[column sep=large]
					\prod_{p\mid N}\bigl(\qIW_m\Omega_{R/A}\bigr)_p^\complete\dar & \prod_{p\mid N}\prod_{d_p\mid m_p}\left(\Omega_{R/A}\lotimes_{A,\psi^{p^{\smash{v_p(m)}}d_p}}A[q]\right)_p^\complete/\Phi_{d_p}\bigl(q^{p^{v_p(m)}}\bigr)\lar["\simeq"',"(\labelcref{prop:qDRWTrivialAfterpCompletion})"]\dar["\left(\phi_{p/A}^{v_p(m/d)}\right)_{p\mid N,\,d\mid m}"]\\
					\prod_{p\mid N}\bigl(\qIW_m\Omega_{R/A}\bigr)_p^\complete\bigl[\localise{p}\bigr]\rar["\simeq","(\gh_{m/d})_{p\mid N,\,d\mid m}"'] & \prod_{p\mid N}\prod_{d\mid m}\left(\Omega_{R/A}\lotimes_{A,\psi^d}A[q]\right)_p^\complete\bigl[\localise{p}\bigr]/\Phi_d(q)
				\end{tikzcd}
			\end{equation*}
		\end{alphanumerate}
		To prove \cref{claim:FrobeniusIdentification}, fix a prime $p$ and a divisor $d\mid m$. Let also $\alpha\coloneqq v_p(m)$ and $i\coloneqq v_p(d)$ for short. By unravelling the proof of \cref{lem:qDRWpLocalDecomposition}, we see that the following diagram commutes:
		\begin{equation*}
			\begin{tikzcd}
				\qIW_m\Omega_{R/A}^*\drar["\gh_{m/d}"',bend right=12]\rar & \left(\qIW_{p^\alpha}\Omega_{R/A}^*\otimes_{A[q],\psi^{p^\alpha d_p}}A[q]/\Phi_{d_p}\bigl(q^{p^\alpha}\bigr)\right)_p^\complete\dar["\gh_{p^\alpha/p^i}"]\\
				&\left(\Omega_{R/A}^*\otimes_{A,\psi^d}A[q]/\Phi_d(q)\right)_p^\complete
			\end{tikzcd}
		\end{equation*}
		Also let $D=D(\pi)$ be a PD-envelope as in \cref{par:CanonicalComparisonMap} and $\breve{\Omega}_{D/A}^*$ the associated PD-de Rham complex. By unravelling \cref{par:CanonicalComparisonMap}, to finish the proof of \cref{claim:FrobeniusIdentification}, it will be enough to show that
		\begin{equation*}
			\begin{tikzcd}
				\bigl(\qIW_{p^\alpha}\Omega_{R/A}^*\bigr)_p^\complete\dar["\gh_{p^\alpha/p^i}"'] & \left(\breve{\Omega}_{D/A}^*\otimes_{A,\psi^{p^\alpha}}A[q]/(q^{p^\alpha}-1)\right)_p^\complete\lar["\simeq"',"(\labelcref{par:CanonicalComparisonMap})"]\dar["\phi_{p/A}^{\alpha-i}"]\\
				\left(\Omega_{R/A}^*\otimes_{A,\psi^{p^i}}A[q]/\Phi_d(q)\right)_p^\complete & \left(\breve{\Omega}_{D/A}^*\otimes_{A,\psi^{p^\alpha}}A[q]/(q^{p^\alpha}-1)\right)_p^\complete\lar
			\end{tikzcd}
		\end{equation*}
		commutes. By the universal property of the PD-de Rham complex, it's enough to show commutativity in degree~$0$. But the map
		\begin{equation*}
			s_{p^\alpha/A}\colon D\otimes_{A,\psi^{p^\alpha}}A[q]/(q^{p^\alpha}-1)\longrightarrow \qIW_{p^\alpha}(D/A)
		\end{equation*}
		considered in \cref{par:CanonicalComparisonMap} satisfies ${\gh_{p^\alpha/p^i}}\circ s_{p^\alpha/A}=\phi_{D/A}^{\alpha-i}$, where now $\phi_{D/A}\colon D\otimes_{A,\psi^p}A\rightarrow D$ denotes the linearised Frobenius of the $\delta$-$A$-algebra $D$. From this observation it becomes obvious that the diagram above commutes in degree~$0$, thus in any degree. This finishes the proof of \cref{claim:FrobeniusIdentification}.
	\end{proof}
	Let $\overtilde{F}_{m/d}\colon \qIW_m\Omega_{R/A}^*\rightarrow \qIW_d\Omega_{R/A}^*$ be given by $(m/d)^iF_{m/d}$ in degree~$i$, so that $\overtilde{F}_{m/d}$ is a map of differential-graded $A[q]$-algebras. The effect of $\overtilde{F}_{m/d}$ on arithmetic fracture squares can be determined using the following corollary. 
	\begin{cor}
		Let $R$ be a smooth $A$-algebra, let~$p$ be a prime and let $\phi_{p/A}$ denote the crystalline Frobenius as in \cref{cor:qDRWArithmeticFractureSquare}. Then the following diagram commutes:
		\begin{equation*}
			\begin{tikzcd}
				\bigl(\Omega_{R/A}\lotimes_{A,\phi^\alpha}A[q]/(q^{p^\alpha}-1)\bigr)_p^\complete\rar["\simeq","(\labelcref{prop:qDRWTrivialAfterpCompletion})"']\dar["\phi_{p/A}"'] & \bigl(\qIW_{p^\alpha}\Omega_{R/A}\bigr)_p^\complete\dar["\overtilde{F}_p"]\\ \bigl(\Omega_{R/A}\lotimes_{A,\phi^{\alpha-1}}A[q]/(q^{p^{\alpha-1}}-1)\bigr)_p^\complete\rar["\simeq","(\labelcref{prop:qDRWTrivialAfterpCompletion})"'] & \bigl(\qIW_{p^{\alpha-1}}\Omega_{R/A}\bigr)_p^\complete
			\end{tikzcd}
		\end{equation*}
	\end{cor}
	\begin{proof}
		Let $D=D(\pi)$ be a PD-envelope as in \cref{par:CanonicalComparisonMap} and let $\phi_{D/A}\colon D\otimes_{A,\phi}A\rightarrow D$ denote the relative Frobenius of the $\delta$-$A$-algebra $D$. By the universal property of the PD-de Rham complex, it induces a map of differential graded algebras
		\begin{equation*}
			\bigl(\breve{\Omega}_{D/A}^*\otimes_{A,\phi}A\bigr)_p^\complete\rightarrow \bigl(\breve{\Omega}_{D/A}^*\bigr)_p^\complete\,,
		\end{equation*}
		which represents the $\IE_\infty$-$A[q]$-algebra map $\phi_{p/A}$. Again by the universal property, whether this map of PD-de Rham complexes is compatible with $\overtilde{F}_p$ can be checked in degree~$0$. Therefore, it's enough to check that the comparison map $s_{p^\alpha/A}$ from \cref{par:CanonicalComparisonMap} satisfies $F_p\circ s_{p^\alpha/A}=s_{p^{\alpha-1}/A}\circ \phi_{D/A}$, which it does by construction.
	\end{proof}

	\newpage
	
	\section{A no-go result for functoriality of the \texorpdfstring{$q$}{q}-Hodge complex}\label{sec:Functoriality}
	In this final section, we'll show the following result, which perhaps comes as an unwelcome surprise after our very promising \cref{thm:qDeRhamWittqHodge}.
	\begin{thm}\label{thm:qHodgeNotFunctorial}
		Let $A$ be a perfectly covered $\Lambda$-ring \embrace{in the sense of \cref{rem:FaithfullyFlatCoverByPerfectLambdaRing}}. If $A$ is not a $\IQ$-algebra, then there can be no functor
		\begin{equation*}
			\qHodge_{-/A}\colon \cat{Sm}_A\longrightarrow \widehat{\Dd}_{(q-1)}\bigl(A\qpower\bigr)
		\end{equation*}
		from the category of smooth $A$-algebras into the $\infty$-category of derived $(q-1)$-complete $A\qpower$-modules in such a way that for all $m\in \IN$ there's also a functorial graded $A\qpower$-module isomorphism 
		\begin{equation*}
			\bigl(\qIW_m\Omega_{-/A}^*\bigr)_{(q-1)}^\complete\overset{\cong}{\longrightarrow}\H^*\bigl(\qHodge_{-/A}/^\L(q^m-1)\bigr)
		\end{equation*}
		and for all $d\mid m$ the canonical projection $\qHodge_{-/A}/^\L(q^m-1)\rightarrow \qHodge_{-/A}/^\L(q^d-1)$ induces the Frobenius $F_{m/d}$ on $q$-de Rham--Witt complexes. 
	\end{thm}
	\begin{rem}
		Over $\IQ$, the $q$-derivatives from \cref{par:qDeRhamqHodge} can be expressed in terms of the usual derivatives; see the argument in \cite[Lemma~\href{https://arxiv.org/pdf/1606.01796\#theorem.4.1}{4.1}]{Toulouse}. Hence in this case the $q$-Hodge complex can be made functorial, but it's no more interesting than the usual de Rham complex itself.
	\end{rem}
	\begin{rem}
		In \cite{qHodge}, we'll explain how in certain situations a functorial \emph{derived $q$-Hodge complex} can be constructed. This will constitute a partial fix for the non-existence result in \cref{thm:qHodgeNotFunctorial}.
	\end{rem}
	We'll now start the proof of \cref{thm:qHodgeNotFunctorial}. %Let us explain the idea first. If a functor as in \cref{thm:qHodgeNotFunctorial} would exist, we could form its \emph{non-abelian derived functor} (see \cref{par:DerivedqHodge} below). The assumed functorial isomorphism $(\qIW_m\Omega_{-/A}^*)_{(q-1)}^\complete\cong \H^*\bigl(\qHodge_{-/A}/^\L(q^m-1))$ provides us with some information about the non-abelian derived functor. As we'll see, this information leads to a contradiction.
	
	%If $A$ is not a $\IQ$-algebra, then we find a prime $p$ such that $\widehat{A}_p\neq 0$. Consider $R\coloneqq (\Oo_C\otimes_\IZ A)_p^\complete$, where $\Oo_C$ is the ring of integers in a complete algebraically closed extension of $\IQ_p$.  
	\begin{numpar}[Derived $q$-Hodge complexes.]\label{par:DerivedqHodge}
		Suppose a functor as in \cref{thm:qHodgeNotFunctorial} would exist. Let $\cat{AniAlg}_A$ denote the $\infty$-category of \emph{animated $A$-algebras} in the sense of Clausen. It can be explicitly constructed as the $\infty$-categorical localisation of the category of simplicial commutative $A$-algebras at all weak equivalences. Via left Kan extension from polynomial $A$-algebras (formerly known as forming the \emph{non-abelian derived functor}), we can define a \emph{derived $q$-Hodge complex}
		\begin{equation*}
			\L\!\qHodge_{-/A}\colon \cat{AniAlg}_A\longrightarrow \widehat{\Dd}_{(q-1)}\bigl(A\qpower\bigr)\,.
		\end{equation*}
	\end{numpar}
	\begin{numpar}[$q$-de Rham--Witt filtrations.]\label{par:Filtrations}
		By left Kan extending (or \emph{non-abelian deriving}) the Postnikov filtration $\tau^{\leqslant i}(\qHodge_{-/A}/^\L(q^m-1))$, we obtain for any animated $A$-algebra $R$ an ascending filtration $\Fil_*^{\qIW\Omega}(\L\!\qHodge_{R/A}/^\L(q^m-1))$ which we call the \emph{$q$-de Rham--Witt filtration}. Since the Postnikov filtration is exhaustive, we get
		\begin{equation*}
			\L\!\qHodge_{R/A}/^\L(q^m-1)\simeq \Bigl(\colimit_{i\geqslant 0}\Fil_i^{\qIW\Omega}\bigl(\L\!\qHodge_{R/A}/^\L(q^m-1)\bigr)\Bigr)_{(q-1)}^\complete\,.
		\end{equation*}
		Furthermore, if a functorial isomorphism $(\qIW_m\Omega_{-/A}^*)_{(q-1)}^\complete\cong \H^*\bigl(\qHodge_{-/A}/^\L(q^m-1))$ exists, then associated graded of the $q$-de Rham-Witt filtration is given by
		\begin{equation*}
			\gr_i^{\qIW\Omega}\bigl(\qHodge_{R/A}/^\L(q^m-1)\bigr)\simeq \bigl(\L\!\qIW_m\Omega_{R/A}^i\bigr)_{(q-1)}^\complete[-i]\,.
		\end{equation*}
		Let us also remark that the $0$\textsuperscript{th} filtration step of the $q$-de Rham--Witt filtration is always given by $\Fil_0^{\qIW\Omega}(\qHodge_{R/A}/^\L(q^m-1))\simeq \qIW_m(R/A)_{(q-1)}^\complete$. Indeed, using the above description of the associated graded $\gr_0^{\qIW\Omega}$, we only have to show that $\L\!\qIW_m(R/A)\rightarrow \qIW_m(R/A)$ is an equivalence. Base change along the faithfully flat map $A\rightarrow A_\infty$ and \cref{lem:RelativeqWittBaseChange} reduce this to a question about absolute $q$-Witt vectors, which follows inductively from \cref{prop:qWittKoszulExactSequence}.
	\end{numpar}
	\begin{proof}[Proof of \cref{thm:qHodgeNotFunctorial}]
		If $A$ is not a $\IQ$-algebra, then we find a prime $p$ such that $\widehat{A}_p\not\cong 0$. Consider $R\coloneqq (\Oo_C\otimes_\IZ A)_p^\complete$, where $\Oo_C$ is the ring of integers in a complete algebraically closed extension of $\IQ_p$. Note that $R\not\cong 0$ since $\Oo_C$ is the $p$-completion of a free $\IZ$-module. We also note that $\Oo_C$ is an integral perfectoid ring in the sense of \cite[Definition~\href{https://people.mpim-bonn.mpg.de/scholze/integralpadicHodge.pdf\#theorem.3.5}{3.5}]{BMS1}, so we can write $\Oo_C\cong \mathrm{A}_\inf(\Oo_C)/\xi$ for some nonzerodivisor $\xi\in \mathrm{A}_{\inf}(\Oo_C)$ such that $\delta(\xi)$ is a unit. Here $\delta$ refers to the usual $\delta$-structure on $\mathrm{A}_\inf(\Oo_C)$. If we define $W\coloneqq (\mathrm{A}_\inf(\Oo_C)\otimes_\IZ A)_p^\complete$, then $\xi$ is also a nonzerodivisor on and $W/\xi\cong R$.
		
		As we'll see in \cref{lem:DerivedqHodge} below, $(\L\!\qHodge_{W/A})_p^\complete\simeq W\qpower$; in particular, it is static in the sense of \cref{par:Notation}. Similarly, $(\L\!\qHodge_{R/A})_p^\complete$ is static and flat over $\IZ_p\qpower$. Hence all derived quotients $(-)/^\L(q^m-1)$ can be identified with actual quotients. We'll now play around with the filtrations on $(\L\!\qHodge_{R/A})_p^\complete/(q^m-1)$ and $(\L\!\qHodge_{R/A})_p^\complete/(q^m-1)$ for $m=1$ and $m=p$ and derive a contradiction.
		
		\emph{Case~1: $m=1$.} Consider the element $\xi\in W\qpower\cong (\L\!\qHodge_{W/A})_p^\complete$. Since $\xi$ vanishes under $\qIW_1(W/A)\rightarrow \qIW_1(R/A)$ and the diagram
		\begin{equation*}
			\begin{tikzcd}
				\qIW_1(W/A)\rar\dar & \Fil_0^{\qIW\Omega}\bigl(\L\!\qHodge_{W/A}/(q-1)\bigr)\rar\dar & \bigl(\L\!\qHodge_{W/A}\bigr)_p^\complete/(q-1)\dar\\
				\qIW_1(R/A)\rar & \Fil_0^{\qIW\Omega}\bigl(\L\!\qHodge_{R/A}/(q-1)\bigr)\rar& \bigl(\L\!\qHodge_{R/A}\bigr)_p^\complete/(q-1)
			\end{tikzcd}
		\end{equation*}
		commutes, we see that $\xi$ vanishes in $(\L\!\qHodge_{R/A})_p^\complete/(q-1)$ and so $\xi$ must be divisible by $(q-1)$ in $(\L\!\qHodge_{R/A})_p^\complete$.
		
		\emph{Case~2: $m=p$.} Consider the element $\phi(\xi)-\Phi_p(q)\delta(\xi)\in W\qpower\cong (\L\!\qHodge_{W/A})_p^\complete$, where $\phi$ denotes the Frobenius of the $\delta$-ring $\mathrm{A}_\inf(\Oo_C)$. Then the image of $\xi$ modulo $q^p-1$ agrees with the image of the Teichmüller lift $\tau_p(\xi)$ under $\qIW_p(W/A)\rightarrow (\L\!\qHodge_{W/A})_p^\complete/(q^p-1)$. Indeed, this follows from an unravelling of first the proof of \cref{lem:DerivedqHodge}\cref{enum:DerivedqHodgeW} and then the map from \cref{lem:WittToCyclicRing}: We must check that $\epsilon_p(\tau_p(\xi))=-\delta(\xi)$, which follows from the fact that the section $s_p\colon \mathrm{A}_\inf(\Oo_C)\rightarrow \IW_p(\mathrm{A}_\inf(\Oo_C))$ coming from the $\delta$-structure on $\mathrm{A}_\inf(\Oo_C)$ satisfies $s_p(\xi)=(\xi,\delta(\xi))=\tau_p(\xi)+V_p(\delta(\xi))$. Now the Teichmüller lift $\tau_p(\xi)$ vanishes under $\qIW_p(W/A)\rightarrow \qIW_p(R/A)$. Since the diagram
		\begin{equation*}
			\begin{tikzcd}
				\qIW_p(W/A)\rar\dar & \Fil_0^{\qIW\Omega}\bigl(\L\!\qHodge_{W/A}/(q^p-1)\bigr)\rar\dar & \bigl(\L\!\qHodge_{W/A}\bigr)_p^\complete/(q^p-1)\dar\\
				\qIW_p(R/A)\rar & \Fil_0^{\qIW\Omega}\bigl(\L\!\qHodge_{R/A}/(q^p-1)\bigr)\rar& \bigl(\L\!\qHodge_{R/A}\bigr)_p^\complete/(q^p-1)
			\end{tikzcd}
		\end{equation*}
		commutes, it follows that the image of $\phi(\xi)-\Phi_p(q)\delta(\xi)$ vanishes in $(\L\!\qHodge_{R/A})_p^\complete/(q^p-1)$ and so it must be divisible by $(q^p-1)$ in $(\L\!\qHodge_{R/A})_p^\complete$.
		
		We're ready to derive our contradiction. In the mod~$p$ reduction $(\L\!\qHodge_{R/A})_p^\complete/p$, we see that $\phi(\xi)-\Phi_p(q)\delta(\xi)\equiv \xi^p-(q-1)^{p-1}\delta(\xi)\mod p$ is divisible by $q^p-1\equiv (q-1)^p\mod p$. Since $\xi^p$ is also divisible by $(q-1)^p$ and $(\L\!\qHodge_{R/A})_p^\complete/p$ is flat over $\IF_p\qpower$, it follows that $\delta(\xi)$ is divisible by $(q-1)$. In particular, $\delta(\xi)$ vanishes in $(\L\!\qHodge_{R/A})_p^\complete/(p,q-1)$. By \cref{lem:DerivedqHodge}\cref{enum:DerivedqHodgeFirstFiltrationStepInjective} we see that $\delta(\xi)$ already vanishes in $R/p$. Since $\delta(\xi)$ is a unit by assumption, this forces $R/p\cong 0$, hence $R\cong 0$ by the derived Nakayama lemma \cite[\stackstag{09B9}]{Stacks}. This is the desired contradiction!
	\end{proof}
	The following technical lemma was used in the proof.
	\begin{lem}\label{lem:DerivedqHodge}
		With notation as above, the following are true:
		\begin{alphanumerate}
			\item $(\L\!\qHodge_{W/A})_p^\complete\simeq W\qpower$.\label{enum:DerivedqHodgeW}
			\item $(\L\!\qHodge_{R/A})_p^\complete$ is a static $A\qpower$-module and flat over $\IZ_p\qpower$.\label{enum:DerivedqHodgeR}
			\item The map $R/p\cong \qIW_1(R)/p\rightarrow (\L\!\qHodge_{R/A})_p^\complete/(p,q-1)$ induced by the $0$\textsuperscript{th} step in the $q$-de Rham--Witt filtration is injective.\label{enum:DerivedqHodgeFirstFiltrationStepInjective}
		\end{alphanumerate}
	\end{lem}
	\begin{proof}
		It's well-known that the $p$-completed cotangent complex $(\L_{\mathrm{A}_\inf(\Oo_C)/\IZ_p})_p^\complete$ vanishes. Hence the graded pieces of the $p$-completed $q$-de Rham--Witt filtration for $(\L\!\qHodge_{W/A}/^\L(q-1))_p^\complete$ are
		\begin{equation*}
			\gr_i^{\qIW\Omega}\bigl(\L\!\qHodge_{W/A}/^\L(q-1)\bigr)_p^\complete\simeq \bigl(\L\Omega_{\mathrm{A}_\inf(\Oo_C)/\IZ_p}^i\lotimes_\IZ A\bigr)_p^\complete[-i]\simeq 0
		\end{equation*}
		for $i>0$. It follows that $(\L\!\qHodge_{W/A})_p^\complete/^\L(q-1)\simeq \qIW_1(W/A)_p^\complete\simeq W$. In general, for any derived $(p,q-1)$-complete object $M\in\widehat{\Dd}_{(p,q-1)}(\IZ_p\qpower)$ we have $M\simeq \R\!\limit_{\alpha\geqslant 0}M/^\L(q^{p^\alpha}-1)$. Indeed, by the derived Nakayama lemma \cite[\stackstag{0G1U}]{Stacks} this can be checked after applying $(-)/^\L p$, and then $M/^\L(p,q^{p^\alpha}-1)\simeq M/^\L(p,(q-1)^{p^\alpha})$, so we recover the condition that $M/^\L p$ is derived $(q-1)$-complete. In particular, we obtain a map
		\begin{equation*}
			\R\!\limit_{\alpha\geqslant 0}\qIW_{p^\alpha}(W/A)_{(p,q-1)}^\complete\longrightarrow \R\!\limit_{\alpha\geqslant 0}\bigl(\L\!\qHodge_{W/A}\bigr)_p^\complete/^\L(q^{p^\alpha}-1)\,.
		\end{equation*}
		where the limit on the left-hand side is taken along the $q$-Witt vector Frobenii. Here's the only time we use our assumption that $\L\!\qHodge_{W/A}/^\L(q^{p^\alpha+1}-1)\rightarrow \L\!\qHodge_{W/A}/^\L(q^{p^\alpha}-1)$ induces the Frobenii on $q$-de Rham--Witt complexes. Using \cref{lem:RelativeqWittBaseChange} and \cref{cor:qWittOfPerfectLambdaRing}, we get
		\begin{equation*}
			\R\!\limit_{\alpha\geqslant 0}\qIW_{p^\alpha}(W/A)_{(p,q-1)}^\complete\simeq\Bigl(\R\!\limit_{\alpha\geqslant 0}W[q]/(q^{p^\alpha}-1)\Bigr)_{(p,q-1)}^\complete\simeq W\qpower
		\end{equation*}
		In summary, we've constructed a map $W\qpower\rightarrow (\L\!\qHodge_{W/A})_p^\complete$. By construction, after $(-)/^\L(q-1)$ this map is the identity on $W$, hence it is an isomorphism by the derived Nakayama lemma. This finishes the proof of \cref{enum:DerivedqHodgeW}.
		
		For~\cref{enum:DerivedqHodgeR} and~\cref{enum:DerivedqHodgeFirstFiltrationStepInjective}, we argue as above to see that the graded pieces of the $p$-completed $q$-de Rham--Witt filtration for $(\L\!\qHodge_{R/A}/^\L(q-1))_p^\complete$ are
		\begin{equation*}
			\gr_i^{\qIW\Omega}\bigl(\L\!\qHodge_{R/A}/^\L(q-1)\bigr)_p^\complete\simeq \bigl(\L\Omega_{\Oo_C/\IZ_p}^i\lotimes_\IZ A\bigr)_p^\complete[-i]\simeq 0
		\end{equation*}
		By a standard fact about perfectoid rings (see \cite[Proposition~\href{https://www.math.uni-bonn.de/people/scholze/bms2.pdf\#theorem.4.19}{4.19}]{BMS2} for example), we have $(\L\Omega_{\Oo_C/\IZ_p}^i)_p^\complete[-i]\simeq \Oo_C$ for all $i\geqslant 0$. Hence the graded pieces above are all equivalent to $R$. In particular, they are all static and $p$-torsion free. Inductively, this implies that all steps in the $p$-completed $q$-de Rham--Witt filtration must be static and $p$-torsion free. Furthermore, the transition maps must be injective. The same conclusion holds modulo $p$, which immediately shows~\cref{enum:DerivedqHodgeFirstFiltrationStepInjective}. For~\cref{enum:DerivedqHodgeR}, we conclude that $(\L\!\qHodge_{R/A})_p^\complete/^\L(q-1)$ is the derived $p$-completion of a static $p$-torsion free $\IZ_p$-module, hence it must be static and $p$-torsion free itself. In general, if $M\in \widehat{\Dd}_{(q-1)}(\IZ\qpower)$ is derived $(q-1)$-complete and $M/^\L(q-1)$ is static, then $M$ itself is static. Indeed, the map $\H^i(M)/(q-1)\rightarrow \H^i(M/^\L(q-1))$ is always injective; together with \cite[\stackstag{09B9}]{Stacks} this implies $\H^i(M)\cong0$ for $i\neq 0$, hence $M$ must indeed be static. It follows that $(\L\!\qHodge_{R/A})_p^\complete$ is static. Moreover, $p$-torsion freeness implies that $(\L\!\qHodge_{R/A})_p^\complete$ is $(q-1)$-completely flat in the sense of \cref{par:Notation}. Since $\IZ_p\qpower$ is noetherian, it follows that $(\L\!\qHodge_{R/A})_p^\complete$ is flat on the nose. This finishes the proof of~\cref{enum:DerivedqHodgeR}.
	\end{proof}

%	\appendix

	\newpage
	\renewcommand{\SectionPrefix}{}
	\renewcommand{\bibfont}{\small}
	\printbibliography
\end{document}